\documentclass[10pt]{article}
\usepackage{epsfig,mst-stylefile,amssymb,amsmath,harvard,dsfont}

\usepackage[svgnames]{xcolor}

\usepackage{xr}
\usepackage{caption}
\newcommand{\bbZ}{{\Bbb Z}}
\newcommand{\bbR}{{\Bbb R}}
\newcommand{\bbN}{{\Bbb N}}
\newcommand{\bbC}{{\Bbb C}}

\newcommand{\bbP}{{\Bbb P}}

\newcommand{\bbE}{{\Bbb E}}
\newcommand{\E}{{\Bbb E}}

\newcommand{\bbS}{{\Bbb S}}

\DeclareMathOperator*{\plim}{\mathit{p}-lim}

\newcommand{\imag}{{\mathbf i}}

\newcommand{\vecoper}{\textnormal{vec}}
\usepackage{xcolor}

\newcommand{\W}{\mathbf{W}}
\newcommand{\B}{\mathbf{B}}
\newcommand{\p}{\mathbf{p}}

\renewcommand{\cite}{\citeyear}

\graphicspath{{../figures_highdim_main/}}

\begin{document}

\title{On high-dimensional wavelet eigenanalysis
\thanks{H.W.~was partially supported by ANR-18-CE45-0007 MUTATION, France. G.D.'s long term visits to ENS de Lyon were supported by the school, the CNRS, the Carol Lavin Bernick faculty grant and the Simons Foundation collaboration grant $\#714014$. The authors would like to thank Alice Guionnet for her comments and suggestions in the initial stages of this work.}
}

\author{Patrice Abry \\  Laboratoire de Physique\\\ ENS de Lyon,  CNRS
\and   B.\ Cooper Boniece \\  Department of Mathematics \\ Drexel University
\and   Gustavo Didier \\ Mathematics Department\\ Tulane University
\and   Herwig Wendt\\ IRIT\\\ Universit\'{e} de Toulouse, CNRS}

\maketitle

\begin{abstract}
In this paper, we characterize the asymptotic and large scale behavior of the eigenvalues of wavelet random matrices in high dimensions. We assume that possibly non-Gaussian, finite-variance $p$-variate measurements are made of a low-dimensional $r$-variate ($r \ll p$) fractional stochastic process with non-canonical scaling coordinates and in the presence of additive high-dimensional noise. The measurements are correlated both time-wise and between rows. We show that the $r$ largest eigenvalues of the wavelet random matrices, when appropriately rescaled, converge in probability to scale-invariant functions in the high-dimensional limit. By contrast, the remaining $p-r$ eigenvalues remain bounded in probability.  Under additional assumptions, we show that the $r$ largest log-eigenvalues of wavelet random matrices exhibit asymptotically Gaussian distributions. The results have direct consequences for statistical inference.
\end{abstract}

\section{Introduction}

A \textit{wavelet} is an oscillatory function in $L^2(\bbR)$ with unit norm (see \eqref{e:N_psi}). For a fixed (octave) $j \in \bbN \cup \{0\}$, the \textit{wavelet transform} of a $p$-variate stochastic process $Y = \{Y(t)\}_{t \in \bbZ}$ at the dyadic scale $a2^j$ and shift $k \in \bbZ$ is defined by the entry-wise convolution
\begin{equation}\label{e:D(a(n)2^j,k)_intro}
\bbR^p \ni D(a2^j,k) = \sum_{\ell\in\bbZ} Y(\ell)\hspace{0.5mm}\widetilde{h}_{a2^j k-\ell}.
\end{equation}
In \eqref{e:D(a(n)2^j,k)_intro}, the terms $\widetilde{h}_{\cdot}$ are real-valued coefficients that depend on the scale and on the underlying wavelet function. The entries of $\{D(a2^j,k)\}_{k \in \bbZ}$ are generally correlated. A \textit{fractal} is an object or phenomenon that displays the property of \textit{self-similarity}, in some sense, across a range of scales (Mandelbrot \cite{mandelbrot:1982}). Due to its intrinsic multiscale character and fine-tuned mathematical properties, the wavelet transform \eqref{e:D(a(n)2^j,k)_intro} has been widely used in the study and characterization of fractals (e.g., Wornell \cite{wornell:1996}, Doukhan et al.\ \cite{doukhan:2003}, Massopust \cite{massopust:2014}).

For any $p \in \bbN$, a $p\times p$ \textit{wavelet random matrix} is given by
\begin{equation}\label{e:W(a2^j)_intro}
 {\mathbf W}(a2^j) = \frac{1}{n_{a,j}}\sum^{n_{a,j}}_{k=1}D(a2^j,k)\hspace{0.5mm} D(a2^j,k)^*.
\end{equation}
In \eqref{e:W(a2^j)_intro}, $^*$ denotes transposition, and $n_{a,j}$ is the (dyadic) number of wavelet-domain observations (we defer to Section \ref{s:wavelet_analysis} for a precise description of $n_{a,j}$). The so-named \textit{wavelet eigenanalysis} methodology consists in using the behavior across scales of the eigenvalues of wavelet random matrices to study the fractality of stochastic systems (Abry and Didier \cite{abry:didier:2018:n-variate,abry:didier:2018:dim2}). In this paper, we characterize the asymptotic and large-scale behavior of the eigenvalues of wavelet random matrices in high dimensions. Observations of a  (possibly non-Gaussian) underlying stochastic process $Y$ in \eqref{e:D(a(n)2^j,k)_intro} are assumed to have the form
\begin{equation}\label{e:Y(t)}
Y(t) = {\mathbf P} X(t) + Z(t). %
\end{equation}
In \eqref{e:Y(t)}, both $Y$ and the noise term $Z = \{Z(t)\}_{t\in\bbZ}$ are (high-dimensional) $p$-variate processes, ${\mathbf P}$ is a rectangular, deterministic coordinates matrix and, for fixed $r \in \bbN$, $X = \{X(t)\}_{t \in\bbZ}$ is a (low-dimensional) $r$-variate fractional process. One can assume $X$ and $Z$ are second order, uncorrelated and zero-mean stochastic processes. In particular, the measurements $Y$ are \textit{correlated} both time-wise and between rows. Assuming the observations $Y(1),Y(2),\ldots, Y(n)$ of \eqref{e:Y(t)} are available, we show that, as $n\to\infty$, if $a=a(n)\to\infty$ and $p(n)\hspace{0.5mm}a(n)/n=O(1)$ (with $r$ fixed and possibly $p(n)\to\infty$) then, under a suitable normalization based on \textit{scaling exponents}, the $r$ largest eigenvalues of ${\mathbf W}(a(n)2^j)$ converge in probability to scale-invariant functions. By contrast, the remaining $p(n)-r$ eigenvalues remain bounded in probability. In addition, we show that the $r$ largest log-eigenvalues of ${\mathbf W}(a(n)2^j)$ exhibit asymptotically Gaussian distributions. The results bear direct consequences for statistical inference starting from high-dimensional measurements of the form of a signal-plus-noise system \eqref{e:Y(t)}, where $X$ is a latent process containing fractal (scaling) information and ${\mathbf P}$ (as well as $Z$) is unknown.\\

In this paper, we combine two mathematical frameworks that are rarely considered jointly: $(i)$ high-dimensional probability theory; $(ii)$ fractal analysis. This is done by bringing together the study of large random matrices and scaling analysis in the wavelet domain.

Since the 1950s, the spectral behavior of large-dimensional random matrices has attracted considerable attention from the mathematical research community. In quantum mechanics, for example, random matrices are of great interest as statistical mechanical models of infinite-dimensional and possibly unknown Hamiltonian operators (e.g., Mehta and Gaudin \cite{mehta:gaudin:1960}, Dyson \cite{dyson:1962}, Ben Arous and Guionnet \cite{benarous:guionnet:1997}, Soshnikov \cite{soshnikov:1999}, Mehta \cite{mehta:2004}, Deift \cite{deift:2007}, Anderson et al.\ \cite{anderson:guionnet:zeitouni:2010}, Tao and Vu \cite{tao:vu:2011}, Erd\H{o}s et al.\ \cite{erdos:yau:yin:2012}). Random matrices have also naturally emerged as one essential mathematical framework for the modern era of ``Big Data" (Briody \cite{briody:2011}). When hundreds to several tens of thousands of time series get recorded and stored on a daily basis, one is often interested in understanding the behavior of random constructs such as the spectral distribution of sample covariance matrices for which the dimension $p$ is comparable to the sample size $n$ (e.g., Tao and Vu \cite{tao:vu:2012}, Xia et al.\ \cite{xia:qin:bai:2013}, Paul and Aue \cite{paul:aue:2014}, Yao et al.\ \cite{yao:zheng:bai:2015}). The literature on random matrices under dependence as well as on high-dimensional stochastic processes has been expanding at a fast pace (e.g., Basu and Michailidis \cite{basu:michailidis:2015}, Chakrabarty et al.\ \cite{chakrabarty:hazra:sarkat:2016}, Merlev\`{e}de and Peligrad \cite{merlevede:peligrad:2016}, Che \cite{che:2017}, Steland and von Sachs \cite{steland:vonsachs:2017}, Taylor and Salhi \cite{taylor:salhi:2017}, Wang et al.\ \cite{wang:aue:paul:2017}, Zhang and Wu \cite{zhang:wu:2017}, Erd\H{o}s et al.\ \cite{erdos:kruger:schroder:2019}, Merlev\`{e}de et al.\ \cite{merlevede:najim:tian:2019}, Bourguin et al.~\cite{bourguin:diez:tudor:2021}, Shen et al.\ \cite{shen:stoev:hsing:2022}).

In turn, recall that the emergence of a fractal is typically the signature of a physical mechanism that generates \textit{scale invariance} (e.g., Peitgen et al.\ \cite{peitgen:jurgens:saupe:feigenbaum:2004}, West et al.\ \cite{west:brown:enquist:1999}, Zheng et al.\ \cite{zheng:shen:wang:li:dunphy:hasan:brinker:su:2017}, He \cite{he:2018}). Unlike traditional statistical mechanical systems (e.g., Reif \cite{reif:2009}), a scale-invariant system does not display a characteristic scale, namely, one that dominates its statistical behavior. Instead, the behavior of the system across scales is determined by specific parameters called \textit{scaling exponents}.  Scale invariance manifests itself in a wide range of natural and social phenomena such as in criticality (Sornette \cite{sornette:2006}), turbulence (Kolmogorov~\cite{Kolmogorovturbulence}), climate studies (Isotta et al.\ \cite{isotta:etal:2014}), dendrochronology (Bai and Taqqu~\cite{bai:taqqu:2018}) and hydrology (Benson et al.\ \cite{benson:baeumer:scheffler:2006}). Mathematically, it is a topic of central importance in Markovian settings (e.g., diffusion, lattice models, universality classes) as well as in non-Markovian ones (e.g., anomalous diffusion, long-range dependence, non-central limit theorems).

In the univariate context $p = 1$, wavelets have proven to be powerful tools for the multiscale analysis of broad classes of stochastic processes. Among many reasons, this is so because, in applications, the computational complexity of \eqref{e:D(a(n)2^j,k)_intro} is very low, sometimes even surpassing that of the fast Fourier transform (e.g., Daubechies \cite{daubechies:1992}, Mallat \cite{mallat:2009}). On the other hand, in theoretical research, the multiscale sequence \eqref{e:D(a(n)2^j,k)_intro} often displays improved stochastic properties by comparison to the original measurements. In particular, wavelets provide a natural analytical arena for non-stationary or fractional processes, due to the usual stationarity and rapidly decaying correlation structure of \eqref{e:D(a(n)2^j,k)_intro} for fixed $j$ (e.g., Meyer et al.\ \cite{meyer:sellan:taqqu:1999}, Moulines \cite{moulines:roueff:taqqu:2007:JTSA}). For $p = 1$, there is now a vast literature on the use of \eqref{e:W(a2^j)_intro} in the characterization of the scaling behavior -- as parametrized by scaling exponents -- of univariate fractional processes (e.g., Flandrin \cite{flandrin:1992}, Wornell and Oppenheim \cite{wornell:oppenheim:1992}, Clausel et al.\ \cite{clausel:roueff:taqqu:tudor:2014:waveletestimation}; see also the initial discussion in Section \ref{s:consequences_for_statistics} of this paper).

In a multidimensional framework, scaling behavior does not always appear along standard coordinate axes, and often involves multiple scaling relations. A $\bbR^r$-valued stochastic process $X$ is called \textit{operator self-similar} (o.s.s.; Laha and Rohatgi \cite{laha:rohatgi:1981}, Hudson and Mason \cite{hudson:mason:1982}) if it exhibits the scaling property
\begin{equation}\label{e:def_ss}
\{X(ct)\}_{t\in\bbR} \stackrel {\textnormal{f.d.d.}}{=} \{c^{\mathbf H} X(t)\}_{t\in\bbR}, \quad c>0.
\end{equation}
In \eqref{e:def_ss}, ${\mathbf H}$ is some (Hurst) matrix whose eigenvalues have real parts lying in the interval $(0,1]$ and $c^{\mathbf H} := \exp\{\log(c) {\mathbf H}\} = \sum^{\infty}_{k=0} \frac{(\log(c) {\mathbf H})^k}{k!}$. A canonical model for multivariate fractional systems is operator fractional Brownian motion (ofBm), namely, a Gaussian, o.s.s., stationary-increment stochastic process (Maejima and Mason \cite{maejima:mason:1994}, Mason and Xiao \cite{mason:xiao:2002}, Didier and Pipiras \cite{didier:pipiras:2012}). In particular, ofBm is the natural multivariate generalization of the classical fractional Brownian motion (fBm; Embrechts and Maejima \cite{embrechts:maejima:2002}).

The importance of the role of multiple scaling laws in applications is now well established. For example, in econometrics, the detection of
distinct scaling laws in multivariate fractional time series is indicative of the key property of (fractional) \textit{cointegration} -- namely,
the existence of meaningful and statistically useful long-run relationships among the individual series (e.g.,
Engle and Granger \cite{engle:granger:1987}, NobelPrize.org \cite{nobelprize:2003}, Hualde and Robinson \cite{hualde:robinson:2010}, Shimotsu \cite{shimotsu:2012}). From a different perspective, it has been shown that ignoring the presence of multiple scaling laws in statistical inference may lead to severe biases (see Section \ref{s:consequences_for_statistics}).

The model \eqref{e:Y(t)} provides a natural formulation of a fractal, or scaling system, in high dimensions. Besides being very general -- in particular, the measurements are possibly non-Gaussian --, it subsumes the fundamental idea behind the modeling of high-dimensional stochastic systems. In other words, a low-dimensional component, containing all the relevant physical information, is embedded in high-dimensional noise (e.g., Giraud \cite{giraud:2015}, Wainwright \cite{wainwright:2019}). In fact, \eqref{e:Y(t)} and closely related models appear in numerous applications such as, for example, in neuroscience and fMRI imaging (Ciuciu et al.\ \cite{ciuciu:varoquaux:abry:sadaghiani:kleinschmidt:2012}, Liu et al.\ \cite{liu:aue:paul:2015}, Ting et al.\ \cite{ting:ombao:samdin:salleh:2017}, Li et al.\ \cite{li:pluta:shahbaba:fortin:ombao:baldi:2019}, Gotts et al.\ \cite{gotts:gilmore:martin:2020}; cf.\ Chauduri et al.\ \cite{chaudhuri:gercek:pandey:peyrache:fiete:2019}, Stringer et al.\ \cite{stringer:pachitariu:steinmetz:carandini:harris:2019}), in factor modeling (Bai \cite{bai:2003}, Cheung \cite{cheung:2022}, Ergemen and Rodr\'{i}guez-Caballero \cite{ergemen:rodriguez-caballero:2023}) and in econometrics (Brown \cite{brown:1989}, Stock and Watson \cite{stock:watson:2011}, Lam and Yao \cite{lam:yao:2012}, Chan et al.\ \cite{chan:lu:yau:2017}, to name a few).

In the characterization of scaling properties, the use of eigenanalysis was first proposed in Meerschaert and Scheffler \cite{meerschaert:scheffler:1999,meerschaert:scheffler:2003} and Becker-Kern and Pap \cite{becker-kern:pap:2008}. It has also been applied in the cointegration literature (e.g., Phillips and Ouliaris \cite{phillips:ouliaris:1988}, Li et al.\ \cite{li:pan:yao:2009}, Zhang et al.\ \cite{zhang:robinson:yao:2018}). In Abry and Didier \cite{abry:didier:2018:n-variate,abry:didier:2018:dim2}, \textit{wavelet eigenanalysis} is put forward in the construction of a general methodology for the statistical identification of the scaling (Hurst) structure of ofBm in low dimensions.

In Abry et al.\ \cite{abry:wendt:didier:2018:detecting_highdim} and Boniece et al.\ \cite{boniece:wendt:didier:abry:2019}, presented without proofs, wavelet random matrices were first used in the modeling of high-dimensional systems. In this paper, we construct the mathematical foundations of wavelet eigenanalysis in high dimensions by investigating the properties of the eigenvalues $\lambda_{1}\big({\mathbf W}(a(n)2^j)\big) \leq \hdots \leq \lambda_{p(n)}\big({\mathbf W}(a(n)2^j)\big)$ of large wavelet random matrices ${\mathbf W}(a(n)2^j)$. We assume measurements given by \eqref{e:Y(t)}, where the fractional behavior of $X$ is characterized by a scaling matrix of the Jordan form
\begin{equation}\label{e:H=PHdiag(h1,...,hn)P^(-1)H}
{\mathbf H} = {\mathbf P}_H \hspace{0.5mm}\textnormal{diag}(h_1,\hdots,h_r)\hspace{0.5mm}{\mathbf P}^{-1}_{H}, \quad h_1 \leq \hdots \leq h_r.
\end{equation}
The term $Z$ can be generally thought of as high-dimensional colored noise, displaying a weak dependence structure \textit{by comparison} to $X$. The measurements $Y$ display correlation time-wise and between rows. We consider the three-way limit as the sample size ($n$),  scale ($a(n)$) and, possibly, the dimension ($p(n)$) go to infinity  simultaneously ($n,p(n),a(n)\to\infty$) and satisfying the condition%
\begin{equation}\label{e:three-fold_lim}
\frac{p(n) \hspace{0.5mm}a(n)}{n}=O(1).
\end{equation}
It is by considering the three-way limit \eqref{e:three-fold_lim}, which includes a \textit{scaling limit}, that large (wavelet) random matrices may be used in the characterization of low-frequency behavior in a high-dimensional framework. In fact, in this paper we show that the $r$ largest eigenvalues of wavelet random matrices display fractal -- or scaling -- properties determined by \eqref{e:H=PHdiag(h1,...,hn)P^(-1)H} as well as, under additional assumptions, asymptotically Gaussian fluctuations. In particular, such eigenvalues are explosive. By contrast, the remaining $p(n)-r$ eigenvalues do not exhibit fractality and remain bounded.

To be more precise, under very general assumptions, we establish that, for positive functions $\xi_q(\cdot)$,
\begin{equation}\label{e:eigenvalue_scaling_intro}
\lambda_{p(n)-r+q}\Big(\frac{{\mathbf W}(a(n)2^j)}{a(n)^{2h_{q}+1}} \Big) \stackrel{\bbP}\rightarrow \xi_{q}(2^j), \quad q = 1,\hdots,r,
\end{equation}
whereas $\lambda_{\ell}\big({\mathbf W}(a(n)2^j)\big)$, $\ell = 1,\hdots,p(n)-r$, are bounded in probability (see Theorem \ref{t:lim_n_a_times_lambda/a^(2h+1)}; see also Figure \ref{fig:logeig} for an illustration). Moreover, under slightly stronger conditions, we show that the random vector
\begin{equation}\label{e:log-eigenvalue_normality_intro}
\sqrt{\frac{n}{a(n)2^j}}\Big(\hspace{1mm}\log \lambda_{p(n)-r+q}\Big(\frac{{\mathbf W}(a(n)2^j)}{a(n)^{2h_{q}+1}} \Big) - \log \lambda_{p(n)-r+q}\Big(\frac{\bbE{\mathbf W}(a(n)2^j)}{a(n)^{2h_{q}+1}} \Big) \hspace{1mm}\Big), \quad q = 1,\hdots, r,
\end{equation}
is asymptotically Gaussian (see Theorem \ref{t:asympt_normality_lambdap-r+q}; see also Figure \ref{fig:logeig_fluctuations} for an illustration). In particular, the convergence rate in \eqref{e:log-eigenvalue_normality_intro} also involves the scaling limit. Note that this stands in sharp contrast with traditional high-dimensional analysis of sample covariance matrices, in which one considers the ratio $\lim_{n \rightarrow \infty} p(n)/n$ and the largest eigenvalue often exhibits universality in the form of Tracy--Widom fluctuations (e.g., Bai and Silverstein \cite{bai:silverstein:2010}, Lee and Schnelli \cite{lee:schnelli:2016}; on a comparison of \eqref{e:log-eigenvalue_normality_intro} with the potentially Gaussian fluctuations of the largest eigenvalues in spiked covariance models, see Remark \ref{r:asympt_rescaled_eigenvalues}, $(vi)$).

From the standpoint of probability theory, to the best of our knowledge this paper provides the first mathematical study of the high-dimensional properties of wavelet random matrices. It is also the first time, again to the best of our knowledge, that the role of scaling -- or low-frequency behavior -- is given special attention in the context of large random matrices, i.e., in the form of the three-way limit \eqref{e:three-fold_lim}.

From the standpoint of fractal analysis, this paper takes a decisive step in the expansion, to the high-dimensional context, of the study of scale-invariant and non-Markovian phenomena started by Kolmogorov \cite{kolmogorov:1940} and Mandelbrot and Van Ness \cite{mandelbrot:vanness:1968}, and later taken up by the likes of Flandrin \cite{flandrin:1992}, Wornell and Oppenheim \cite{wornell:oppenheim:1992}, Meyer et al.\ \cite{meyer:sellan:taqqu:1999}, among many others (see Pipiras and Taqqu \cite{pipiras:taqqu:2017}).

The expressions for the top $r$ wavelet eigenvalues involve discrepant scaling rates, leading to the presence of potentially explosive terms. For this reason, establishing \eqref{e:eigenvalue_scaling_intro} requires constructing a squeeze-type argument based on lower and upper bounds where such terms have been replaced by finite and convergent sequences. In turn, proving \eqref{e:log-eigenvalue_normality_intro} involves handling Taylor expansions of wavelet log-eigenvalues both in the high-dimensional limit and in the presence of potentially explosive terms. High-level discussions of the main technical issues involved in showing \eqref{e:eigenvalue_scaling_intro} and \eqref{e:log-eigenvalue_normality_intro} are provided at the beginning of Sections \ref{s:proof_of_Prop_conv_subseq_of_rescaled_eigenvalues} and \ref{s:proving_theorem_3.2}, respectively. The proofs of both Theorems \ref{t:lim_n_a_times_lambda/a^(2h+1)} and \ref{t:asympt_normality_lambdap-r+q} are original and involve nontrivial extensions and enhancements of the techniques first developed in Abry and Didier \cite{abry:didier:2018:n-variate,abry:didier:2018:dim2} for handling eigenvalues of fixed-dimensional wavelet random matrices.

For the sake of clarity and mathematical generality, our assumptions are stated directly in the wavelet domain, namely, in terms of properties of wavelet random matrices (see Section \ref{s:framework}). Our results have direct consequences for the empirical identification and description of fractality in high-dimensional systems, as briefly discussed in Section \ref{s:consequences_for_statistics} (see also Abry et al.\ \cite{abry:boniece:didier:wendt:2023:regression} on a multiscale regression-type methodology based on the theory of wavelet random matrices constructed in this paper). In Section \ref{s:examples}, we further provide encompassing classes of examples covered by the assumptions used in Section \ref{s:framework}. This includes the cases where $X$ is an ofBm, and also where, for each $p$, the $p$-variate noise term $Z$ is a classical, ARMA-type Gaussian linear process. We illustrate the flexibility of the framework provided by Theorems \ref{t:lim_n_a_times_lambda/a^(2h+1)} and \ref{t:asympt_normality_lambdap-r+q} by applying them to a class of (Gaussian) factor models. We also discuss some simple finite-variance and non-Gaussian instances $X$ of interest, hence illustrating the broad scope of the assumptions (see Section \ref{s:examples}). Detailed proofs for Section \ref{s:examples} can be found in the extended version of this paper (Abry et al.\ \cite{abry:boniece:didier:wendt:2023:extended}), which is self-contained and available online. See also Remark \ref{r:asympt_rescaled_eigenvalues} on the use of assumptions in Theorems \ref{t:lim_n_a_times_lambda/a^(2h+1)} and \ref{t:asympt_normality_lambdap-r+q}.

This paper is organized as follows. In Section \ref{s:framework}, we provide the basic wavelet framework, definitions and wavelet-domain assumptions used throughout the paper. In Section \ref{s:main}, we state and discuss the main results on the asymptotic and large-scale behavior of wavelet eigenvalues in high dimensions. In Section \ref{s:examples}, we provide Gaussian and non-Gaussian examples.  In Section \ref{s:proof_main_results}, we prove the main results, stated in Section \ref{s:main}. In Section \ref{s:conclusion}, we lay out conclusions and discuss several open problems that this work leads to. This includes new aspects of the theory of wavelet random matrices, as well as consequences for statistical inference and modeling. The appendix contains the statements and proofs of auxiliary results.

\begin{center}
\begin{figure}[ht!]
\centering
\begin{minipage}{.45\linewidth}
\includegraphics[width=\linewidth]{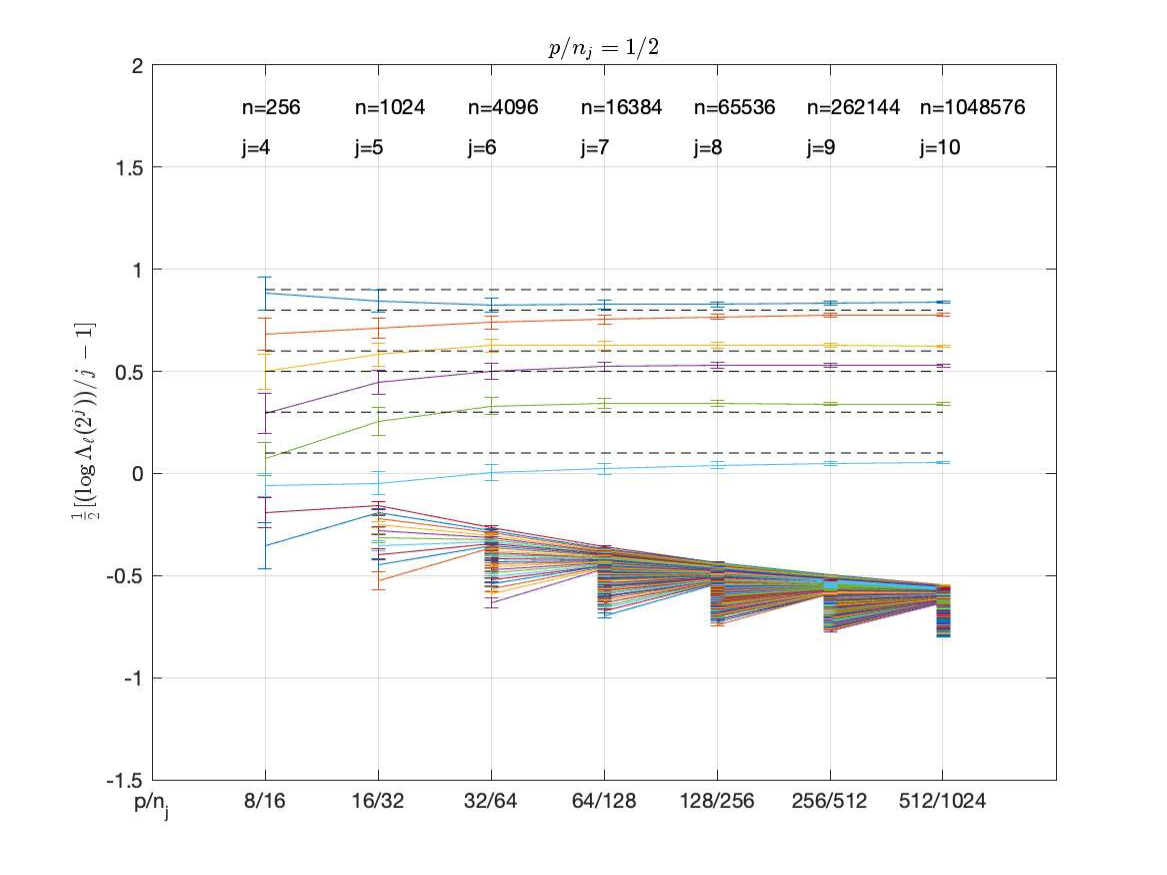}
\end{minipage}
\begin{minipage}{.45\linewidth}
\includegraphics[width=\linewidth]{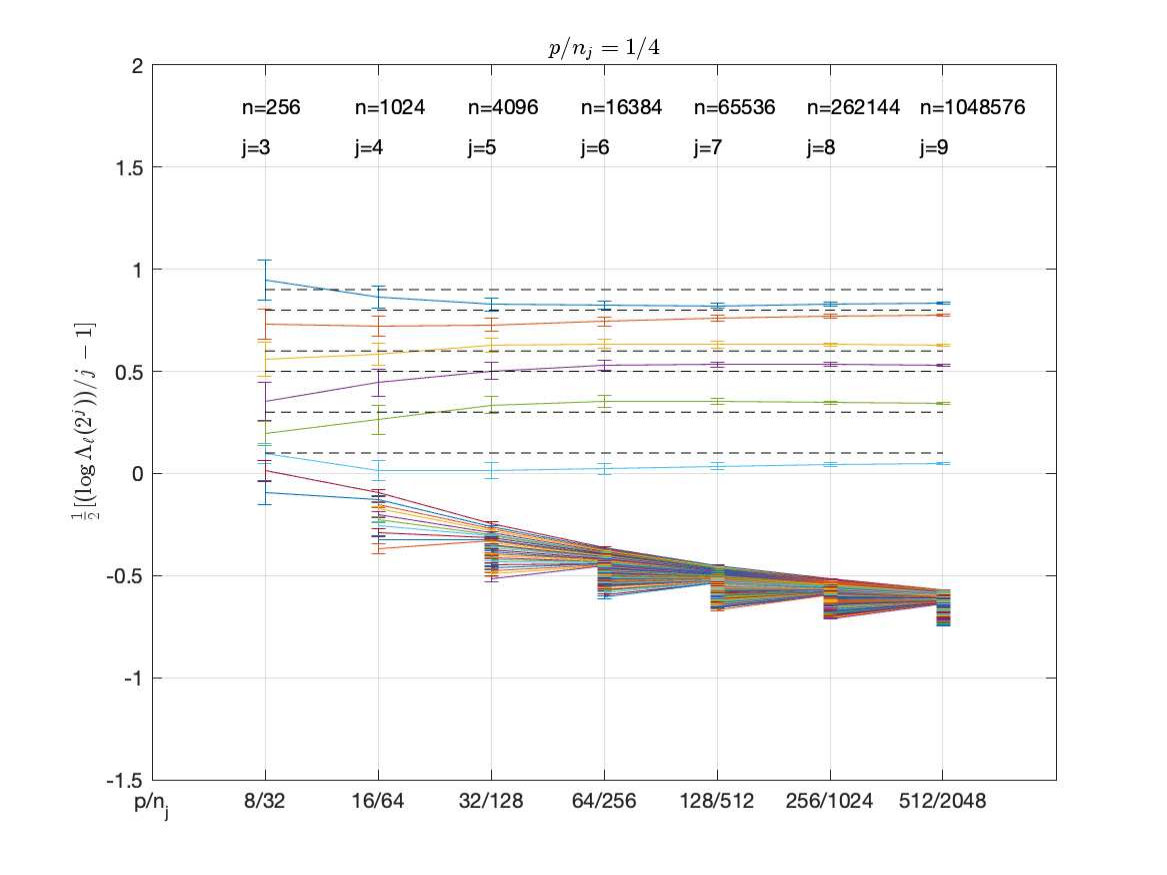}
\end{minipage}
\begin{minipage}{.45\linewidth}
\includegraphics[width=\linewidth]{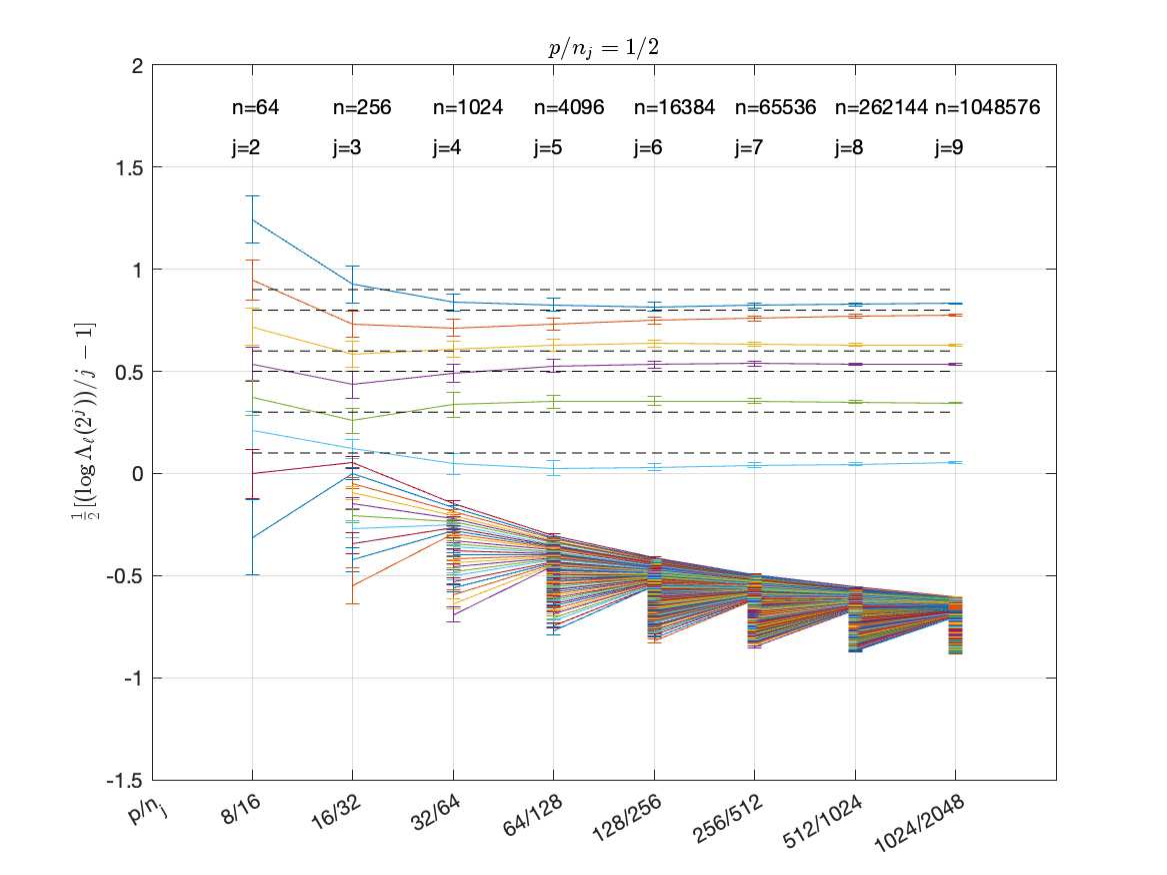}
\end{minipage}
\begin{minipage}{.45\linewidth}
\includegraphics[width=\linewidth]{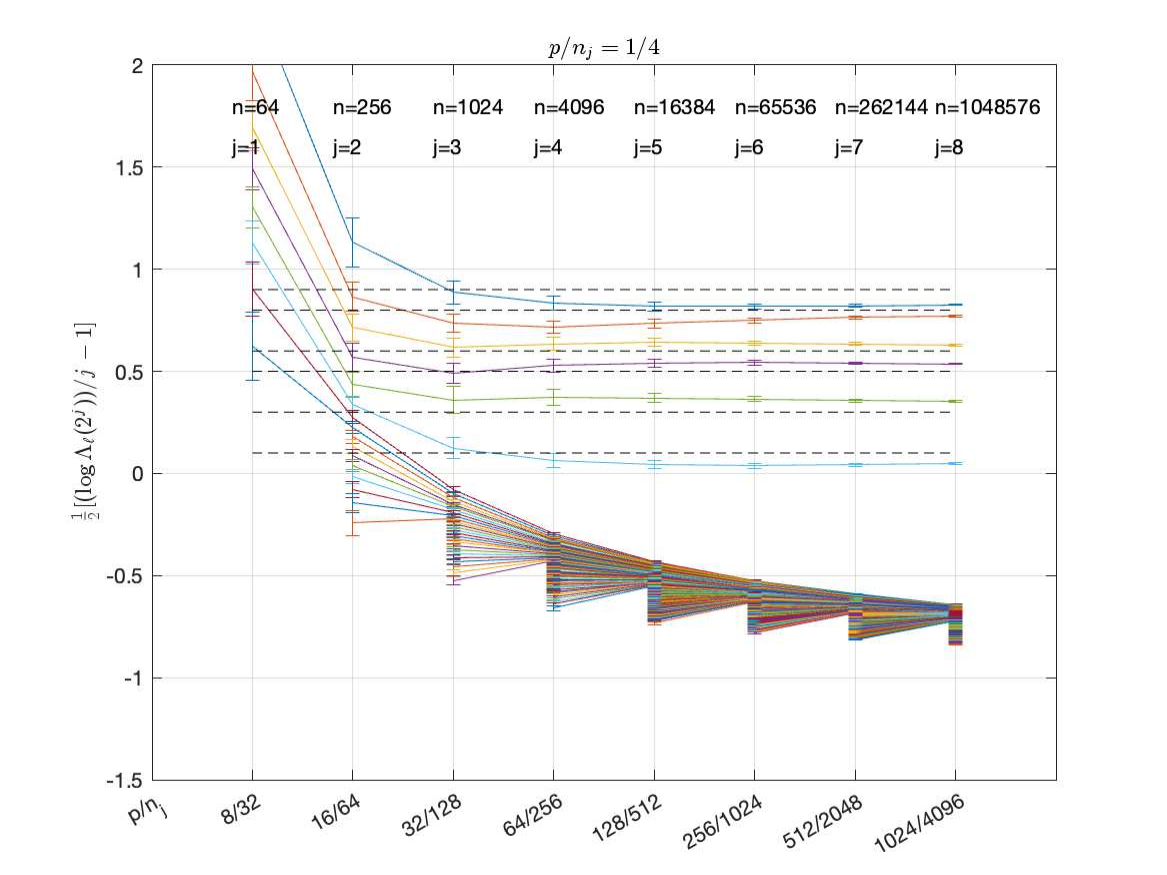}
\end{minipage}
\caption[The high-dimensional asymptotic behavior of log-eigenvalues]{\label{fig:logeig}\textbf{The convergence of the rescaled wavelet log-eigenvalues in the three-way limit $\frac{p\hspace{0.2mm}2^j}{n} \rightarrow c \in [0,\infty)$.} In this simulation exercise, $X$ is an ofBm and $Z$ is a vector of Gaussian white noise processes (see Section \ref{s:examples} for a discussion of these models). $X$ and $Z$ are generated independently. For each $n$, the coordinates matrix ${\mathbf P} {\mathbf P}_H$ is randomly drawn based on i.i.d.\ standard Gaussian entries, and then normalized to have unit-norm columns. For notational simplicity, we reexpressed the scaling factor as $a(n) = 2^j$, where $j = j_n \rightarrow \infty$. For $r=6$ and $h_q \in \{0.1,0.3, 0.5,0.6,0.8,0.9\}$, $q = 1,\hdots,r$, the plots display the asymptotic behavior of $\frac{1}{2}[(\log \Lambda_{\ell}(2^j))/j-1]$, where $\Lambda_{\ell}(2^j):=\lambda_{\ell}(\mathbf{W}(2^j))$, $\ell = 1,\hdots,p$. The six dashed lines correspond to the values $\{0.1,0.3, 0.5,0.6,0.8,0.9\}$. In all plots, $p$, $n$ and $j$ increase while their ratio remains fixed at $p \hspace{0.2mm} 2^j/n =: p/n_j= c = 1/2$ (left column) and $c=1/4$ (right column). In light of Theorem \ref{t:lim_n_a_times_lambda/a^(2h+1)}, for $j = j_n \rightarrow \infty$, $(1/j)\log \Lambda_{p-r+q}(2^j) \stackrel{\bbP}\rightarrow 2h_q+1$, $q=1,\hdots,6$, and $(1/j)\log \Lambda_{p-r}(2^j) \stackrel{\bbP}\rightarrow 0$ in the three-way limit. As expected, the $r=6$ largest $\frac{1}{2}[(\log \Lambda_{\ell}(2^j))/j-1]$ approach $h_1,\hdots,h_6$ in the plots as $n$, $p$ and $j$ grow. By contrast, the remaining $\frac{1}{2}[(\log \Lambda_{\ell}(2^j))/j-1]$ tend toward zero as $j$ increases in all instances. For a fixed $c$ (column), going from the top to the bottom row, the magnitudes of $p$ and $2^j$ are larger and \textnormal{smaller}, respectively, by a factor of 2 for each value of $n$. As a result, we observe near-convergence at \textnormal{smaller} octaves $j$ (\textbf{n.b.}: axes have been shifted to align the curves). Similarly, going from the left to the right column (i.e., as $c$ decreases), the near-convergence also occurs at \textnormal{smaller} $j$, which is reflective of a less extreme high-dimensional regime $c$.}
\end{figure}
\end{center}

\begin{center}
\begin{figure}[ht!]
\centering
\setlength{\tabcolsep}{2pt}\footnotesize
\noindent\begin{tabular}{ccccc}
\includegraphics[height=0.22\linewidth, trim=0 0 15 0, clip]{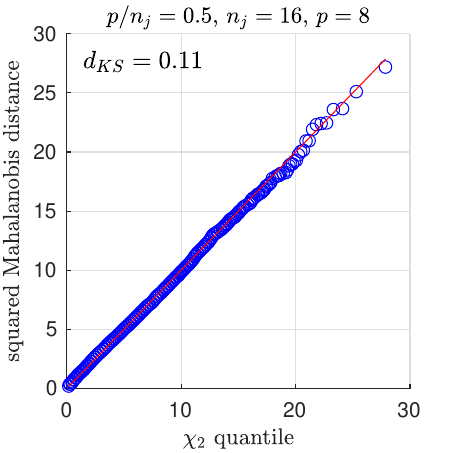}
&\includegraphics[height=0.22\linewidth, trim=15 0 15 0, clip]{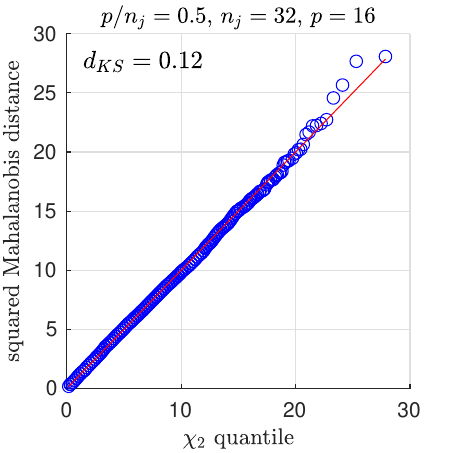}\vspace{1ex}&
&\includegraphics[height=0.22\linewidth, trim=15 0 15 0, clip]{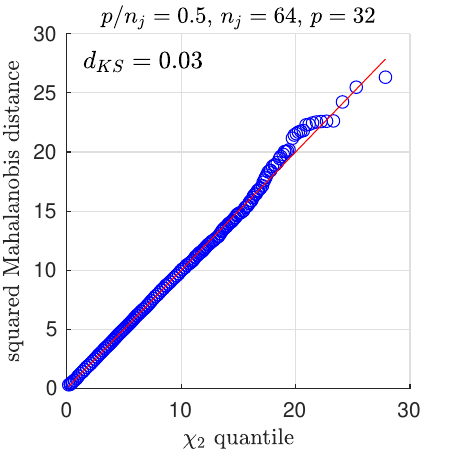}&
\includegraphics[height=0.22\linewidth, trim=15 0 15 0, clip]{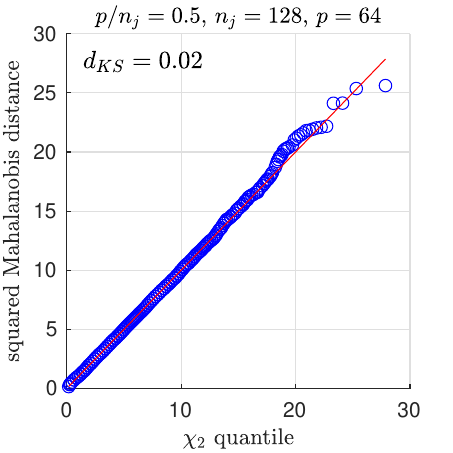}
\end{tabular}\vspace{-1ex}
\caption[Fluctuations]{\label{fig:logeig_fluctuations}\textbf{The fluctuations of wavelet log-eigenvalues in the three-way limit $\frac{p\hspace{0.2mm}2^j}{n} \rightarrow c=1/2$}. In the same simulation framework as for Figure \ref{fig:logeig}, Theorem \ref{t:asympt_normality_lambdap-r+q} predicts that the joint distribution of $\{\log \Lambda_{p-r+q}(2^j))\}_{q=1,\hdots,6}$ is asymptotically Gaussian in the three-way limit, after centering and rescaling. In fact, this convergence can be seen in the so-called Gamma plots displayed above, which are expected to look close to a straight line under joint Gaussianity. The plots show the empirical quantiles of the squared Mahalanobis distance statistic vs.\ the theoretical quantiles of a $\chi^2_6$ distribution based on 5000 realizations (e.g., Johnson and Wichern \cite{johnson:wichern:2002}). The (effective) sample size $n_j=n/2^j$ increases from left-to-right and $p/n_j$ is set equal to $1/2.$ The plots also display Kolmogorov-Smirnov distance statistics $d_{KS}$, which tend to shrink for larger values of $n_j$.}
\end{figure}
\end{center}

\section{Framework}\label{s:framework}

\subsection{Notation} For $m \in \bbN$, let ${\mathcal M}(m,\bbR)$ and ${\mathcal M}(m,\bbC)$ be the spaces of $m \times m$ real- and complex-valued matrices, respectively. Also, let ${\mathcal M}(m_1,m_2,\bbR)$ be the space of $m_1 \times m_2$ real-valued matrices. Let ${\mathcal S}(m,\bbR)$ and ${\mathcal S}(m,\bbC)$ be the spaces of $m \times m$ symmetric and Hermitian symmetric matrices, respectively. We use the notation ${\mathcal S}_{\geq 0}(m,\bbR)$ and ${\mathcal S}_{> 0}(m,\bbR)$ to denote the sets of symmetric positive semidefinite and symmetric positive definite matrices, respectively. The groups of real- or complex-valued invertible matrices are denoted by $GL(m,\bbR)$ and $GL(m,\bbC)$, respectively. The symbol $\mathbf I_m$ denotes the identity matrix in ${\mathcal M}(m,\bbR)$. For convenience, we may write $\mathbf I$ when the dimension is unambiguous. The notation $\bbS^{m-1}=\{\mathbf u\in \bbR^m: \mathbf u^*\mathbf u =1\}$ represents the $m-1$ dimensional sphere. Throughout the manuscript, $\|{\mathbf M}\|$ denotes the spectral norm of a matrix ${\mathbf M} \in {\mathcal M}(p,\bbR)$ in arbitrary dimension $p$, i.e.,
$\|{\mathbf M}\| = \sqrt{\sup_{{\mathbf u} \in \bbS^{p-1}} {\mathbf u}^* {\mathbf M} {\mathbf M}^* {\mathbf u}} = \sqrt{\sup_{{\mathbf u} \in \bbS^{p-1}} {\mathbf u}^* {\mathbf M}^* {\mathbf M} {\mathbf u}}$. Also, the norm $\|{\mathbf M}\|$ is analogously defined when ${\mathbf M}$ is rectangular. For any ${\mathbf M} \in {\mathcal S}(p,\bbR)$,
\begin{equation}\label{e:lambda1(M)=<...=<lambdap(M)}
- \infty < \lambda_1({\mathbf M}) \leq \hdots \leq \lambda_q({\mathbf M}) \leq \hdots \leq \lambda_p({\mathbf M}) < \infty
\end{equation}
denotes the set of ordered eigenvalues of the matrix ${\mathbf M}$. For ${\mathbf M} \in {\mathcal M}(r_1,r_2,\mathbb{R})$ and for any $i_1 \in \{1,\hdots,r_1\}$ and $i_2 \in \{1,\hdots,r_2\}$,
\begin{equation}\label{e:pi_i1,i2}
\pi_{i_1,i_2}({\mathbf M}) = m_{i_1,i_2}
\end{equation}
denotes entry $(i_1,i_2)$ of the matrix ${\mathbf M}$. Also,
\begin{equation}\label{e:vec_non-symm_def}
\textnormal{vec}({\mathbf M})=(m_{11},m_{21},\dots,m_{r_1 1},m_{12},m_{22},\dots,m_{r_1 2},\dots,m_{r_1 r_2}).
\end{equation}
For ${\mathbf S} = (s_{i_1,i_2})_{i_1,i_2=1,\dots,p}\in {\mathcal S}(p,\mathbb{R})$, we define the operator
\begin{equation}\label{e:vec_definitions}
\textnormal{vec}_{{\mathcal S}}({\mathbf S})=(s_{11},s_{21},\dots,s_{p1},s_{22},s_{32},\dots,s_{p2},\dots,s_{pp}),
\end{equation}
which gives the free entries of ${\mathbf S}$. Further recall that any matrix ${\mathbf M} \in {\mathcal M}(p,r,\bbR)$ admits a $QR$ decomposition
\begin{equation}\label{e:M=QR}
{\mathbf M} = {\mathbf Q}{\mathbf R},
\end{equation}
where ${\mathbf Q}\in {\mathcal M}(p,r,\bbR)$ has orthonormal columns and ${\mathbf R}\in GL(r,\bbR)$ (e.g., Horn and Johnson \cite{horn:johnson:2013}, Theorem 2.1.14, (a)). Given any matrix $\mathbf M \in {\mathcal M}(r,s,\bbR)$, $r,s \in \bbN$, for simplicity we write
\begin{equation}\label{e:A^perp}
\mathbf M^\perp = \{\mathbf v\in\bbR^r : \mathbf M^* \mathbf v=\mathbf 0\} = \text{nullspace}\{ \mathbf M^*\}.
\end{equation}
For any collection of vectors ${\mathbf v}_1,\hdots,{\mathbf v}_m$, $\textnormal{span}\{{\mathbf v}_1,\hdots,{\mathbf v}_m\} =
\textnormal{span}_{\ell=1,\hdots,m}\{{\mathbf v}_\ell\}$ denotes the linear space generated by these vectors. Likewise, for any collection of matrices $\mathbf M_i\in {\mathcal M}(r,s_i,\bbR)$, $i=1,\ldots, m$,
\begin{equation}\label{e:span(M1,...,Mm)}
\text{span}\{\mathbf M_1,\ldots, \mathbf M_m\}
\end{equation}
denotes the column space of the matrix $(\mathbf M_1\ldots \mathbf M_m)$.  We use the asymptotic notation
\begin{equation}\label{e:OP(1),OmegaP(1)}
o_\bbP(1), ~ O_\bbP(1), ~ o(1) ~\textnormal{ and }~ O(1)
\end{equation}
to describe sequences of matrices (or vectors) whose \textit{spectral norms} vanish or are bounded above, respectively, in probability or deterministically, as both $n,p\to\infty$ in accordance with \eqref{e:three-fold_lim}.

\subsection{Measurements}\label{s:measurements} Throughout the paper, we assume observations stem from the model \eqref{e:Y(t)}. The ``signal" $X = \{X(t)\}_{t =1,\ldots,n}$ and the ``noise" component $Z = \{Z(t)\}_{t =1,\ldots,n}$ are $\bbR^r$-valued and $\bbR^p$-valued stochastic processes, respectively, where $r$ is fixed and $p = p(n)$. Though not explicitly assumed, one can think that
\begin{equation}\label{e:X_and_Z_are_independent}
\textnormal{$X$ and $Z$ are second order, uncorrelated and zero-mean stochastic processes}
\end{equation}
(see also Remark \ref{r:asympt_rescaled_eigenvalues}, $(iv)$). The deterministic matrix ${\mathbf P} = {\mathbf P}(n)$ can be expressed as
\begin{equation}\label{e:P(n)}
{\mathcal M}(p,r,\bbR) \ni {\mathbf P}(n) = \Big( \mathbf{p}_{ 1}(n),\hdots,\mathbf{p}_{ r}(n) \Big), \quad \|\mathbf{p}_{ q}(n)\| = 1, \hspace{3mm} q = 1,\hdots,r.
\end{equation}
For the sake of clarity and mathematical generality, in Section \ref{s:WRMs_and_assumptions} we state directly in the wavelet domain the conditions for the convergence in probability as well as for the asymptotic normality of wavelet log-eigenvalues. Before doing so, in Section \ref{s:wavelet_analysis} we recap the basic framework of wavelet multiresolution analysis.

\subsection{Wavelet analysis}\label{s:wavelet_analysis}

Recall that a \textit{wavelet} $\psi$ is a unit $L^2(\bbR)$-norm function that annihilates polynomials (see \eqref{e:N_psi}). Throughout the paper, we make use of a wavelet multiresolution analysis (MRA; see Mallat \cite{mallat:1999}, chapter 7), which decomposes $L^2(\mathbb{R})$ into a sequence of \textit{approximation} (low-frequency) and \textit{detail} (high-frequency) subspaces $V_j$ and $W_j$, respectively, associated with different scales of analysis $2^{j}$, $j \in \bbZ$. In particular, given a wavelet $\psi$, there is a related \textit{scaling function} $\phi \in L^2(\bbR)$. Appropriate rescalings and shifts of $\phi$ and $\psi$ form bases for the subspaces $V_j$ and $W_j$, respectively (see Mallat \cite{mallat:1999}, Theorems 7.1 and 7.3).

In almost all mathematical statements, we make assumptions ($W1-W3$) on the underlying wavelet MRA. Such assumptions are standard in the wavelet literature and are accurately described in Section \ref{s:assumptions_on_the_MRA}. In particular, we make use of a compactly supported wavelet basis.

So, let $\phi$ and $\psi$ be the scaling and wavelet functions, respectively, associated with the wavelet MRA. We further suppose the wavelet coefficients stem from Mallat's pyramidal algorithm (Mallat \cite{mallat:1999}, chapter 7). For expositional simplicity, in our description of the algorithm we use the $\bbR^p$-valued process $Y$ in \eqref{e:Y(t)}, though analogous developments also hold for both $X$ and $Z$. Initially, suppose an infinite sequence of (generally dependent) random vectors
\begin{equation}\label{e:infinite_sample}
\{Y(\ell)\}_{\ell \in \bbZ},
\end{equation}
associated with the starting scale $2^j = 1$ (or octave $j = 0$), is available. Then, we can apply Mallat's algorithm to extract the so-named \textit{approximation} $(A(2^{j+1},\cdot))$ and \textit{detail} $(D(2^{j+1},\cdot))$ coefficients at coarser scales $2^{j+1}$ by means of an iterative procedure. In fact, as commonly done in the wavelet literature, we initialize the algorithm with the process
\begin{equation}\label{e:Btilde}
\bbR^p \ni \widetilde{Y}(t) := \sum_{k \in \bbZ} Y(k)\phi(t-k), \quad  t\in \bbR.
\end{equation}
By the orthogonality of the shifted scaling functions $\{\phi(\cdot - k)\}_{k \in \bbZ}$,
\begin{equation}\label{e:a(0,k)}
\bbR^p \ni  A(2^0,k)= \int_\bbR \widetilde{Y}(t)\phi(t-k)dt= Y(k), \quad k \in \bbZ
\end{equation}
(see Stoev et al.\ \cite{stoev:pipiras:taqqu:2002}, proof of Lemma 6.1, or Moulines et al.\ \cite{moulines:roueff:taqqu:2007:JTSA}, p.\ 160; cf.\ Abry and Flandrin \cite{abry:flandrin:1994}, p.\ 33). In other words, the initial sequence, at octave $j = 0$, of approximation coefficients is given by the original sequence of random vectors. To obtain approximation and detail coefficients at coarser scales, we use Mallat's iterative procedure
\begin{equation}\label{e:Mallat}
A(2^{j+1},k) = \sum_{k'\in \mathbb{Z}} u_{k'-2k}A(2^j,k'),\quad D(2^{j+1},k) =\sum_{k'\in \mathbb{Z}}v_{k'-2k} A(2^j,k'),  \quad k \in \mathbb{Z},
\end{equation}
for each $ j \in \mathbb{N} \cup \{0\}$. In \eqref{e:Mallat}, the (scalar) filter sequences $\{ u_k:=2^{-1/2}\int_\bbR \phi(t/2)\phi(t-k)dt \}_{k\in\bbZ}$ and $\{v_k:=2^{-1/2}\int_\bbR\psi(t/2)\phi(t-k)dt\}_{k\in\bbZ}$ are called low- and high-pass MRA filters, respectively. Due to the assumed compactness of the supports of $\psi$ and of the associated scaling function $\phi$ (see condition \eqref{e:supp_psi=compact}), only a finite number of filter terms is nonzero, which is convenient for computational purposes (Daubechies \cite{daubechies:1992}). {Hereinafter, we assume without loss of generality that $\text{supp}(\phi) = \text{supp}(\psi)=[0,T]$ (cf.\ Moulines et al.\ \cite{moulines:roueff:taqqu:2007:JTSA}, p.\ 160).} Moreover, the wavelet (detail) coefficients $D(2^j,k)$ of $Y$ can be expressed as
\begin{equation}\label{e:disc2}
\bbR^p \ni D(2^j,k) = \sum_{\ell\in\bbZ} Y(\ell)h_{j,2^jk-\ell},
\end{equation}
where the filter terms are defined as $ h_{j,\ell} =2^{-j/2}\int_\bbR\phi(t+\ell)\psi(2^{-j}t)dt$
(in the notation of \eqref{e:D(a(n)2^j,k)_intro}, $\widetilde{h}_{2^jk-\ell} = h_{j,2^jk-\ell}$). If we replace \eqref{e:infinite_sample} with the realistic assumption that only a finite length series
\begin{equation}\label{e:finite_sample}
\{Y(\ell)\}_{\ell=1,\hdots,n}
\end{equation}
is available, writing $\widetilde{Y}^{(n)}(t) := \sum_{\ell=1}^{n}Y(\ell)\phi(t-\ell)$, we have $\widetilde{Y}^{(n)}(t) = \widetilde{Y}(t)$ for all $t\in (T,n+1)$ (cf.\ Moulines et al.\ \cite{moulines:roueff:taqqu:2007:JTSA}). Noting that $D(2^j,k) = \int_\bbR \widetilde Y(t) 2^{-j/2}\psi(2^{-j}t - k)dt$ and $D^{(n)}(2^j,k) = \int_\bbR \widetilde Y^{(n)}(t)2^{-j/2}\psi(2^{-j}t - k)dt$, it follows that the finite-sample wavelet coefficients $D^{(n)}(2^j,k)$ of $\widetilde{Y}^{(n)}(t) $ are equal to $D(2^j,k)$ whenever $\textnormal{supp } \psi(2^{-j} \cdot - k) = (2^jk, 2^j (k+T)) \subseteq (T,n+1)$. In other words,
\begin{equation}\label{e:d^n=d}
D^{(n)}(2^j,k) = D(2^j,k),\qquad \forall (j,k)\in\{(j,k): 2^{-j}T\leq k\leq 2^{-j}(n+1)-T\}.
\end{equation}
Equivalently, such subset of finite-sample wavelet coefficients is not affected by the so-named \textit{border effect} (cf.\ Craigmile et al.\ \cite{craigmile:guttorp:Percival:2005}, Percival and Walden \cite{percival:walden:2006}). Moreover, by \eqref{e:d^n=d} the number of such coefficients at octave $j$ is given by $n_j = \lfloor 2^{-j}(n+1-T)-T\rfloor$. Hence, $n_j \sim 2^{-j}n$ for large $n$. Thus, for notational simplicity we suppose
 \begin{equation}\label{e:nj=n/2^j}
 n_{j} = \frac{n}{2^j}
 \end{equation}
 holds exactly and only work with wavelet coefficients unaffected by the border effect.

\subsection{Wavelet random matrices and assumptions}\label{s:WRMs_and_assumptions} For $j \in \bbN \cup \{0\}$, $k \in \bbZ$ and a dyadic sequence $\{a(n)\}_{n \in \bbN}$, the random vectors
\begin{equation}\label{e:wavelet_transform_Y_X_Z}
D(a(n)2^j,k)  \in \bbR^p, \quad  D_X(a(n)2^j,k)   \in \bbR^r \quad \textnormal{and} \quad D_Z(a(n)2^j,k) \in \bbR^p
\end{equation}
denote the \textit{wavelet transform} at scale $a(n)2^j$ of the stochastic processes $Y$, $X$ or $Z$, respectively. Whenever well defined, the \textit{wavelet random matrix} -- or \textit{sample wavelet (co)variance} -- of $Y$ at scale $a(n)2^j$ is denoted by
\begin{equation}\label{e:W(a(nu))}
{\mathcal S}_{\geq 0}(p,\bbR) \ni \mathbf{W}(a(n) 2^j) = \frac{1}{n_{a,j}}\sum^{n_{a,j}}_{k=1}D(a(n)2^j,k)D(a(n)2^j,k)^*, \quad n_{a,j} = \frac{n}{a(n)2^j}.
\end{equation}
The remaining wavelet random matrix terms $\W_X,\W_{X,Z},\W_Z$ are naturally defined as
$$
\mathbf{W}_X(a(n) 2^j) = \frac{1}{n_{a,j}}\sum^{n_{a,j}}_{k=1}D_X(a(n)2^j,k)D_X(a(n)2^j,k)^* \in {\mathcal S}_{\geq 0}(r,\bbR),
$$
$$
\mathbf{W}_{X,Z}(a(n) 2^j) = \frac{1}{n_{a,j}}\sum^{n_{a,j}}_{k=1}D_X(a(n)2^j,k)D_Z(a(n)2^j,k)^* \in {\mathcal M}(r,p,\bbR)
$$
\begin{equation}\label{e:WZ(a(n)2^j)}
\textnormal{and}\quad \mathbf{W}_{Z}(a(n) 2^j) = \frac{1}{n_{a,j}}\sum^{n_{a,j}}_{k=1}D_Z(a(n)2^j,k)D_Z(a(n)2^j,k)^* \in {\mathcal S}_{\geq 0}(p,\bbR).
\end{equation}
In particular, since $p = p(n) \rightarrow \infty$ in general (see \eqref{e:p(n),a(n)_conditions}), only $\mathbf{W}_X(a(n) 2^j)$ in \eqref{e:WZ(a(n)2^j)} has fixed dimensions.

We further define the auxiliary random matrix
\begin{equation}\label{e:B-hat_a(2^j)}
\widehat{{\mathbf B}}_a(2^j) = {\mathbf P}_H^{-1}\big\{a(n)^{-{\mathbf H}-(1/2){\mathbf I}}\hspace{1mm}{\mathbf W}_X(a(n)2^j)\hspace{1mm}a(n)^{-{\mathbf H}^*-(1/2){\mathbf I}} \big\}({\mathbf P}_H^*)^{-1}\in {\mathcal S}_{\geq 0}(r,\bbR).
\end{equation}
Its mean is denoted by
\begin{equation}\label{e:B_a(2^j)=EB-hat_a(2^j)}
{\mathbf B}_a(2^j) := \bbE \widehat{{\mathbf B}}_a(2^j)
\end{equation}
whenever it exists. In \eqref{e:B-hat_a(2^j)}, we assume that the \textit{scaling matrix} ${\mathbf H}$ has the Jordan form
\begin{equation}\label{e:H_is_diagonalizable}
{\mathbf H} = {\mathbf P}_H \textnormal{diag}(h_1,\hdots,h_r) {\mathbf P}^{-1}_H, \quad  {\mathbf P}_H \in GL(r,\bbR), \quad -1/2 < h_1 \leq \hdots \leq h_r <\infty.
\end{equation}
For the sake of illustration, when $X$ is an ofBm, ${\mathbf H}$ is a \textit{Hurst matrix} whose ordered eigenvalues satisfy $0 < h_1 \leq \hdots \leq h_r \leq 1$ (see Section \ref{s:examples}). Moreover, in this case it can be shown that the relation
\begin{equation}\label{e:W_X(a2^j)_approx}
{\mathbf W}_X(a(n)2^j) \approx a(n)^{{\mathbf H}+(1/2){\mathbf I}}{\mathbf W}_X(2^j) a(n)^{{\mathbf H}^*+(1/2){\mathbf I}} %
\end{equation}
holds approximately in law (cf.\ Abry and Didier \cite{abry:didier:2018:dim2}, Proposition 3.1), where ${\mathbf W}_X(2^j) := \frac{1}{n_{a,j}}\sum^{n_{a,j}}_{k=1}D_X(2^j,k)D_X(2^j,k)^*$. Hence, the matrix $\widehat{{\mathbf B}}_a(2^j)$ can be interpreted as a version of ${\mathbf W}_X(a(n)2^j)$ after compensating for scaling (i.e., multiplication by $a(n)^{-{\mathbf H}-(1/2){\mathbf I}}$ and its transpose) and non-canonical coordinates (i.e., multiplication by ${\mathbf P}_H^{-1}$ and its transpose).

We make use of the following assumptions in the main results of this paper (Section \ref{s:main}). For expository purposes, we first state the assumptions, and then provide some interpretation. Throughout the sequel, we fix a finite number $m \in \bbN$ of integers
\begin{equation}\label{e:def_j1,jm}
0 \leq j_1<\ldots<j_m.
\end{equation}
They correspond to the entries of the vector of random matrices $\big({\mathbf W}(a(n)2^{j})\big)_{j=j_1,\hdots,j_m}$, whose spectral behavior in the three-way limit \eqref{e:three-fold_lim} is the central focus of this work.

\medskip

\noindent {\sc Assumption $(A1)$}: Given \eqref{e:P(n)}, \eqref{e:W(a(nu))} and \eqref{e:WZ(a(n)2^j)}, for $j = j_1,\ldots,j_m$ and any $n \in \bbN$, the wavelet random matrix
$$
{\mathbf W}(a(n)2^j) = {\mathbf P}(n){\mathbf W}_X(a(n)2^j){\mathbf P}^*(n) + {\mathbf W}_Z(a(n)2^j)
$$
\begin{equation}\label{e:W=PWXP*+WZ+PWXZ+WXZ*P*}
 + {\mathbf P}(n){\mathbf W}_{X,Z}(a(n)2^j)+ {\mathbf W}^*_{X,Z}(a(n)2^j){\mathbf P}^*(n),
\end{equation}
and each sum term on the right-hand side of \eqref{e:W=PWXP*+WZ+PWXZ+WXZ*P*} are well defined a.s. Also, all entry-wise moments of the random matrices in \eqref{e:W=PWXP*+WZ+PWXZ+WXZ*P*} exist and
\begin{equation}\label{e:EW_X,Z(a(n)2^j)=0}
\bbE {\mathbf W}_{X,Z}(a(n)2^j) = {\mathbf 0}.
\end{equation}

\medskip

\noindent {\sc Assumption $(A2)$}: In \eqref{e:W=PWXP*+WZ+PWXZ+WXZ*P*}, %
\begin{equation}\label{e:assumptions_WZ=OP(1)}
\|\W_Z(a(n)2^j)\|  = O_{\bbP}(1) \quad\text{and}\quad \|\E \mathbf{W}_Z(a(n)2^j)\| = O(1).
\end{equation}

\medskip

\noindent {\sc Assumption $(A3)$}: the random matrix $\widehat{{\mathbf B}}_a(2^j)$ as in \eqref{e:B-hat_a(2^j)} satisfies
\begin{equation}\label{e:sqrt(K)(B^-B)->N(0,sigma^2)}
\Big(\hspace{1mm}\sqrt{n_{a,j}}\hspace{0.5mm}(\vecoper_{{\mathcal S}}\widehat{{\mathbf B}}_a(2^j) - \vecoper_{{\mathcal S}}{\mathbf B}_a(2^j))\hspace{1mm}\Big)_{j=j_1,\hdots,j_m}\stackrel{d}\rightarrow {\mathcal N}\big(0,\Sigma_B(j_1,\hdots,j_m)\big), \quad n \rightarrow \infty,
\end{equation}
for some $\Sigma_B(j_1,\hdots,j_m) \in {\mathcal S}_{\geq 0}(m \cdot r,\bbR)$. In addition, its mean ${\mathbf B}_a(2^j) = \bbE \widehat{{\mathbf B}}_a(2^j)$ satisfies
\begin{equation}\label{e:|bfB_a(2^j)-B(2^j)|=O(shrinking)}
\|{{\mathbf B}}_a(2^j) - \mathbf B(2^j) \| = o(1), \quad  n\to\infty, \quad j=j_1,\hdots,j_m,
\end{equation}
where $\mathbf B(2^j)$ is some matrix such that
\begin{equation}\label{e:B(2^j)_full_rank}
\mathbf B(2^j) \in \mathcal S_{>0}(r,\bbR). %
\end{equation}

\medskip

\noindent {\sc Assumption $(A4)$}: The dimension $p(n) \geq r$ and the dyadic scaling factor $a(n)$ satisfy the relations
\begin{equation}\label{e:p(n),a(n)_conditions}
{a(n) \leq \frac{n}{2^{j_m}}}, \quad \frac{a(n)}{n}+\frac{n}{a(n)^{h_1 + 3/2}} \rightarrow 0, \quad\frac{p(n)}{n/a(n)}=O(1),\quad n \rightarrow \infty.
\end{equation}
\medskip

\noindent {\sc Assumption $(A5)$}: Let ${\mathbf P}(n) \in {\mathcal M}(p,r,\bbR)$ and ${\mathbf P}_H \in GL(r,\bbR)$ be deterministic matrices as in \eqref{e:P(n)} and \eqref{e:H_is_diagonalizable}, respectively. Let
\begin{equation}\label{e:P(n)=Q(n)R(n)}
{\mathbf P}(n){\mathbf P}_H={\mathbf Q}(n){\mathbf R}(n)
\end{equation}
be the $QR$ decomposition of ${\mathbf P}(n){\mathbf P}_H$ (cf.\ \eqref{e:M=QR}). Then, there exists a (deterministic) matrix ${\mathbf A} \in \mathcal{S}_{>0}(r,\bbR)$ with Cholesky decomposition
${\mathbf A} = {\mathbf R}^* {\mathbf R}$
such that
\begin{equation}\label{e:<p1,p2>=c12_2}
\|{\mathbf R}(n)-{\mathbf R}\|  = o(1).
\end{equation}

\medskip

Assumptions $(A1 - A3)$ pertain to wavelet domain behavior.  Assumption $(A1)$ holds under very general conditions. In fact, under $(W1-W3)$, it is satisfied assuming \eqref{e:X_and_Z_are_independent}. Assumption $(A2)$ ensures that the influence of the matrices ${\mathbf W}_{Z}(a(n)2^j)$ and $\bbE {\mathbf W}_{Z}(a(n)2^j)$ is not too large on the behavior of ${\mathbf W}(a(n)2^j)$ and $\bbE {\mathbf W}(a(n)2^j)$, respectively. In particular, the matrix ${\mathbf W}_{Z}(a(n)2^j)$ for the noise term displays no explosive scaling behavior.  Assumption $(A3)$ posits the asymptotic normality of the (wavelet domain) fractional component ${\mathbf W}_{X}(a(n)2^j)$ after compensating for scaling and non-canonical coordinates.

In turn, assumption $(A4)$ controls the divergence rates among $n$, $a(n)$ and $p(n)$ in the three-way limit. In particular, it states that the scaling factor $a(n)$ must blow up slower than $n$, and that the three-component ratio $\frac{p(n) \hspace{0.5mm}a(n)}{n}$ must converge to a constant (cf.\ the traditional ratio $\lim_{n \rightarrow \infty} p(n)/n$ for high-dimensional sample covariance matrices). Assumption $(A5)$ ensures that, asymptotically speaking, the angles between the column vectors of the matrix ${\mathbf P}(n) {\mathbf P}_H$ converge in such a way that the matrix  $\lim_{n \rightarrow \infty}{\mathbf P}^*_H {\mathbf P}^*(n){\mathbf P}(n){\mathbf P}_H = \lim_{n \rightarrow \infty}{\mathbf R}^*(n){\mathbf R}(n)  = {\mathbf A}$ has full rank. This entails that ${\mathbf P}(n)$ does not strongly impact the scaling properties of the hidden random matrix ${\mathbf W}_{X}(a(n)2^j)$.\vspace{2mm}

A discussion of some broad Gaussian and non-Gaussian contexts where assumptions $(A1-A3)$ are satisfied is deferred to Section \ref{s:examples}. Heuristically, assuming a large enough $N_{\psi}$, these assumptions hold for several instances of $X$ and $Z$. This is so, for example, when $Z$ is an ARMA-type $p$-variate process and, for some appropriate matrix ${\mathbf H}$, $X$ is a $r$-variate, stationary-increment (Gaussian) process satisfying the scaling relation $\{X(ct)\}_{t \in \bbR} \approx \{c^{{\mathbf H}}X(t)\}_{t \in \bbR}$ for all $c > 0$ (see the examples in Section \ref{s:examples}).

\section{Main results}\label{s:main}

\subsection{Asymptotic behavior of wavelet eigenvalues}

In our first theorem, we establish that, after proper rescaling, the $r$ largest eigenvalues of a wavelet random matrix $\mathbf{W}(a(n)2^j)$ in high dimensions converge in probability to deterministic functions $\xi_q(2^j)$, $q = 1,\hdots,r$. Thus, these functions can be interpreted as asymptotic rescaled eigenvalues. Notably, they display a scaling property. Moreover, the remaining $p(n)-r$ eigenvalues of a wavelet random matrix are bounded in probability.
\begin{theorem}\label{t:lim_n_a_times_lambda/a^(2h+1)}
Fix any $j$ as in \eqref{e:def_j1,jm} and assume $(W1 - W3)$ and $(A1 - A5)$ hold. Then, for $p = p(n)$, the limits
\begin{equation}\label{e:lim_n_a*lambda/a^(2h+1)}
\plim_{n \rightarrow \infty}\frac{\lambda_{p-r+q}\big(\mathbf{W}(a(n)2^j)\big)}{a(n)^{2h_q+ 1}} =: \xi_q(2^j) > 0, \quad q=1,\ldots,r,
\end{equation}
exist, and the deterministic functions $\xi_q$ satisfy the scaling relation
\begin{equation}\label{e:xi-i0_scales}
\xi_{q}(2^j)  = 2^{j \hspace{0.5mm}(2 h_{q}+1)} \xi_{q}(1).
\end{equation}
In addition,
\begin{equation}\label{e:lim_lambda_p-r(W)}
0\leq \lambda_{1}\big(\mathbf{W}(a(n)2^j)\big) \leq \ldots \leq \lambda_{p-r}\big(\mathbf{W}(a(n)2^j)\big) = O_{\bbP}(1).
\end{equation}
\end{theorem}

\begin{remark}\label{r:explicit_expression_for_xiq}
For any fixed $j$ as in \eqref{e:def_j1,jm} and $q \in \{1,\hdots,r\}$, relation \eqref{e:xi_ell(2^j)=lambda_r-r2+ell} in the proof of Theorem \ref{t:lim_n_a_times_lambda/a^(2h+1)} provides the explicit expression $\xi_q(2^j) = \lambda_{r-r_2+q}(\boldsymbol \Lambda)$, where $\boldsymbol \Lambda = \boldsymbol \Lambda(2^j)$ is determined from relations \eqref{e:M=B22-B23*B33(-1)B32}, \eqref{e:Pi3} and \eqref{e:def_Lambda}. In particular, $\xi_q(2^j)$ depends on ${\mathbf B}(2^j)$ and on the limiting behavior of the coordinates matrix ${\mathbf P}(n)$ (see \eqref{e:P=(P1(n)_P2(n)_P3(n))} and \eqref{e:R=(R1,R2,R3)}). See also Example \ref{ex:proof_of_Theo_3.1} for an illustration based on a simplified case.
\end{remark}
In our second theorem, we establish the asymptotic normality of the $r$ largest wavelet log-eigenvalues in high dimensions. Note that this theorem requires stronger assumptions than the previous one (see Remark \ref{r:asympt_rescaled_eigenvalues}, $(i)-(v)$, on the use of assumptions in Theorems \ref{t:lim_n_a_times_lambda/a^(2h+1)} and \ref{t:asympt_normality_lambdap-r+q}).
\begin{theorem}\label{t:asympt_normality_lambdap-r+q}
Fix integers $0\leq j_1<j_2<\ldots < j_m$ as in \eqref{e:def_j1,jm} and assume $(W1 - W3)$ and $(A1 - A5)$ hold. Further suppose that
\begin{equation}\label{e:xiq_distinct}
\text{ whenever $h_{\ell_1}=h_{\ell_2}$ with $\ell_1\neq\ell_2$,} \quad \text{then} \quad \xi_{\ell_1}(1)\neq \xi_{\ell_2}(1).
\end{equation}
Then, for $p = p(n)$, as $n \rightarrow \infty$,
$$
 \Big( \sqrt{n_{a,j}}\Big( \log \lambda_{p-r+q}\big({\mathbf W}(a(n)2^{j})\big)  - \log \lambda_{p-r+q}\big(\bbE {\mathbf W} (a(n)2^{j})\big) \Big)_{q=1,\hdots,r} \Big)_{j=j_1,\hdots,j_m}
 $$
 \begin{equation}\label{e:asympt_normality_lambda2}
 \hspace{8cm}\stackrel{d}\rightarrow {\mathcal N}(0,\Sigma_{\lambda})
\end{equation}
for some $\Sigma_{\lambda} \in {\mathcal S}_{\geq 0}(m \cdot r,\bbR)$.
\end{theorem}

\begin{remark}\label{r:theo_3.2}
Condition \eqref{e:xiq_distinct} covers the central subcases where the scaling eigenvalues are simple ($h_1 < \hdots < h_r$) or identical ($h_1 = \hdots = h_r$) with distinct constants $\xi_q(1)$, $q = 1,\hdots,r$.
\end{remark}

\begin{remark}\label{r:asympt_rescaled_eigenvalues}
Some comments are in order on the use of each assumption within the theorems and also on the statements of the theorems.
\begin{itemize}
\item [$(i)$] In Theorems \ref{t:lim_n_a_times_lambda/a^(2h+1)} and \ref{t:asympt_normality_lambdap-r+q}, only assumptions $(A1-A5)$ are directly used. Nevertheless, assumptions $(W1-W3)$ on the underlying wavelet basis are implicitly used in the definition of wavelet random matrices. Also, they are applied in the construction of examples of frameworks where conditions $(A1-A5)$ hold. As anticipated in the Introduction, these examples are developed in Section \ref{s:examples}.
\item [$(ii)$] In Theorem \ref{t:lim_n_a_times_lambda/a^(2h+1)}, the only aspect of \eqref{e:sqrt(K)(B^-B)->N(0,sigma^2)} that is used is the fact that
\begin{equation}\label{e:||B-hat_a(2^j)-B(2^j)||->0}
\big\|\big(\vecoper_{{\mathcal S}}\widehat{{\mathbf B}}_a(2^j) - \vecoper_{{\mathcal S}}{\mathbf B}_a(2^j)\big)_{j=j_1,\hdots,j_m}\big\|\stackrel{\bbP}\rightarrow 0, \quad n \rightarrow \infty.
\end{equation}
By contrast, the convergence in distribution \eqref{e:sqrt(K)(B^-B)->N(0,sigma^2)} is fully used in the proof of Theorem \ref{t:asympt_normality_lambdap-r+q}.
\item [$(iii)$] By comparison to Theorem \ref{t:lim_n_a_times_lambda/a^(2h+1)}, the asymptotic normality of wavelet log-eigenvalues obtained in Theorem \ref{t:asympt_normality_lambdap-r+q} requires the additional condition \eqref{e:xiq_distinct} so as to ensure the simplicity of wavelet eigenvalues. Without condition \eqref{e:xiq_distinct}, due to the lack of smoothness of eigenvalues, the asymptotic distribution of wavelet log-eigenvalues in high dimensions is expected to be generally non-Gaussian. A broad characterization of such distribution remains an open problem.

\item[$(iv)$] The assumption that $\bbE {\mathbf W}_{X,Z}(a(n)2^j) = {\mathbf 0}$ (see \eqref{e:EW_X,Z(a(n)2^j)=0}; cf.\ \eqref{e:X_and_Z_are_independent}) is \textit{not} essential. It just conveniently simplifies some expressions in the proof of Theorem \ref{t:asympt_normality_lambdap-r+q}, hence rendering the argument more readable.
\item[$(v)$]  For illustration, suppose in  \eqref{e:three-fold_lim} that the limit $c:=\lim_{n \rightarrow \infty}\frac{p(n) \hspace{0.5mm}a(n)}{n}\in[0,\infty)$ exists. In this case, unlike in a traditional Mar$\breve{\textnormal{c}}$enko-Pastur limit (see Bai and Silverstein \cite{bai:silverstein:2010}, Chapter 3), the particular value of $c\geq 0$
    does not play any role in the claims of either one of the two theorems. The boundedness of the ratio \eqref{e:three-fold_lim} is critically used in the bounds \eqref{e:term3_bound} and \eqref{e:sqrt(n/a)(fnq2(rescaled_WXZ)-fnq2(0))_bound}, when proving Theorem \ref{t:asympt_normality_lambdap-r+q}. Namely, \eqref{e:three-fold_lim} guarantees that the centered Taylor expansions of the functions $f_{n,q,2}$ and $f_{n,q,3}$ (see \eqref{e:fv} and \eqref{e:fv_3}) vanish.
\item[$(vi)$] In the rich literature on spiked covariance models (Johnstone \cite{johnstone:2001}, Baik and Silverstein \cite{baik:silverstein:2006}, Wang and Fan \cite{wang:fan:2017}, Cai et al.\ \cite{cai:han:pan:2020}, Diaconu \cite{diaconu:2023}), the top eigenvalues of sample covariance matrices may also display asymptotically Gaussian fluctuations under conditions (Bai and Yao \cite{bai:yao:2008}). In this case, though, Gaussianity is a fixed-scale phenomenon, stemming from direct assumptions on the magnitude of top population eigenvalues and low-rank perturbations of sample covariance matrices (e.g., Bai and Yao \cite{bai:yao:2012}). Similar remarks can be made about related phenomena appearing in the vast literature on principal components analysis (e.g., Johnstone \cite{johnstone:2001}, Johnstone and Paul \cite{johnstone:paul:2018}, Wang and Fan \cite{wang:fan:2017}). By contrast, in Theorem \ref{t:asympt_normality_lambdap-r+q} the asymptotically Gaussian fluctuations are a large-scale phenomenon. They are fundamentally based on the distinct scaling behavior displayed by the latent process $X$ and by the noise term $Z$, captured in the eigenvalues of wavelet random matrices.
\end{itemize}
\end{remark}

\begin{remark}\label{r:multiresolution_RM}
The mathematical framework of Theorems \ref{t:lim_n_a_times_lambda/a^(2h+1)} and \ref{t:asympt_normality_lambdap-r+q} applies to a much larger class of random matrices which includes sample covariance matrices. This is so because, as mentioned in Remark \ref{r:asympt_rescaled_eigenvalues}, $(i)$, only the assumptions $(A1 - A5)$ are directly used in the proofs of the theorems.

To see this, let $Y = \{Y(t)\}_{t \in \bbZ}$ be the $p$-variate stochastic process \eqref{e:Y(t)}, with $X = \{X(t)\}_{t \in \bbZ}$ and $Z = \{Z(t)\}_{t \in \bbZ}$ as described in Section \ref{s:WRMs_and_assumptions}. For any scale $\widetilde{a} \in \bbN \cup \{0\}$ and ``time" parameter $\kappa \in \bbZ$, a \textit{multiresolution random vector}
\begin{equation}\label{e:T(a(n)2^j,k)}
T(\widetilde{a},\kappa) \in \bbR^p
\end{equation}
associated with $Y$ is a measurable function of the process $Y$ that depends on $\widetilde{a}$ and $\kappa$ (cf.\ Jaffard, Lashermes and Abry \cite{jaffard:lashermes:abry:2006} and Jaffard, Seuret et al.\ \cite{jaffard:seuret:wendt:leonarduzzi:roux:abry:2019}).  %
Examples of sequences of multiresolution random vectors include the wavelet transform \eqref{e:wavelet_transform_Y_X_Z} itself (for the choices $\widetilde{a} = a(n)2^j$, $\kappa = a(n)2^j k$), as well as the increments
\begin{equation}\label{e:increments}
T(\widetilde{a},\kappa) = Y(\kappa + \widetilde{a}) - Y(\kappa).
\end{equation}
Let $\bbZ_{\widetilde{a}}$ be the set of values of $\kappa$ available at scale $\widetilde{a}$, where $\widetilde{n}_{a}:=\textnormal{card}(\bbZ_{\widetilde{a}})$ (for wavelet random matrices, $\widetilde{n}_{a} = n_{a,j}$). Let $\{T(\widetilde{a},\kappa)\}_{\kappa \in \bbZ_{\widetilde{a}}}$ be a sequence of multiresolution random vectors as in \eqref{e:T(a(n)2^j,k)}. The associated \textit{multiresolution random matrix} is defined as
\begin{equation}\label{e:multires_RM}
{\mathbf W}(\widetilde{a}) = \frac{1}{\widetilde{n}_{a}} \sum_{\kappa \in \bbZ_{\widetilde{a}}} T(\widetilde{a},\kappa)T(\widetilde{a},\kappa)^* \in {\mathcal S}_{\geq 0}(p,\bbR).
\end{equation}
When $T(\widetilde{a},\kappa)$ is given by the increments \eqref{e:increments}, then ${\mathbf W}(\widetilde{a})$ is a classical \textit{sample covariance matrix} (at scale $\widetilde{a}$). Analogously, we can define multiresolution random vectors and random matrices associated with $X$ and $Z$. Then, \textit{mutatis mutandis}, under assumptions $(A1 - A5)$ the proofs of Theorems \ref{t:lim_n_a_times_lambda/a^(2h+1)} and \ref{t:asympt_normality_lambdap-r+q} show that the conclusions of these theorems hold for \eqref{e:multires_RM}.
\end{remark}

\subsection{Consequences for statistical inference: a short discussion}\label{s:consequences_for_statistics}

For the sake of illustration, consider first the classical univariate context $p=1$. Suppose \eqref{e:D(a(n)2^j,k)_intro} is the wavelet transform of a fBm with Hurst (scaling) exponent $h \in (0,1)$. Then, under mild assumptions on the wavelet basis it can be shown that \eqref{e:W(a2^j)_intro} satisfies
\begin{equation}\label{e:W(a(n)2^j)_fBm}
\bbR \ni {\mathbf W}(a(n)2^j) \approx a(n)^{2h+1}{\mathbf W}(2^j), \quad \sqrt{n_{a,j}}\big({\mathbf W}(2^j)- \bbE{\mathbf W}(2^j)) \stackrel{d}\rightarrow {\mathcal N}(0,\sigma^2_B(j)),
\end{equation}
for large $n$ and $a(n)$ (e.g., Bardet \cite{bardet:2002}, Moulines et al.\ \cite{moulines:roueff:taqqu:2008}). After linearizing the first relation in \eqref{e:W(a(n)2^j)_fBm} by means of a logarithmic transformation, a multiscale regression-type procedure can be used for statistical inference on $h$ and other parameters (Veitch and Abry \cite{veitch:abry:1999}).

The difficulties involved in high-dimensional statistical inference are much greater. Non-canonical scaling coordinates ${\mathbf P}(n){\mathbf P}_H$ (see \eqref{e:P(n)=Q(n)R(n)}) generally mix together slow and fast scaling laws present in the behavior of high-dimensional fractional stochastic processes. This leads to the so-called \textit{amplitude }and \textit{dominance effects} (see Abry and Didier \cite{abry:didier:2018:dim2} for a detailed discussion). These effects manifest themselves in the form of strong biases in standard, univariate-like statistical methodology when applied to measurements of multidimensional phenomena such as Internet traffic (e.g., Abry and Didier \cite{abry:didier:2018:n-variate}, Section 6), cointegration (see, for instance, Kaufmann and Stern \cite{kaufmann:stern:2002}, Schmith et al.\ \cite{schmith:johansen:thejll:2012} on climate science) and systems modeled in blind source separation problems (e.g., Comon and Jutten \cite{comon:jutten:2010}; see also Section \ref{s:examples} in this paper).

Theorems \ref{t:lim_n_a_times_lambda/a^(2h+1)} and \ref{t:asympt_normality_lambdap-r+q} bear direct consequences for statistical inference. This is so because they provide a framework for the high-dimensional estimation of the parameters $h_1,\hdots,h_r$ and $r$, and hence, of the scaling properties of the system \eqref{e:Y(t)} even in the presence of non-canonical coordinates ${\mathbf P}(n){\mathbf P}_H$ and high-dimensional, non-Gaussian noise.

In fact, fix $j$. In light of \eqref{e:lim_n_a*lambda/a^(2h+1)}, \eqref{e:xi-i0_scales} and \eqref{e:asympt_normality_lambda2}, the random vector
\begin{equation}\label{e:log_lambda_p-r+q/2log_a}
\Big\{\frac{\log \lambda_{p(n)-r+q}\big({\mathbf W}(a(n)2^j)\big)}{2 \log a(n)} - \frac{1}{2}\Big\}_{q = 1,\hdots,r}
\end{equation}
can be interpreted, in the language of statistics, as \textit{consistent} and \textit{asymptotically normal} estimators of the vector of scaling parameters $\{h_{q}\}_{q=1,\hdots,r}$ in high dimensions. Under the same conditions, the lowest $p(n)-r$ wavelet log-eigenvalues stay bounded, whence
\begin{equation}\label{e:log_lambda_p-r/2log_a}
\frac{\log \lambda_{p(n)-r}\big({\mathbf W}(a(n)2^j )\big)}{2 \log a(n)} - \frac{1}{2}
\end{equation}
converges to zero in probability (see Figure \ref{fig:logeig}). Moreover, Theorem \ref{t:asympt_normality_lambdap-r+q} can be used in testing the hypothesis of the equality of scaling eigenvalues (cf.\ Remark \ref{r:theo_3.2}).

For significantly improved finite-sample and asymptotic estimation properties, Theorems \ref{t:lim_n_a_times_lambda/a^(2h+1)} and \ref{t:asympt_normality_lambdap-r+q} can be used as a theoretical basis for the development of a multiscale regression-type statistical methodology in the wavelet eigenvalue domain. On this topic, see Abry et al.\ \cite{abry:boniece:didier:wendt:2023:regression} (see also Section \ref{s:conclusion} in this paper).

On a related note, in Section \ref{s:Gaussian} we provide examples to which the comments made in this section apply.

\section{Examples}\label{s:examples}

Recall that assumptions $(A1)$, $(A2)$ and $(A3)$ (i.e., \eqref{e:W=PWXP*+WZ+PWXZ+WXZ*P*}, \eqref{e:assumptions_WZ=OP(1)} and \eqref{e:sqrt(K)(B^-B)->N(0,sigma^2)}--\eqref{e:B(2^j)_full_rank}, respectively) are stated in the wavelet domain. Under ($W1-W3$), for any choice of pair of zero-mean, uncorrelated second order processes $Z$ and $X$, assumption $(A1)$ is satisfied. Hence, the key assumptions to be verified for specific instances of $Z$ and $X$ are $(A2)$ and $(A3)$.

For this reason, in this section we provide broad Gaussian and non-Gaussian classes of examples where assumptions $(A2)$ and $(A3)$ are satisfied (further details about ($A4$) are available in Appendix C of Abry et al.\ \cite{abry:boniece:didier:wendt:2023:extended}). Throughout this section, we suppose assumptions ($W1-W3$) hold. For the discussions, recall that the number of vanishing moments $N_{\psi}$ is given by \eqref{e:N_psi}, and the time-domain smoothness of the underlying wavelet basis is controlled by $\alpha>1$ in \eqref{e:psihat_is_slower_than_a_power_function}.

\subsection{Gaussian instances}\label{s:Gaussian}

When $X$ and $Z$ are each marginally Gaussian, it can be shown that conditions \eqref{e:assumptions_WZ=OP(1)}--\eqref{e:B(2^j)_full_rank} hold under very general assumptions.  For illustration, in this section we consider a few examples.

\begin{example}\label{ex:oss_ofBm}
As discussed in the Introduction, ofBm is the natural multivariate generalization of fBm. It is defined as a $(i)$ Gaussian; $(ii)$ o.s.s.; $(iii)$ stationary-increment stochastic process. Suppose
\begin{equation}\label{e:oss_ofBm}
\textnormal{$X$ is an $\bbR^r$-valued ofBm}
\end{equation}
with (generalized) spectral density
\begin{equation}\label{e:ofBm_harmonizable_repres}
{\mathfrak{g}}_X(x) = x^{-({\mathbf H}+(1/2){\mathbf I})}_+ \hspace{0.25mm} {\mathcal A}{\mathcal A}^* \hspace{0.25mm}x^{-({\mathbf H}^*+(1/2){\mathbf I})}_+ + x^{-({\mathbf H}+(1/2){\mathbf I})}_- \hspace{0.25mm}\overline{{\mathcal A}{\mathcal A}^*} \hspace{0.25mm}x^{-({\mathbf H}^*+(1/2){\mathbf I})}_-, \quad x \in \bbR \backslash\{0\}
\end{equation}
(Didier and Pipiras \cite{didier:pipiras:2011}, Theorem 3.1). In \eqref{e:ofBm_harmonizable_repres}, $x_{\pm}=\max\{\pm x,0\}$, the constant matrix ${\mathcal A} \in {\mathcal M}(r,\bbC)$ satisfies $\det\Re(\mathcal{A}\mathcal{A}^*) >0$, and the Hurst matrix is given by
\begin{equation}\label{e:Hurst_matrix_ofBm}
{\mathbf H} = {\mathbf P}_H \textnormal{diag}(h_1,\hdots,h_r){\mathbf P}^{-1}_H \in {\mathcal M}(r,\bbR),\quad {\mathbf P}_H \in GL(r,\bbC),
\end{equation}
where $h_1,\ldots,h_r \in (0,1) \backslash\{1/2\}$. Now fix any $j$ as in \eqref{e:def_j1,jm}. If, in addition, $N_\psi \geq 2$, then assumption ($A3$) is satisfied as a consequence of Proposition C.1 %
($i$), Abry et al.\ \cite{abry:boniece:didier:wendt:2023:extended}. In other words, relations \eqref{e:sqrt(K)(B^-B)->N(0,sigma^2)}--\eqref{e:B(2^j)_full_rank} hold for a matrix sequence ${\mathbf B}_a(2^j) \in {\mathcal S}_{\geq0}(\bbR,r)$ and a matrix ${\mathbf B}(2^j) \in {\mathcal S}_{>0}(\bbR,r)$.
\end{example}

\begin{example}\label{ex:ARMA-type}
Let $\{\mathbf A_\ell(p)\}_{\ell \in \bbN} \subseteq \mathcal M(p,\bbR)$. Consider the $p$-variate white noise sequence ${\boldsymbol \varepsilon}=\{{\boldsymbol \varepsilon}(t)\}_{t\in\bbZ}$, where $ \bbE {\boldsymbol \varepsilon}_t {\boldsymbol \varepsilon}_s^*= {\mathbf 1}_{\{t=s\}}\Sigma_{\boldsymbol \varepsilon}(p)$ for some $\Sigma_{{\boldsymbol \varepsilon}}(p) \in {\mathcal S}_{\geq 0}(p,\bbR)$. Suppose
\begin{equation}\label{e:sup-p_Sigma(p)<infty}
\sup_{p \in \bbN} \|\Sigma_{{\boldsymbol \varepsilon}}(p)\|<\infty,\quad \sup_{p \in \bbN} \sum_{\ell \in \bbZ}\|\mathbf A_{\ell}(p)\| <\infty.
\end{equation}
Then, for each $p$,
\begin{equation}\label{e:Z(t)=ARMA}
Z(t)= \sum_{\ell \in \bbZ} \mathbf A_\ell(p) {\boldsymbol \varepsilon}(t-\ell), \quad t \in \bbZ,
\end{equation}
is a $p$-variate, (weakly) stationary linear process. In particular, all classical ARMA-type multivariate processes can be written in this form (e.g., Brockwell and Davis \cite{brockwell:davis:1991}). Furthermore, under \eqref{e:sup-p_Sigma(p)<infty}, assumption ($A2$) holds under ($A4$) as a consequence of Proposition C.1, $(ii)$, Abry et al.\ \cite{abry:boniece:didier:wendt:2023:extended} (with $d = 0$ and, trivially, $N_{\psi} \geq 1 > d$).
\end{example}

The following example provides another illustration of the breadth and flexibility of the wavelet-domain framework defined by assumptions $(A1-A3)$. %

\begin{example} \label{ex:3}
Let
\begin{equation}\label{e:ofBm_factor_model_example}
\{\mathcal X(t)\}_{t \in \bbZ}
\end{equation}
be an ofBm as in \eqref{e:oss_ofBm}--\eqref{e:Hurst_matrix_ofBm}. In addition, assume
\begin{equation}\label{e:h1<...<hr}
h_1 < \hdots < h_r.
\end{equation}
Also, let $\{{\mathcal Z}(t)\}_{t \in \bbZ}$ be a $p$-variate, Gaussian process with maximal memory parameter $d$ uniformly in $p$ satisfying
\begin{equation}\label{e:d<(3/2)*(h1+1/2)}
0 < d< \min\{\tfrac{3}{2} (h_1+1/2),~2h_1+1/2\}
\end{equation}
(see Definition C.1 in Abry et al.\ \cite{abry:boniece:didier:wendt:2023:extended} for a precise description of this type of process). In particular, the (asymptotic) scaling laws present in the dynamics of $\mathcal Z$ may actually exceed those in $\mathcal X$. Now consider the process
\begin{equation}\label{e:mathcal_Y(t)=mathcal_P*mathcal_X(t)+mathcal_Z(t)}
\mathcal Y (t)= \boldsymbol{\mathcal P}(n)  \mathcal X(t) +  \mathcal Z(t), \quad t\in \bbZ,
\end{equation}
where the matrix $\boldsymbol{\mathcal P}(n) \in\mathcal M(p,r)$ satisfies
\begin{equation}\label{e:strongfactor}
\tfrac{1}{p(n)} \boldsymbol{\mathcal P}^*(n)\boldsymbol{\mathcal P}(n) \to \mathbf A, \quad\text{for some}~\mathbf A \in \mathcal S_{> 0}(r,\bbR).
\end{equation}
Expression \eqref{e:mathcal_Y(t)=mathcal_P*mathcal_X(t)+mathcal_Z(t)} defines an example of a so-called high-dimensional \textit{factor model}, which is the subject of a vast literature (e.g., Bai \cite{bai:2003} and Bai and Ng \cite{bai:ng:2013}). Condition \eqref{e:strongfactor} is sometimes called the ``strong factor'' assumption (e.g., Bai and Ng \cite{bai:ng:2002,bai:ng:2023}). It stands in sharp contrast with assumption $(A5)$, which in this context may be interpreted as a ``weak factor'' one.

Theorems \ref{t:lim_n_a_times_lambda/a^(2h+1)} and \ref{t:asympt_normality_lambdap-r+q} can be applied in the study of the scaling behavior of \eqref{e:mathcal_Y(t)=mathcal_P*mathcal_X(t)+mathcal_Z(t)}. In fact, for any
\begin{equation}\label{e:a(n),b}
\textnormal{$\{p(n)\}_{n \in \bbN} \subseteq \bbN$, \hspace{1mm}dyadic $\{a(n)\}_{n \in \bbN}$ \hspace{1mm}and \hspace{1mm}scalar $b \in \bbR$,}
\end{equation}
multiplication by $\frac{a(n)^{-b}}{\sqrt {p(n)}}$ converts the system \eqref{e:mathcal_Y(t)=mathcal_P*mathcal_X(t)+mathcal_Z(t)} into the format \eqref{e:Y(t)}, where $Y(t):=\tfrac{a(n)^{-b}}{\sqrt{p(n)}}  \mathcal Y(t)$, $\mathbf P(n): =\frac{1}{\sqrt{p(n)} }\boldsymbol{\mathcal P}(n)$,
\begin{equation}\label{e:Z(t)=residual_factor_model}
Z(t):= \frac{a(n)^{-b}}{\sqrt {p(n)}}\mathcal Z(t)~\textnormal{ and } ~X(t) := a(n)^{-b}{\mathcal X}(t).
\end{equation}
It is clear that $ \{{\mathbf P}(n)\}_{n \in \bbN}$ satisfies ($A5$). Now, assume   $N_\psi \geq 2$ satisfies $N_\psi> d+\frac12$ and that, in addition to \eqref{e:p(n),a(n)_conditions}, $n a(n)^{-2\alpha}=O(1)$. Proposition C.1, $(iii)$, of Abry et al.\ \cite{abry:boniece:didier:wendt:2023:extended} implies that there is a choice of \eqref{e:a(n),b} for which ($A4$) holds with $\frac{p(n)}{n/a(n)}\to c>0$, and also such that the associated wavelet random matrices $\mathbf{W}_{Z}(a(n)2^j)$ and
\begin{equation}\label{e:W_X(a2^j)_factor_model_example}
\mathbf{W}_{ X}(a(n)2^j)
\end{equation}
satisfy assumptions $(A2)$ and $(A3)$, respectively. In the latter case, as shown in the proposition, the eigenvalues of the scaling matrix $\widetilde{{\mathbf H}}$  (in place of ${\mathbf H}$ as in \eqref{e:H=PHdiag(h1,...,hn)P^(-1)H}) are given by
\begin{equation}\label{e:h-tilde_q=h_q-b>1/2}
\widetilde h_q=h_q-b >1/2, \quad q=1,\ldots, r.
\end{equation}
Consequently, since $\mathbf{W}_{\mathcal Y}(a(n)2^j) = \frac{a(n)^{-2b}}{p(n)}\mathbf W(a(n)2^j)$, Theorem \ref{t:lim_n_a_times_lambda/a^(2h+1)} implies that, as $n \rightarrow \infty$,
$$
\frac{\lambda_{p-r+q}\big(\mathbf{W}_{\mathcal Y}(a(n)2^j)\big)}{p(n)a(n)^{2{ h}_q+ 1}} =\frac{\lambda_{p-r+q}\big(\mathbf{W}(a(n)2^j)\big)}{a(n)^{2{\widetilde h}_q+ 1}} \stackrel{\bbP}\to  \xi_q(2^j) , \quad q=1,\ldots r.
$$
Furthermore, by Theorem \ref{t:asympt_normality_lambdap-r+q}, under \eqref{e:h1<...<hr} the fluctuations of the log-eigenvalues of $\mathbf W_{\mathcal Y}(a(n)2^j)$ are characterized by the limit \eqref{e:asympt_normality_lambda2}, notably with the same rate $\sqrt{n_{a,j}}$.

These developments may be further extended so as to include ``weak factor'' assumptions (e.g., Bai and Ng \cite{bai:ng:2023}), among other possible factor modeling contexts.

One noteworthy consequence of these calculations is that the scaling behavior of the ``factors'' $\mathcal X(t)$ stands out in the wavelet spectral (eigenvalue) domain without any direct knowledge of the ``factors'' themselves.  This is roughly analogous to a similar phenomenon in factor model inference, wherein detection of a factor (or estimation of its variance) is relatively easier in comparison to estimating the factor direction itself. In contrast, in the related statistical inference literature concerning factor models with long memory, multi-step methods have been used where factors themselves are estimated at a first step (e.g.,  Cheung \cite{cheung:2022}, Ergemen and Rodr\'{i}guez-Caballero \cite{ergemen:rodriguez-caballero:2023}). In a related vein, eigenanlaysis approaches  in the presence of multiple scaling laws have also been proposed in the cointegration literature (e.g., Zhang et al.\ \cite{zhang:robinson:yao:2018}).
\vspace{0.5ex}
\end{example}

\begin{remark}\label{r:on_prop_Gaussian_framework}
Supposing $X$ is an ofBm --  or even that it has (first order) stationary increments -- is \textit{not} crucial for assumption $(A3)$ to hold. It is well known that wavelet frameworks are suitable for stationary increment processes of any order. This is so because, for such processes, wavelet coefficients are, in general, stationary as long as the chosen number of vanishing moments of the underlying wavelet basis (see \eqref{e:N_psi}) is sufficiently large. This topic has been broadly explored in the literature (e.g., Flandrin \cite{flandrin:1992}, Wornell and Oppenheim \cite{wornell:oppenheim:1992}, Veitch and Abry \cite{veitch:abry:1999}, Moulines et al.\ \cite{moulines:roueff:taqqu:2008}, Roueff and Taqqu \cite{roueff:taqqu:2009}; see Abry et al.\ \cite{abry:didier:li:2019} on the multivariate stochastic processes). In particular, condition \eqref{e:H_is_diagonalizable} on scaling eigenvalues allows for stationary fractional processes $X$ exhibiting long-range dependence (see, for instance, Embrechts and Maejima \cite{embrechts:maejima:2002}, Pipiras and Taqqu \cite{pipiras:taqqu:2017}).
\end{remark}

\subsection{Non-Gaussian instances: a short discussion}\label{s:non-Gaussian}

Characterizing the wavelet domain behavior of high-dimensional, non-Gaussian, second-order and fractional frameworks is a very broad topic that lies well outside the scope of this paper. In this section, we restrict ourselves to discussing certain non-Gaussian instances so as to help illustrate the fact that the wavelet domain properties established in Section \ref{s:main} do not fundamentally require the system \eqref{e:Y(t)} to be Gaussian.

\begin{example}\label{ex:non-Gaussian_linear}
Let $X = \{X(t)\}_{t \in \bbZ}$ be a possibly non-Gaussian, $\bbR^r$-valued stochastic processes made up of independent linear fractional processes with finite fourth moments. Then, based on the framework constructed in Roueff and Taqqu \cite{roueff:taqqu:2009}, we can show that the associated random matrix $\widehat{{\mathbf B}}_a(2^j) \in {\mathcal S}_{>0}(r,\bbR)$ satisfies conditions \eqref{e:sqrt(K)(B^-B)->N(0,sigma^2)} and \eqref{e:|bfB_a(2^j)-B(2^j)|=O(shrinking)} (i.e., in assumption $(A3)$) under conditions $(W1-W3)$ and mild additional assumptions on the wavelet $\psi$ and on the process $X$. This claim is made precise in Proposition C.2, Abry et al.\ \cite{abry:boniece:didier:wendt:2023:extended}.

Even though -- in this particular case -- the $r$ components of $X$ are assumed independent, note that statistically estimating the $r$ univariate scaling exponents based on the associated model \eqref{e:Y(t)} is still, in general, a nontrivial problem. This is so due to the presence of the unknown coordinates matrix ${\mathbf P}(n){\mathbf P}_H = {\mathbf P}(n)$, as well as of the high-dimensional noise component $Z$. In fact, mathematically speaking, this model is a high-dimensional version of the so-called \textit{blind source separation} problems, which are of great interest in the field of signal processing (e.g., Naik and Wang \cite{naik:wang:2014}; in a fractional context, see Abry et al.\ \cite{abry:didier:li:2019}).
\end{example}

\begin{example}\label{ex:Haar_sub-Gaussian}
Recall that a distribution is called \textit{sub-Gaussian} when its tails are no heavier than those of the Gaussian distribution (Vershynin \cite{vershynin:2018}, Proposition 2.5.2). Sub-Gaussian distributions form a broad family that includes the Gaussian distribution itself, as well as compactly supported distributions, for example. Suppose the noise process $\{Z(t)\}_{t \in \bbZ}$ consists of (discrete-time) i.i.d.\ sub-Gaussian observations. Consider the Haar wavelet framework, where the wavelet coefficients are computed by means of Mallat's iterative procedure \eqref{e:Mallat}. Then, for any $j \in \bbN \cup \{0\}$, the nonzero coefficients in the sequences $\{h_{j,2^{j}k - \ell}\}_{\ell \in \bbZ}$ and $\{h_{j,2^{j}k' - \ell}\}_{\ell \in \bbZ}$ do not overlap for $k \neq k'$ (cf.\ \eqref{e:disc2}). For this reason, the wavelet coefficients are independent at any fixed scale. Although we do not provide a proof due to space constraints, it is then possible to use traditional concentration of measure techniques to show that, under ($A4$), the wavelet random matrix ${\mathbf W}_{Z}$ satisfies condition \eqref{e:assumptions_WZ=OP(1)} in assumption $(A2)$. The study of the properties of wavelet random matrices ${\mathbf W}_{Z}$ under other wavelet bases or other classes of non-Gaussian observations is currently a topic of research.
\end{example}

\begin{remark}\label{r:non-Gaussian_B-hat}
In general, depending on the properties of the latent process $X$, the random matrix $\widehat{{\mathbf B}}_a(2^j)$ may not be asymptotically Gaussian, i.e., assumption $(A3)$ may not hold. In the univariate context ($r=1$), see, for instance, Bardet and Tudor \cite{bardet:tudor:2010}, Clausel et al.\ \cite{clausel:roueff:taqqu:tudor:2014:waveletestimation}. Nevertheless, for $r > 1$, the literature still lacks broad characterizations of conditions under which $\widehat{{\mathbf B}}_a(2^j)$ (or ${\mathbf W}_X(a(n)2^j)$) is asymptotically non-Gaussian.
\end{remark}

\section{Proofs of the main results}\label{s:proof_main_results}

In this section, by redefining $\mathbf P(n)$ if necessary we may assume without loss of generality that $\mathbf P_H=I$. In addition, whenever convenient we omit dependence on $n$ and write
\begin{equation}\label{e:P(n)PH_equiv_P}
p = p(n), \quad a = a(n), \quad {\mathbf P}={\mathbf P}(n) = {\mathbf P}(n){\mathbf P}_H.
\end{equation}
In \eqref{e:P(n)PH_equiv_P}, the column vectors of ${\mathbf P}$ are denoted by ${\mathbf p}_{\ell}={\mathbf p}_{\ell}(n)$, $\ell = 1,\hdots,r$.

In the proofs, we fix an arbitrary $q\in\{1,\ldots,r\}$ and focus on the associated rescaled eigenvalue $\frac{\lambda_{p-r+q}(\W(a(n)2^j))}{a(n)^{2h_q+1}}$. We define the associated sets of indices
\begin{equation} \label{e:def_indexsets}
 \begin{array}{lll}
\mathcal{I}_- := \{\ell: h_\ell < h_q\}, \quad \mathcal{I}_0 := \{\ell: h_\ell =h_q\}, \quad \mathcal{I}_+ := \{\ell: h_\ell > h_q\}.
 \end{array}
 \end{equation}
Note that $\mathcal{I}_-$ and $\mathcal{I}_+$ are possibly empty. Also write their respective cardinalities as
\begin{equation}\label{e:r1,r2,r3}
r_1:= \textnormal{card}(\mathcal{I}_-), \quad r_2:= \textnormal{card}(\mathcal{I}_0) \geq 1, \quad r_3:= \textnormal{card}(\mathcal{I}_+).
\end{equation}

Throughout this section, we often make use of the asymptotic notation \eqref{e:OP(1),OmegaP(1)} to denote residual random matrix terms, where convergence or boundedness in probability or deterministically refers to their \textit{spectral norms}. In particular, in the depiction of the residual terms, it will be notationally convenient to build upon condition \eqref{e:assumptions_WZ=OP(1)} and relation \eqref{e:|a(n)(-D)W_X,Z(a(n)2j))|=OP(1)} by writing
\begin{equation}\label{e:WZ=OP(1),a^(-h-1/2I)WXZ_repeat}
\begin{gathered}
{\mathbf W}_Z(a(n)2^{j})=O_{\bbP}(1), \quad \bbE {\mathbf W}_Z(a(n)2^{j})=O(1),\\
a(n)^{-{\mathbf H}-(1/2){\mathbf I}}{\mathbf W}_{X,Z}(a(n)2^{j})=O_{\bbP}(1), \quad a(n)^{-{\mathbf H}-(1/2){\mathbf I}}\bbE {\mathbf W}_{X,Z}(a(n)2^{j})= {\mathbf 0},
\end{gathered}
\end{equation}
where the last equality follows from \eqref{e:EW_X,Z(a(n)2^j)=0}. In \eqref{e:WZ=OP(1),a^(-h-1/2I)WXZ_repeat}, the matrix dimensions of the $O(1)$, $O_\bbP(1)$ and ${\mathbf 0}$ terms are implicit.

\subsection{Proving Theorem \ref{t:lim_n_a_times_lambda/a^(2h+1)}}\label{s:proving_theorem_3.1}

Fix $q \in \{1,\hdots,r\}$ and define the diagonal matrix
\begin{equation}\label{e:h-bf=diag(h1,...,hr)}
\mathbf h = \text{diag}(h_1,\ldots, h_r), \quad -1/2 < h_1\leq \ldots \leq h_r < \infty.
\end{equation}
Bearing in mind \eqref{e:P(n)PH_equiv_P}, we can use \eqref{e:B-hat_a(2^j)}, \eqref{e:W=PWXP*+WZ+PWXZ+WXZ*P*}, \eqref{e:WZ=OP(1),a^(-h-1/2I)WXZ_repeat} and \eqref{e:h-bf=diag(h1,...,hr)} to write
\begin{equation}\label{e:wave_RM_P=I}%
\frac{{\mathbf W}(a(n)2^j)}{a(n)^{2h_q+1}}  = \underbrace{\frac{{\mathbf P}(n){\mathbf W}_X(a(n)2^j){\mathbf P}^*(n)}{a(n)^{2h_q+1}} }_{\textnormal{main scaling term}}
\end{equation}
$$
+ \underbrace{\frac{{\mathbf W}_Z(a(n)2^j)}{a(n)^{2h_q+1}}
 + \frac{{\mathbf P}(n){\mathbf W}_{X,Z}(a(n)2^j)}{a(n)^{2h_q+1}}+ \frac{{\mathbf W}^*_{X,Z}(a(n)2^j){\mathbf P}^*(n)}{a(n)^{2h_q+1}}}_{\textnormal{residual}}
$$
\begin{equation*}%
= \underbrace{ \frac{{\mathbf P}(n) a(n)^{\mathbf h} \widehat{{\mathbf B}}_a(2^j) a(n)^{\mathbf h}{\mathbf P}^*(n) }{a(n)^{2h_q}} }_{\textnormal{main scaling term}} +
\underbrace{\frac{O_{\bbP}(1)}{a(n)^{2h_q+1}}+
\frac{{\mathbf P}(n)a(n)^{{\mathbf h}}O_{\bbP}(1)}{a(n)^{2h_q+1/2}}+\frac{O^*_{\bbP}(1)a(n)^{{\mathbf h}}{\mathbf P}^*(n)}{a(n)^{2h_q+1/2}}}_{\textnormal{residual}}
\end{equation*}
\begin{equation}\label{e:rescaled_W_conv_in_prob}
=: \underbrace{ \frac{{\mathbf P}(n) a(n)^{\mathbf h} \widehat{{\mathbf B}}_a(2^j) a(n)^{\mathbf h}{\mathbf P}^*(n) }{a(n)^{2h_q}} }_{\textnormal{main scaling term}} + \underbrace{ \Xi_q(n) }_{\textnormal{residual}}.
\end{equation}

Likewise, based on relations \eqref{e:B_a(2^j)=EB-hat_a(2^j)}, \eqref{e:EW_X,Z(a(n)2^j)=0}, \eqref{e:assumptions_WZ=OP(1)} and \eqref{e:WZ=OP(1),a^(-h-1/2I)WXZ_repeat}, the deterministic counterpart of \eqref{e:rescaled_W_conv_in_prob}
can be re-expressed as
\begin{equation}\label{e:rescaled_EW_original}
\frac{\bbE{\mathbf W}(a(n)2^j)}{a(n)^{2h_q+1}} = \underbrace{\frac{{\mathbf P}(n)a(n)^{{\mathbf h}} {\mathbf B}_a(2^j)a(n)^{{\mathbf h}}{\mathbf P}^*(n)}{a(n)^{2h_q}}}_{\textnormal{main scaling term}}
+ \underbrace{\frac{O(1)}{a(n)^{2h_q+1}}}_{\textnormal{residual}}.
\end{equation}

We break up the analysis of the convergence of the top eigenvalues of the wavelet random matrix ${\mathbf W}(a(n)2^j)/a(n)^{2h_q+1}$ as in \eqref{e:wave_RM_P=I} into two main results, namely, Theorem \ref{t:lim_n_a_times_lambda/a^(2h+1)} itself and Proposition \ref{p:conv_w-by-w_rescaled_eigenvalues}. In Theorem \ref{t:lim_n_a_times_lambda/a^(2h+1)}, the proof of the main claim \eqref{e:lim_n_a*lambda/a^(2h+1)} contains the backbone of the overall mathematical framework, i.e., a two-step argument with random sub-subsequences.

 On the other hand, Proposition \ref{p:conv_w-by-w_rescaled_eigenvalues} -- which is assumed in the proof of Theorem \ref{t:lim_n_a_times_lambda/a^(2h+1)} -- contains the bulk of the technical argument. The proposition pertains to subsequences $\frac{\lambda_{p-r+\ell}\big(\W(a(n')2^j)\big)}{a(n')^{2h_\ell+1}}$ and involves the crucial role played by wavelet \textit{eigenvectors}. The main difficulty involved in proving Proposition \ref{p:conv_w-by-w_rescaled_eigenvalues} -- and, ultimately, \eqref{e:lim_n_a*lambda/a^(2h+1)} -- lies in dealing with potentially divergent terms in the expressions for rescaled eigenvalues.  Specifically, it is possible that $\|\Xi_q(n)\|\stackrel \bbP \to \infty$, since%
\begin{equation}\label{e:||Xi_q||_explosive?}
\Big\| \frac{{\mathbf P}(n)a(n)^{{\mathbf h}}}{a(n)^{2h_q+1/2}}\Big\|\to \infty %
\end{equation}
whenever $h_r > 2h_q+1/2$. In other words, even the residual term $\Xi_q(n)$ in \eqref{e:rescaled_W_conv_in_prob} must be \textit{a priori} treated as explosive in norm.%

This issue is not restricted to the residual term. In fact, given the fixed $q \in \{1,\hdots,r\}$, let $\ell \in {\mathcal I}_0$. Now let
\begin{equation}\label{e:u_p-r+ell(n)_def}
{\mathbf u}_{p-r+\ell} := {\mathbf u}_{p-r+\ell}(n) = {\mathbf u}_{p-r+\ell}(n,\omega)
\end{equation}
be a (random) unit eigenvector associated with the $(p-r+\ell)$-th eigenvalue of the random matrix ${\mathbf W}(a(n)2^j)$. %
Also, rewrite entry-wise
\begin{equation}\label{e:B-hat_a(2^j)_entry-wise}
\widehat{{\mathbf B}}_a(2^j)=\big(\widehat{b}_{\ell_1, \ell_2} \big)_{\ell_1,\ell_2=1,\ldots,r}.
\end{equation}
After post- and pre-multiplying \eqref{e:rescaled_W_conv_in_prob} by the vector \eqref{e:u_p-r+ell(n)_def} and its transpose, respectively, we can re-express the rescaled wavelet eigenvalue as
$$
\frac{\lambda_{p-r+\ell}\big(\W(a(n)2^j)\big)}{{a(n)^{2h_q+1}}}
$$
\begin{equation}\label{e:rescaled_eigen_p-r+ell_quad_form_in_terms_of_scaling_terms}
= {\mathbf u}^*_{p-r+\ell}\underbrace{{\mathbf P} a^{\mathbf h -  h_q{\mathbf I}} \widehat{{\mathbf B}}_a(2^j)
a^{\mathbf h -  h_q{\mathbf I}}{\mathbf P}^*}_{\textnormal{main scaling term}} {\mathbf u}_{p-r+\ell} + {\mathbf u}^*_{p-r+\ell}  \underbrace{\Xi_q(n)}
_{\textnormal{residual}}{\mathbf u}_{p-r+\ell}
\end{equation}
$$
= \sum_{\ell_1 \in {\mathcal I}_0} \sum_{\ell_2 \in {\mathcal I}_0} \widehat{b}_{\ell_1, \ell_2} \hspace{0.5mm}\langle {\mathbf p}_{\ell_1},{\mathbf u}_{p-r+\ell}\rangle  \hspace{0.5mm} \langle {\mathbf p}_{\ell_2},{\mathbf u}_{p-r+\ell} \rangle
$$
$$
+ 2  \sum_{\ell_1 \in {\mathcal I}_0}\sum_{\ell_2 \in {\mathcal I}_+}\widehat{b}_{\ell_1 , \ell_2} \hspace{0.5mm}\langle {\mathbf p}_{\ell_1},{\mathbf u}_{p-r+\ell}\rangle \hspace{0.5mm} \underbrace{a^{h_{\ell_2}-h_{q}}\langle {\mathbf p}_{\ell_2},{\mathbf u}_{p-r+\ell}\rangle }_{\text{possibly divergent}}
$$
\begin{equation}\label{e:rescaled_eigen_p-r+ell_quad_form}
+ \sum_{\ell_1 \in  {\mathcal I}_+} \sum_{\ell_2 \in {\mathcal I}_+} \widehat{b}_{\ell_1, \ell_2}\hspace{1mm} \underbrace{a^{h_{\ell_1}-h_{q}}\langle {\mathbf p}_{\ell_1},{\mathbf u}_{p-r+\ell}\rangle }_{\text{possibly divergent}}\hspace{1mm}\underbrace{a^{h_{\ell_2}-h_{q}}\langle {\mathbf p}_{\ell_2},{\mathbf u}_{p-r+\ell}\rangle }_{\text{possibly divergent}}
+o_{\bbP}(1).
\end{equation}
The projected main scaling term in \eqref{e:rescaled_eigen_p-r+ell_quad_form_in_terms_of_scaling_terms} yields all three double summation terms in \eqref{e:rescaled_eigen_p-r+ell_quad_form} (plus some vanishing terms). Also, to fix ideas, we take for granted the implicit claim that ${\mathbf u}^*_{p-r+\ell}  \Xi_q(n){\mathbf u}_{p-r+\ell} = o_{\bbP}(1)$ (\textbf{n.b.}: in general, ${\mathbf u}^*_{p-r+\ell}  \Xi_q(n){\mathbf u}_{p-r+\ell} \neq \|\Xi_q(n)\|$; cf.\ \eqref{e:||Xi_q||_explosive?}). So, note that, in principle, there is no guarantee that angular terms such as $\langle {\mathbf p}_{\ell_1},{\mathbf u}_{p-r+\ell}\rangle$ shrink fast. Since $h_{\ell_1} > h_{q}$ for $\ell_1 \in {\mathcal I}_+$, then the terms marked my braces in \eqref{e:rescaled_eigen_p-r+ell_quad_form} represent  potentially explosive terms (in probability).

For expository purposes, we first state Proposition \ref{p:conv_w-by-w_rescaled_eigenvalues} and then provide some interpretation.
\begin{proposition}\label{p:conv_w-by-w_rescaled_eigenvalues}
Suppose the assumptions of Theorem \ref{t:lim_n_a_times_lambda/a^(2h+1)} hold. Fix $q \in\{1,\ldots,r\}$, together with its associated indices \eqref{e:def_indexsets}.  For each $n$, let
$\p_{r+1}(n),\ldots,\p_p(n)$
be an orthonormal basis for $\textnormal{nullspace}\{{\mathbf P}^*(n)\}$.  Let $n'\in\bbN'$ be any subsequence along which
\begin{equation}\label{e:subseq_condition_1_top}
\sum_{\ell=1}^r\sum^{p(n')}_{i=r+1}\langle {\mathbf p}_{  i}(n), {\mathbf u}_{p-r+\ell}(n')\rangle^2 \to  0\quad \text{a.s.},
\end{equation}
and also
\begin{equation}\label{e:subseq_condition_2_top}
\sum_{i \in \mathcal I_+}\sum_{\ell \in \mathcal I_0}\langle {\mathbf p}_{  i}(n'),{\mathbf u}_{p-r+\ell}(n')\rangle^2 + \sum_{i \in \mathcal I_+\cup\mathcal I_0}\sum_{\ell \in \mathcal I_-}\langle {\mathbf p}_{  i}(n'),{\mathbf u}_{p-r+\ell}(n')\rangle^2 \to 0 \quad \text{a.s.},
\end{equation}
\begin{equation}\label{e:B-hat(2j)->B(2j)_prop}
  \max_{\ell \in {\mathcal I}_0}\big|{\mathbf u}^*_{p-r+\ell}(n')  \Xi_q(n'){\mathbf u}_{p-r+\ell}(n')|\to 0, \quad \widehat{\B}_{a}(2^j) \to \B(2^j) \quad \text{a.s.}%
\end{equation}
Then, there exists a deterministic matrix  ${\boldsymbol \Lambda} = {\boldsymbol \Lambda}(2^j)$, independent of the choice of $n'$, such that
\begin{equation}\label{e:lim_nu_rescaled_eigenvalue}
\lim_{n'\to\infty}\frac{\lambda_{p-r+\ell}\big(\W(a(n')2^j)\big)}{a(n')^{2h_\ell+1}}=\lambda_{r-r_2+\ell}(\boldsymbol\Lambda)\quad \text{a.s.,}\quad \ell\in\mathcal I_0,
\end{equation}
where $\mathcal I_0$ is given by \eqref{e:def_indexsets}. In particular, the limits in \eqref{e:lim_nu_rescaled_eigenvalue} are constant a.s.
\end{proposition}

In Proposition \ref{p:conv_w-by-w_rescaled_eigenvalues}, relations \eqref{e:subseq_condition_1_top}--\eqref{e:B-hat(2j)->B(2j)_prop} are required to hold almost surely. The existence of a subsequence on which all statements hold is guaranteed in Lemmas \ref{l:|<p3,uq(n)>|a(n)^{h_3-h_1}=O(1)} and \ref{l:sum<pi(n),up-r+q(n)>2-o(a-varpi)_first}, where the convergence statements are shown to hold \textit{in probability} along the full sequence. For this reason, it will be useful to consider the full sequence $n \in \bbN$ in the following discussion.%

Condition \eqref{e:subseq_condition_1_top} expresses that
\begin{equation}\label{e:subseq_condition_1_top_explanation}
\textnormal{span}_{\ell =1,\hdots,r}\{{\mathbf u}_{p-r+\ell}(n)\} \textnormal{ and }\textnormal{span}\{{\mathbf P}(n)\} \textnormal{ asymptotically align}.
\end{equation}
Equivalently, the top $r$ wavelet eigenvectors jointly align, in the limit, with the subspace of $\bbR^p$ containing the image of the scaling process $X$. From an eigenvector perspective, this is why the top $r$ wavelet eigenvalues display scaling behavior.

Condition \eqref{e:subseq_condition_2_top} states that, again in angular terms,
\begin{equation}\label{e:subseq_condition_2_top_explanation}
\textnormal{span}_{\ell \in {\mathcal I}_0}\{ {\mathbf u}_{p-r+\ell}(n)\} \textnormal{ and }\textnormal{span}_{i \in {\mathcal I}_+}\{{\mathbf p}_{i}(n)\} \textnormal{ are asymptotically orthogonal.}
\end{equation}
It should be noticed that the latter space is associated with directions corresponding to scaling exponents strictly larger than $h_q$, where potentially explosive terms appear (see \eqref{e:rescaled_eigen_p-r+ell_quad_form}). Namely, relation \eqref{e:subseq_condition_2_top_explanation} expresses that the wavelet eigenvectors ${\mathbf u}_{p-r+\ell}(n)$, $\ell \in {\mathcal I}_0$, eventually turn away from these directions, which ultimately prevents divergent behavior in \eqref{e:rescaled_eigen_p-r+ell_quad_form}.

Likewise, \eqref{e:subseq_condition_2_top} also states that
\begin{equation}\label{e:subseq_condition_2_top_explanation-2}
\textnormal{span}_{\ell \in {\mathcal I}_-}\{ {\mathbf u}_{p-r+\ell}(n)\}\textnormal{ and }\textnormal{span}_{i \in {\mathcal I}_0 \cup {\mathcal I}_+}\{{\mathbf p}_{i}(n)\} \textnormal{ are asymptotically orthogonal},
\end{equation}
expressing, analogously, that eigenvectors corresponding to exponents strictly less than $h_q$  turn away from spaces containing scaling behavior associated with exponents equal to $h_q$ or larger.
Figure \ref{fig:eigenvectors} provides a schematic illustration of relations \eqref{e:subseq_condition_1_top_explanation}--\eqref{e:subseq_condition_2_top_explanation-2}.

In turn, even though $\|\Xi_q(n)\|$ is potentially explosive (cf.\ \eqref{e:||Xi_q||_explosive?}), the first expression in \eqref{e:B-hat(2j)->B(2j)_prop} states that the residual $\Xi_q(n)$ is negligible \textit{along the directions given by the eigenvectors} ${\mathbf u}_{p-r+\ell}(n)$, $\ell \in {\mathcal I}_0$, of ${\mathbf W}(a(n)2^j)$.

Lastly,  the second expression in \eqref{e:B-hat(2j)->B(2j)_prop} is a simple statement about the convergence of the auxiliary wavelet random matrix $\widehat{\B}_{a}(2^j) $.

\begin{figure}
    \centering
    \includegraphics[scale=0.25]{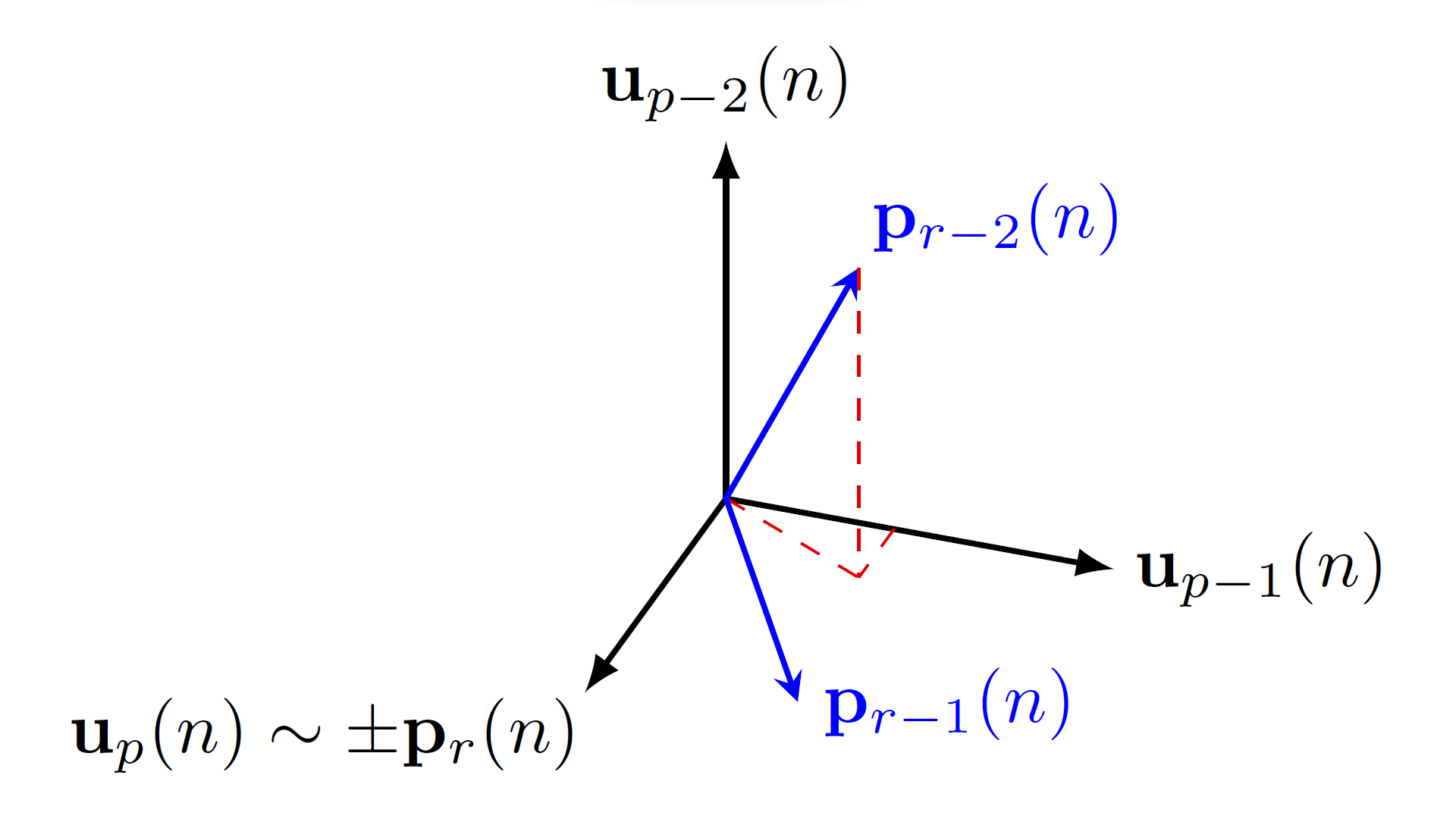}
 \caption{ \textbf{Schematic representation of the behavior of wavelet eigenvectors vis-\`{a}-vis the coordinate vectors in the three-way limit \eqref{e:three-fold_lim}.} The plot depicts the case $r = 3$ and $r_1 = r_2 = r_3 = 1$ (i.e., $h_1 < h_2 < h_3$). It visually represents the asymptotic behavior of the \textnormal{angles} among the high-dimensional vectors. In the limit, we observe that $\textnormal{span}_{\ell=1,2,3}\{{\mathbf u}_{p-r+\ell}(n)\}\sim \textnormal{span}_{\ell=1,2,3}\{{\mathbf p}_{\ell}(n)\}$, $|\langle {\mathbf u}_{p-1}(n),{\mathbf p}_r(n)\rangle| \sim 0$ and $\max\big\{|\langle {\mathbf u}_{p-2}(n),{\mathbf p}_r(n)\rangle|, |\langle {\mathbf u}_{p-2}(n),{\mathbf p}_{r-1}(n)\rangle| \big\}\sim 0$. In particular, $\langle {\mathbf u}_p(n),{\mathbf p}_r(n)\rangle^2 \sim 1$ and, \textnormal{approximately}, ${\mathbf u}_{p-1}(n) \in \textnormal{span}\{{\mathbf p}_{r-1}(n),{\mathbf p}_r(n)\}$. In the general case, the former three relations are given by \eqref{e:subseq_condition_1_top_explanation}, \eqref{e:subseq_condition_2_top_explanation} and \eqref{e:subseq_condition_2_top_explanation-2}, respectively.}
    \label{fig:eigenvectors}
\end{figure}

\vspace{2mm}
Assuming Proposition \ref{p:conv_w-by-w_rescaled_eigenvalues}, we are now in a position to establish Theorem \ref{t:lim_n_a_times_lambda/a^(2h+1)}.\\

{\noindent \sc Proof of Theorem \ref{t:lim_n_a_times_lambda/a^(2h+1)}}: Statement \eqref{e:lim_lambda_p-r(W)} is a consequence of expression \eqref{e:lambdap-r=O_P(1)} in Lemma \ref{l:gutted_log_eig_consistency}. So, assume for the moment that \eqref{e:lim_n_a*lambda/a^(2h+1)} holds. To establish the scaling relationship \eqref{e:xi-i0_scales}, fix an arbitrary $q\in\{1,\ldots,r\}$ and let $\widetilde a(n) = a(n)2^j$. By \eqref{e:lim_n_a*lambda/a^(2h+1)},
$$
\xi_q(2^j)=\plim_{n \rightarrow \infty}\frac{\lambda_{p-r+q}\big(\mathbf{W}(a(n)2^j)\big)}{a(n)^{2h_q+ 1}} =\plim_{n \rightarrow \infty}\frac{\lambda_{p-r+q}\big(\mathbf{W}(\widetilde a(n))\big)}{(\widetilde a(n)2^{-j})^{2h_q+ 1}} = 2^{j(2h_q+1)}\xi_q(1),
$$
as we wanted to show.

So, we now prove \eqref{e:lim_n_a*lambda/a^(2h+1)}. Let $\bbN'$ be any subsequence of $\bbN$. As a consequence of condition \eqref{e:sqrt(K)(B^-B)->N(0,sigma^2)}, $\widehat{\B}_{a}(2^j)\stackrel \bbP\to \B(2^j)$ as $n' \rightarrow \infty$. Also, recall that, for any $\ell \in {\mathcal I}_0$, $h_{\ell} = h_{q}$. Then, again for $\ell \in {\mathcal I}_0$, $\Xi_{\ell}(n') = \Xi_q(n')$ (cf.\ \eqref{e:rescaled_W_conv_in_prob}). Thus, Lemma \ref{l:|<p3,uq(n)>|a(n)^{h_3-h_1}=O(1)}, $(ii)$, implies that $\max_{\ell \in {\mathcal I}_0}\big|{\mathbf u}^*_{p-r+\ell}(n') \Xi_q(n'){\mathbf u}_{p-r+\ell}(n') \big| \stackrel \bbP \to 0$ as $n'  \rightarrow \infty$. Now note that, by Lemma \ref{l:|<p3,uq(n)>|a(n)^{h_3-h_1}=O(1)}, $(i)$, the eigenvector ${\mathbf u}_{p-r+\ell}(n')$ is asymptotically orthogonal to a coordinate vector ${\mathbf p}_i(n')$ associated with a larger scaling exponent $h_{i} > h_{\ell}$. In particular, for the fixed $q$ and the associated index sets \eqref{e:def_indexsets}, by considering ${\mathbf u}_{p-r+\ell}(n')$ for either index range $\ell \in {\mathcal I}_0$ or $\ell \in {\mathcal I}_-$, Lemma \ref{l:|<p3,uq(n)>|a(n)^{h_3-h_1}=O(1)}, $(i)$, implies that
$$
\sum_{i \in \mathcal I_+}\sum_{\ell \in \mathcal I_0}\langle {\mathbf p}_{  i}(n'),{\mathbf u}_{p-r+\ell}(n')\rangle^2 + \sum_{i \in \mathcal I_+\cup\mathcal I_0}\sum_{\ell \in \mathcal I_-}\langle {\mathbf p}_{  i}(n'),{\mathbf u}_{p-r+\ell}(n')\rangle^2 \stackrel\bbP \to 0, \quad n' \rightarrow \infty.
$$
Furthermore, by Lemma \ref{l:sum<pi(n),up-r+q(n)>2-o(a-varpi)_first},  $\sum_{\ell=1}^r\sum^{p(n')}_{i=r+1}\langle {\mathbf p}_{  i}(n'), {\mathbf u}_{p-r+\ell}(n')\rangle^2 \stackrel \bbP\to  0.$ Thus, there exists a further subsequence $\bbN''\ni n''$ such that \eqref{e:subseq_condition_1_top}, \eqref{e:subseq_condition_2_top} and \eqref{e:B-hat(2j)->B(2j)_prop} hold along $n''$.
By Proposition \ref{p:conv_w-by-w_rescaled_eigenvalues}, along this same subsequence $n''\in \bbN''$,
\begin{equation*}
\lim_{n''\to\infty}\frac{\lambda_{p-r+\ell}(\W(a(n'')2^j))}{a({n''})^{2h_\ell+1}}=\lambda_{r-r_2+\ell}(\boldsymbol\Lambda) \quad \textnormal{a.s.},\quad \ell\in\mathcal I_0.
\end{equation*}
Hence, \eqref{e:lim_n_a*lambda/a^(2h+1)} holds for
\begin{equation}\label{e:xi_ell(2^j)=lambda_r-r2+ell}
\xi_\ell(2^j) := \lambda_{r-r_2+\ell}(\boldsymbol\Lambda). \quad \Box
\end{equation}

\subsection{Proving Proposition \ref{p:conv_w-by-w_rescaled_eigenvalues}}\label{s:proof_of_Prop_conv_subseq_of_rescaled_eigenvalues}

It remains to establish Proposition \ref{p:conv_w-by-w_rescaled_eigenvalues}. For the reader's convenience, we now provide a short discussion of the proof method.

So, consider a subsequence $\bbN'$ as defined in the assumptions of the proposition. Define the event
\begin{equation}\label{e:event_A}
A=\Big\{\omega: \text{ \eqref{e:subseq_condition_1_top}, \eqref{e:subseq_condition_2_top} and \eqref{e:B-hat(2j)->B(2j)_prop} hold} \Big\}.
\end{equation}
In particular, $\bbP(A)=1$. Hereinafter, we use the expression
\begin{equation*}%
\textnormal{``for each $\omega \in A$ a.s."}
\end{equation*}
to mean ``for each $\omega \in A$ up to intersection with a probability 1 event". Then, we show that, for each $\omega \in A$ a.s., any arbitrary (sub)subsequence
\begin{equation}\label{e:nu_n(w)_contained_in_N'}
\{\nu_n(\omega)\}_{n\in\bbN}\subseteq \bbN'
\end{equation}
contains a refinement (still denoted $\{\nu_n(\omega)\}_{n\in\bbN}$, for notational simplicity) such that
\begin{equation}\label{e:nu_n_mainlimit}
\lim_{n\to\infty}\frac{\lambda_{p-r+\ell}\big(\W(a(\nu_n(\omega))2^j)\big)}{a(\nu_n(\omega))^{2h_\ell+1}}(\omega)=\lambda_{r-r_2+\ell}(\boldsymbol\Lambda),
\end{equation}
where the matrix $\boldsymbol\Lambda$ is deterministic. This establishes the almost sure limit \eqref{e:lim_nu_rescaled_eigenvalue}.

In turn, constructing the refined subsequence $\{\nu_n(\omega)\}_{n\in\bbN}$ over which \eqref{e:nu_n_mainlimit} holds requires four steps, labeled \textbf{(a)--(d)}.
\begin{itemize}
\item [\textbf{(a)}]{ \textbf{Passing from high- to fixed-dimensional coordinates.}\vspace{0.1cm} }

 We use the $QR$ decomposition of ${\mathbf P} = {\mathbf P}(n){\mathbf P}_H = {\mathbf Q}(n){\mathbf R}(n)$ to make a change-of-coordinates from the high-dimensional eigenvectors (notation: ${\mathbf u}_{p-r+\ell}(n) \in \bbR^p$) of $\frac{\W(a(n)2^j)}{a(n)^{2h_q+1}}$ to \textit{fixed}-dimensional \textit{bounded} coordinates (notation: ${\boldsymbol \tau}_{\ell}(n) = {\mathbf Q}^*(n){\mathbf u}_{p-r+\ell}(n) \in \bbR^r$). This is convenient because it allows us to refine the subsequence $\{\nu_n(\omega)\}_{n\in\bbN}$ so as to obtain the almost sure convergence of these coordinates $\{{\boldsymbol \tau}_{\ell}(\nu_n)={\boldsymbol \tau}_{\ell}(\nu_n,\omega): \ell = 1,\hdots,r\} \subseteq \bbR^r$ to a set of possibly random orthonormal vectors $\{{\boldsymbol \tau}_{\ell}={\boldsymbol \tau}_{\ell}(\omega): \ell = 1, \hdots, r\} \subseteq \bbR^r$. \vspace{0.1cm}
\item [\textbf{(b)}] \textbf{Dealing with the potentially explosive terms by replacing each of them with a generic variable $x$.}\vspace{0.1cm}

 Consider the refined subsequence $\{\nu_n(\omega)\}_{n\in\bbN}$ obtained in \textbf{(a)}. Starting from the expression for $\frac{\lambda_{p-r+\ell}\big(\W(a(\nu_n)2^j)\big)}{a(\nu_n)^{2h_q+1}}$, we define two functions
    \begin{equation}\label{e:f-hat_n(x,u)=varphi-hat_n(x,tau)}
    \widehat{f}_{\nu_n}({\mathbf x},{\mathbf u})= \widehat{\varphi}_{\nu_n}({\mathbf x},{\boldsymbol \tau})
    \end{equation}
    that, up to residuals, \textit{express and generalize} the main term on the right-hand side of \eqref{e:rescaled_eigen_p-r+ell_quad_form} in the following two senses.
\begin{itemize}
 \item [$(i)$] In $\widehat{f}_{\nu_n}$ and $\widehat{\varphi}_{\nu_n}$, respectively, eigenvectors are \textit{replaced} by a general argument in high-dimensional (${\mathbf u}$) and fixed-dimensional coordinates (${\boldsymbol \tau}$).
\item [$(ii)$] In both functions, the potentially divergent terms in \eqref{e:rescaled_eigen_p-r+ell_quad_form} are \textit{replaced} by a vector of generic variables ${\mathbf x}$.
\end{itemize}
In particular, if we set ${\mathbf u} = {\mathbf u}_{p-r+\ell}(\nu_n)$, then there exists a vector ${\mathbf x}_{\ell} \in \bbR^{r_3}$ such that we can express
$$
\frac{\lambda_{p-r+ \ell}\big({\mathbf W}(a(\nu_n)2^j)\big)}{a(\nu_n)^{2h_{q}+1}} - {\mathbf u}^*_{p-r+\ell}(\nu_n) \underbrace{ \Xi_q(\nu_n) }_{\textnormal{residual}}{\mathbf u}_{p-r+\ell}(\nu_n) $$
\begin{equation}\label{e:rescaled_eigen_p-r+ell_in_terms_of_fn_and_varphin}
 = \widehat{f}_{\nu_n} \big({\mathbf x}_{\ell},{\mathbf u}_{p-r+\ell}(\nu_n)\big)= \widehat{\varphi}_{\nu_n}\big({\mathbf x}_{\ell},{\boldsymbol \tau}_{\ell}(\nu_n)\big)
\end{equation}
(cf.\ relation \eqref{e:rescaled_eigen_p-r+ell_quad_form_in_terms_of_scaling_terms}). In \eqref{e:rescaled_eigen_p-r+ell_in_terms_of_fn_and_varphin}, the last equality follows from \eqref{e:f-hat_n(x,u)=varphi-hat_n(x,tau)} for ${\boldsymbol \tau} = {\boldsymbol \tau}_{\ell}(\nu_n) = {\mathbf Q}^*(\nu_n){\mathbf u}_{p-r+\ell}(\nu_n) \in \bbR^r$.

    We further define a function $\varphi$ that may be interpreted as a \textit{pointwise limit}
    \begin{equation}\label{e:varphi=lim_varphi-hat_nu-n}
    \varphi= \lim_{n \rightarrow \infty}\widehat{\varphi}_{\nu_n}.
    \end{equation}
\item [\textbf{(c)}] \textbf{Obtaining $\lambda_{r-r_2+\ell}(\boldsymbol\Lambda)$ by minimizing $\varphi$ with respect to the vector ${\mathbf x}$ of generic variables as in (b).} \vspace{0.1cm}

    For the function $\varphi$ obtained in \textbf{(b)} (see \eqref{e:varphi=lim_varphi-hat_nu-n}), let ${\mathbf x}_{*}({\boldsymbol \tau})$ be the \textit{minimizer} of $\varphi$ in ${\mathbf x}$ for a given ${\boldsymbol \tau}$. As it turns out, we can reexpress
    \begin{equation}\label{e:varphi(x*(tau),tau)}
    \varphi\big({\mathbf x}_*(\boldsymbol \tau),\boldsymbol \tau \big) =\boldsymbol \tau^* \boldsymbol \Lambda \boldsymbol \tau \quad \textnormal{for any }\boldsymbol \tau \in \mathcal W: = \text{span}\{\boldsymbol \Lambda\}.
    \end{equation}
    The newly defined matrix $\boldsymbol \Lambda \in {\mathcal S}_{\geq 0}(r,\bbR)$ is deterministic and can be shown to have rank $r_2 = {\mathcal I}_0$. It is based on this deterministic function \eqref{e:varphi(x*(tau),tau)} that the limit of $\frac{\lambda_{p-r+\ell }\big(\W(a(\nu_n)2^j)\big)}{a(\nu_n)^{2h_q+1}}$ will be obtained in step \textbf{(d)}.\vspace{0.1cm}
\item [\textbf{(d)}] \textbf{Squeezing $\frac{\lambda_{p-r+(r_1 + 1)}\big(\W(a(\nu_n)2^j)\big)}{a(\nu_n)^{2h_q+1}}$ based on $\varphi$.} \vspace{0.1cm}

For $\ell = r_1 + 1 =\min (\mathcal I_0)$, we use the (unambiguously bounded) functions $\widehat{f}_{\nu_n}$ and $\widehat{\varphi}_{\nu_n}$ defined in \textbf{(b)} to construct \textit{lower} and \textit{upper bounds} for
\begin{equation}\label{e:lambda_p-r+(r1+1)/a^(2hq+1)_intuition}
\frac{\lambda_{p-r+(r_1 + 1)}\big(\W(a(\nu_n)2^j)\big)}{a(\nu_n)^{2h_q+1}}.
\end{equation}
The convergence of \eqref{e:lambda_p-r+(r1+1)/a^(2hq+1)_intuition} is then obtained by means of \textit{squeezing}.

  To be slightly more precise, let
    \begin{equation}\label{e:w_in_R^r_intuition}
    {\mathbf w} \in \bbR^r
    \end{equation}
    be an eigenvector associated with the smallest nonzero eigenvalue of the matrix ${\boldsymbol \Lambda}$ defined in \textbf{(c)}. Namely, from \eqref{e:varphi(x*(tau),tau)},
\begin{equation}\label{e:varphi(x*(w),w)=<varphi(x*(tau),tau)}
\varphi\big({\mathbf x}_*({\mathbf w}),{\mathbf w}\big) \leq    \varphi\big({\mathbf x}_*(\boldsymbol \tau),\boldsymbol \tau \big) \quad \textnormal{for any }\boldsymbol \tau \in \mathcal W: = \text{span}\{\boldsymbol \Lambda\}.
\end{equation}
    A lower bound for \eqref{e:lambda_p-r+(r1+1)/a^(2hq+1)_intuition} can be naturally constructed based on \eqref{e:rescaled_eigen_p-r+ell_in_terms_of_fn_and_varphin} by minimizing the functions $\widehat{\varphi}_{\nu_n}$, $\varphi$ with respect to each argument. In fact, analogously to \eqref{e:varphi(x*(tau),tau)}, let ${\mathbf x}_{*,\nu_n}({\boldsymbol \tau})$ be the minimizer of $\widehat{\varphi}_{\nu_n}$ in ${\mathbf x}$ for a given ${\boldsymbol \tau}$. Then,
$$
\frac{\lambda_{p-r+(r_1 + 1)}\big(\W(a(\nu_n)2^j)\big)}{a(\nu_n)^{2h_q+1}} \geq \widehat{\varphi}_{\nu_n}\big({\mathbf x}_{*,\nu_n}({\boldsymbol \tau}_{\ell}(\nu_n)),{\boldsymbol \tau}_{\ell}(\nu_n)\big) + o(1)
$$
\begin{equation}\label{e:lower_bound_intuition}
\rightarrow \varphi\big({\mathbf x}_{*}({\boldsymbol \tau}_{\ell}),{\boldsymbol \tau}_{\ell}\big)  \geq \varphi\big({\mathbf x}_{*}({\mathbf w}),{\mathbf w}\big), \quad n \rightarrow \infty.
\end{equation}
In \eqref{e:lower_bound_intuition}, the first inequality follows from \eqref{e:rescaled_eigen_p-r+ell_in_terms_of_fn_and_varphin} and from minimization with respect to ${\mathbf x}$. The second inequality stems from \eqref{e:varphi(x*(w),w)=<varphi(x*(tau),tau)}.

Constructing an upper bound is more elaborate. Note that all the potentially explosive terms appearing in \eqref{e:rescaled_eigen_p-r+ell_quad_form} involve the coordinate vectors ${\mathbf p}_{\ell}(\nu_n)$, $\ell \in {\mathcal I}_+$. Nevertheless, %
it can be shown (Lemma \ref{l:<p,w>=infinitesimal}) that
 one can always find a sequence of unit vectors ${\mathbf v}(\nu_n)$ with the following three key properties.%
\begin{itemize}
\item [$(i)$] ${\mathbf v}(\nu_n) \in \textnormal{span}\{{\boldsymbol {\mathbf u}}_{p-r+(r_1+1)}(\nu_n),\ldots,{\boldsymbol {\mathbf u}}_{p}(\nu_n)\}$.
\item [$(ii)$] Let ${\mathbf x}_*({\mathbf w}) = \big( x_{*,r-r_3+1}({\mathbf w}),\hdots, x_{*,r}({\mathbf w})\big)$ and ${\mathbf w}$ be as in \eqref{e:varphi(x*(tau),tau)} and \eqref{e:w_in_R^r_intuition}, respectively. Then, $\langle \mathbf{p}_i(\nu_n), {\mathbf v}(\nu_n) \rangle = \frac{x_{*,i}({\mathbf w})}{a(\nu_n)^{h_i - h_q}}$, $i \in \mathcal{I}_+$, for large $n$. In particular, such angles shrink fast enough so that they cancel any explosive scaling terms $a(\nu_n)^{h_i - h_q}$ (cf.\ \eqref{e:rescaled_eigen_p-r+ell_quad_form}).
\item [$(iii)$] In ${\boldsymbol\tau}$ coordinates (see \textbf{(a)}), $\bbR^r \ni {\mathbf w}(\nu_n) := {\mathbf Q}^*(\nu_n){\mathbf v}(\nu_n) \rightarrow {\mathbf w}$, $n \rightarrow \infty$. \vspace{2mm}
\end{itemize}

\noindent (\textbf{intuition on the existence of a unit vector ${\mathbf v}(\nu_n)$ satisfying $(i)$, $(ii)$ and $(iii)$}: in Figure \ref{fig:eigenvectors}, even if $\textnormal{span}\{{\mathbf p}_{r-1}(n),{\mathbf p}_{r}(n)\}$ and $\textnormal{span}\{{\mathbf u}_{p-1}(n),{\mathbf u}_{p}(n)\}$ do not exactly coincide, in general one can find a unit vector ${\mathbf v}(n)$ in the latter space displaying any preset and small -- possibly zero -- angular magnitude $\langle {\mathbf p}_{r}(n),{\mathbf v}(n) \rangle$).\vspace{2mm}

We arrive at
$$
\frac{\lambda_{p-r+(r_1 + 1)}\big(\W(a(\nu_n)2^j)\big)}{a(\nu_n)^{2h_q+1}} \leq {\mathbf v}^*(\nu_n)
\frac{\W(a(\nu_n)2^j) }{a(\nu_n)^{2h_q+1}} {\mathbf v}(\nu_n)
$$
\begin{equation}\label{e:high-level_system_ineqs}
= \widehat \varphi_{\nu_n}\big(\mathbf x_*({\mathbf w}),{\mathbf w}(\nu_n)\big) +o(1)\rightarrow \varphi\big(\mathbf x_*({\mathbf w}),{\mathbf w}\big), \quad n \rightarrow \infty.
\end{equation}
In \eqref{e:high-level_system_ineqs}, the inequality, equality and limit are consequences of $(i)$, $(ii)$ and $(iii)$, respectively.

From the lower and upper bounds \eqref{e:lower_bound_intuition} and \eqref{e:high-level_system_ineqs}, we conclude that \eqref{e:lambda_p-r+(r1+1)/a^(2hq+1)_intuition} converges to $\varphi\big(\mathbf x_*({\mathbf w}),{\mathbf w}\big)$ as $n \rightarrow \infty$.

To finish the proof of \eqref{e:lim_nu_rescaled_eigenvalue}, the conclusion is then extended to any $\ell \in {\mathcal I}_0 =\{r_1+1,\ldots,r_1+r_2\}$ by induction.

\end{itemize}

For the sake of illustration, the example right below the proof of Proposition \ref{p:conv_w-by-w_rescaled_eigenvalues} contains steps \textbf{(a)--(d)} described in a simple context.

\begin{remark}\label{r:nonmeasurable}
Generally speaking, the technique of constructing $\omega$-dependent indices \eqref{e:nu_n(w)_contained_in_N'} yields possibly non-measurable sequences such as \eqref{e:nu_n_mainlimit}. Nevertheless, this poses no difficulties in the framework of Proposition \ref{p:conv_w-by-w_rescaled_eigenvalues} since the event of the convergence of such sequence is, indeed, a measurable set. In fact, as shown in the proof of Proposition \ref{p:conv_w-by-w_rescaled_eigenvalues}, it occurs with probability 1. For terminological simplicity, throughout the proof of the proposition, as well as in Lemmas \ref{l:fullrank_limit_P*U}, \ref{l:<p,w>=infinitesimal}, \ref{l:max|<up-r+q(n),pi(n)>|a(n)^(hi-hq)=OP(1)} and \ref{l:supR'max(angles*powerlaws)_bounded_for_subseq}, the word ``random" is applied in the extended sense of $\omega$-dependent constructs, regardless of measurability.
\end{remark}

We are now in a position to prove Proposition \ref{p:conv_w-by-w_rescaled_eigenvalues}.\\

\noindent {\sc Proof of Proposition \ref{p:conv_w-by-w_rescaled_eigenvalues}}:  Let $\bbN'$ be a subsequence as in condition \eqref{e:B-hat(2j)->B(2j)_prop} and consider $\omega$ in the event $A$ as in \eqref{e:event_A}. Take an arbitrary random subsequence $\{\nu_n(\omega)\}_{n \in \bbN}$ as in \eqref{e:nu_n(w)_contained_in_N'}. We now follow the four steps \textbf{(a)--(d)} as described at the beginning of this section (Section \ref{s:proof_of_Prop_conv_subseq_of_rescaled_eigenvalues}).\\

We first tackle \textbf{(a)}. Define the sequence of rectangular random matrices
\begin{equation}\label{e:def_Un}
\mathbf U_r(n) = \mathbf U_r(n,\omega) :=(\mathbf u_{p-r+1}(n)\ldots \mathbf u_p(n))  \in {\mathcal M}(p,r,\bbR),
\end{equation}
where each $\mathbf u_\ell(n)=\mathbf u_\ell(n,\omega)$ is a.s.\ a (random) unit eigenvector associated with the $\ell$--th eigenvalue of $\W(a(n)2^j)$ (cf.\ \eqref{e:u_p-r+ell(n)_def}). Consider the matrix $\mathbf Q(n)$ from the $QR$ decomposition $\mathbf P(n) =  \mathbf P(n){\mathbf P}_H=\mathbf Q(n)\mathbf R(n)$ as in \eqref{e:P(n)=Q(n)R(n)}. Define
\begin{equation}\label{e:T(n)}
\mathbf T(n) =\mathbf T(n,\omega) :=(\boldsymbol \tau_1(n),\ldots,\boldsymbol \tau_r(n)): = \mathbf Q^*(n) \mathbf U_r(n) \in {\mathcal M}(r,\bbR),
\end{equation}
where each $\boldsymbol \tau_i(n) = \boldsymbol \tau_i(n, \omega)$ denotes a (random) column of $\mathbf T(n)=\mathbf T(n,\omega)$. Note that, for $\omega \in A$ a.s., the fixed-dimensional sequence $\{{\mathbf T}(\nu_n,\omega)\}_{n \in \bbN}$ is bounded in norm a.s. So, by applying the Bolzano-Weierstrass theorem for each $\omega \in A$  a.s., we may refine the subsequence $\{\nu_n(\omega)\}_{n\in\bbN}$ further (still denoted $\nu_n$, for notational simplicity) so as to obtain the limit
\begin{equation}\label{e:Gamma_a.s._limit}
\lim_{n\to\infty} \mathbf T(\nu_n)=:\mathbf T = (\boldsymbol \tau_1, \ldots, \boldsymbol \tau_r) \in {\mathcal M}(r,\bbR) \quad \textnormal{a.s.}
\end{equation}
In \eqref{e:Gamma_a.s._limit}, each $\boldsymbol \tau_i = \boldsymbol \tau_i( \omega)$ denotes a column of $\mathbf T=\mathbf T(\omega)$. Moreover, by Lemma \ref{l:fullrank_limit_P*U}, $(i)$, the limiting column vectors $\{{\boldsymbol \tau}_{\ell}: \ell = 1,\hdots,r\} $ are orthonormal a.s.\\

We now turn to step \textbf{(b)}. It will be useful to introduce some notation. Starting from the limit $\widehat{{\mathbf B}}_a(2^j) \rightarrow {\mathbf B}(2^j)$  a.s.\ (see \eqref{e:B-hat(2j)->B(2j)_prop}), recast
\begin{equation}\label{e:B(2^j)=(B_i,ell)}
\widehat {\mathbf B}_a(2^j) =(\widehat {\mathbf B}_{i\ell})_{i,\ell=1,2,3}, \quad {\mathbf B}(2^j) =(\mathbf B_{i\ell})_{i,\ell=1,2,3},
\end{equation}
where ${\mathbf B}_{i\ell}$ and $\widehat {\mathbf B}_{i\ell}$ denote blocks of size $r_i\times r_\ell$.  Similarly, define
\begin{equation}\label{e:P=(P1(n)_P2(n)_P3(n))}
\mathbf P  = \big(\mathbf P_1(n) ~\mathbf P_2(n) ~\mathbf P_3(n)\big), \quad \mathbf R(n) \ = \big(\mathbf R_1(n) ~\mathbf R_2(n)~\mathbf R_3(n)\big),
\end{equation}
where, for $i = 1,2,3$, $\mathbf P_i(n) \in {\mathcal M}(p,r_i,\bbR)$ and $\mathbf R_i(n) \in {\mathcal M}(r,r_i,\bbR)$.  Also, considering the limit ${\mathbf R}(n) \rightarrow {\mathbf R}$ as in \eqref{e:<p1,p2>=c12_2}, recast
\begin{equation}\label{e:R=(R1,R2,R3)}
\mathbf R \ = \big( \mathbf R_1 ~\mathbf R_2~\mathbf R_3 \big),
\end{equation}
where each $\mathbf R_i$ is a full rank matrix of size $r\times r_i$. Now define the diagonal matrix
\begin{equation}\label{e:h1}
\mathbf h_1 = \text{diag}(h_1,\ldots, h_{r_1})
\end{equation}
(cf.\ \eqref{e:h-bf=diag(h1,...,hr)}), and let $\mathbf x\in \mathbb R^{r_3}$, ${\mathbf u} \in \bbS^{p-1}$ and ${\boldsymbol \tau} \in \bbS^{r-1}$. Further define the scalar-valued functions ${\widehat f}_{\nu_n}(\mathbf x,{\mathbf u})$, ${\widehat \varphi}_{\nu_n}(\mathbf x,{\boldsymbol \tau})$ and $\varphi(\mathbf x,{\boldsymbol \tau})$ by means of the relations
\begin{equation} \label{e:def_fn(x,u)}
0 \leq {\widehat f}_{\nu_n}(\mathbf x,{\mathbf u}) =\begin{pmatrix}
a^{\mathbf h_1 - h_q {\mathbf I}}\mathbf P^*_1(\nu_n) {\mathbf u}\\
\mathbf P^*_2(\nu_n) {\mathbf u}\\
\mathbf x
\end{pmatrix}^*\widehat {\mathbf B}_a(2^j) \begin{pmatrix}
a^{\mathbf h_1 - h_q  {\mathbf I}}\mathbf P^*_1(\nu_n) {\mathbf u}\\
\mathbf P^*_2(\nu_n) {\mathbf u}\\
\mathbf x
\end{pmatrix},
\end{equation}
\begin{equation} \label{e:def_phi_n(x,u)}
0 \leq {\widehat \varphi}_{\nu_n}(\mathbf x,{\boldsymbol \tau}) =\begin{pmatrix}
a^{\mathbf h_1 - h_q  {\mathbf I}}\mathbf R^*_1(\nu_n) {\boldsymbol \tau}\\
\mathbf R^*_2(\nu_n) {\boldsymbol \tau}\\
\mathbf x
\end{pmatrix}^*\widehat {\mathbf B}_a(2^j) \begin{pmatrix}
a^{\mathbf h_1 - h_q {\mathbf I}}\mathbf R^*_1(\nu_n) {\boldsymbol \tau}\\
\mathbf R^*_2(\nu_n) {\boldsymbol \tau}\\
\mathbf x
\end{pmatrix}
\end{equation}
and
\begin{equation} \label{e:def_varphi(x,u)}
0 \leq \varphi(\mathbf x,{\boldsymbol \tau}) = \begin{pmatrix}
\mathbf 0\\
{\mathbf R}^*_2 {\boldsymbol \tau}\\
\mathbf x
\end{pmatrix}^* {\mathbf B}(2^j) \begin{pmatrix}
\mathbf 0\\
{\mathbf R}^*_2 {\boldsymbol \tau}\\
\mathbf x
\end{pmatrix}
\end{equation}
(cf.\ \eqref{e:f-hat_n(x,u)=varphi-hat_n(x,tau)} and \eqref{e:rescaled_eigen_p-r+ell_in_terms_of_fn_and_varphin}). In \eqref{e:def_fn(x,u)} and \eqref{e:def_phi_n(x,u)}, for notational simplicity we keep writing $\widehat {\mathbf B}_a(2^j)$ along $\nu_n$. Since $\mathbf B(2^j)\in \mathcal S_{>0}(r,\bbR)$, then $\mathbf B_{33}$ is invertible. Hence, for $\omega \in A$ a.s.\ and large enough $n$, $\widehat{\mathbf B}_{33}$ is also invertible. Thus, for any (large) $n$ and for each fixed $\mathbf u \in \mathcal \bbS^{p-1}$ and ${\boldsymbol \tau} \in \mathcal \bbS^{r-1}$, the functions $\widehat{f}_{\nu_n}(\cdot,\mathbf u)$, $\widehat{\varphi}_{\nu_n}(\cdot,{\boldsymbol \tau})$ and $\varphi(\cdot,{\boldsymbol \tau})$ have unique minimizers ${\mathbf x}_{*,\nu_n}({\mathbf u})$, ${\mathbf x}_{*,\nu_n}({\boldsymbol \tau})$ and $\mathbf x_*({\boldsymbol \tau})$, respectively, in the argument ${\mathbf x}$. In particular, we can express
\begin{equation}\label{e:def_x-star(u)}
{\mathbf x}_{*,\nu_n}({\boldsymbol \tau}) = -\big(\widehat{\mathbf B}^{-1}_{33}\widehat{{\mathbf B}}_{31}a^{\mathbf h_1 - h_q{\mathbf I}} {\mathbf R}_1^*(\nu_n)+ \widehat{\mathbf B}^{-1}_{33}\widehat{\mathbf B}_{32} \mathbf R^*_2(\nu_n)\big) {\boldsymbol \tau}, \quad \mathbf x_*({\boldsymbol \tau}) = - \mathbf B^{-1}_{33}\mathbf B_{32} \mathbf R_2^*{\boldsymbol \tau}.
\end{equation}
For notational simplicity, we further define
\begin{equation}\label{e:f-hat_nu(u),varphi-hat_nu(tau),varphi(tau)}
\widehat f_{\nu_n}(\mathbf u):=\widehat f_{\nu_n}\big(\mathbf x_{*,\nu_n}(\mathbf u),\mathbf u\big), \quad
\widehat \varphi_{\nu_n}({\boldsymbol \tau}):=\widehat \varphi_{\nu_n}\big( \mathbf x_{*,\nu_n}({\boldsymbol \tau}),{\boldsymbol \tau}\big),\quad
\varphi({\boldsymbol \tau}):=\varphi\big(\mathbf x_*({\boldsymbol \tau}),{\boldsymbol \tau}\big).
\end{equation}

In regard to \textbf{(c)}, define the matrix
\begin{equation}\label{e:M=B22-B23*B33(-1)B32}
 \mathbf M:= \mathbf B_{22} - \mathbf B_{23}\mathbf B_{33}^{-1}\mathbf B_{32}  \in \mathcal S(r_2,\bbR).
\end{equation}
From \eqref{e:def_varphi(x,u)}--\eqref{e:M=B22-B23*B33(-1)B32}, we can conveniently write
\begin{equation}\label{e:varphi(u)=varphi(x*(u),u)}
\varphi({\boldsymbol \tau}) = {\boldsymbol \tau}^* \mathbf R_2 \mathbf M\mathbf R_2^* {\boldsymbol \tau}.
\end{equation}
Let
\begin{equation}\label{e:Pi3}
{\mathbf \Pi}_3 \in {\mathcal M}(r,\bbR)
\end{equation}
be the projection matrix onto $\mathbf R_3^\perp$ (see the notation \eqref{e:A^perp}). Bearing in mind the matrix $\mathbf R_2 \mathbf M\mathbf R_2^*$ in \eqref{e:varphi(u)=varphi(x*(u),u)}, we define
\begin{equation} \label{e:def_Lambda}
\boldsymbol\Lambda := {\mathbf \Pi}_3^*\mathbf R_2 \mathbf  M \mathbf R_2^*{\mathbf \Pi}_3\in \mathcal S_{\geq 0}(r,\bbR).
\end{equation}
However, by Lemma \ref{l:M_in_S>0(r2,R)}, $(ii)$,
\begin{equation}\label{e:def_mathcal_W}
{\mathcal W}:=\text{span}\{{\boldsymbol\Lambda}\} = \text{span}\{\mathbf R_2,\mathbf R_3\}\cap \mathbf R_3^\perp.
\end{equation}
Then, $\mathbf \Pi_3 \boldsymbol \tau = \boldsymbol \tau$ for any $\boldsymbol\tau \in \mathcal{W}$. Thus, based on \eqref{e:varphi(u)=varphi(x*(u),u)}, we can recast
\begin{equation}\label{e:varphi(tau)=tau*_Lambda_tau}
\varphi({\boldsymbol \tau}) = {\boldsymbol \tau}^* {\boldsymbol \Lambda}{\boldsymbol \tau}, \quad {\boldsymbol \tau} \in {\mathcal W}.
\end{equation}

We are now in possession of all the elements described in \textbf{(a)--(c)}. Following the description of \textbf{(d)}, we establish \eqref{e:lim_nu_rescaled_eigenvalue} first for $\ell=\min \hspace{1mm}\mathcal I_0 = r_1+1$, and then proceed by induction.

In the construction of the argument, it will be convenient to consider the decomposition of $\mathcal W$ (see \eqref{e:def_mathcal_W}) given by
\begin{equation}\label{e:span_equals_W}
\text{span}\{{\boldsymbol \tau}_{\ell}(\omega): \ell \in \mathcal I_0\} = \mathcal W \hspace{3mm} \textnormal{a.s.},
\end{equation}
itself a consequence of Lemma \ref{l:fullrank_limit_P*U}, $(iii)$. On the other hand, by Lemma \ref{l:M_in_S>0(r2,R)}, $(iii)$, $\text{rank}(\boldsymbol\Lambda)=r_2 \geq 1$. So, for some $\eta \in \bbN$, let
\begin{equation}\label{e:chi1,...,chi_eta}
0 < \chi_1 < \chi_2 < \ldots  <\chi_\eta
\end{equation}
be the distinct values among the strictly positive $r_2$ eigenvalues of ${\boldsymbol \Lambda}$. Also, let $\mathcal M_1,\ldots,\mathcal M_\eta$ be the (deterministic) eigenspaces associated with each of the $\eta\geq 1$ distinct positive eigenvalues \eqref{e:chi1,...,chi_eta} of $\boldsymbol \Lambda$. Then, as a consequence of \eqref{e:def_mathcal_W} and \eqref{e:span_equals_W}, we can further write
\begin{equation}\label{e:span(t_ell)_I0=W=M1+...+Meta}
\text{span}\{{\boldsymbol \tau}_{\ell}(\omega): \ell \in \mathcal I_0\} = \mathcal W = \mathcal M_1 \oplus \ldots \oplus\mathcal M_\eta,
\end{equation}
where the first equality in \eqref{e:span(t_ell)_I0=W=M1+...+Meta} holds a.s.\ (\textbf{n.b.}: relation \eqref{e:span(t_ell)_I0=W=M1+...+Meta} does not per se determine the connection between subsets of vectors ${\boldsymbol \tau}_{\ell}(\omega)$ and the eigenspaces $\mathcal M_{i}$. This connection will be established in the next stages of this proof.)\vspace{2mm}

\noindent \textbf{Step} ${\boldsymbol \ell} = {\mathbf r_1 + 1}$. First, we establish a lower bound, as well as its limit, for the rescaled eigenvalue $\frac{\lambda_{p-r+(r_1+1)}\big(\W(a(\nu_n)2^j)\big)}{a(\nu_n)^{2h_q+1}}$ (cf.\ \eqref{e:high-level_system_ineqs}). Recall that the vectors $\boldsymbol \tau_\ell(n) = \boldsymbol \tau_\ell(n,\omega)$, $\ell = 1,\hdots,r$, are given by \eqref{e:T(n)}. Also, let $\widehat f_{\nu_n}$ and ${\widehat \varphi}_{\nu_n}$ be as in \eqref{e:def_fn(x,u)}. Then, for $\omega \in A$ a.s., as $n \rightarrow \infty$,
$$
\frac{\lambda_{p-r+(r_1+1)}\big(\W(a(\nu_n)2^j)\big)}{a(\nu_n)^{2h_q+1}}
$$
$$
\geq \widehat f_{\nu_n}\big( {\mathbf x}_{*,\nu_n}({\mathbf u}_{p-r+r_1+1}(\nu_n)),{\mathbf u}_{p-r+r_1+1}(\nu_n)\big) + \mathbf u^*_{p-r+r_1+1}(\nu_n) \Xi_q(\nu_n)\mathbf u_{p-r+r_1+1}(\nu_n)
$$
\begin{equation}\label{e:rescaled_lambda2_lower_bound_first}
\geq {\widehat \varphi}_{\nu_n}\big( {\mathbf x}_{*,\nu_n}({\boldsymbol \tau}_{r_1+1}(\nu_n)),{\boldsymbol \tau}_{r_1+1}(\nu_n)\big)+ o(1) \rightarrow \varphi(\boldsymbol \tau _{{r_1+1}}) \geq \chi_1.%
\end{equation}
In \eqref{e:rescaled_lambda2_lower_bound_first}, the first inequality follows from \eqref{e:rescaled_eigen_p-r+ell_in_terms_of_fn_and_varphin} and the fact that ${\mathbf x}_{*,\nu_n}({\mathbf u}_{p-r+r_1+1}(\nu_n))$ is a minimizer of $\widehat f_{\nu_n}\big( {\mathbf x},{\mathbf u}_{p-r+r_1+1}(\nu_n)\big)$ in the argument ${\mathbf x}$. The second inequality in \eqref{e:rescaled_lambda2_lower_bound_first} holds by \eqref{e:rescaled_eigen_p-r+ell_in_terms_of_fn_and_varphin} (\textbf{n.b.:} ${\mathbf x}_{*,\nu_n}({\mathbf u}_{p-r+r_1+1}(\nu_n))={\mathbf x}_{*,\nu_n}({\boldsymbol \tau}_{r_1+1}(\nu_n))$). The $o(1)$ term appearing in \eqref{e:rescaled_lambda2_lower_bound_first} is a consequence of the fact that  $\mathbf u^*_{p-r+r_1+1}(\nu_n)\Xi_q(\nu_n)\mathbf u_{p-r+r_1+1}(\nu_n) = o(1)$ for the given $\omega \in A$ a.s.\ due to condition \eqref{e:B-hat(2j)->B(2j)_prop}. Also, ${\widehat \varphi}_{n}\big( {\mathbf x}_{*,\nu_n}({\boldsymbol \tau}_{r_1+1}(\nu_n)),{\boldsymbol \tau}_{r_1+1}(\nu_n)\big) \rightarrow \varphi(\boldsymbol \tau _{{r_1+1}})$ due to expressions \eqref{e:def_fn(x,u)}--\eqref{e:def_x-star(u)}. In addition, the last inequality in \eqref{e:rescaled_lambda2_lower_bound_first} holds since $\boldsymbol \tau _{{r_1+1}} \in {\mathcal W} \cap \bbS^{r-1}$ (see \eqref{e:span(t_ell)_I0=W=M1+...+Meta}) and $\chi_1$ is the smallest value $\varphi$ can take on ${\mathcal W} \cap \bbS^{r-1}$. This establishes the lower bound.

We now construct an upper bound, as well as its limit, for the rescaled eigenvalue $\frac{\lambda_{p-r+(r_1+1)}\big(\W(a(\nu_n)2^j)\big)}{a(\nu_n)^{2h_q+1}}$ (cf.\ \eqref{e:high-level_system_ineqs}). Fix an arbitrary  (deterministic) vector
\begin{equation}\label{e:w_in_M1_in_Sr-1}
\mathbf w\in \mathcal M_1\cap \mathcal \bbS^{r-1}
\end{equation}
(namely, $\mathbf w$ is any unit eigenvector of $\boldsymbol \Lambda$ associated with its smallest positive eigenvalue $\chi_1$). By relations \eqref{e:varphi(tau)=tau*_Lambda_tau} and \eqref{e:chi1,...,chi_eta},
\begin{equation}\label{e:f(w11)=chi_1}
\varphi(\mathbf w) = \chi_1.
\end{equation}
In view of \eqref{e:span_equals_W}, we can use the a.s.\ orthonormal vectors $\{{\boldsymbol \tau}_{\ell}(\omega): \ell \in \mathcal I_0\}$ to write
\begin{equation}\label{e:w=sum_ell_alpha(omega)*tau_ell(omega)}
\mathbf w = \sum_{\ell \in \mathcal I_0}\alpha_\ell(\omega) \boldsymbol\tau_\ell(\omega).
\end{equation}
Let $\mathbf x_*(\mathbf w) = (x_{*,r_1+r_2+1}({\mathbf w} ),\ldots, x_{*,r}({\mathbf w}))$ be the minimizer  of $\varphi(\cdot,{\mathbf w})$, as defined by \eqref{e:def_x-star(u)}.  In view of the convergence conditions \eqref{e:subseq_condition_1_top} and \eqref{e:subseq_condition_2_top}, for  $\omega \in A$ a.s.\ we may apply Lemma \ref{l:<p,w>=infinitesimal}  to extract a sequence $\{{\mathbf v}(\nu_n)\}_{n \in \bbN} = \{{\mathbf v}(\nu_n,\omega)\}_{n \in \bbN}$  of unit vectors
\begin{equation}\label{e:v(n)_in_span(u_p-r+r1+1,...,u_p(n))}
{\mathbf v}(\nu_n) \in \textnormal{span}\{{\mathbf u}_{{p-r+r_1+1}}(\nu_n),\hdots,{\mathbf u}_{p}(\nu_n)\}
\end{equation}
such that, as $n \rightarrow \infty$,
\begin{equation}\label{e:pi(n),v(n)=x*/a^(hi-hq)}
\langle \mathbf p_i(\nu_n), {\mathbf v}(\nu_n) \rangle = \frac{x_{*,i}({\mathbf w})}{a(\nu_n)^{h_i-h_q}},\quad i\in\mathcal I_+, \quad \textnormal{and} \quad \mathbf w(\nu_n)={\mathbf Q}^*(\nu_n)\mathbf v(\nu_n)\rightarrow \mathbf w.
\end{equation}
Then, as $n \rightarrow \infty$,
$$
\frac{\lambda_{p-r+(r_1+1)}\big(\W(a(\nu_n)2^j)\big)}{a(\nu_n)^{2h_q+1}}  \leq {\mathbf v}^*(\nu_n)
\frac{\W(a(\nu_n)2^j) }{a(\nu_n)^{2h_q+1}} {\mathbf v}(\nu_n)
$$
\begin{equation}\label{e:rescaled_lambda2_upper_bound1}
={\widehat \varphi}_n\big(\mathbf x_*(\mathbf w),\mathbf w(\nu_n)\big)
 +{\mathbf v}^*_{{r_1+1}}(\nu_n) \Xi_q(\nu_n){\mathbf v}_{{r_1+1}}(\nu_n)
\rightarrow \varphi({\mathbf w}) =\chi_1.
\end{equation}
In \eqref{e:rescaled_lambda2_upper_bound1}, the inequality is a consequence of \eqref{e:v(n)_in_span(u_p-r+r1+1,...,u_p(n))} and \eqref{e:lambdaq(M)_based_on_eigenvecs}. The convergence follows since ${\mathbf v}^*_{r_1+1}(\nu_n) \Xi_q(\nu_n){\mathbf v}_{r_1+1}(\nu_n) =o(1)$, and also because ${\widehat \varphi}_{\nu_n}\big(\mathbf x_*(\mathbf w),\mathbf w(\nu_n)\big)
  \to \varphi(\mathbf w)$, itself a consequence of expressions \eqref{e:def_phi_n(x,u)}, \eqref{e:def_varphi(x,u)} and of the limit in \eqref{e:pi(n),v(n)=x*/a^(hi-hq)}. This establishes the upper bound.

So, expressions \eqref{e:rescaled_lambda2_lower_bound_first} and \eqref{e:rescaled_lambda2_upper_bound1} show that, for $\omega \in A$ a.s.,
\begin{equation}\label{e:f(ul(omega)=f(u-tilde))}
\frac{\lambda_{p-r+(r_1+1)}\big(\W(a(\nu_n)2^j)\big)}{a(\nu_n)^{2h_q+1}}(\omega) \to \varphi(\boldsymbol \tau _{{r_1+1}}(\omega)) = \varphi({\mathbf w})=\chi_1,
\quad n \rightarrow \infty,
\end{equation}
where the last equality follows from \eqref{e:f(w11)=chi_1}. In addition, expression \eqref{e:f(ul(omega)=f(u-tilde))} implies that ${\boldsymbol \tau}_{{r_1+1}}(\omega)\in \mathcal M_1$. This establishes \eqref{e:lim_nu_rescaled_eigenvalue} for the index value $\ell=r_1+1=\min \hspace{1mm}\mathcal I_0$.\vspace{3mm}

\noindent \textbf{Step general }${\boldsymbol \ell} \in {\mathcal I}_0$. We now proceed by induction through the set $\mathcal I_0$. Consider the double decomposition \eqref{e:span(t_ell)_I0=W=M1+...+Meta} of ${\mathcal W}$, and let $\mathcal M_0=\emptyset$. For the induction hypothesis, assume that, for ${\boldsymbol \ell}-1 \in {\mathcal I}_0$, there exists $k \in \bbN$ such that
\begin{equation}\label{e:inductive_hypothesis}
\mathcal M_0\oplus \ldots \oplus \mathcal M_{k-1} \subsetneq \text{span}\{\boldsymbol \tau_{r_1+1},\ldots,\boldsymbol\tau_{\ell-1}\} \subseteq \mathcal M_0\oplus \ldots \oplus \mathcal M_{k} \hspace{3mm}\textnormal{a.s.}
\end{equation}
Further assume that, for $i = r_1+1,\hdots,\ell-1$ and for $\omega \in A$ a.s.,
\begin{equation}\label{e:rescaled_eigen_p-r+i/a^(2*hq+1)->varphi(tau_i)}
\lim_{n \rightarrow \infty}\frac{\lambda_{p-r+i}\big(\W(a(\nu_n)2^j)\big)}{a(\nu_n)^{2h_q+1}} = \varphi({\boldsymbol \tau}_{i}) \in \{\chi_1,\hdots,\chi_\eta\}.
\end{equation}
In \eqref{e:rescaled_eigen_p-r+i/a^(2*hq+1)->varphi(tau_i)}, for ${\boldsymbol \tau}_i={\boldsymbol \tau}_i(\omega)$ as in \eqref{e:Gamma_a.s._limit}, we suppose
\begin{equation}\label{e:chi1=varphi(tau_1)=<...=<varphi(tau_ell-1)=chi_k}
\chi_1 = \varphi({\boldsymbol \tau}_{ r_1+1 }) \leq \hdots \leq \varphi({\boldsymbol \tau}_{\ell-1}) = \chi_k
\end{equation}
(\textbf{n.b.:} $k$ does \textit{not} depend on $\omega$ -- cf.\ the decomposition $\mathcal W = \mathcal M_1 \oplus \ldots \oplus\mathcal M_\eta$ in \eqref{e:span(t_ell)_I0=W=M1+...+Meta}, which is deterministic).

So, starting from the induction hypothesis \eqref{e:inductive_hypothesis}--\eqref{e:chi1=varphi(tau_1)=<...=<varphi(tau_ell-1)=chi_k}, note that $\textnormal{dim}(\mathcal M_0\oplus \ldots \oplus \mathcal M_{k}) \geq \ell -1 -r_1$. Our goal is to show that
\begin{equation}\label{e:case_(i)_induction}
(i) \quad \textnormal{if }~\textnormal{dim}(\mathcal M_0\oplus \ldots \oplus \mathcal M_{k}) > \ell -1 -r_1,
\end{equation}
then, almost surely,
\begin{equation}\label{e:rescaled_eigen_p-r+ell/a^(2*hq+1)->varphi(tau_ell)_case_(i)}
\lim_{n \rightarrow \infty}\frac{\lambda_{p-r+\ell}\big(\W(a(\nu_n)2^j)\big)}{a(\nu_n)^{2h_q+1}} = \varphi({\boldsymbol \tau}_{\ell}) = \chi_k, \quad {\boldsymbol \tau}_{\ell} \in \text{span}\{\boldsymbol \tau_{r_1+1},\ldots,\boldsymbol\tau_{\ell-1}\}^\perp \cap \mathcal M_{k};
\end{equation}
or
\begin{equation}\label{e:case_(ii)_induction}
(ii) \quad \textnormal{if }~\textnormal{dim}(\mathcal M_0\oplus \ldots \oplus \mathcal M_{k}) = \ell -1 -r_1,
 \end{equation}
 then, almost surely,
\begin{equation}\label{e:rescaled_eigen_p-r+ell/a^(2*hq+1)->varphi(tau_ell)_case_(ii)}
\lim_{n \rightarrow \infty }\frac{\lambda_{p-r+\ell}\big(\W(a(\nu_n)2^j)\big)}{a(\nu_n)^{2h_q+1}} = \varphi({\boldsymbol \tau}_{\ell}) = \chi_{k+1}, \quad {\boldsymbol \tau}_{\ell} \in \mathcal M_{k+1}.
\end{equation}
In either case, \eqref{e:inductive_hypothesis}--\eqref{e:chi1=varphi(tau_1)=<...=<varphi(tau_ell-1)=chi_k} are extended to $\ell > \ell -1$, which establishes the induction.

So, under \eqref{e:case_(i)_induction} and \eqref{e:case_(ii)_induction}, respectively, either
\begin{equation}\label{e:case_(i)}
(i) \quad \text{span}\{\boldsymbol \tau_{ r_1+1 },\ldots,\boldsymbol\tau_{\ell-1}\} \subsetneq \mathcal M_0\oplus \ldots \oplus \mathcal M_{k} \hspace{3mm}\textnormal{a.s.},
\end{equation}
or
\begin{equation}\label{e:case_(ii)}
(ii) \quad \text{span}\{\boldsymbol \tau_{r_1+1 },\ldots,\boldsymbol\tau_{\ell-1}\} =\mathcal M_0\oplus \ldots \oplus \mathcal M_{k} \hspace{3mm}\textnormal{a.s.}
 \end{equation}
 We consider each case $(i)$ and $(ii)$ separately. To avoid the introduction of cumbersome notation and to facilitate comparison with the inductive step $\ell = r_1 + 1$, we reuse the notation $\mathbf w$, $\mathbf w(\nu_n)$, and $\mathbf v(\nu_n)$ according to convenience.

We begin with $(i)$. Again in view of \eqref{e:span(t_ell)_I0=W=M1+...+Meta}, since the vectors $\boldsymbol \tau_i$ are orthonormal a.s., the first inclusion in  \eqref{e:inductive_hypothesis} implies that
\begin{equation}\label{e:tau_ell_in_M-k+...+M-eta_and_sphere}
\boldsymbol \tau_\ell \in \mathcal W \cap \text{span}\{\boldsymbol \tau_{r_1+1},\ldots,\boldsymbol\tau_{\ell-1}\}^\perp \subseteq  \mathcal M_k \oplus \ldots \oplus \mathcal M_\eta \hspace{3mm}\textnormal{a.s.}
\end{equation}
However,  ${\mathcal M}_{k}, \ldots, {\mathcal M}_\eta$ are the eigenspaces of ${\boldsymbol \Lambda}$ associated with the distinct eigenvalues $\chi_{k} < \hdots < \chi_{\eta}$, respectively. Hence, relation \eqref{e:tau_ell_in_M-k+...+M-eta_and_sphere} implies that the unit vector $\boldsymbol \tau_\ell = \boldsymbol \tau_\ell(\omega)$ is a convex combination of eigenvectors of ${\boldsymbol \Lambda}$ associated with eigenvalues no smaller than $\chi_k$. Therefore, by expression \eqref{e:varphi(tau)=tau*_Lambda_tau},
\begin{equation}\label{e:varphi(tau_ell)>=chi_k}
\varphi({\boldsymbol \tau}_\ell) \geq \chi_k \hspace{3mm}\textnormal{a.s.}
\end{equation}
So, arguing as in \eqref{e:rescaled_lambda2_lower_bound_first} with $\ell$ replacing $r_1 + 1$, as $n \rightarrow \infty$,
\begin{equation}\label{e:rescaled_lambda_lower_bound_q(k,ell+1)2}
\frac{\lambda_{p-r+\ell}\big(\W(a(\nu_n)2^j)\big)}{a(\nu_n)^{2h_q+1}} \geq {\widehat \varphi}_{n}\big( {\mathbf x}_{*,\nu_n}({\boldsymbol \tau}_{\ell}(\nu_n)),{\boldsymbol \tau}_{\ell}(\nu_n)\big)+ o(1) \rightarrow \varphi(\boldsymbol \tau _{{\ell}}) \geq \chi_k.
\end{equation}

On the other hand, under \eqref{e:case_(i)}, relations \eqref{e:span(t_ell)_I0=W=M1+...+Meta} and \eqref{e:inductive_hypothesis} imply that there exists a random unit vector $\mathbf w=\mathbf w(\omega) \in \textnormal{span}\{\boldsymbol \tau_\ell,\ldots,\boldsymbol\tau_{r-r_3}\}\cap \mathcal M_k$. In particular,
\begin{equation}\label{e:varphi(w)=chi_k}
\varphi({\mathbf w}) = \chi_k.
\end{equation}
Moreover, there are random coefficients $\beta_i = \beta_i(\omega)$, $i = \ell, \hdots,r-r_3$, based on which we may express
$$
{\mathbf w}(\omega)= \sum_{ i= \ell}^{r-r_3}\beta_i(\omega) {\boldsymbol\tau}_{i}(\omega) \in {\mathcal M}_k \cap \bbS^{r-1} %
$$
(cf.\ relation \eqref{e:w=sum_ell_alpha(omega)*tau_ell(omega)}, where the left-hand side is deterministic). Again in view of conditions \eqref{e:subseq_condition_1_top} and \eqref{e:subseq_condition_2_top}, for $\omega \in A$ a.s.\ Lemma \ref{l:<p,w>=infinitesimal} implies that we may pick a sequence $\{{\mathbf v}(\nu_n)\}_{n \in \bbN}$ of unit vectors ${\mathbf v}(\nu_n) \in \textnormal{span}\{{\mathbf u}_{p-r+\ell}(\nu_n),\hdots,{\mathbf u}_{p}(\nu_n)\}$ such that
\begin{equation}\label{e:pick_the_vector_main_manuscript}
\langle \mathbf p_i(\nu_n), {\mathbf v}(\nu_n) \rangle = \frac{{\mathbf x}_{*,i}({\mathbf w})}{a(\nu_n)^{h_i-h_q}},\quad i\in\mathcal I_+, \quad \textnormal{and} \quad \mathbf Q^*(\nu_n) {\mathbf v}(\nu_n) \rightarrow  \mathbf w.
\end{equation}
Thus, as $n\rightarrow \infty$, arguing similarly as in  \eqref{e:rescaled_lambda2_upper_bound1},
\begin{equation}\label{e:rescaled_lambda_upperbound2}
\frac{\lambda_{p-r+\ell}\big(\W(a(\nu_n)2^j)\big)}{a(\nu_n)^{2h_q+1}} \leq {\mathbf v}^*(\nu_n)
\frac{\W(a(\nu_n)2^j) }{a(\nu_n)^{2h_q+1}}  {\mathbf v}(\nu_n)\rightarrow \varphi({\mathbf w}).
\end{equation}
As a consequence of \eqref{e:rescaled_lambda_lower_bound_q(k,ell+1)2}, \eqref{e:varphi(w)=chi_k} and \eqref{e:rescaled_lambda_upperbound2}, for $\omega \in A$ a.s.\ and ${\mathbf w}={\mathbf w}(\omega) \in {\mathcal M}_k$,
\begin{equation}\label{e:prop_case_(i)_final_step}
\frac{\lambda_{p-r+\ell}\big(\W(a(\nu_n)2^j)\big)}{a(\nu_n)^{2h_q+1}}(\omega) \rightarrow \varphi({\mathbf w}(\omega)) = \chi_k, \quad n \rightarrow \infty.
\end{equation}
Relations \eqref{e:rescaled_lambda_lower_bound_q(k,ell+1)2}, \eqref{e:rescaled_lambda_upperbound2} and \eqref{e:prop_case_(i)_final_step} further imply that
\begin{equation}\label{E:varphi(tau_ell)=chi_k}
\varphi({\boldsymbol \tau}_{\ell}) = \chi_k.
\end{equation}
Together with \eqref{e:tau_ell_in_M-k+...+M-eta_and_sphere}, expressions \eqref{e:prop_case_(i)_final_step} and \eqref{E:varphi(tau_ell)=chi_k} establish \eqref{e:rescaled_eigen_p-r+ell/a^(2*hq+1)->varphi(tau_ell)_case_(i)} in case $(i)$.

In case $(ii)$, first note that
\begin{equation}\label{e:span(tau-ell,...,tau-r-r1)=M-k+1,...,M-eta}
  \text{span}\{\boldsymbol \tau_\ell,\ldots,\boldsymbol\tau_{r-r_3}\} =\mathcal M_{k+1} \oplus\ldots\oplus \mathcal M_\eta \hspace{3mm}\textnormal{a.s.}
\end{equation}
by \eqref{e:span(t_ell)_I0=W=M1+...+Meta} and \eqref{e:case_(ii)}. Then, since $\chi_{k+1}$ is the smallest value $\varphi$ can take on $\mathcal M_{k+1} \oplus\ldots\oplus \mathcal M_\eta\cap \mathcal \bbS^{r-1}$, by relation \eqref{e:varphi(tau)=tau*_Lambda_tau},
\begin{equation}\label{e:varphi(tau_ell)_>=_varphi(w)=chi_k+1}
\varphi({\boldsymbol \tau}_{\ell}) \geq \chi_{k+1}.
\end{equation}
By analogous arguments to those for \eqref{e:rescaled_lambda2_lower_bound_first}, we obtain, for $\omega \in A$ a.s.,
$$
\frac{\lambda_{p-r+\ell}\big(\W(a(\nu_n)2^j)\big)}{a(\nu_n)^{2h_q+1}} \geq {\widehat \varphi}_{n}\big( {\mathbf x}_{*,n}({\boldsymbol \tau}_{\ell}(\nu_n)),{\boldsymbol \tau}_{\ell}(\nu_n)\big)+o(1) \rightarrow \varphi(\boldsymbol \tau _{{\ell}}) \geq \chi_{k+1}, \quad n \rightarrow \infty.
$$
On the other hand, fix any (deterministic) $\mathbf w \in \mathcal M_{k+1} \cap \bbS^{r-1}$ (i.e.,  ${\mathbf w}$ is an eigenvector of $\boldsymbol \Lambda$ associated with $\chi_{k+1}$). Given \eqref{e:subseq_condition_1_top} and \eqref{e:subseq_condition_2_top}, for $\omega \in A$ a.s.\ Lemma \ref{l:<p,w>=infinitesimal} implies that there exists a sequence $\{{\mathbf v}(\nu_n)\}_{n \in \bbN}$ of unit vectors in $\textnormal{span}\{{\mathbf u}_{p-r+\ell}(\nu_n),\hdots,{\mathbf u}_{p}(\nu_n)\}$ satisfying \eqref{e:pick_the_vector_main_manuscript}.  Then, the analogous limit \eqref{e:rescaled_lambda_upperbound2} follows, with $\varphi(\mathbf w)=\chi_{k+1}$. From the lower and the upper limits, we conclude that, for $\omega \in A$ a.s.,
\begin{equation}\label{e:case_(ii)_lim_rescaled_eigenvalue}
\frac{\lambda_{p-r+\ell}\big(\W(a(\nu_n)2^j)\big)}{a(\nu_n)^{2h_q+1}}(\omega) \rightarrow \varphi({\mathbf w}) =  \varphi(\boldsymbol \tau _{{\ell}}) = \chi_{k+1}, \quad n \rightarrow \infty.
\end{equation}
Moreover, \eqref{e:span(tau-ell,...,tau-r-r1)=M-k+1,...,M-eta} and \eqref{e:case_(ii)_lim_rescaled_eigenvalue} imply that ${\boldsymbol \tau}_\ell \in {\mathcal M}_{k+1}$. This establishes \eqref{e:rescaled_eigen_p-r+ell/a^(2*hq+1)->varphi(tau_ell)_case_(ii)} in case $(ii)$. So, the induction is complete, which in turn establishes \eqref{e:lim_nu_rescaled_eigenvalue}. $\Box$\\

In the following example, we illustrate some of the main aspects involved in the proof of Proposition \ref{p:conv_w-by-w_rescaled_eigenvalues} and Theorem \ref{t:lim_n_a_times_lambda/a^(2h+1)}. To facilitate comparison, the description is broken up into the same steps \textbf{(a)--(d)} used in the proof of Proposition \ref{p:conv_w-by-w_rescaled_eigenvalues}. The example involves the simplest possible instance where there are both a slower and a faster eigenvalue than the reference eigenvalue $\lambda_{p-r+q}\big({\mathbf W}(a(n)2^j)\big)= \lambda_{p-1}\big({\mathbf W}(a(n)2^j)\big)$.

\begin{example}\label{ex:proof_of_Theo_3.1} Suppose $r = 3$ and $r_1 = r_2 = r_3 = 1$ (i.e., $h_1 < h_2 < h_3$). Hence, ${\mathbf P}(n) = \big({\mathbf p}_{1}(n),{\mathbf p}_{2}(n),{\mathbf p}_{3}(n) \big) \in {\mathcal M}(p,3,\bbR)$ where $p = p(n) \rightarrow \infty$. For ease of interpretation, suppose in addition that, for all $n \in \bbN$,
\begin{equation}\label{e:R(n)=R=(e1,e2,e3)}
{\mathbf R}(n) = {\mathbf R} = \big({\mathbf e}_{1} \hspace{2mm} {\mathbf e}_{2} \hspace{2mm}{\mathbf e}_{3}\big)\in {\mathcal M}(3,\bbR).
\end{equation}
In high-dimensional coordinates, \eqref{e:R(n)=R=(e1,e2,e3)} implies that the vectors ${\mathbf p}_{1}(n),{\mathbf p}_{2}(n),{\mathbf p}_{3}(n)$ are orthonormal. Moreover, the wavelet eigenvectors satisfy, in the three-way limit \eqref{e:three-fold_lim},
$$
\langle {\mathbf u}_{p-3+\ell}(n),{\mathbf p}_{\ell}(n)\rangle^2 \stackrel{\bbP}\rightarrow 1, \quad \ell = 1,2,3
$$
(cf.\ Figure \ref{fig:eigenvectors}, which displays the general case of non-orthogonal ${\mathbf p}_{1}(n),{\mathbf p}_{2}(n),{\mathbf p}_{3}(n)$).
\vspace{0.2mm}

\noindent \textbf{(a)--(b)} For $\nu_n$ as in \eqref{e:Gamma_a.s._limit}, we can write
\begin{equation*}%
\widehat{f}_{\nu_n}(x,{\mathbf u}) = \widehat{b}_{2 2} \hspace{0.5mm} \langle {\mathbf p}_{2}(\nu_n),{\mathbf u}\rangle^2
+ 2\hspace{1mm}\widehat{b}_{23} \hspace{0.5mm}\langle {\mathbf p}_{2}(\nu_n),{\mathbf u}\rangle \hspace{0.5mm} x + \widehat{b}_{33}\hspace{1mm} x^2
+ o_{\textnormal{a.s.}}(1),
\end{equation*}
\begin{equation*}%
\widehat{\varphi}_{\nu_n}(x,{\boldsymbol \tau}) = \widehat{b}_{2 2} \hspace{0.5mm} \langle {\mathbf r}_{2}(\nu_n),{\boldsymbol \tau}\rangle^2
+ 2\hspace{1mm}\widehat{b}_{23} \hspace{0.5mm}\langle {\mathbf r}_{2}(\nu_n),{\boldsymbol \tau}\rangle \hspace{0.5mm} x + \widehat{b}_{33}\hspace{1mm} x^2 + o_{\textnormal{a.s.}}(1)
\end{equation*}
(cf.\ \eqref{e:def_fn(x,u)} and \eqref{e:def_phi_n(x,u)}). By analogy to \eqref{e:B-hat_a(2^j)_entry-wise}, let ${\mathbf B}(2^j)=\big(b_{\ell_1, \ell_2} \big)_{\ell_1,\ell_2=1,2,3}$. Then, as $n \rightarrow \infty$,
\begin{equation*}%
\widehat{\varphi}_{\nu_n}(x,{\boldsymbol \tau}) \rightarrow \varphi(x,{\boldsymbol \tau}) = b_{2 2} \hspace{0.5mm} \langle {\mathbf e}_{2},{\boldsymbol \tau}\rangle^2
+ 2\hspace{1mm} b_{23} \hspace{0.5mm}\langle {\mathbf e}_{2},{\boldsymbol \tau}\rangle \hspace{0.5mm} x + b_{33}\hspace{1mm} x^2
\end{equation*}
(cf.\ \eqref{e:varphi=lim_varphi-hat_nu-n} and \eqref{e:def_varphi(x,u)}).\vspace{0.4mm}

\noindent \textbf{(c)} For any fixed ${\boldsymbol \tau}$, the global minimizer of $\varphi(x,{\boldsymbol \tau})$ in $x$ is given by $x_{*}({\boldsymbol \tau}) = - \frac{b_{23}}{b_{33}}\langle {\mathbf e}_2, {\boldsymbol \tau}\rangle$. Therefore,
\begin{equation}\label{e:varphi_tau_example}
\varphi({\boldsymbol \tau})= \varphi(x_{*}({\boldsymbol \tau}),{\boldsymbol \tau}) = \frac{1}{b_{33}} \big(b_{22}b_{33} - b^2_{23} \big) \langle {\mathbf e}_2, {\boldsymbol \tau}\rangle^2.
\end{equation}\vspace{0.4mm}

\noindent \textbf{(d)} By \eqref{e:def_mathcal_W} and \eqref{e:R(n)=R=(e1,e2,e3)}, ${\mathcal W} = \textnormal{span}\{{\mathbf e}_2\} \subseteq \bbR^3$. In particular, ${\mathcal M}_1 = \{\pm {\mathbf e}_2\} \ni {\mathbf w}$. Hence, by \eqref{e:varphi_tau_example}, the lower and upper bounds in \eqref{e:rescaled_lambda2_lower_bound_first} and \eqref{e:rescaled_lambda2_upper_bound1} are given by $\varphi({\mathbf w}) = \frac{1}{b_{33}} \big(b_{22}b_{33} - b^2_{23} \big)$.\vspace{0.6mm}

In other words, we ultimately conclude that
$$
\frac{\lambda_{p-1}\big(\W(a(n)2^j)\big)}{a(n)^{2h_2+1}} \stackrel{\bbP}\rightarrow \varphi({\mathbf w}) = \frac{1}{b_{33}} \big(b_{22}b_{33} - b^2_{23} \big), \quad n \rightarrow \infty.
$$
In addition, $\frac{\lambda_{p-2}\big(\W(a(n)2^j)\big)}{a(n)^{2h_2+1}} \stackrel{\bbP}\rightarrow 0$ and $\frac{\lambda_{p}\big(\W(a(n)2^j)\big)}{a(n)^{2h_2+1}} \stackrel{\bbP}\rightarrow \infty$ as $n \rightarrow \infty$ (cf.\ Corollary \ref{c:PWXP^*+R_asymptotics}).
\end{example}

\subsection{Proving Theorem \ref{t:asympt_normality_lambdap-r+q}}\label{s:proving_theorem_3.2}

In this section, again for notational simplicity we work under \eqref{e:P(n)PH_equiv_P}, namely, $p = p(n)$, $a = a(n)$ and ${\mathbf P} = {\mathbf P}(n) = {\mathbf P}(n){\mathbf P}_H$.

As with Proposition \ref{p:conv_w-by-w_rescaled_eigenvalues}, before showing Theorem \ref{t:asympt_normality_lambdap-r+q} for the reader's convenience we provide a summary of the proof method.

The argument is based on mean value theorem-type expansions of the expressions on the left-hand side of \eqref{e:asympt_normality_lambda2}. So, recall expressions \eqref{e:wave_RM_P=I} and \eqref{e:rescaled_EW_original}, namely,
\begin{equation}\label{e:rescaled_W}
\frac{{\mathbf W}(a2^j)}{a^{2h_q+1}} = \underbrace{\frac{{\mathbf P}a^{{\mathbf h}} \widehat{{\mathbf B}}_a(2^j)a^{{\mathbf h}}{\mathbf P}^*}{a^{2h_q}}}_{\textnormal{main scaling term}} + \underbrace{\frac{O_{\bbP}(1)}{a^{2h_q+1}}+
\frac{{\mathbf P}a^{{\mathbf h}}O_{\bbP}(1)}{a^{2h_q+1/2}}+\frac{O^*_{\bbP}(1)a^{{\mathbf h}}{\mathbf P}^*}{a^{2h_q+1/2}}}_{\textnormal{residual}}
\end{equation}
and
\begin{equation}\label{e:rescaled_EW}
\frac{\bbE{\mathbf W}(a2^j)}{a^{2h_q+1}} = \underbrace{\frac{{\mathbf P}a^{{\mathbf h}} {\mathbf B}_a(2^j)a^{{\mathbf h}}{\mathbf P}^*}{a^{2h_q}}}_{\textnormal{main scaling term}} + \underbrace{\frac{O(1)}{a^{2h_q+1}}}_{\textnormal{residual}}.
\end{equation}
Further recall that each of the two $O_{\bbP}(1)$ terms in \eqref{e:rescaled_W} and the $O(1)$ term in \eqref{e:rescaled_EW} correspond to, respectively, ${\mathbf W}_Z(a(n)2^{j})$, $a(n)^{-{\mathbf H}-(1/2){\mathbf I}}{\mathbf W}_{X,Z}(a(n)2^{j})$ and $\bbE {\mathbf W}_Z(a(n)2^{j})$ in \eqref{e:WZ=OP(1),a^(-h-1/2I)WXZ_repeat}. Ultimately, the asymptotic fluctuations of the log-eigenvalues of $\frac{{\mathbf W}(a2^j)}{a^{2h_q+1}}$ will stem from the main scaling terms in \eqref{e:rescaled_W} and \eqref{e:rescaled_EW}.

We can apply \eqref{e:rescaled_W} and \eqref{e:rescaled_EW} so as to decompose
\begin{equation}\label{e:loglambda-loglambdaE_3_terms}
\log \lambda_{p-r+q}\Big( \frac{{\mathbf W}(a(n)2^{j})}{a^{2h_q+1}} \Big)  - \log \lambda_{p-r+q}\Big( \frac{\bbE {\mathbf W} (a(n)2^{j})}{a^{2h_q+1}} \Big)
\end{equation}
$$
= \Big\{\log \lambda_{p-r+q}\Big( \frac{{\mathbf P}a^{{\mathbf h}} \widehat{{\mathbf B}}_a(2^j)a^{{\mathbf h}}{\mathbf P}^*}{a^{2h_q}} + \underbrace{ \frac{O_{\bbP}(1)}{a^{2h_q+1}}+
\frac{{\mathbf P}a^{{\mathbf h}}O_{\bbP}(1)}{a^{2h_q+1/2}}+\frac{O^*_{\bbP}(1)a^{{\mathbf h}}{\mathbf P}^*}{a^{2h_q+1/2}} }_{(*)}\Big)
$$
\begin{equation}\label{e:loglambda-loglambdaE_3_terms_term1}
- \log \lambda_{p-r+q}\Big( \frac{{\mathbf P}a^{{\mathbf h}} {\mathbf B}_a(2^j) a^{{\mathbf h}}{\mathbf P}^*}{a^{2h_q}} + \underbrace{\frac{O_{\bbP}(1)}{a^{2h_q+1}}+
 \frac{{\mathbf P}a^{{\mathbf h}}O_{\bbP}(1)}{a^{2h_q+1/2}}+\frac{O^*_{\bbP}(1)a^{{\mathbf h}}{\mathbf P}^*}{a^{2h_q+1/2}} }_{(*)}\Big) \Big\}
\end{equation}
$$
+ \Big\{\log \lambda_{p-r+q}\Big( \underbrace{ \frac{{\mathbf P}a^{{\mathbf h}} {\mathbf B}_a(2^j) a^{{\mathbf h}}{\mathbf P}^*}{a^{2h_q}}  + \frac{O_{\bbP}(1)}{a^{2h_q+1}} }_{(**)}+ \frac{{\mathbf P}a^{{\mathbf h}}}{a^{h_q}}
\frac{O_{\bbP}(1)}{a^{h_q+1/2}}+\frac{O^*_{\bbP}(1)}{a^{h_q+1/2}}\frac{a^{{\mathbf h}}{\mathbf P}^*}{a^{h_q}}\Big)
$$
\begin{equation}\label{e:loglambda-loglambdaE_3_terms_term2}
- \log \lambda_{p-r+q}\Big( \underbrace{ \frac{{\mathbf P}a^{{\mathbf h}} {\mathbf B}_a(2^j) a^{{\mathbf h}}{\mathbf P}^*}{a^{2h_q}}  + \frac{O_{\bbP}(1)}{a^{2h_q+1}} }_{(**)}\Big) \Big\}
\end{equation}
\begin{equation}\label{e:loglambda-loglambdaE_3_terms_term3}
+ \Big\{\log \lambda_{p-r+q}\Big( \underbrace{ \frac{{\mathbf P}a^{{\mathbf h}} {\mathbf B}_a(2^j) a^{{\mathbf h}}{\mathbf P}^*}{a^{2h_q}} }_{(***)} + \frac{O_{\bbP}(1)}{a^{2h_q+1}}\Big)
- \log \lambda_{p-r+q}\Big( \underbrace{  \frac{{\mathbf P}a^{{\mathbf h}} {\mathbf B}_a(2^j) a^{{\mathbf h}}{\mathbf P}^*}{a^{2h_q}} }_{(***)} + \frac{O(1)}{a^{2h_q+1}}\Big)\Big\}.
\end{equation}
Then, in the proof we consider each sum term \eqref{e:loglambda-loglambdaE_3_terms_term1}, \eqref{e:loglambda-loglambdaE_3_terms_term2} and \eqref{e:loglambda-loglambdaE_3_terms_term3} separately. For each one of them, the common factors in the arguments are marked by underbraces ($(*)$, $(**$), $(***)$). Accordingly, for each sum term the expansions are based on the differences
$$
\widehat{{\mathbf B}}_a(2^j)- {\mathbf B}_a(2^j) \in {\mathcal M}(r,\bbR), \quad \frac{a^{- {\mathbf h}- (1/2){\mathbf I}}{\mathbf W}_{X,Z}(a2^j)}{a^{h_q + 1/2}} - {\mathbf 0} =
\frac{O_{\bbP}(1)}{a^{h_q + 1/2}} - {\mathbf 0} \in {\mathcal M}(r,p,\bbR),
$$
\begin{equation}\label{e:asympt_normality_differences}
\textnormal{and}\quad \frac{{\mathbf W}_{Z}(a2^j)}{a^{2h_q+1}}- \frac{\bbE {\mathbf W}_{Z}(a2^j)}{a^{2h_q+1}} = \frac{O_{\bbP}(1)}{a^{2h_q+1}}- \frac{O(1)}{a^{2h_q+1}} \in {\mathcal S}(p,\bbR),
\end{equation}
respectively (see expressions \eqref{e:Taylor}, \eqref{e:sqrt(n/a)(fnq2(o(1))-fnq0(0))} and \eqref{e:sqrt(n/a)(fnq3(WZ/a^2hq+1)-fnq(0))}). For all three terms, differentiability can be proven to hold in a suitably defined neighborhood containing the terms appearing in \eqref{e:asympt_normality_differences}. This allows us to construct mean value theorem-type expansions.

Then, after multiplication by the rate $\sqrt{n_{a,j}}$, we show that the term \eqref{e:loglambda-loglambdaE_3_terms_term1} converges to a Gaussian distribution, where the fluctuations fundamentally originate in condition \eqref{e:sqrt(K)(B^-B)->N(0,sigma^2)}. We further show that, again after multiplication by the rate $\sqrt{n_{a,j}}$, the terms \eqref{e:loglambda-loglambdaE_3_terms_term2} and \eqref{e:loglambda-loglambdaE_3_terms_term3} converge to zero in probability (see expressions \eqref{e:f_1_weakconv}, \eqref{e:f_2_to0} and \eqref{e:f_3_to0} for the precise statements). \\

We are now in a position to show Theorem \ref{t:asympt_normality_lambdap-r+q}. Even though some steps involved in tackling each term \eqref{e:loglambda-loglambdaE_3_terms_term1}, \eqref{e:loglambda-loglambdaE_3_terms_term2} and \eqref{e:loglambda-loglambdaE_3_terms_term3} display formal similarities, we opted for repeating them so as to facilitate reading. In regard to the notation, in the proof we use \eqref{e:pi_i1,i2} and also express the generic matrices ${\mathbf B} = \big(b_{\ell,\ell' }\big)$, ${\mathbf K} = \big(\kappa_{\ell,\ell'}\big)$ entry-wise. \\

\noindent {\sc Proof of Theorem \ref{t:asympt_normality_lambdap-r+q}}: Fix an arbitrary $q\in \{1,\ldots,r\}$, and let $\mathcal{I}_0$ and (the possibly empty sets) $\mathcal{I}_-,\mathcal{I}_+$ be as in \eqref{e:def_indexsets}. For the sake of concision, we focus on the case where
\begin{equation}\label{e:I_-,I_+_neq_emptyset}
{\mathcal I}_{-}, {\mathcal I}_{+} \neq \emptyset,
\end{equation}
since the remaining cases can be promptly established by a natural simplification of the argument for the case \eqref{e:I_-,I_+_neq_emptyset}.

Consider a generic matrix term ${\mathbf B} \in {\mathcal S}_{\geq 0}(r,\bbR)$ and matrix residual terms ${\mathbf K}$ in either ${\mathcal M}(r,p,\bbR)$ or ${\mathcal S}(p,\bbR)$. For notational simplicity, it is convenient to define the sequences of symmetric random matrices
\begin{equation}\label{e:W-tilde1(a2^j,B)}
\widetilde{{\mathbf W}}_1(a2^j,{\mathbf B}) = a^{2h_q+1}\Big(\frac{{\mathbf P}a^{{\mathbf h}} {\mathbf B} a^{{\mathbf h}}{\mathbf P}^*}{a^{2h_q}}  + \frac{O_{\bbP}(1)}{a^{2h_q+1}}+
\frac{{\mathbf P}a^{{\mathbf h}}O_{\bbP}(1)}{a^{2h_q+1/2}}+\frac{O^*_{\bbP}(1)a^{{\mathbf h}}{\mathbf P}^*}{a^{2h_q+1/2}} \Big),
\end{equation}
\begin{equation}\label{e:W-tilde2(a2^j,K)}
\widetilde{{\mathbf W}}_2(a2^j,{\mathbf K}) = a^{2h_q+1}\Big( \frac{{\mathbf P}a^{{\mathbf h}} {\mathbf B}_a(2^j) a^{{\mathbf h}}{\mathbf P}^*}{a^{2h_q}} + \frac{O_{\bbP}(1)}{a^{2h_q+1}} + \frac{{\mathbf P}a^{{\mathbf h}}}{a^{h_q}} {\mathbf K} + {\mathbf K}^*\frac{a^{{\mathbf h}}{\mathbf P}^*}{a^{h_q}}\Big)
\end{equation}
and
\begin{equation}\label{e:W-tilde3(a2^j,K)}
\widetilde{{\mathbf W}}_3(a2^j,{\mathbf K}) = a^{2h_q+1}\Big( \frac{{\mathbf P}a^{{\mathbf h}} {\mathbf B}_a(2^j) a^{{\mathbf h}}{\mathbf P}^*}{a^{2h_q}} + {\mathbf K}\Big).
\end{equation}
Now define the functions
\begin{equation}\label{e:fv}
f_{n,q,1}({\mathbf B}) := \log \lambda_{p-r+q}\Bigg(\frac{\widetilde{{\mathbf W}}_1(a2^j,{\mathbf B})}{a^{2h_q+1}}\Bigg), \quad
f_{n,q,2}({\mathbf K})  := \log \lambda_{p-r+q}\Bigg( \frac{\widetilde{{\mathbf W}}_2(a2^j,{\mathbf K}) }{a^{2h_q+1}}\Bigg)
\end{equation}
and
\begin{equation}\label{e:fv_3}
f_{n,q,3}({\mathbf K}):= \log \lambda_{p-r+q}\Bigg( \frac{\widetilde{{\mathbf W}}_3(a2^j,{\mathbf K}) }{a^{2h_q+1}}\Bigg).
\end{equation}
For the sake of interpretation, ${\mathbf B}$, ${\mathbf K}$ in \eqref{e:fv} and ${\mathbf K}$ in \eqref{e:fv_3}, respectively, replace and generalize the arguments $\widehat{{\mathbf B}}_a(2^j)$ and ${\mathbf B}_a(2^j)$ in \eqref{e:loglambda-loglambdaE_3_terms_term1}, the argument $\frac{O_{\bbP}(1)}{a^{h_q+1/2}} $ in \eqref{e:loglambda-loglambdaE_3_terms_term2}, and the arguments $\frac{O_{\bbP}(1)}{a^{2h_q+1}}$ and $\frac{O(1)}{a^{2h_q+1}} $ in \eqref{e:loglambda-loglambdaE_3_terms_term3}. In the course of this proof, we will establish in what sense the functions in \eqref{e:fv} and \eqref{e:fv_3} are well defined.

The layout of the proof is as follows.  We will establish the  convergence
\begin{equation}\label{e:f_1_weakconv}
\Big(\sqrt{n_{a,j}}\hspace{0.5mm}\Big(f_{n,q,1}(\widehat{{\mathbf B}}_a(2^j)) - f_{n,q,1}({\mathbf B}_a(2^j))\Big)_{q=1,\hdots,r}\Big)_{j=j_1,\hdots,j_m}\stackrel{d}\rightarrow {\mathcal N}(0,\Sigma_{\lambda}),
\end{equation}
and also that
\begin{equation}\label{e:f_2_to0}
\sqrt{n_{a,j}}\Big(f_{n,q,2}\Big( \frac{a^{- {\mathbf h}- (1/2){\mathbf I}}{\mathbf W}_{X,Z}(a2^j)}{a^{h_q + 1/2}} \Big) - f_{n,q,2}(
\mathbf{0})\Big) = o_{\bbP}(1),
\end{equation}
\begin{equation}\label{e:f_3_to0}
 \sqrt{n_{a,j}}\Big(f_{n,q,3}\Big( \frac{{\mathbf W}_Z(a2^j)}{a^{2h_q+1}}\Big) -f_{n,q,3}\Big( \frac{\bbE{\mathbf W}_Z(a2^j)}{a^{2h_q+1}}\Big)\Big) = o_{\bbP}(1).
\end{equation}
Then, as a consequence of \eqref{e:loglambda-loglambdaE_3_terms}, \eqref{e:f_1_weakconv}, \eqref{e:f_2_to0} and \eqref{e:f_3_to0},
$$
 \Big( \sqrt{n_{a,j}}\Big( \log \lambda_{p-r+q}\big({\mathbf W}(a(n)2^{j})\big)  - \log \lambda_{p-r+q}\big(\bbE {\mathbf W} (a(n)2^{j}) \big) \Big)_{q=1,\hdots,r} \Big)_{j=j_1,\hdots,j_m}
$$
$$
 =  \Big(\sqrt{n_{a,j}}\hspace{0.5mm}\Big(f_{n,q,1}(\widehat{{\mathbf B}}_a(2^j)) - f_{n,q,1}({\mathbf B}_a(2^j))\Big)_{q=1,\hdots,r}\Big)_{j=j_1,\hdots,j_m}+ o_{\bbP}(1) \stackrel{d}\rightarrow {\mathcal N}(0,\Sigma_{\lambda}),
$$
as $n \rightarrow \infty$, which proves \eqref{e:asympt_normality_lambda2}.\\

So, we proceed first to establish \eqref{e:f_1_weakconv}.  Recall that, for any ${\mathbf M} \in {\mathcal S}(p,\bbR)$, the differential of a \emph{simple} eigenvalue $\lambda_{i}({\mathbf M})$ exists in a connected vicinity of ${\mathbf M}$ and is given by
\begin{equation}\label{e:dlambdal}
d\lambda_{i}({\mathbf M}) = {\mathbf u}^*_{i} \hspace{0.5mm}\{d {\mathbf M} \} \hspace{0.5mm}{\mathbf u}_{i},
\end{equation}
where ${\mathbf u}_{i}$ is a unit eigenvector of ${\mathbf M}$ associated with $\lambda_{i}({\mathbf M})$ (Magnus \cite{magnus:1985}, p.\ 182, Theorem 1).

Consider expression \eqref{e:W-tilde(a2^j,B,K1,K2)} for the matrix $\overline{{\mathbf W}} \in {\mathcal S}(p,\bbR)$. Note that
$$
\widetilde{{\mathbf W}}_1(a2^j,{\mathbf B}) = \overline{{\mathbf W}}\big(a2^j,{\mathbf B},O_{\bbP}(1)/a^{2h_q+1},O_{\bbP}(1)/a^{h_q+1/2}\big),
$$
where $O_{\bbP}(1)/a^{2h_q+1} = o_{\bbP}(1)$, $O_{\bbP}(1)/a^{h_q+1/2} = o_{\bbP}(1)$. Thus, under condition \eqref{e:xiq_distinct}, Lemma \ref{l:f1,f2,f3_well_defined} implies that, for large enough $n$ and with probability going to 1, the eigenvalue $\lambda_{p-r+q}\Big( \frac{\widetilde{{\mathbf W}}_1(a2^j,{\mathbf B}) }{a^{2h_q+1}}\Big)$ must be simple and strictly positive for any ${\mathbf B}$ in some open and connected set
\begin{equation}\label{e:O-neighborhood_delta0,r}
{\mathcal O}_{\delta_0,r}  \subseteq {\mathcal S}_{\geq 0}(r,\bbR), \quad {\mathcal O}_{\delta_0,r} \ni {\mathbf B}(2^j),
\end{equation}
in the topology of ${\mathcal S}(r,\bbR)$.
In particular, the logarithmic function $f_{n,q,1}$ in \eqref{e:fv} is well defined in the vicinity \eqref{e:O-neighborhood_delta0,r}. Then, again for large $n$ with probability going to 1, the derivative of the function $f_{n,q,1}$ exists in the vicinity \eqref{e:O-neighborhood_delta0,r}. On the other hand, by condition \eqref{e:|bfB_a(2^j)-B(2^j)|=O(shrinking)}, ${\mathbf B}_a(2^j) \rightarrow {\mathbf B}(2^j)$ as $n \rightarrow \infty$. Hence, with probability going to 1, for large enough $n$ and for any ${\mathbf B} \in {\mathcal O}_{\delta_0,r}$, an application of Lemma \ref{l:mean_value_theorem} (for the choices $T_0 = \vecoper_{{\mathcal S}} \hspace{0.5mm}{\mathbf B}_a(2^j)$, $T_1 = \vecoper_{{\mathcal S}} \hspace{0.5mm}{\mathbf B}$, $m = r(1+r)/2$, $G(\vecoper_{{\mathcal S}} {\mathbf B}) = f_{n,q,1}({\mathbf B})$) yields
\begin{equation}\label{e:Taylor_determ}
f_{n,q,1}({\mathbf B}) - f_{n,q,1}({\mathbf B}_a(2^j)) = \sum_{1 \leq \ell \leq \ell' \leq r}\frac{\partial}{\partial b_{\ell,\ell'}}f_{n,q,1}(\breve{{\mathbf B}}) \hspace{1mm} \pi_{\ell,\ell'}({\mathbf B} - {\mathbf B}_a(2^j))
\end{equation}
for some matrix $\breve{{\mathbf B}} \in {\mathcal S}_{\geq 0}(r,\bbR)$ lying in a segment connecting ${\mathbf B}$ and ${\mathbf B}_a(2^j)$ across ${\mathcal  S}_{\geq 0}(r,\bbR)$. Define the event $A_n = \big\{ \omega:  \widehat{{\mathbf B}}_{a}(2^j) \in {\mathcal O}_{\delta_0,r} \big\}$. By \eqref{e:sqrt(K)(B^-B)->N(0,sigma^2)} and \eqref{e:|bfB_a(2^j)-B(2^j)|=O(shrinking)},
\begin{equation}\label{e:W_has_pairwise_distinct_eigens}
\bbP (A_n) \rightarrow 1, \quad n \rightarrow \infty.
\end{equation}
By \eqref{e:Taylor_determ} and \eqref{e:W_has_pairwise_distinct_eigens}, with probability going to 1, for large enough $n$ the expansion
\begin{equation}\label{e:Taylor}
f_{n,q,1}(\widehat{{\mathbf B}}_{a} (2^j)) - f_{n,q,1}({\mathbf B}_a(2^j)) = \sum_{1 \leq \ell \leq \ell' \leq r}\frac{\partial}{\partial b_{\ell,\ell'}}f_{n,q,1}(\breve{{\mathbf B}}_{a} (2^j)) \hspace{1mm} \pi_{\ell,\ell'}(\widehat{{\mathbf B}}_{a}(2^j) - {\mathbf B}_a(2^j))
\end{equation}
holds for some matrix $\breve{{\mathbf B}}_{a}(2^j)$ lying in a segment connecting $\widehat{{\mathbf B}}_{a}(2^j)$ and ${\mathbf B}_a(2^j)$ across ${\mathcal  S}_{\geq 0}(r,\bbR)$.

So, for a generic matrix $\breve{{\mathbf B}} \in {\mathcal S}_{\geq 0}(r,\bbR)$, consider the matrix of derivatives
$$
\Big\{\frac{\partial}{\partial b_{\ell,\ell'}}f_{n,q,1}(\breve{{\mathbf B}})\Big\}_{1 \leq \ell \leq \ell' \leq r}
$$
\begin{equation}\label{e:df1/dbi1,i2}= \Big\{ \lambda^{-1}_{p-r+q}\Big(\frac{ \widetilde{{\mathbf W}}_1(a2^j,\breve{{\mathbf B}} ) }{a^{2h_q+1}}\Big)
\cdot \frac{\partial}{\partial b_{\ell,\ell'}}\lambda_{p-r+q}\Big(\frac{ \widetilde{{\mathbf W}}_1(a2^j,\breve{{\mathbf B}} )}{a^{2h_q+1}}\Big)\Big\}_{1 \leq \ell \leq \ell' \leq r}.
\end{equation}
In \eqref{e:df1/dbi1,i2}, the differential of the eigenvalue $\lambda_{p-r+q}\Big(\frac{\widetilde{{\mathbf W}}_1(a2^j,\breve{{\mathbf B}} )}{a^{2h_q+1}}\Big)$ is given by expression \eqref{e:dlambdal} with $\frac{\widetilde{{\mathbf W}}_1(a2^j,\breve{{\mathbf B}} )}{a^{2h_q+1}}$ in place of ${\mathbf M}$ and ${\mathbf u}_{i}:={\mathbf u}_{p-r+q}(n) = {\mathbf u}_{p-r+q}\big(n,\widetilde{{\mathbf W}}_1(a2^j,\breve{{\mathbf B}} )\big)$ denoting a unit eigenvector of $\frac{\widetilde{{\mathbf W}}_1(a2^j,\breve{{\mathbf B}} )}{a^{2h_q+1}}$ associated with its $(p-r+q)$--th eigenvalue. Moreover,
$$
\frac{\partial}{\partial b_{\ell,\ell'}}\lambda_{p-r+q}\Big(\frac{\widetilde{{\mathbf W}}_1(a2^j,\breve{{\mathbf B}} )}{a^{2h_q+1}}\Big)
$$
$$
= {\mathbf u}^*_{p-r+q}(n)\frac{\partial}{\partial b_{\ell,\ell'}}\Big[\frac{{\mathbf P}a^{{\mathbf h}} \breve{{\mathbf B}} a^{{\mathbf h}}{\mathbf P}^*}{a^{2h_q}} +\frac{O_{\bbP}(1)}{a^{2h_q+1}} + \frac{{\mathbf P}a^{{\mathbf h}}}{a^{h_q}} \frac{O_{\bbP}(1)}{a^{h_q+1/2}} + \frac{O^*_{\bbP}(1)}{a^{h_q+1/2}}\frac{a^{{\mathbf h}}{\mathbf P}^*}{a^{h_q}}\Big]{\mathbf u}_{p-r+q}(n)
$$
\begin{equation}\label{e:deriv_lambda_p-r+q(W-tilde_1)}
= {\mathbf u}^*_{p-r+q}(n)\Big( \frac{{\mathbf P}a^{{\mathbf h}} {\mathbf 1}_{(\ell,\ell') \cup (\ell',\ell)} a^{{\mathbf h}}{\mathbf P}^*}{a^{2h_q}}\Big){\mathbf u}_{p-r+q}(n), \quad \quad  1 \leq \ell \leq \ell' \leq r,
\end{equation}
where ${\boldsymbol 1}_{(\ell,\ell') \cup (\ell',\ell)}$ is a matrix with 1 on entries $(\ell,\ell')$ and $(\ell',\ell)$, and zeroes elsewhere. Now consider using $\breve{{\mathbf B}}_{a}(2^j)$ in place of $\breve{{\mathbf B}}$ and $\widetilde{{\mathbf W}}_1(a2^j,\breve{{\mathbf B}}_{a}(2^j))$ in place of $\widetilde{{\mathbf W}}_1(a2^j,\breve{{\mathbf B}})$ in \eqref{e:deriv_lambda_p-r+q(W-tilde_1)}. By relation \eqref{e:dlambdal}, under condition \eqref{e:xiq_distinct}, we can pick the sequence ${\mathbf u}_{p-r+q}(n)$ provided in Proposition \ref{p:|lambdaq(EW)-xiq(2^j)|_bound} and Lemma \ref{l:max|<up-r+q(n),pi(n)>|a(n)^(hi-hq)=OP(1)} so as to obtain, for $1 \leq \ell \leq \ell' \leq r$,
$$
a^{-(2h_{q}+1)} {\mathbf u}^*_{p-r+q}(n)\Big\{\frac{\partial}{\partial b_{\ell,\ell'}}\widetilde{{\mathbf W}}_1(a2^j,\breve{{\mathbf B}}_{a}(2^j))\Big\}{\mathbf u}_{p-r+q}(n)
$$
\begin{equation}\label{e:a^(-(2hq+1))u*_p-r+q(n)partial_W_u_p-r+q(n)}
= {\mathbf u}^*_{p-r+q}(n){\mathbf P} \textnormal{diag}(a^{h_1 - h_{q}},\hdots,a^{h_r - h_{q}}){\boldsymbol 1}_{(\ell,\ell') \cup (\ell',\ell)}
\textnormal{diag}(a^{h_1 - h_{q}},\hdots,a^{h_r - h_{q}}){\mathbf P}^*{\mathbf u}_{p-r+q}(n).
\end{equation}
If the indices are such that $\ell < \ell'$, then \eqref{e:a^(-(2hq+1))u*_p-r+q(n)partial_W_u_p-r+q(n)} is equal to
\begin{equation}\label{e:limit_derivative_eigenvalue}
2 \langle {\mathbf p}_{ \ell}(n),{\mathbf u}_{p-r+q}(n)\rangle a^{h_{\ell} - h_{q}} \langle {\mathbf p}_{ \ell'}(n),{\mathbf u}_{p-r+q}(n)\rangle a^{h_{\ell'} - h_{q}}
\stackrel{\bbP}\rightarrow \left\{\begin{array}{cc}
0, & \ell \in\mathcal{I}_-;\\
2\hspace{0.5mm}\gamma_{{\ell}q }\gamma_{\ell'q}, & \ell,\ell'\in\mathcal{I}_0, \hspace{1mm}\ell < \ell';\\
2\hspace{0.5mm}\gamma_{\ell q} x_{\ell',*}, & \ell \in \mathcal{I}_0, \ell'\in \mathcal{I}_+;\\
2\hspace{0.5mm}x_{\ell,*}x_{\ell',*}, &  \ell,\ell'\in \mathcal{I}_+, \hspace{1mm}\ell < \ell'.\\
\end{array}\right.
\end{equation}
Otherwise, i.e., if $\ell = \ell'$, then \eqref{e:a^(-(2hq+1))u*_p-r+q(n)partial_W_u_p-r+q(n)} is equal to
\begin{equation}\label{e:limit_derivative_eigenvalue_ell=ell'}
a^{2(h_{\ell} - h_{q})}  \langle {\mathbf p}_{ \ell}(n),{\mathbf u}_{p-r+q}(n)\rangle^2 \stackrel{\bbP}\rightarrow \left\{\begin{array}{cc}
0, & \ell \in\mathcal{I}_-;\\
\gamma^2_{{\ell}q }, & \ell=\ell'\in\mathcal{I}_0;\\
x^2_{\ell,*}, &  \ell=\ell'\in \mathcal{I}_+.\\
\end{array}\right.
\end{equation}
In both \eqref{e:limit_derivative_eigenvalue} and \eqref{e:limit_derivative_eigenvalue_ell=ell'}, the entries $x_{\ell,*}$ (depending on $q$), $\ell \in \mathcal{I}_+$, of the vector $\textbf{x}_{q,*}(2^j)$ are given by expression \eqref{e:inner*scaling=o(1)} in Lemma \ref{l:max|<up-r+q(n),pi(n)>|a(n)^(hi-hq)=OP(1)}. In turn, the entries $\gamma_{\ell,q}$, $\ell = 1,\hdots,r$, of the vector ${\boldsymbol \gamma}_{\ell}$ are given by expression \eqref{e:<p,u>_to_gamma} in Proposition \ref{p:|lambdaq(EW)-xiq(2^j)|_bound}. In addition, since $\breve{{\mathbf B}}_a(2^j) \stackrel{\bbP}\rightarrow {\mathbf B}(2^j)$, $n\rightarrow \infty$, as a consequence of conditions \eqref{e:sqrt(K)(B^-B)->N(0,sigma^2)} and \eqref{e:|bfB_a(2^j)-B(2^j)|=O(shrinking)}, Corollary \ref{c:PWXP^*+R_asymptotics} implies that
\begin{equation}\label{e:lambda_p-r+q(W1)/a^(2hq+1)->xi_q}
\lambda_{p-r+q}\Big(\frac{\widetilde{{\mathbf W}}_1(a2^j,\breve{{\mathbf B}}_{a}(2^j))}{a^{2h_q+1}}\Big) \stackrel{\bbP}\rightarrow \xi_q(2^j)>0, \quad n \rightarrow \infty.
\end{equation}
Then, by \eqref{e:limit_derivative_eigenvalue}--\eqref{e:lambda_p-r+q(W1)/a^(2hq+1)->xi_q}, the limit in probability of expression \eqref{e:df1/dbi1,i2} (with  $\breve{{\mathbf B}}_{a}(2^j)$ in place of $\breve{{\mathbf B}}$) can be pictorially represented by the upper triangular scheme
\begin{equation}\label{e:limit_derivative}
\frac{1}{\xi_q(2^j)}\times
\left(\begin{array}{cccccc}
{\mathbf 0} & {\mathbf 0} & {\mathbf 0} \\
 & \hspace{3mm}\big\{{\boldsymbol 1}_{\ell=\ell'} + 2 {\boldsymbol 1}_{\ell < \ell'}\big\}\hspace{0.5mm} (\gamma_{{\ell}q }\gamma_{\ell'q})_{\ell,\ell'\in\mathcal{I}_0,\hspace{0.5mm}\ell \leq \ell'} & 2(\gamma_{\ell q} x_{\ell',*})_{\ell \in \mathcal{I}_0, \ell'\in \mathcal{I}_+, \hspace{0.5mm}\ell \leq \ell'}\\
 &   & \big\{{\boldsymbol 1}_{\ell=\ell'} + 2 {\boldsymbol 1}_{\ell < \ell'}\big\}\hspace{0.5mm}(x_{\ell,*}x_{\ell',*})_{\ell,\ell'\in \mathcal{I}_+, \hspace{0.5mm}\ell \leq \ell'}
\end{array}\right).
\end{equation}
In \eqref{e:limit_derivative}, the empty entries are not used.
The ${\mathbf 0}$ on the upper left corner is a placeholder for a triangular array of zeroes, the other two ${\mathbf 0}$s being placeholders for rectangular ones.

Turning back to \eqref{e:Taylor}, by the arbitrariness of $j$ and $q$, expression \eqref{e:limit_derivative} and condition \eqref{e:sqrt(K)(B^-B)->N(0,sigma^2)} imply that
\begin{equation}\label{e:term1_conv_in_law}
\Big(\sqrt{n_{a,j}}\hspace{0.5mm}\Big(f_{n,q,1}(\widehat{{\mathbf B}}_a(2^j)) - f_{n,q,1}({\mathbf B}_a(2^j))\Big)_{q=1,\hdots,r}\Big)_{j=j_1,\hdots,j_m}\stackrel{d}\rightarrow {\mathcal N}(0,\Sigma_{\lambda}),
\end{equation}
as $n \rightarrow \infty$, for some $\Sigma_\lambda \in \mathcal S_{\geq 0}(r\cdot m,\bbR)$, i.e.,  \eqref{e:f_1_weakconv} holds.

We now turn to \eqref{e:f_3_to0}. Consider expression \eqref{e:W-tilde(a2^j,B,K1,K2)} for the matrix $\overline{{\mathbf W}} \in {\mathcal S}(p,\bbR)$. Note that
$$
\widetilde{{\mathbf W}}_3(a2^j,{\mathbf K}) = \overline{{\mathbf W}}\big(a2^j,{\mathbf B}_a(2^j),{\mathbf K},{\mathbf 0}\big),
$$
where ${\mathbf B}_a(2^j)$ satisfies \eqref{e:|bfB_a(2^j)-B(2^j)|=O(shrinking)}. Hence, under condition \eqref{e:xiq_distinct}, Lemma \ref{l:f1,f2,f3_well_defined} implies that, for large enough $n$, the deterministic eigenvalue $\lambda_{p-r+q}\Big( \frac{\widetilde{{\mathbf W}}_3(a2^j,{\mathbf K}) }{a^{2h_q+1}}\Big)$ must be simple and strictly positive for any ${\mathbf K}$ in some open and connected vicinity
\begin{equation}\label{e:O-neighborhood_zeta01,p}
{\mathcal O}_{\zeta_{01},p} = \{ {\mathbf K} \in {\mathcal S}(p,\bbR): \|{\mathbf K}\| < \zeta_{01}\}
\end{equation}
(\textbf{n.b.}: $ p = p(n)$). In particular, the logarithmic function $f_{n,q,3}$ in \eqref{e:fv_3} is well defined in the vicinity \eqref{e:O-neighborhood_zeta01,p}. Hence, an application of Lemma \ref{l:mean_value_theorem} (for $T_0 = \vecoper_{{\mathcal S}} \hspace{0.5mm}\big(O_{\bbP}(1)/a^{2h_q+1}\big)$, $T_1 = \vecoper_{{\mathcal S}} \hspace{0.5mm}{\mathbf K}$, $m = p(1+p)/2$, $G(\vecoper_{{\mathcal S}}\hspace{0.5mm} {\mathbf K}) = f_{n,q,3}({\mathbf K})$) implies that, for large $n$, we can write
$$
f_{n,q,3}({\mathbf K}) - f_{n,q,3}\Big(\frac{O(1)}{a^{2h_q+1}}\Big)
$$
$$
 = \log \lambda_{p-r+q}\Big(\frac{\widetilde{{\mathbf W}}_3(a2^j,{\mathbf K})}{a^{2h_q+1}}\Big)
- \log \lambda_{p-r+q}\Big(\frac{\widetilde{{\mathbf W}}_3(a2^j,O(1)/a^{2h_q+1})}{a^{2h_q+1}}\Big)
$$
\begin{equation}\label{e:fnq3(R)-fnq3(0)}
= \sum_{1 \leq i_1 \leq i_2 \leq p} \frac{\partial}{\partial \kappa_{i_1,i_2}}f_{n,q,3}(\breve{{\mathbf K}}) \hspace{1mm} \pi_{i_1,i_2}\Big({\mathbf K} -  \frac{O(1)}{a^{2h_q+1}}   \Big)
\end{equation}
for some matrix $\breve{{\mathbf K}} \in {\mathcal S}_{\geq 0}(p,\bbR)$ lying in a segment connecting ${\mathbf K}$ and $\frac{O(1)}{a^{2h_q+1}} $ across ${\mathcal S}_{\geq 0}(p,\bbR)$. For a generic matrix $\breve{{\mathbf K}} \in {\mathcal S}_{\geq 0}(p,\bbR)$, consider the matrix of derivatives
$$
\Big\{\frac{\partial}{\partial \kappa_{i_1,i_2}}f_{n,q,3}(\breve{{\mathbf K}})\Big\}_{1 \leq i_1 \leq i_2 \leq p}
$$
\begin{equation}\label{e:df/dbi1,i2}
= \Big\{ \lambda^{-1}_{p-r+q}\Big(\frac{\widetilde{{\mathbf W}}_3(a2^j,\breve{{\mathbf K}})}{a^{2h_q+1}}\Big)
\cdot \frac{\partial}{\partial \kappa_{i_1,i_2}}\lambda_{p-r+q} \Big(\frac{\widetilde{{\mathbf W}}_3(a2^j,\breve{{\mathbf K}})}{a^{2h_q+1}}\Big) \Big\}_{1 \leq i_1 \leq i_2 \leq p}.
\end{equation}
In \eqref{e:df/dbi1,i2}, the differential of the eigenvalue $\lambda_{p-r+q} \big(\frac{\widetilde{{\mathbf W}}_3(a2^j,\breve{{\mathbf K}})}{a^{2h_q+1}}\big)$ is given by expression \eqref{e:dlambdal} with $\frac{\widetilde{{\mathbf W}}_3(a2^j,\breve{{\mathbf K}})}{a^{2h_q+1}}$
in place of ${\mathbf M}$ and ${\mathbf u}_i := {\mathbf u}_{p-r+q}(n) = {\mathbf u}_{p-r+q}\big(n,\widetilde{{\mathbf W}}_3(a2^j,\breve{{\mathbf K}})\big)$ denoting a unit eigenvector of $\frac{\widetilde{{\mathbf W}}_3(a2^j,\breve{{\mathbf K}})}{a^{2h_q+1}}$ associated with its $(p-r+q)$--th eigenvalue. In addition,
$$
\frac{\partial}{\partial \kappa_{i_1,i_2}}\lambda_{p-r+q}\Big(\frac{{\mathbf P}a^{{\mathbf h}} {\mathbf B}_a(2^j) a^{{\mathbf h}}{\mathbf P}^*}{a^{2h_q}} + {\mathbf K}\Big)
$$$$
= {\mathbf u}^*_{p-r+q}(n)\frac{\partial}{\partial \kappa_{i_1,i_2}}\Big[\frac{{\mathbf P}a^{{\mathbf h}} {\mathbf B}_a(2^j) a^{{\mathbf h}}{\mathbf P}^*}{a^{2h_q}} + {\mathbf K}\Big]{\mathbf u}_{p-r+q}(n)
$$
$$
= {\mathbf u}^*_{p-r+q}(n) {\mathbf 1}_{(i_1,i_2) \cup (i_2,i_1)}{\mathbf u}_{p-r+q}(n), \quad  1 \leq i_1 \leq i_2 \leq p.
$$
Under condition \eqref{e:assumptions_WZ=OP(1)}, ${\mathbf W}_Z(a2^j)/a^{2h_q + 1}  = O_{\bbP}(1)/a^{2h_q+1}=o_{\bbP}(1)$. So, for large enough $n$ with probability going to 1, expression \eqref{e:fnq3(R)-fnq3(0)} implies that
$$
\sqrt{\frac{n}{a}}\Big(f_{n,q,3}\Big( \frac{{\mathbf W}_Z(a2^j)}{a^{2h_q+1}}\Big) -f_{n,q,3}\Big( \frac{\bbE{\mathbf W}_Z(a2^j)}{a^{2h_q+1}}  \Big)\Big)
$$
$$
= \sum_{1 \leq i_1 \leq i_2 \leq p}\frac{\partial}{\partial \kappa_{i_1,i_2}}f_{n,q,3}(\breve{{\mathbf K}}_n) \hspace{1mm} \sqrt{\frac{n}{a}}\pi_{i_1,i_2} \Big( \frac{{\mathbf W}_Z(a2^j)}{a^{2h_q+1}} - \frac{\bbE{\mathbf W}_Z(a2^j)}{a^{2h_q+1}} \Big)
$$
$$
= \lambda^{-1}_{p-r+q}\Big( \frac{\widetilde{{\mathbf W}}_3(a2^j,\breve{{\mathbf K}}_n) }{a^{2h_q+1}}\Big)$$
\begin{equation}\label{e:sqrt(n/a)(fnq3(WZ/a^2hq+1)-fnq(0))}
\quad \times  \sum_{1 \leq i_1 \leq i_2 \leq p} {\mathbf u}^*_{p-r+q}(n) {\mathbf 1}_{(i_1,i_2) \cup (i_2,i_1)}{\mathbf u}_{p-r+q}(n) \hspace{1mm} \sqrt{\frac{n}{a}}\pi_{i_1,i_2}  \Big( \frac{{\mathbf W}_Z(a2^j)}{a^{2h_q+1}} - \frac{\bbE{\mathbf W}_Z(a2^j)}{a^{2h_q+1}} \Big) .
\end{equation}
In \eqref{e:sqrt(n/a)(fnq3(WZ/a^2hq+1)-fnq(0))}, the matrix $\breve{{\mathbf K}}_n$ lies in a segment connecting ${\mathbf W}_Z(a2^j)/a^{2h_q+1} = o_{\bbP}(1)$ and $\bbE{\mathbf W}_Z(a2^j)/a^{2h_q+1}= o(1)$ across ${\mathcal S}_{\geq 0}(p,\bbR)$. Hence, $\breve{{\mathbf K}}_n=o_{\bbP}(1)$. Thus, Corollary \ref{c:PWXP^*+R_asymptotics} implies that
\begin{equation} \label{e:lambda_p-r+q(W3)/a^(2hq+1)->xi_q(2^j)}
\lambda_{p-r+q}\Big(\frac{\widetilde{{\mathbf W}}_3(a2^j,\breve{{\mathbf K}}_n)}{a^{2h_q+1}}\Big) \stackrel{\bbP}\rightarrow \xi_{q}(2^j) > 0,
\end{equation}
$n \rightarrow \infty$. Now recall that, for ${\mathbf M} = (m_{i_1,i_2}) \in {\mathcal M}(p_1,p_2,\bbR)$,
\begin{equation}\label{e:max|sil(p)|=<C'}
\max_{i_1=1,\hdots,p_1; \hspace{.5mm}i_2 =1,\hdots,p_2}|m_{i_1, i_2}| \leq \| {\mathbf M}\|.
\end{equation}
Thus, by condition \eqref{e:assumptions_WZ=OP(1)},
\begin{equation}\label{e:max_pi_i1,i2(WZ-EWZ)=OP(1)}
\max_{i_1,i_2=1,\hdots,p}\big|\pi_{i_1,i_2}\big( {\mathbf W}_Z(a2^j) - \bbE {\mathbf W}_Z(a2^j) \big)\big| = O_{\bbP}(1).
\end{equation}
Also recall that, for a vector $\textbf{x} \in \bbR^p$,
\begin{equation}\label{e:|x|1=<sqrt(p)|x|2}
\|\textbf{x}\|_{1}\leq \sqrt{p} \hspace{0.5mm}\|\textbf{x}\|_{2}, \quad \text{where } \|\textbf{x}\|_{\ell}:= \Big(\sum^{p}_{i=1} |x_i|^{\ell}\Big)^{1/\ell}, \quad \ell = 1,2.
\end{equation}
Then, by \eqref{e:lambda_p-r+q(W3)/a^(2hq+1)->xi_q(2^j)} and \eqref{e:max_pi_i1,i2(WZ-EWZ)=OP(1)}, with probability going to 1 expression \eqref{e:sqrt(n/a)(fnq3(WZ/a^2hq+1)-fnq(0))} is bounded in absolute value by
$$
 \frac{O_{\bbP}(1)}{a^{2h_q+1}}\sqrt{\frac{n}{a}} \sum^{p}_{i_1 = 1}\sum^{p}_{i_2 = 1} |{\mathbf u}_{p-r+q}(n)_{i_1}{\mathbf u}_{p-r+q}(n)_{i_2}| =
\frac{O_{\bbP}(1)}{a^{2h_q+1}}\sqrt{\frac{n}{a}} \hspace{0.5mm}\|{\mathbf u}_{p-r+q}(n)\|^2_{1}
$$
\begin{equation}\label{e:term3_bound}
\leq \frac{O_{\bbP}(1)}{a^{2h_q+1}}\sqrt{\frac{n}{a}} \hspace{0.5mm}p = O_{\bbP}(1)\Big(\frac{n}{a^{(4/3) h_q + 5/3}}\Big)^{3/2} \frac{p}{n/a} \stackrel{\bbP}\to 0, \quad n \rightarrow \infty.
\end{equation}
In \eqref{e:term3_bound}, the inequality follows from \eqref{e:|x|1=<sqrt(p)|x|2} and the limit follows from condition \eqref{e:p(n),a(n)_conditions}, since $(4/3) h_q + 5/3> h_1 + 3/2$.  Therefore,
$
\sqrt{\frac{n}{a}}\Big(f_{n,q,3}\Big( \frac{{\mathbf W}_Z(a2^j)}{a^{2h_q+1}}\Big) -f_{n,q,3}\Big( \frac{\bbE{\mathbf W}_Z(a2^j)}{a^{2h_q+1}}  \Big)\Big) = o_{\bbP}(1),
$
which corresponds to \eqref{e:f_3_to0}.

We now turn to \eqref{e:f_2_to0}. Consider expression \eqref{e:W-tilde(a2^j,B,K1,K2)} for the matrix $\overline{{\mathbf W}} \in {\mathcal S}(p,\bbR)$. Note that
$$
\widetilde{{\mathbf W}}_2(a2^j,{\mathbf K}) = \overline{{\mathbf W}}\big(a2^j,{\mathbf B}_a(2^j),O_{\bbP}(1)/a^{2h_q+1},{\mathbf K}\big),
$$
where ${\mathbf B}_a(2^j)$ satisfies \eqref{e:|bfB_a(2^j)-B(2^j)|=O(shrinking)} and $O_{\bbP}(1)/a^{2h_q+1} = o_{\bbP}(1)$. Thus, under condition \eqref{e:xiq_distinct}, Lemma \ref{l:f1,f2,f3_well_defined} implies that, for large enough $n$ and with probability going to 1, the eigenvalue $\lambda_{p-r+q}\Big( \frac{\widetilde{{\mathbf W}}_2(a2^j,{\mathbf K}) }{a^{2h_q+1}}\Big)$ must be simple and positive for any ${\mathbf K}$ in some open and connected vicinity
\begin{equation}\label{e:O-neighborhood_zeta02,p}
{\mathcal O}_{\zeta_{02},r,p} = \{ {\mathbf K} \in {\mathcal M}(r,p,\bbR): \|{\mathbf K}\| < \zeta_{02}\}
\end{equation}
in the topology of ${\mathcal M}(r,p,\bbR)$ (\textbf{n.b.}: $ p = p(n)$). In particular, the logarithmic function $f_{n,q,2}$ in \eqref{e:fv} is well defined in the vicinity \eqref{e:O-neighborhood_zeta02,p}. Hence, for ${\mathbf K }\in {\mathcal O}_{\zeta_{02},r,p}$, an application of Lemma \ref{l:mean_value_theorem} (for $T_0 = {\mathbf 0}$, $T_1 = \vecoper \hspace{0.5mm} {\mathbf K}$ with $\vecoper$ as in \eqref{e:vec_non-symm_def}, $m = r\cdot p$, $G(\vecoper \hspace{0.5mm}{\mathbf K}) = f_{n,q,2}({\mathbf K})$) implies that
$$
f_{n,q,2}({\mathbf K}) - f_{n,q,2}({\mathbf 0}) = \log \lambda_{p-r+q}\Big(\frac{\widetilde{{\mathbf W}}_2(a2^j,{\mathbf K})}{a^{2h_q+1}}\Big)
- \log \lambda_{p-r+q}\Big(\frac{\widetilde{{\mathbf W}}_2(a2^j,{\mathbf 0})}{a^{2h_q+1}}\Big)
$$
\begin{equation}\label{e:f_n,q,2(R)-f_n,q,2(0)=MVT}
= \sum^{r}_{i_1 = 1}\sum^{p}_{i_2 = 1}\frac{\partial}{\partial \kappa_{i_1,i_2}}f_{n,q,2}(\breve{{\mathbf K}}) \hspace{1mm} \pi_{i_1,i_2}({\mathbf K})
\end{equation}
for some matrix $\breve{{\mathbf K}} \in {\mathcal M}(r,p,\bbR)$ lying in a segment connecting ${\mathbf K}$ and ${\mathbf 0}$ across ${\mathcal M}(r,p,\bbR)$. For a generic matrix ${\mathbf K} \in {\mathcal M}(r,p,\bbR)$, consider the matrix of derivatives
$$
\Big\{\frac{\partial}{\partial \kappa_{i_1,i_2}}f_{n,q,2}(\breve{{\mathbf K}})\Big\}_{i_1= 1,\hdots,r; \hspace{1mm}i_2 = 1,\hdots,p}
$$
\begin{equation}\label{e:df/dri1,i2}
= \Big\{ \lambda^{-1}_{p-r+q}\Big(\frac{\widetilde{{\mathbf W}}_2(a2^j,\breve{{\mathbf K}})}{a^{2h_q+1}}\Big)
\cdot \frac{\partial}{\partial \kappa_{i_1,i_2}}\lambda_{p-r+q}\Big(\frac{\widetilde{{\mathbf W}}_2(a2^j,\breve{{\mathbf K}})}{a^{2h_q+1}}\Big)\Big\}_{i_1= 1,\hdots,r; \hspace{1mm}i_2 = 1,\hdots,p}.
\end{equation}
In \eqref{e:df/dri1,i2}, the differential of the eigenvalue $\lambda_{p-r+q}\big(\frac{\widetilde{{\mathbf W}}_2(a2^j,\breve{{\mathbf K}})}{a^{2h_q+1}}\big)$ is given by expression \eqref{e:dlambdal} with $\frac{\widetilde{{\mathbf W}}_2(a2^j,\breve{{\mathbf K}})}{a^{2h_q+1}}$
in place of ${\mathbf M}$ and ${\mathbf u}_i := {\mathbf u}_{p-r+q}(n) = {\mathbf u}_{p-r+q}(n,\widetilde{{\mathbf W}}_2(a2^j,\breve{{\mathbf K}}))$ denoting a unit eigenvector of $\frac{\widetilde{{\mathbf W}}_2(a2^j,\breve{{\mathbf K}})}{a^{2h_q+1}}$ associated with its $(p-r+q)$--th eigenvalue. Moreover,
$$
\frac{\partial}{\partial \kappa_{i_1,i_2}}\lambda_{p-r+q}\Big(\frac{\widetilde{{\mathbf W}}_2(a2^j,\breve{{\mathbf K}})}{a^{2h_q+1}}\Big)
$$
$$
= {\mathbf u}^*_{p-r+q}(n)\frac{\partial}{\partial \kappa_{i_1,i_2}}\Big[\frac{{\mathbf P}a^{{\mathbf h}} \mathbf{B}_a(2^j) a^{{\mathbf h}}{\mathbf P}^*}{a^{2h_q}} +\frac{O_{\bbP}(1)}{a^{2h_q+1}} + \frac{{\mathbf P}a^{{\mathbf h}}}{a^{h_q}} \breve{{\mathbf K}} + \breve{{\mathbf K}}^* \frac{a^{{\mathbf h}}{\mathbf P}^*}{a^{h_q}}\Big]{\mathbf u}_{p-r+q}(n)
$$
$$
= {\mathbf u}^*_{p-r+q}(n)\Big( \frac{{\mathbf P}a^{{\mathbf h}}}{a^{h_q}} {\mathbf 1}_{i_1,i_2}  + {\mathbf 1}_{i_2,i_1} \frac{a^{{\mathbf h}}{\mathbf P}^*}{a^{h_q}}\Big){\mathbf u}_{p-r+q}(n), \quad \quad  i_1= 1,\hdots,r, \quad i_2= 1,\hdots,p.
$$
Therefore,
\begin{equation}\label{e:sqrt(n)(fnq2(R)-fnq2(0))}
\sqrt{\frac{n}{a}}\Big(f_{n,q,2}({\mathbf K}) - f_{n,q,2}(\mathbf{0})\Big) = \lambda^{-1}_{p-r+q}\Big(\frac{\widetilde{{\mathbf W}}_2(a2^j,\breve{{\mathbf K}})}{a^{2h_q+1}}\Big)
\end{equation}
$$
\times \sum^{r}_{i_1 = 1}\sum^{p}_{i_2 = 1} {\mathbf u}^*_{p-r+q}(n)\Big( \frac{{\mathbf P}a^{{\mathbf h}}}{a^{h_q}} {\mathbf 1}_{i_1,i_2} + {\mathbf 1}_{i_2,i_1}\frac{a^{{\mathbf h}}{\mathbf P}^*}{a^{h_q}}\Big){\mathbf u}_{p-r+q}(n)
\hspace{1mm}\sqrt{\frac{n}{a}} \pi_{i_1,i_2} ({\mathbf K}).
$$
Note that, by Lemma \ref{l:|a(n)(-D)W_X,Z(a(n)2j))|=OP(1)}, $a^{- {\mathbf h}- (1/2){\mathbf I}}{\mathbf W}_{X,Z}(a2^j)/a^{h_q+1/2} = O_{\bbP}(1)/a^{h_q+1/2}= o_{\bbP}(1)$. Thus, for large $n$ with probability going to 1, expression  \eqref{e:sqrt(n)(fnq2(R)-fnq2(0))} implies that
\begin{equation}\label{e:sqrt(n/a)(fnq2(o(1))-fnq0(0))}
\sqrt{\frac{n}{a}}\Big(f_{n,q,2}\Big( \frac{a^{- {\mathbf h}- (1/2){\mathbf I}}{\mathbf W}_{X,Z}(a2^j)}{a^{h_q + 1/2}} \Big) - f_{n,q,2}(\mathbf{0})\Big)
\end{equation}
$$
= \lambda^{-1}_{p-r+q}\Big(\frac{\widetilde{{\mathbf W}}_2(a2^j,\breve{{\mathbf K}}_n)}{a^{2h_q+1}}\Big)\sum^{r}_{i_1 = 1}\sum^{p}_{i_2 = 1}\Bigg({\mathbf u}^*_{p-r+q}(n)\Big( \frac{{\mathbf P}a^{{\mathbf h}}}{a^{h_q}} {\mathbf 1}_{i_1,i_2} + {\mathbf 1}_{i_2,i_1}\frac{a^{{\mathbf h}}{\mathbf P}^*}{a^{h_q}}\Big){\mathbf u}_{p-r+q}(n)
$$
$$
\times \sqrt{\frac{n}{a}} \pi_{i_1,i_2} \Big( \frac{a^{- {\mathbf h}- (1/2){\mathbf I}}{\mathbf W}_{X,Z}(a2^j)}{a^{h_q + 1/2}} \Big)\Bigg).
$$
Note that $\breve {\mathbf K}_n$ lies in a segment connecting $\mathbf 0$ and the matrix $a^{- {\mathbf h}- (1/2){\mathbf I}}{\mathbf W}_{X,Z}(a2^j)/ a^{h_q + 1/2} = o_{\bbP}(1)$ across ${\mathcal M}(r,p,\bbR)$. Hence, $\breve {\mathbf K}_n= o_{\bbP}(1)$. Thus, Corollary \ref{c:PWXP^*+R_asymptotics} implies that
\begin{equation}\label{e:max_(up-r+q,p_ell(n))a^(h_ll-hq)}
\lambda_{p-r+q}\Big(\frac{\widetilde{{\mathbf W}}_2(a2^j,\breve{{\mathbf K}}_n)}{a^{2h_q+1}}\Big)\stackrel{\bbP}\rightarrow \xi_{q}(2^j), \quad n \rightarrow \infty.
\end{equation}
As a consequence of Lemma \ref{l:|a(n)(-D)W_X,Z(a(n)2j))|=OP(1)} and of \eqref{e:max|sil(p)|=<C'},
$
\max_{i_1,i_2}|\pi_{i_1,i_2}( a^{-\mathbf{h}-(1/2){\mathbf I}}{\mathbf W}_{X,Z}(a2^j) )| = O_{\bbP}(1).
$
Therefore, by expressions \eqref{e:|x|1=<sqrt(p)|x|2}, \eqref{e:max_(up-r+q,p_ell(n))a^(h_ll-hq)} and by condition \eqref{e:p(n),a(n)_conditions}, with probability going to 1 the right-hand side of \eqref{e:sqrt(n/a)(fnq2(o(1))-fnq0(0))} is bounded, in absolute value, by
$$
O_{\bbP}(1) \sqrt{\frac{n}{a}} \Big( \frac{1}{a^{h_q + 1/2}} \Big)\sum^{r}_{i_1 = 1}\sum^{p}_{i_2 = 1} \Big|{\mathbf u}^*_{p-r+q}(n)\Big( \frac{{\mathbf P}a^{{\mathbf h}}}{a^{h_q}} {\mathbf 1}_{i_1,i_2} + {\mathbf 1}_{i_2,i_1}\frac{a^{{\mathbf h}}{\mathbf P}^*}{a^{h_q}}\Big){\mathbf u}_{p-r+q}(n)\Big|
$$
$$
= O_{\bbP}(1) \sqrt{\frac{n}{a}} \Big( \frac{1}{a^{h_q + 1/2}} \Big)\sum^{r}_{i_1 = 1}\Big| \langle{\mathbf p}_{i_1},{\mathbf u}_{p-r+q}(n)\rangle a^{h_{i_1}-h_q} \Big| \hspace{1mm}\sum^{p}_{i_2 = 1}  | {\mathbf u}_{p-r+q}(n)_{i_2} |
$$
\begin{equation}\label{e:sqrt(n/a)(fnq2(rescaled_WXZ)-fnq2(0))_bound}
\leq O_{\bbP}(1) \sqrt{\frac{n}{a}} \Big( \frac{1}{a^{h_q + 1/2}} \Big) \hspace{1mm} \sqrt{p}
= O_{\bbP}(1) \sqrt{\frac{p}{n/a}} \cdot  \frac{n}{a^{h_q + 3/2}} \stackrel \bbP \to 0.
\end{equation}
In \eqref{e:sqrt(n/a)(fnq2(rescaled_WXZ)-fnq2(0))_bound}, $\sum^{r}_{i_1 = 1} | \langle{\mathbf p}_{i_1},{\mathbf u}_{p-r+q}(n)\rangle a^{h_{i_1}-h_q} | = O_{\bbP}(1)$ is a consequence of Lemma \ref{l:|<p3,uq(n)>|a(n)^{h_3-h_1}=O(1)}, $(i)$, and the limit follows from \eqref{e:p(n),a(n)_conditions}. In other words, %
\eqref{e:f_2_to0} holds. This concludes the proof of \eqref{e:asympt_normality_lambda2}. $\Box$\\

\begin{remark}
Note that, in \eqref{e:loglambda-loglambdaE_3_terms_term1}, there is functional dependence among the matrices $\widehat{{\mathbf B}}_a(2^j) \in {\mathcal S}_{\geq 0}(r,\bbR)$, $o_{\bbP}(1) \in {\mathcal S}_{\geq 0}(p,\bbR)$ and $O_{\bbP}(1) \in {\mathcal M}(r,p,\bbR)$. Also, analogous statements hold for \eqref{e:loglambda-loglambdaE_3_terms_term2} and \eqref{e:loglambda-loglambdaE_3_terms_term3}. However, expressions \eqref{e:Taylor_determ}, \eqref{e:fnq3(R)-fnq3(0)} and \eqref{e:f_n,q,2(R)-f_n,q,2(0)=MVT} represent partial mean value theorem-type expansions of each term. Namely, each matrix argument is first treated as a functionally independent variable, and then the actual value of the matrix argument is plugged back in (see \eqref{e:Taylor}, \eqref{e:sqrt(n/a)(fnq3(WZ/a^2hq+1)-fnq(0))} and \eqref{e:sqrt(n/a)(fnq2(o(1))-fnq0(0))}, respectively).
\end{remark}

\section{Conclusion and open problems}\label{s:conclusion}

In this paper, we mathematically characterize the asymptotic and large-scale behavior of the eigenvalues of wavelet random matrices in high dimensions. We assume that possibly non-Gaussian, finite-variance $p$-variate measurements are made of a low-dimensional $r$-variate ($r \ll p$) fractional stochastic process with unknown scaling coordinates and in the presence of additive high-dimensional noise. In the three-way limit where the sample size ($n$), dimension ($p(n)$) and scale ($a(n)$) go to infinity, we establish that the rescaled $r$ largest eigenvalues of the wavelet random matrices converge to scale-invariant functions, whereas the remaining $p(n)-r$ eigenvalues remain bounded. In addition, under slightly stronger assumptions, we show that the $r$ largest log-eigenvalues of wavelet random matrices exhibit asymptotically Gaussian distributions. The results bear direct consequences for high-dimensional modeling starting from measurements of the form of a signal-plus-noise system \eqref{e:Y(t)}, where $X$ is a latent process containing fractal information and ${\mathbf P}$ (as well as $Z$) is unknown.

This research leads to many relevant open problems involving scale invariance in high dimensions, some of which can be briefly described as follows. $(i)$ The results in Section \ref{s:main} build upon broad wavelet domain assumptions, hence providing a rich framework for future research pursuits involving wavelet random matrices. A natural direction of inquiry is the mathematical study of the properties of wavelet random matrices in second order, non-Gaussian fractional instances. This generally involves the mathematical control of the impact of heavier tails under the broad assumptions of Theorems \ref{t:lim_n_a_times_lambda/a^(2h+1)} and \ref{t:asympt_normality_lambdap-r+q}. $(ii)$ In turn, an interesting direction of extension is the characterization of sets of conditions under which the $r$ largest eigenvalues of wavelet random matrices exhibit non-Gaussian fluctuations (cf.\ Remarks \ref{r:asympt_rescaled_eigenvalues} and \ref{r:non-Gaussian_B-hat}). $(iii)$ In modeling, starting from measurements of the general form \eqref{e:Y(t)}, the construction of an extended framework for the detection of scaling laws in high-dimensional systems calls for the investigation of the behavior of wavelet random matrices when $r \rightarrow \infty$. In particular, such undertaking requires a deeper study, in the wavelet domain, of the so-named \textit{eigenvalue repulsion effect} (e.g., Tao \cite{tao:2012}), which may severely skew the observed scaling laws (see Wendt et al.\ \cite{wendt:abry:didier:2019:bootstrap} on preliminary computational studies). $(iv)$ Building upon the discussion in Remark \ref{r:multiresolution_RM}, one can envision the development of a theory of general multiresolution random matrices, encompassing both wavelet random matrices and sample covariance matrices. This includes the study of broad classes of models to which the assumptions apply.  $(v)$ Superior finite-sample statistical performance can be attained by means of a wavelet eigenvalue regression procedure across scales (cf.\ Abry and Didier \cite{abry:didier:2018:dim2,abry:didier:2018:n-variate}). Namely, fix a range of scales
$j = j_1 , j_1 + 1, \hdots , j_m$ and define
\begin{equation}\label{e:hl-hat}
\{ \widehat{\ell}_q \}_{q = 1,\hdots,r} :=  \Big\{\frac{1}{2}\Big(\sum_{j=j_1}^{j_m} w_j \log_2 \lambda_{p-r+q}\big({\mathbf W}(a(n)2^j)\big) -1\Big)\Big\}_{q = 1,\hdots,r}.
\end{equation}
In \eqref{e:hl-hat}, $w_j$, $j = j_1,\hdots,j_m$, are weights satisfying the relations $\sum^{j_m}_{j=j_1}w_j = 0$, $\sum^{j_m}_{j=j_1}j w_j = 1$,
where $w_j = 1$ if $m = 1$. As an immediate consequence of Theorem \ref{t:lim_n_a_times_lambda/a^(2h+1)}, $\{ \widehat{\ell}_q \}_{q = 1,\hdots,r}$ provides a consistent estimator of $\{h_q\}_{q = 1,\hdots,r}$. Moreover, as in low dimensions (Abry and Didier \cite{abry:didier:2018:n-variate}), Theorem \ref{t:asympt_normality_lambdap-r+q} points to asymptotic normality. However, additional results are required. This and other issues are tackled in Abry et al. \cite{abry:boniece:didier:wendt:2023:regression}. $(vi)$ In applications, the results in this paper naturally pave the way for the investigation of scaling behavior in high-dimensional (``Big") data from fields such as physics, neuroscience and signal processing.

\appendix

\section{Assumptions on the wavelet multiresolution analysis}\label{s:assumptions_on_the_MRA}

In the main results of the paper, we make use of the following conditions on the underlying wavelet MRA.

\medskip

\noindent {\sc Assumption $(W1)$}: $\psi \in L^2(\bbR)$ is a wavelet function, namely, it satisfies the relations
\begin{equation}\label{e:N_psi}
\int_{\bbR} \psi^2(t)dt = 1 , \quad \int_{\bbR} t^{p}\psi(t)dt = 0, \quad p = 0,1,\hdots, N_{\psi}-1, \quad \int_{\bbR} t^{N_{\psi}}\psi(t)dt \neq 0,
\end{equation}
for some integer (number of vanishing moments) $N_{\psi} \geq 1$.

\medskip

\noindent {\sc Assumption ($W2$)}: the scaling and wavelet functions
\begin{equation}\label{e:supp_psi=compact}
\textnormal{$\phi\in L^1(\bbR)$ and $\psi\in L^1(\bbR)$ are compactly supported }
\end{equation}
and $\widehat{\phi}(0)=1$.

\medskip

\noindent {\sc Assumption $(W3)$}: there is $\alpha > 1$ such that
\begin{equation}\label{e:psihat_is_slower_than_a_power_function}
\sup_{x \in \bbR} |\widehat{\psi}(x)| (1 + |x|)^{\alpha} < \infty.
\end{equation}

\medskip

Conditions \eqref{e:N_psi} and \eqref{e:supp_psi=compact} imply that $\widehat{\psi}(x)$ exists, is infinitely differentiable everywhere and its first $N_{\psi}-1$ derivatives are zero at $x = 0$. Condition \eqref{e:psihat_is_slower_than_a_power_function}, in turn, implies that $\psi$ is continuous (see Mallat \cite{mallat:1999}, Theorem 6.1) and, hence, bounded.

Note that assumptions ($W1-W3$) are closely related to the broad wavelet framework for the analysis of $\kappa$-th order ($\kappa \in \bbN \cup \{0\}$) stationary-increment stochastic processes laid out in Moulines et al.\ \cite{moulines:roueff:taqqu:2007:Fractals,moulines:roueff:taqqu:2007:JTSA,moulines:roueff:taqqu:2008} and Roueff and Taqqu~\cite{roueff:taqqu:2009}. The Daubechies scaling and wavelet functions generally satisfy ($W1-W3$) (see Moulines et al.\ \cite{moulines:roueff:taqqu:2008}, p.\ 1927, or Mallat \cite{mallat:1999}, p.\ 253). Usually, the parameter $\alpha$ increases to infinity as $N_{\psi}$ goes to infinity (see Moulines et al.\ \cite{moulines:roueff:taqqu:2008}, p.\ 1927, or Cohen \cite{cohen:2003}, Theorem 2.10.1). Also, under the orthogonality of the underlying wavelet and scaling function basis, ($W1-W3$) imply the so-called Strang-Fix condition (see Mallat \cite{mallat:1999}, Theorem 7.4, and Moulines et al.\ \cite{moulines:roueff:taqqu:2007:JTSA}, p.\ 159, condition (W-4)).

\section{Auxiliary statements}\label{s:auxiliary}

The proofs of Section \ref{s:proof_main_results} depend on a series of auxiliary statements that are established in this section.

In the following lemma, we establish some properties of the matrix ${\boldsymbol \Lambda}$ used in the proofs of Proposition \ref{p:conv_w-by-w_rescaled_eigenvalues}/Theorem \ref{t:lim_n_a_times_lambda/a^(2h+1)} and Corollary \ref{c:PWXP^*+R_asymptotics}. Note that part $(i)$ of the lemma, on the matrix ${\mathbf M}$, is needed for establishing part $(ii)$.

In order to prove the lemma, we recap the following basic definitions and identities. For any two subspaces $W,V\subseteq \bbR^r$,
\begin{equation}\label{e:W+V}
W + V := \big\{{\mathbf w} + {\mathbf v}:  {\mathbf w} \in W, \hspace{1mm}{\mathbf v} \in V\big\}.
\end{equation}
We also write \eqref{e:W+V} as $W \oplus V$ (direct sum) when, in addition, $W \cap V = \{{\mathbf 0}\}$ (e.g., Horn and Johnson \cite{horn:johnson:2013}, p.\ 2). Starting from \eqref{e:W+V}, it can be shown that
\begin{equation}\label{e:(W+V)^perp}
\big(W +V \big)^{\perp} = W^{\perp} \cap V^{\perp}
\end{equation}
(e.g., Horn and Johnson \cite{horn:johnson:2013}, p.\ 16). In addition, by the subspace intersection lemma (e.g., Horn and Johnson \cite{horn:johnson:2013}, p.\ 4),
\begin{equation}\label{e:dim(W_and_V)}
\dim(W\cap V) = \dim(W)+\dim(V)-\dim(W+V).
\end{equation}
For $m_1,m_2 \in \bbN$, consider two matrices ${\mathbf M}_1 \in {\mathcal M}(m_1,r,\bbR)$ and ${\mathbf M}_2 \in {\mathcal M}(m_2,r,\bbR)$. Then, based on \eqref{e:span(M1,...,Mm)} and \eqref{e:W+V}, we can write
\begin{equation}\label{e:span(M1,M2)=span(M1)+span(M2)}
\textnormal{span}\{{\mathbf M}_1,{\mathbf M}_2\} = \textnormal{span}\{{\mathbf M}_1\} + \textnormal{span}\{{\mathbf M}_2\}.
\end{equation}
For any matrix ${\mathbf M}$, further recall the notation $\mathbf M^\perp = \text{nullspace}\{ \mathbf M^*\}$ as in \eqref{e:A^perp}.

\begin{lemma}\label{l:M_in_S>0(r2,R)} Let ${\mathbf M}$ and ${\boldsymbol \Lambda}$ be as in \eqref{e:M=B22-B23*B33(-1)B32} and \eqref{e:def_Lambda}, respectively. Then,
\begin{itemize}
\item[$(i)$] $\mathbf M \in\mathcal S_{>0}(r_2,\bbR)$;
\item[$(ii)$]
$\textnormal{span}\{\boldsymbol\Lambda\}=\textnormal{span}\{\mathbf R_2,\mathbf R_3\}\cap \mathbf R_3^\perp$;
\item[$(iii)$]  $\textnormal{rank} (\boldsymbol \Lambda)  = r_2.$
\end{itemize}
\end{lemma}
\begin{proof}
To show $(i)$, let $\mathbf v\in \bbR^{r_2}$ be any unit vector. Then, the vector ${\mathbf w}^* := \big({\mathbf v}^*, - (\mathbf B_{33}^{-1}\mathbf B_{32}{\mathbf v})^*\big) \in \bbR^{r_2+r_3}$ satisfies
$$
0 < \mathbf w^* \begin{pmatrix} \mathbf B_{22} & \mathbf B_{23}\\ \mathbf B_{32} & \mathbf B_{33}\end{pmatrix} \mathbf w = \mathbf v^* \mathbf M \mathbf v,
$$
where the inequality stems from the fact that $(\mathbf B_{i\ell})_{i,\ell=2,3}\in \mathcal S_{>0}(r_2+r_3,\bbR)$. This establishes $(i)$.

We now show $(ii)$. Since the matrix ${\boldsymbol\Lambda}$ is symmetric, $\text{span}\{\boldsymbol\Lambda\}^\perp = \text{nullspace}\{\boldsymbol \Lambda\}$. Therefore, $(ii)$ is proven once we establish the set of equalities
\begin{equation}\label{e:nullspace(Lambda)=span(R3)+R^perp3_and_R^perp2}
{\big(\textnormal{span}\{\mathbf R_2,\mathbf R_3\}\cap \mathbf R_3^\perp\big)}^{\perp} = {\mathbf R}^{\perp}_2 \cap {\mathbf R}^{\perp}_3 \oplus \textnormal{span}\{{\mathbf R}_3\}
= \text{nullspace}\{\boldsymbol\Lambda\}.
\end{equation}
Indeed, to prove the leftmost equality \eqref{e:nullspace(Lambda)=span(R3)+R^perp3_and_R^perp2}, it suffices to show that
\begin{equation}\label{e:span=intersec+perp-perp}
\big(\textnormal{span}\{\mathbf R_2,\mathbf R_3\}\cap \mathbf R_3^\perp\big)^{\perp}=\big(\textnormal{span}\{\mathbf R_2,\mathbf R_3\}\big)^\perp + \big(\mathbf R_3^\perp\big)^{\perp} = ({\mathbf R}^{\perp}_2 \cap {\mathbf R}^{\perp}_3 )\oplus \textnormal{span}\{{\mathbf R}_3\}.
\end{equation}
The first equality in \eqref{e:span=intersec+perp-perp} follows from relation \eqref{e:(W+V)^perp}. The second equality results from the relation $\textnormal{span}\{{\mathbf R}_2,{\mathbf R}_3\}^{\perp} = {\mathbf R}^{\perp}_2 \cap {\mathbf R}^{\perp}_3$ (itself a consequence of \eqref{e:(W+V)^perp} and \eqref{e:span(M1,M2)=span(M1)+span(M2)}), and also from the fact that $\textnormal{span}\{{\mathbf R}_2,{\mathbf R}_3\}^{\perp} \cap \text{span}\{\mathbf R_3\}=\{\mathbf 0\}$. In other words, the leftmost equality \eqref{e:nullspace(Lambda)=span(R3)+R^perp3_and_R^perp2} holds.

In regard to the rightmost equality in \eqref{e:nullspace(Lambda)=span(R3)+R^perp3_and_R^perp2}, we claim that
\begin{equation}\label{e:span(R3)+(R3^perp+R2^perp)^perp}
\text{span} \{\mathbf R_3\} \oplus \big ( \mathbf R_3 ^\perp \cap \mathbf R_2^\perp\big) \subseteq \text{nullspace}\{\boldsymbol\Lambda\}.
\end{equation}
In fact, recall that $\boldsymbol\Lambda = {\mathbf \Pi}_3^*\mathbf R_2 \mathbf  M \mathbf R_2^*{\mathbf \Pi}_3$ (see \eqref{e:def_Lambda}), where ${\mathbf \Pi}_3$ as in \eqref{e:Pi3} is the projection matrix onto $\textnormal{nullspace}\{\mathbf R^*_3\}$. Consider a nonzero vector ${\mathbf v} \in \text{span} \{\mathbf R_3\} \oplus \big ( \mathbf R_3 ^\perp \cap \mathbf R_2^\perp\big)$. Next, decompose ${\mathbf v} = {\mathbf v}_1 + {\mathbf v}_2$, where ${\mathbf v}_1 \in \textnormal{span}\{\mathbf R_3\} = (\textnormal{nullspace}\{\mathbf R^*_3\})^{\perp}$ and ${\mathbf v}_2 \in \mathbf R_3 ^\perp \cap \mathbf R_2^\perp = \textnormal{nullspace}\{\mathbf R^*_3\} \cap \textnormal{nullspace}\{\mathbf R^*_2\}$. Then, ${\mathbf \Pi}_3{\mathbf v}_1 = {\mathbf 0}$ and ${\mathbf R}^*_2 {\mathbf \Pi}_3{\mathbf v}_2 = {\mathbf R}^*_2 {\mathbf v}_2  ={\mathbf 0}$. As a result, ${\mathbf v} \in \textnormal{nullspace}\{\boldsymbol\Lambda\}$, i.e., \eqref{e:span(R3)+(R3^perp+R2^perp)^perp} holds.

So, we now prove the opposite inclusion to that in \eqref{e:span(R3)+(R3^perp+R2^perp)^perp}, i.e.,
\begin{equation}\label{e:span(R3)+(R3^perp+R2^perp)^perp_contains}
\text{span} \{\mathbf R_3\} \oplus \big ( \mathbf R_3 ^\perp \cap \mathbf R_2^\perp\big) \supseteq \text{nullspace}\{\boldsymbol\Lambda\}.
\end{equation}
 First note that, as a consequence of Lemma \ref{l:M_in_S>0(r2,R)}, $(i)$, there exists a (unique) square root ${\mathbf  M}^{1/2}$. Hence, if we define ${\mathbf L} = {\mathbf  M}^{1/2} {\mathbf R}^*_2{\mathbf \Pi}_3$, then we can write ${\mathbf L ^*\mathbf L} ={\boldsymbol \Lambda}$. Now let ${\mathbf v} \in \text{nullspace}\{\boldsymbol\Lambda\}$, i.e., ${\boldsymbol \Lambda}{\mathbf v} = {\mathbf 0}$. Then, ${\mathbf 0} = {\mathbf v}^*{\boldsymbol \Lambda}{\mathbf v} = (\mathbf L {\mathbf v})^*(\mathbf L {\mathbf v})$. Consequently, $\mathbf L {\mathbf v}= {\mathbf 0}$, i.e., ${\mathbf  M}^{1/2} {\mathbf R}^*_2{\mathbf \Pi}_3{\mathbf v} = {\mathbf 0}$. However, again by Lemma \ref{l:M_in_S>0(r2,R)}, $(i)$, the matrix $\mathbf M$ has full rank. Therefore, ${\mathbf R}^*_2{\mathbf \Pi}_3{\mathbf v} = {\mathbf 0}$. Thus, if ${\mathbf \Pi}_3{\mathbf v} \neq {\mathbf 0}$ (i.e., if ${\mathbf v} \notin \textnormal{span}\{{\mathbf R}_3\} = ({\mathbf R}^{\perp}_3)^{\perp}$), then the projection of ${\mathbf v}$ onto ${\mathbf R}^{\perp}_3$ is nonzero and ${\mathbf \Pi}_3{\mathbf v} \in {\mathbf R}^{\perp}_2$. Namely, ${\mathbf v} \in \text{span} \{\mathbf R_3\} \oplus\big ( \mathbf R_3 ^\perp \cap \mathbf R_2^\perp\big)$. Consequently, \eqref{e:span(R3)+(R3^perp+R2^perp)^perp_contains} holds, and so does the right equality in \eqref{e:nullspace(Lambda)=span(R3)+R^perp3_and_R^perp2}. This establishes $(ii)$.

For $(iii)$, note that $\bbR^r \supseteq \text{span}\{\mathbf R_2,\mathbf R_3\} +  \mathbf R_3^\perp \supseteq \text{span}\{\mathbf R_3\} +  \mathbf R_3^\perp = \bbR^r$. Therefore, $\dim(\text{span}\{\mathbf R_2,\mathbf R_3\} +  \mathbf R_3^\perp) = r$. So, by part $(ii)$ of this lemma and by relation \eqref{e:dim(W_and_V)}, $\dim (\textnormal{span}\{{\boldsymbol \Lambda}\})$ is given by
$$
\dim(\text{span}\{\mathbf R_2,\mathbf R_3\}) + \dim( \mathbf R_3^\perp) - \dim(\text{span}\{\mathbf R_2,\mathbf R_3\} +  \mathbf R_3^\perp) = (r_2+r_3) + (r-r_3) -r =  r_2.
$$
This shows $(iii)$.
\end{proof}

Hereinafter, whenever convenient we use again the notational convention \eqref{e:P(n)PH_equiv_P}, i.e., $p = p(n)$, $a = a(n)$ and ${\mathbf P} = {\mathbf P}(n) = {\mathbf P}(n){\mathbf P}_H$.

Several of the following lemmas in this section are stated in terms of the random matrices $\widetilde{\W}(a(n)2^j)$. These matrices are slightly generalized versions of $\W(a(n)2^j)$, and are defined in the statement of Lemma \ref{l:gutted_log_eig_consistency}. They appear multiple times throughout the proof of Theorem \ref{t:asympt_normality_lambdap-r+q}.

For the proofs, it is useful to first recall the variational characterization of the eigenvalues of a matrix ${\mathbf M} \in {\mathcal S}(p,\bbR)$ provided by the Courant--Fischer principle. In other words, fix $\ell \in \{1,\hdots,p\}$ and consider ordered eigenvalues as in \eqref{e:lambda1(M)=<...=<lambdap(M)}. By the Courant--Fischer principle, we can express
\begin{equation}\label{e:Courant--Fischer}
\lambda_{\ell}({\mathbf M}) = \inf_{{\mathcal U}_\ell} \sup_{{\mathbf u} \in {\mathcal U}_\ell \cap \bbS^{p-1}}{\mathbf u}^* {\mathbf M} {\mathbf u} = \sup_{{\mathcal U}_{p-\ell+1}}\inf_{{\mathbf u} \in {\mathcal U}_{p-\ell+1} \cap \bbS^{p-1}} {\mathbf u}^*{\mathbf M}{\mathbf u},
\end{equation}
where ${\mathcal U}_\ell$ is a $\ell$-dimensional subspace of $\bbR^p$ (Horn and Johnson \cite{horn:johnson:2013}, Chapter 4). A related useful fact is the following. For ${\mathbf M} \in {\mathcal S}(p,\bbR)$, let ${\mathbf u}_1,\hdots, {\mathbf u}_p$ be unit eigenvectors of ${\mathbf M}$ associated with the eigenvalues \eqref{e:lambda1(M)=<...=<lambdap(M)} respectively. Then, for $\ell \in \{1,\hdots,p\}$, we can further express
\begin{equation}\label{e:lambdaq(M)_based_on_eigenvecs}
\lambda_{\ell}({\mathbf M}) = \inf_{{\mathbf u} \in \textnormal{span}\{{\mathbf u}_{\ell},\hdots, {\mathbf u}_p\} \cap \bbS^{p-1}}{\mathbf u}^* {\mathbf M} {\mathbf u} = \sup_{{\mathbf u} \in \textnormal{span}\{{\mathbf u}_1,\hdots, {\mathbf u}_{\ell}\} \cap \bbS^{p-1}}{\mathbf u}^* {\mathbf M} {\mathbf u}.
\end{equation}

The following result,  Lemma \ref{l:gutted_log_eig_consistency}, is used in Theorem \ref{t:lim_n_a_times_lambda/a^(2h+1)}. It is also applied in Lemma \ref{l:sum<pi(n),up-r+q(n)>2-o(a-varpi)_first}, which in turn is needed to establish Proposition \ref{p:conv_w-by-w_rescaled_eigenvalues}/Theorem \ref{t:lim_n_a_times_lambda/a^(2h+1)}, Corollary \ref{c:PWXP^*+R_asymptotics} and Lemma \ref{l:max|<up-r+q(n),pi(n)>|a(n)^(hi-hq)=OP(1)}.
\begin{lemma}\label{l:gutted_log_eig_consistency}
Fix any $j$ as in \eqref{e:def_j1,jm}. For ${\mathbf B}(2^j)$ as in \eqref{e:|bfB_a(2^j)-B(2^j)|=O(shrinking)}, let $\mathbf B_n \in {\mathcal S}_{\geq 0}(r,\bbR)$ be a sequence of random (or deterministic) matrices such that, as $n \rightarrow \infty$,
\begin{equation}\label{e:Bn->B(2^j)}
\mathbf B_n \stackrel{\bbP}\rightarrow {\mathbf B}(2^j) \quad (\textnormal{or } \mathbf B_n \rightarrow {\mathbf B}(2^j)).
\end{equation}
Suppose ($A4$) and ($A5$) hold and let $\mathbf h$ be as in \eqref{e:h-bf=diag(h1,...,hr)}.
 Define the sequence of random matrices
\begin{equation}\label{e:W-tilde(a2^j)}
\widetilde{\W}(a(n)2^j) ={\mathbf P}(n)a(n)^{{\mathbf h+ \frac{1}{2}{\mathbf I}}}\hspace{0.25mm}\mathbf B_n\hspace{0.25mm} a(n)^{{\mathbf h + \frac{1}{2}{\mathbf I}}} {\mathbf P}^*(n) + {\mathbf M}_n \in {\mathcal S}(p,\bbR), \quad n \in \bbN,
\end{equation}
where
\begin{equation}\label{e:def_Mn}
{\mathbf M}_n= O_{\bbP}(1) +
{\mathbf P}(n)a(n)^{{\mathbf h + \frac{1}{2}{\mathbf I}}} O_{\bbP}(1)+ O^*_{\bbP}(1)  a(n)^{{\mathbf h+\frac{1}{2}{\mathbf I}}}{\mathbf P}^*(n)  \in {\mathcal S}(p,\bbR).
\end{equation}
In \eqref{e:def_Mn}, the terms $O_\bbP(1)$ denote generic random matrices with values in ${\mathcal S}_{\geq 0}(p,\bbR)$ (first term) and ${\mathcal M}(r,p,\bbR)$ (second and third terms), whose spectral  norms are bounded in probability. Then, as $n \rightarrow \infty$,
\begin{equation}\label{e:lambdap-r=O_P(1)}
\lambda_{p-r}\big(\widetilde{\W}(a(n)2^j)\big) = O_\bbP(1)
\end{equation}
and, for some constant $C > 0$,
\begin{equation}\label{e:log_lambdaq_consistency}
\frac{\lambda_{p-r+q}\big(\widetilde{\W}(a(n)2^j)\big)}{a(n)^{2h_q + 1}} \geq C \big\{\lambda_1({\mathbf B}(2^j))+ o_\bbP(1) \big\},  \quad q=1,\ldots,r.
\end{equation}
In particular, statements \eqref{e:lambdap-r=O_P(1)} and \eqref{e:log_lambdaq_consistency} hold for the wavelet random matrix $\mathbf{W}(a(n)2^j)$ as in \eqref{e:W=PWXP*+WZ+PWXZ+WXZ*P*}, whence $\lambda_{p-r}\big(\mathbf{W}(a(n)2^j)\big) = O_\bbP(1)$ and $\lambda_{p-r+1}\big(\mathbf{W}(a(n) 2^j)\big)\stackrel \bbP \to \infty.$
\end{lemma}

\begin{example}
In \eqref{e:W-tilde(a2^j)}, $\widetilde{\W}(a(n)2^j) = \W(a(n)2^j)$ for the choices ${\mathbf B}_n = \widehat{{\mathbf B}}_a(2^j)$,
$ O_{\bbP}(1) = {\mathbf W}_Z(a(n)2^{j})$ (first term in \eqref{e:def_Mn}) and $ O_{\bbP}(1) = a(n)^{-{\mathbf H}-(1/2){\mathbf I}}{\mathbf W}_{X,Z}(a(n)2^{j})$ (second and third terms in \eqref{e:def_Mn})  (cf.\ \eqref{e:WZ=OP(1),a^(-h-1/2I)WXZ_repeat}).
\end{example}

\noindent {\sc Proof of Lemma \ref{l:gutted_log_eig_consistency}}: To establish \eqref{e:lambdap-r=O_P(1)}, just note that, by relations \eqref{e:Courant--Fischer} and \eqref{e:W-tilde(a2^j)},
$$
0 \leq \lambda_{p-r}\big(\widetilde{\W}(a2^j)\big) \leq \sup_{{\mathbf u} \in \{\mathbf{p}_{1},\hdots,\mathbf{p}_{r}\}^{\perp} \cap \bbS^{p-1}}{\mathbf u}^* \widetilde{\W}(a2^j){\mathbf u} = O_\bbP(1).
$$

We now establish \eqref{e:log_lambdaq_consistency}. Fix $q \in \{1,\hdots,r\}$ and, for any $n$, define the subspace $ \mathcal V(n) =  \{\mathbf p_1(n),\ldots,\mathbf p_{q-1}(n)\}^\perp\cap \text{span}\{\mathbf P(n)\}$. Then, by \eqref{e:Courant--Fischer} and the fact that $\dim  \mathcal V(n)= r-q+1$,
$$
\lambda_{p-r+q}\big(\widetilde{\W}(a2^j)\big)=\sup_{{\mathcal U}_{r-q+1}}\inf_{{\mathbf u} \in {\mathcal U}_{r-q+1} \cap \bbS^{p-1}} {\mathbf u}^*\widetilde{\W}(a2^j){\mathbf u}
$$
\begin{equation}\label{e:lambda_p-q+1(W(a2^j))}
\geq \inf_{{\mathbf u} \in \mathcal V(n) \cap \bbS^{p-1}}{\mathbf u}^*\widetilde{\W}(a2^j){\mathbf u}=: \widetilde{{\mathbf u}}^*(n)\widetilde{\W}(a2^j)\widetilde{{\mathbf u}}(n).
\end{equation}
In \eqref{e:lambda_p-q+1(W(a2^j))}, $\widetilde{{\mathbf u}}(n) = \widetilde{{\mathbf u}}(n,\omega)\in \mathcal V(n)$ is a unit (random) vector at which $\inf_{{\mathbf u} \in \mathcal V(n) \cap \bbS^{p-1}}$ is attained, and, like the (deterministic) vectors ${\mathbf p}_{\ell}(n)$, $\ell = 1,\hdots,r$, it is a function of $n$.

So, let $\widetilde{{\mathbf u}}(n)$ be as in \eqref{e:lambda_p-q+1(W(a2^j))}. Suppose, for the moment, that the norm $\|a^{\mathbf h + \frac{1}{2}\mathbf I}{\mathbf P}^*(n)\widetilde{{\mathbf u}}(n)\|$ grows at least as fast as  $a^{h_q+1/2}$. In other words, suppose that
\begin{equation}\label{e:||a^hP*u-tilde||>=Ca^(h1)}
\|a^{\mathbf h + \frac{1}{2}\mathbf I}{\mathbf P}^*(n)\widetilde{{\mathbf u}}(n)\|^2 \geq C a^{2h_q+1}  \quad \textnormal{a.s.}
\end{equation}
In view of \eqref{e:W-tilde(a2^j)}, we can write
\begin{equation}\label{e:W=PWXP*+WZ+PWXZ+WXZ*P*_2}
a^{-1}\widetilde{\W}(a2^j)= {\mathbf P}(n)a^{\mathbf h} \mathbf B_n a^{\mathbf h}{\mathbf P}^*(n)+ O_\bbP(1) + {\mathbf P}(n) a^{\mathbf h} O_\bbP(1) + O^*_\bbP(1)a^{\mathbf h}{\mathbf P}^*(n).
\end{equation}
Therefore, by \eqref{e:Bn->B(2^j)}, \eqref{e:||a^hP*u-tilde||>=Ca^(h1)} and \eqref{e:W=PWXP*+WZ+PWXZ+WXZ*P*_2},
$$
 a^{-1}\widetilde{{\mathbf u}}^*\widetilde{\W}(a2^j)\widetilde{{\mathbf u}}
 $$
 $$
 =\|a^{\mathbf h}{\mathbf P}^*\widetilde{{\mathbf u}}\|^2 \Big\{ \frac{\widetilde{{\mathbf u}}^*{\mathbf P}a^{\mathbf h}}{\|\widetilde{{\mathbf u}}^*{\mathbf P}a^{\mathbf h}\|} \mathbf B_n \frac{a^{\mathbf h}{\mathbf P}^*\widetilde{{\mathbf u}}}{\|a^{\mathbf h}{\mathbf P}^*\widetilde{{\mathbf u}}\|}
 + \frac{\widetilde{{\mathbf u}}^* O_\bbP(1)\widetilde{{\mathbf u}}}{\|a^{\mathbf h}{\mathbf P}^*\widetilde{{\mathbf u}}\|^2} + \frac{2}{\|\widetilde{{\mathbf u}}^*{\mathbf P} a^{\mathbf h}\|} \widetilde{{\mathbf u}}^*O_\bbP(1)\frac{a^{\mathbf h}{\mathbf P}^*\widetilde{{\mathbf u}}}{\|a^{\mathbf h}{\mathbf P}^*\widetilde{{\mathbf u}}\|} \Big\}
$$
$$
 \geq  \|a^{\mathbf h}{\mathbf P}^*\widetilde{{\mathbf u}}\|^2 \big\{\lambda_1(\mathbf B_n)+ o_\bbP(1) \big\} \geq
 C a^{2h_q} \big\{\lambda_1({\mathbf B}(2^j))+ o_\bbP(1) \big\},
$$
where $\widetilde{{\mathbf u}} = \widetilde{{\mathbf u}}(n)$ and ${\mathbf P} = {\mathbf P}(n)$. Hence, \eqref{e:log_lambdaq_consistency} is established.

So, we now need to prove \eqref{e:||a^hP*u-tilde||>=Ca^(h1)}. To this end, note that $\text{span}\{\mathbf Q(n)\} = \text{span}\{\mathbf P(n)\} \supseteq {\mathcal V}(n) \ni \widetilde{{\mathbf u}}(n)$ a.s.\ by condition \eqref{e:P(n)=Q(n)R(n)}. Then, $\|\mathbf Q(n) \widetilde{{\mathbf u}}(n) \| = \|\widetilde{{\mathbf u}}(n)\|=1$ a.s. Consequently,
$$\sum_{i=q}^{r}\langle \mathbf p_i(n),\widetilde{\mathbf u}(n)\rangle^2=\| {\mathbf P}^*(n)\widetilde{{\mathbf u}}(n) \|^2 $$$$= \|{\mathbf R}^*(n) \mathbf Q^*(n)  \widetilde{{\mathbf u}}(n)\|^2 \geq \inf_{\mathbf u \in \bbS^{r-1}}\mathbf u^* {\mathbf R}(n){\mathbf R}^*(n) {\mathbf u} \geq C$$ for some $C > 0$, where the  last inequality follows from condition \eqref{e:<p1,p2>=c12_2}. This implies that
$$
\|a^{\mathbf h+\frac{1}{2}\mathbf I}{\mathbf P}^*(n)\widetilde{{\mathbf u}}(n)\|^2 = \sum_{i=q}^{r} a^{2h_i+1}\langle  \mathbf p_i(n),\widetilde{\mathbf u}(n)\rangle^2$$
$$ \geq a^{2h_q+1}\sum_{i=q}^{r} \langle  \mathbf p_i(n),\widetilde{\mathbf u}(n)\rangle^2 \geq C a^{2h_q+1}  \quad \textnormal{a.s.}
$$%
In other words, \eqref{e:||a^hP*u-tilde||>=Ca^(h1)} holds, as claimed. \\

Lemma \ref{l:|<p3,uq(n)>|a(n)^{h_3-h_1}=O(1)}, stated and proved next, is used in the proofs of Proposition \ref{p:conv_w-by-w_rescaled_eigenvalues}/Theorem \ref{t:lim_n_a_times_lambda/a^(2h+1)}, Theorem \ref{t:asympt_normality_lambdap-r+q}, Lemma \ref{l:fullrank_limit_P*U} (indirectly) as well as in the proofs of Corollary \ref{c:PWXP^*+R_asymptotics}, Proposition \ref{p:|lambdaq(EW)-xiq(2^j)|_bound}, and Lemmas \ref{l:max|<up-r+q(n),pi(n)>|a(n)^(hi-hq)=OP(1)} and \ref{l:supR'max(angles*powerlaws)_bounded_for_subseq}.

As one consequence of Lemma \ref{l:|<p3,uq(n)>|a(n)^{h_3-h_1}=O(1)}, we can see that the potentially explosive terms in expression \eqref{e:rescaled_eigen_p-r+ell_quad_form} are, in truth, bounded in probability. In particular, this implies relations \eqref{e:subseq_condition_2_top_explanation} and \eqref{e:subseq_condition_2_top_explanation-2} (cf.\ \eqref{e:subseq_condition_2_top}). Moreover, for any fixed $q \in \{1,\hdots,r\}$, the lemma also shows that, for the rescaled random matrix $\widetilde{{\mathbf W}}(a(n)2^j)/a(n)^{2h_q+1}$ (cf.\ \eqref{e:W-tilde(a2^j)}), all terms other than the main scaling term ${\mathbf P}(n) a(n)^{\mathbf h -  h_q{\mathbf I}} {\mathbf B}_n a(n)^{\mathbf h -  h_q{\mathbf I}}{\mathbf P}^*(n)$  behave like residuals in probability \textit{along the direction given by the eigenvector ${\mathbf {\mathfrak u}}_{p-r+q}(n)$} of  $\widetilde{{\mathbf W}}(a(n)2^j)$, in spite of the fact that their spectral norms may be unbounded in general. In other words, the first relation in \eqref{e:B-hat(2j)->B(2j)_prop} holds.

For the purpose of stating the lemma, for a generic fixed $q\in\{1,\ldots,r\}$, it will also be useful to define the random matrix
\begin{equation}\label{e:W_PX(a2^j)}
 \widetilde{ \Xi}_q(n):= \frac{\mathbf M_n}{a(n)^{2h_q+1}}=\frac{\widetilde{\W}(a(n)2^j)}{a(n)^{2h_q+1}}- \frac{ {\mathbf P}(n)a(n)^{{\mathbf h}}\hspace{0.25mm}\mathbf B_n\hspace{0.25mm} a(n)^{{\mathbf h }} {\mathbf P}^*(n) }{a(n)^{2h_q}}
\end{equation}
(cf.\ \eqref{e:rescaled_W_conv_in_prob}). Hereinafter, we use this notation.

\begin{lemma}\label{l:|<p3,uq(n)>|a(n)^{h_3-h_1}=O(1)} Fix any $j$ as in \eqref{e:def_j1,jm}. Let $\widetilde{\mathbf W}(a(n)2^j)$ be the sequence of random matrices given by \eqref{e:W-tilde(a2^j)} (see Lemma \ref{l:gutted_log_eig_consistency}). Fix $q \in\{1,\hdots,r\}$.  Also, let
\begin{equation}\label{e:ufrak_p-r+q(n,omega)}
{\boldsymbol {\mathfrak u}}_{p-r+q}(n) = {\boldsymbol {\mathfrak u}}_{p-r+q}(n,\omega)
\end{equation}
be a unit eigenvector associated with the $(p-r+q)$--th eigenvalue of $\widetilde{\mathbf W}(a(n)2^j)$.
\begin{itemize}
\item [$(i)$] Then,
\begin{equation}\label{e:supR'max(angles*powerlaws)=O_P(1)}
\max_{i \in \mathcal{I}_+}\big\{ |\langle {\mathbf p}_{  i}(n),{\boldsymbol {\mathfrak u}}_{p-r+q}(n)\rangle| \hspace{0.5mm}a(n)^{h_{i}-h_q}\big\} = O_{\bbP}(1), \quad n \rightarrow \infty.
\end{equation}

\item [$(ii)$] For $\widetilde{ \Xi}_q(n)$ as in \eqref{e:W_PX(a2^j)},
\begin{equation}\label{e:residual_of_W-tilde_along_eigenvector_directions}
\Big|{\boldsymbol {\mathfrak u}}^*_{p-r+q}(n)  \widetilde{ \Xi}_q(n)
{\boldsymbol {\mathfrak u}}_{p-r+q}(n) \Big|= o_{\bbP}(1), \quad n \rightarrow \infty.
\end{equation}
\end{itemize}
In particular, statements \eqref{e:supR'max(angles*powerlaws)=O_P(1)} and \eqref{e:residual_of_W-tilde_along_eigenvector_directions} hold for a unit eigenvector ${\mathbf u}_{p-r+q}(n)$ of the wavelet random matrix $\mathbf{W}(a(n)2^j)$ as in \eqref{e:W=PWXP*+WZ+PWXZ+WXZ*P*} and for the random matrix $\Xi_q(n)$ as in \eqref{e:rescaled_W_conv_in_prob}.
\end{lemma}
\begin{proof}
We first show $(i)$. For the fixed $q \in \{1,\hdots,r\}$, if the index set $\mathcal{I}_+$ as in \eqref{e:def_indexsets} is empty, then the claim holds vacuously. So, suppose $\mathcal{I}_+ \neq \emptyset$ and let $\mathbf h$ be the diagonal matrix given by \eqref{e:h-bf=diag(h1,...,hr)}. It suffices to show that
\begin{equation}\label{e:||a^(h-h_qI)P^*u_p-r+q(n)|| = OP(1)}
\| a^{\mathbf h - h_q{\mathbf I}} {\mathbf P}^*{\boldsymbol {\mathfrak u}}_{p-r+q}(n)\| = O_{\bbP}(1).
\end{equation}
For this purpose, we begin by establishing two bounds. In fact, on one hand note that, by \eqref{e:Courant--Fischer},

$$
\frac{\lambda_{p-r+q}(\widetilde{\mathbf W}(a2^j))}{a^{2h_q + 1}} = \inf_{{\mathcal U}_{p-r+q}}\sup_{{\mathbf u} \in {\mathcal U}_{p-r+q} \cap \bbS^{p-1}} {\mathbf u}^*\frac{\widetilde{\mathbf W}(a2^j)}{a^{2h_q + 1}}{\mathbf u} $$
\begin{equation}\label{e:lambdaq/a^(2h1+1)=OP(1)}\leq \sup_{{\mathbf u} \in {\textnormal{span}\{\mathbf p_{q+1},\ldots,\mathbf{p}_r\}}^{\perp} \cap \bbS^{p-1}} {\mathbf u}^*\frac{\widetilde{\mathbf W}(a2^j)}{a^{2h_q + 1}}{\mathbf u}
\leq \sup_{{\mathbf u} \in {\textnormal{span}\{\mathbf p_\ell,\ell\in \mathcal{I}_+\}}^{\perp} \cap \bbS^{p-1}} {\mathbf u}^*\frac{\widetilde{\mathbf W}(a2^j)}{a^{2h_q + 1}}{\mathbf u} = O_{\bbP}(1).
\end{equation}
On the other hand, again in light of \eqref{e:Courant--Fischer}, relation \eqref{e:def_Mn} implies that there are sequences of nonnegative random variables
\begin{equation}\label{e:zeta-n,zeta'-n}
\{\zeta_n\}_{n\in \bbN}=o_\bbP(1) \textnormal{ and }\{\zeta'_n\}_{n\in \bbN}=o_\bbP(1)
\end{equation}
such that
$$
\frac{\lambda_{p-r+q}(\widetilde{\mathbf W}(a2^j))}{a^{2h_q + 1}}  \geq \| a^{\mathbf h - h_q{\mathbf I}} {\mathbf P}^*{\boldsymbol {\mathfrak u}}_{p-r+q}(n)\|^2\lambda_1(\mathbf B_n) - 2  \| a^{\mathbf h - h_q{\mathbf I}} {\mathbf P}^*{\boldsymbol {\mathfrak u}}_{p-r+q}(n)\| \zeta_n  -\zeta'_n
$$
\begin{equation}\label{e:lambdaq/a^(2h1+1)=OP(1)_auxiliary_lower_bound}
=\| a^{\mathbf h - h_q{\mathbf I}} {\mathbf P}^*{\boldsymbol {\mathfrak u}}_{p-r+q}(n)\|\Big( \| a^{\mathbf h - h_q{\mathbf I}} {\mathbf P}^*{\boldsymbol {\mathfrak u}}_{p-r+q}(n)\|\lambda_1(\mathbf B_n) - \zeta_n\Big) - \zeta'_n.
\end{equation}
In turn, in order to use relations \eqref{e:lambdaq/a^(2h1+1)=OP(1)} and \eqref{e:lambdaq/a^(2h1+1)=OP(1)_auxiliary_lower_bound} to bound $\| a^{\mathbf h - h_qI} {\mathbf P}^*{\boldsymbol {\mathfrak u}}_{p-r+q}(n)\|$, we need to account for the role of the term $\| a^{\mathbf h - h_qI} {\mathbf P}^*{\boldsymbol {\mathfrak u}}_{p-r+q}(n)\|\lambda_1(\mathbf B_n) - \zeta_n$ in \eqref{e:lambdaq/a^(2h1+1)=OP(1)_auxiliary_lower_bound}. Indeed, $\lambda_1(\mathbf B_n) \stackrel{\bbP}\to \lambda_1(\mathbf B(2^j))>0$ by \eqref{e:Bn->B(2^j)} and \eqref{e:B(2^j)_full_rank}. Then, there exists $\delta>0$ such that the sequence of events $A_n:= \{\omega \in \Omega: \lambda_1(\mathbf B_n)-\zeta_n> \delta\}$ satisfies
\begin{equation}\label{e:P(Omega'n)->1}
\bbP( A_n )\to 1, \quad n \rightarrow \infty
\end{equation}
(n.b.: not to be confused with $A_n$ appearing in \eqref{e:W_has_pairwise_distinct_eigens}). Then, by dividing through by $\| a^{\mathbf h - h_q{\mathbf I}} {\mathbf P}^*{\boldsymbol {\mathfrak u}}_{p-r+q}(n)\|\lambda_1(\mathbf B_n) - \zeta_n$ in  \eqref{e:lambdaq/a^(2h1+1)=OP(1)_auxiliary_lower_bound}, we obtain, for each $\omega \in A_n$,
$$
\| a^{\mathbf h - h_q{\mathbf I}} {\mathbf P}^*{\boldsymbol {\mathfrak u}}_{p-r+q}(n)\|  \leq 1 +  \| a^{\mathbf h - h_q{\mathbf I}} {\mathbf P}^*{\boldsymbol {\mathfrak u}}_{p-r+q}(n)\|  \mathbf 1_{\{\| a^{\mathbf h - h_q{\mathbf I}} {\mathbf P}^*{\boldsymbol {\mathfrak u}}_{p-r+q}(n)\| > 1\}}
$$
\begin{equation}\label{e:||a^(h-hqI)P^*ufrak_p-r+q(n)||_bound }
\leq  1  + \frac{\frac{\lambda_{p-r+q}(\widetilde{\mathbf W}(a2^j))}{a^{2h_q + 1}}  +\zeta'_n}{ \lambda_1(\mathbf B_n) -\zeta_n}.
\end{equation}
So, fix $\varepsilon>0$ and let $\zeta'_n$ be as in \eqref{e:lambdaq/a^(2h1+1)=OP(1)_auxiliary_lower_bound}. By relations \eqref{e:lambdaq/a^(2h1+1)=OP(1)} and \eqref{e:zeta-n,zeta'-n}, there exists $M = M(\varepsilon) > 0$ such that
\begin{equation}\label{e:P(rescaled_eigenval_+_zeta-n>M)_bound}
\bbP\bigg(  \frac{\lambda_{p-r+q}(\widetilde{\mathbf W}(a2^j))}{a^{2h_q + 1}} + \zeta'_n >  M  \bigg) < \frac{\varepsilon}{2}.
\end{equation}
Pick $M'> 0$ such that $\delta(M' - 1) = M$. Then, relations \eqref{e:||a^(h-hqI)P^*ufrak_p-r+q(n)||_bound } and \eqref{e:P(rescaled_eigenval_+_zeta-n>M)_bound} imply that
$$
 \bbP\big(\| a^{\mathbf h - h_q{\mathbf I}} {\mathbf P}^*{\boldsymbol {\mathfrak u}}_{p-r+q}(n)\| > M' \big)
 $$$$\leq \bbP\big(A_n^c\big) + \bbP\bigg(\Big\{\frac{\lambda_{p-r+q}(\widetilde{\mathbf W}(a2^j))}{a^{2h_q + 1}} + \zeta'_n >  \delta(M' -1)\Big\} \cap   A_n \bigg) < \varepsilon,
$$
where the last inequality holds for large enough $n$ as a consequence of the limit \eqref{e:P(Omega'n)->1}. In other words, \eqref{e:||a^(h-h_qI)P^*u_p-r+q(n)|| = OP(1)} holds. This establishes \eqref{e:supR'max(angles*powerlaws)=O_P(1)}, and hence, $(i)$.

We now show $(ii)$. If $\mathcal{I}_+ = \emptyset$, then the statement holds trivially. So, suppose $\mathcal{I}_+ \neq \emptyset$. We can rewrite the left-hand side of \eqref{e:residual_of_W-tilde_along_eigenvector_directions} as
$$
\Big|{\boldsymbol {\mathfrak u}}^*_{p-r+q}(n) \Big(\frac{O_{\bbP}(1)}{a^{2h_q+1}}+
\frac{{\mathbf P}a^{{\mathbf h}}}{a^{h_q}}\frac{O_{\bbP}(1)}{a^{h_q + 1/2}}+\frac{O^*_{\bbP}(1)}{a^{h_q + 1/2}}\frac{a^{{\mathbf h}}{\mathbf P}^*}{a^{h_q}}\Big){\boldsymbol {\mathfrak u}}_{p-r+q}(n)\Big|
$$
$$
\leq \Big|{\boldsymbol {\mathfrak u}}^*_{p-r+q}(n) \frac{O_{\bbP}(1)}{a^{2h_q+1}}{\boldsymbol {\mathfrak u}}_{p-r+q}(n) \Big|+
\Big|{\boldsymbol {\mathfrak u}}^*_{p-r+q}(n) \Big(\frac{{\mathbf P}a^{{\mathbf h}}}{a^{h_q}}\frac{O_{\bbP}(1)}{a^{h_q + 1/2}}+\frac{O^*_{\bbP}(1)}{a^{h_q + 1/2}}\frac{a^{{\mathbf h}}{\mathbf P}^*}{a^{h_q}}\Big){\boldsymbol {\mathfrak u}}_{p-r+q}(n)\Big|
$$
\begin{equation}
= o_{\bbP}(1) + 2 \Big|{\boldsymbol {\mathfrak u}}^*_{p-r+q}(n) \Big(\frac{{\mathbf P}a^{{\mathbf h}}}{a^{h_q}}\frac{O_{\bbP}(1)}{a^{h_q + 1/2}}\Big){\boldsymbol {\mathfrak u}}_{p-r+q}(n)\Big|.
\end{equation}
However, $\frac{O_{\bbP}(1)}{a^{h_q + 1/2}}{\boldsymbol {\mathfrak u}}_{p-r+q}(n) = o_{\bbP}(1) \in \bbR^r$. Moreover, ${\boldsymbol {\mathfrak u}}^*_{p-r+q}(n) \frac{{\mathbf P}a^{{\mathbf h}}}{a^{h_q}} = O_{\bbP}(1) \in \bbR^r$, by \eqref{e:supR'max(angles*powerlaws)=O_P(1)}. This establishes \eqref{e:residual_of_W-tilde_along_eigenvector_directions} and, hence, $(ii)$. \end{proof}

The following lemma is used in the proofs of Proposition \ref{p:conv_w-by-w_rescaled_eigenvalues}/Theorem \ref{t:lim_n_a_times_lambda/a^(2h+1)}, Corollary \ref{c:PWXP^*+R_asymptotics} and Lemma \ref{l:max|<up-r+q(n),pi(n)>|a(n)^(hi-hq)=OP(1)}. It shows that, for any fixed $q \in \{1,\hdots,r\}$, the norm of an eigenvector ${\boldsymbol {\mathfrak u}}_{p-r+q}(n)$ of the random matrix $\widetilde{\mathbf W}(a(n)2^j)$  eventually concentrates in the space $\text{span}\{\mathbf P(n)\}$. As a consequence, relation \eqref{e:subseq_condition_1_top_explanation} holds (cf.\ \eqref{e:subseq_condition_1_top}).
\begin{lemma}\label{l:sum<pi(n),up-r+q(n)>2-o(a-varpi)_first} Fix any $j$ as in \eqref{e:def_j1,jm}. Let $\widetilde{\mathbf W}(a(n)2^j)$ be the sequence of random matrices given by \eqref{e:W-tilde(a2^j)} (see Lemma \ref{l:gutted_log_eig_consistency}). For each $n$, let
\begin{equation}\label{e:p_r+1(n),...,p_p(n)}
\p_{r+1}(n),\ldots,\p_p(n)
\end{equation}
be an orthonormal basis for $\textnormal{nullspace}({\mathbf P}^*(n))$. Fix any $q \in\{ 1,\hdots,r\}$, and let ${\boldsymbol {\mathfrak u}}_{p-r+q}(n) = {\boldsymbol {\mathfrak u}}_{p-r+q}(n,\omega)$ denote a unit eigenvector associated with $\widetilde{\mathbf W}(a(n)2^j)$, as in \eqref{e:ufrak_p-r+q(n,omega)}. Then,
\begin{equation}\label{e:sum<pi(n),up-r+q(n)>2-o(a-varpi)_first}
 \sum^{p}_{i=r+1}\langle {\mathbf p}_{  i}(n), {\boldsymbol {\mathfrak u}}_{p-r+q}(n)\rangle^2 = O_\bbP\big(a(n)^{-(h_q+ \frac{1}{2})}\big).
\end{equation}
In particular, \eqref{e:sum<pi(n),up-r+q(n)>2-o(a-varpi)_first} holds for a unit eigenvector $\mathbf u_{p-r+q}(n)= \mathbf u_{p-r+q}(n,\omega)$ associated with the random matrix ${\mathbf W}(a(n)2^j)$.
\end{lemma}
\begin{proof}
For ${\mathbf P} = {\mathbf P}(n)$, consider the projection ${\boldsymbol \upsilon}(n)$ of ${\boldsymbol{\mathfrak u}}_{p-r+q}(n)$ onto $\textnormal{nullspace}({\mathbf P}^*)$, i.e.,
\begin{equation}\label{e:v*(n)}
{\boldsymbol \upsilon}(n) =  \sum^{p}_{i=r+1} \langle \p_{i}(n),{\boldsymbol {\mathfrak u}}_{p-r+q}(n)\rangle \hspace{0.5mm} \p_{i}(n).
\end{equation}
Starting from expression \eqref{e:def_Mn}, it is clear that ${\boldsymbol \upsilon}^*(n)\big\{{\mathbf P}a^{{\mathbf h + \frac{1}{2}{\mathbf I}}}\hspace{0.25mm}\mathbf B_n\hspace{0.25mm} a^{{\mathbf h + \frac{1}{2}{\mathbf I}}} {\mathbf P}^* + {\mathbf P}a^{{\mathbf h + \frac{1}{2}{\mathbf I}}}O_\bbP(1)+ O^*_\bbP(1) a^{{\mathbf h + \frac{1}{2}{\mathbf I}}}{\mathbf P}^*\big\} {\boldsymbol \upsilon}(n) = {\mathbf 0}$. Hence, by \eqref{e:W-tilde(a2^j)},
\begin{equation}\label{e:v*(n)W(a2j)v(n)=OP(1)}
{\boldsymbol \upsilon}^*(n)\widetilde {\mathbf W}(a2^j) {\boldsymbol \upsilon}(n) = O_{\bbP}(1).
\end{equation}
However, let ${\mathcal O}(n) = ( \mathbf {\boldsymbol {\mathfrak u}}_{1}(n)\hdots \mathbf {\boldsymbol {\mathfrak u}}_{p}(n) )$ be a matrix of eigenvectors of $\widetilde {\mathbf W}(a2^j)$. Then,
$$
{\boldsymbol \upsilon}^*(n)\widetilde{\mathbf W}(a2^j){\boldsymbol \upsilon}(n)= {\boldsymbol \upsilon}^*(n) {\mathcal O}(n) \textnormal{diag}\Big(\lambda_1(\widetilde {\mathbf W}(a2^j)),\hdots,\lambda_{p}(\widetilde {\mathbf W}(a2^j))\Big){\mathcal O}^*(n){\boldsymbol \upsilon}(n)
$$
$$
= \sum^{p}_{i=1} \langle{\boldsymbol {\mathfrak u}}_{i}(n), {\boldsymbol \upsilon}(n)\rangle^2 \lambda_i(\widetilde {\mathbf W}(a2^j)) \geq \langle {\boldsymbol {\mathfrak u}}_{p-r+q}(n),{\boldsymbol \upsilon}(n)\rangle^2 \lambda_{p-r+q}(\widetilde {\mathbf W}(a2^j))
$$
$$
= \Big( \sum^{p}_{i=r+1} \langle \p_{i}(n), {\boldsymbol {\mathfrak u}}_{p-r+q}(n)\rangle^2 \Big)^2 \lambda_{p-r+q}(\widetilde{\mathbf W}(a2^j))
$$
\begin{equation}\label{e:v*(n)W(a2j)v(n)_lower_bound}
\geq \Big( \sum^{p}_{i=r+1} \langle \p_{i}(n), {\boldsymbol {\mathfrak u}}_{p-r+q}(n)\rangle^2 \Big)^2 a^{2h_q+1} \cdot C \big\{\lambda_1({\mathbf B}(2^j))+ o_\bbP(1) \big\} .
\end{equation}
In \eqref{e:v*(n)W(a2j)v(n)_lower_bound}, the last equality and the last inequality follow from \eqref{e:v*(n)} and \eqref{e:log_lambdaq_consistency} (see Lemma \ref{l:gutted_log_eig_consistency}), respectively. We claim that %
\begin{equation}\label{e:sum_<pi(n),u_p-r+q(n)>^2_=<_O(a^(-(2hq+1)))}
\Big( \sum^{p}_{i=r+1} \langle \p_{i}(n), {\boldsymbol {\mathfrak u}}_{p-r+q}(n)\rangle^2 \Big)^2 \leq  O_\bbP\big(a^{-(2h_q+1)}\big).
\end{equation}
In fact, fix any $\varepsilon > 0$ and any $M > 0$. Then, there exists $n_0 \in \bbN$ such that $n \geq n_0$ implies $\lambda_1({\mathbf B}(2^j))+ o_\bbP(1) \geq \frac{\lambda_1({\mathbf B}(2^j))}{2} > 0$ with probability at least $1 - \varepsilon/2$, where the strict inequality is a consequence of \eqref{e:B(2^j)_full_rank}. Also, \eqref{e:v*(n)W(a2j)v(n)=OP(1)} implies that there exists $n_1 \in \bbN$ such that $n \geq n_1$ implies $C \frac{\lambda_1({\mathbf B}(2^j))}{2} M \geq {\boldsymbol \upsilon}^*(n)\widetilde {\mathbf W}(a2^j) {\boldsymbol \upsilon}(n)$ with probability at least $1 - \varepsilon/2$. Now recall the elementary inequality $\bbP(A \cap B) \geq 1 - \bbP(A^c) - \bbP(B^c)$ for any two events $A$ and $B$. Then, for $n \geq \max\{n_0,n_1\}$, \eqref{e:v*(n)W(a2j)v(n)_lower_bound} implies that
$$
\Big( \sum^{p}_{i=r+1} \langle \p_{i}(n), {\boldsymbol {\mathfrak u}}_{p-r+q}(n)\rangle^2 \Big)^2 a^{2h_q+1}  \leq M
$$
with probability at least $1 - \varepsilon$. This establishes \eqref{e:sum_<pi(n),u_p-r+q(n)>^2_=<_O(a^(-(2hq+1)))}.  Hence, \eqref{e:sum<pi(n),up-r+q(n)>2-o(a-varpi)_first} holds.
\end{proof}

Note that, as a consequence of Lemmas \ref{l:|<p3,uq(n)>|a(n)^{h_3-h_1}=O(1)} and \ref{l:sum<pi(n),up-r+q(n)>2-o(a-varpi)_first}, for the vectors $\mathbf p_{r+1}(n),\ldots,\mathbf p_p(n)$ as in \eqref{e:p_r+1(n),...,p_p(n)}, there exist a subsequence $n'\in\bbN'$ along which \eqref{e:subseq_condition_1_top}, \eqref{e:subseq_condition_2_top} and the first relation in \eqref{e:B-hat(2j)->B(2j)_prop} hold. This fact is used in the proofs of Theorem \ref{t:lim_n_a_times_lambda/a^(2h+1)} and Proposition \ref{p:|lambdaq(EW)-xiq(2^j)|_bound}, as well as in the assumptions of Proposition \ref{p:conv_w-by-w_rescaled_eigenvalues}, Lemma \ref{l:fullrank_limit_P*U} and Lemma \ref{l:<p,w>=infinitesimal}.\\

The following result, Lemma \ref{l:fullrank_limit_P*U}, is used in the proofs of Proposition \ref{p:conv_w-by-w_rescaled_eigenvalues}, Corollary \ref{c:PWXP^*+R_asymptotics}, and Lemma \ref{l:<p,w>=infinitesimal}.  It pertains to the behavior of $\widetilde{\mathbf T}(\nu_n)$, i.e., the wavelet eigenvectors after switching to $\tau$ coordinates (see \eqref{e:f-hat_n(x,u)=varphi-hat_n(x,tau)} and \eqref{e:T(n)}), along an appropriate random subsequence $\nu_n(\omega)$.

We proceed to state the lemma, and then provide some interpretation before proving it. For $\ell = 1,\hdots,p$, as in \eqref{e:ufrak_p-r+q(n,omega)} let ${\boldsymbol {\mathfrak u}}_{\ell}(n)= {\boldsymbol {\mathfrak u}}_{\ell}(n,\omega)$ be an eigenvector associated with the $\ell$--th eigenvalue of a random matrix $\widetilde{\mathbf W}(a(n)2^j)$ of the form given by \eqref{e:W-tilde(a2^j)}. In the proof of Lemma \ref{l:fullrank_limit_P*U} and also elsewhere, we make use of the sequence of rectangular random matrices
\begin{equation}\label{e:U_r(n)}
\widetilde{{\mathbf U}}_r(n) = \widetilde{{\mathbf U}}_r(n,\omega) =\big( \boldsymbol{\mathfrak u}_{p-r+1}(n),\ldots,\boldsymbol{\mathfrak u}_{p}(n)\big) \in {\mathcal M}(p,r,\bbR), \quad n \in \bbN,
\end{equation}
which contains eigenvectors associated with the $r$ top eigenvalues of the random matrix $\widetilde{\mathbf W}(a(n)2^j)$ (cf.\ \eqref{e:def_Un}).

In the lemma, we also make use of the following basic fact. For a matrix ${\mathbf M}=(m_{i,\kappa})\in {\mathcal M}(r,\bbR)$, we can write
\begin{equation}\label{e:tr(M*M)=sum_singval^2}
\sum^{r}_{i,\kappa=1} m_{i,\kappa}^2= \textnormal{tr}({\mathbf M}^* {\mathbf M}) = \sum_{i=1}^r\sigma_i^2({\mathbf M}),
\end{equation}
where $\sigma_i({\mathbf M})$, $i = 1,\hdots,r$, are the singular values of ${\mathbf M}$.

\begin{lemma}\label{l:fullrank_limit_P*U}  Fix any $j$ as in \eqref{e:def_j1,jm}. Let $\widetilde{\mathbf W}(a(n)2^j)$, $\widetilde{{\mathbf U}}_r(n)$ and ${\mathbf Q}(n)$ be the sequences of random matrices given by \eqref{e:W-tilde(a2^j)} (see Lemma \ref{l:gutted_log_eig_consistency}), \eqref{e:U_r(n)} and \eqref{e:P(n)=Q(n)R(n)}, respectively. Consider the random matrix
\begin{equation}\label{e:T-tilde(n)}
\widetilde{\mathbf T}(n) = \widetilde{\mathbf T}(n,\omega) :=(\boldsymbol \tau_1(n),\ldots,\boldsymbol \tau_r(n)): = \mathbf Q^*(n) \widetilde{\mathbf U}_r(n) \in {\mathcal M}(r,\bbR)
\end{equation}
(cf.\ \eqref{e:T(n)}). Let $n'\in\bbN'$ be any sequence along which \eqref{e:subseq_condition_1_top} and \eqref{e:subseq_condition_2_top} hold with ${\boldsymbol{\mathfrak u}}_{p-r+\ell}(n')$ in place of ${\mathbf u}_{p-r+\ell}(n')$, and consider any random subsequence $\nu_1(\omega) <\nu_2(\omega) <\ldots$ of $n' \in \bbN'$ along which
\begin{equation}\label{e:def_Gamma}
\lim_{n\to\infty} \widetilde{\mathbf T}(\nu_n)=:\widetilde{\mathbf T} = (\boldsymbol \tau_1, \ldots, \boldsymbol \tau_r)   \quad a.s.
\end{equation}
\begin{itemize}
\item [$(i)$] Then, the vectors $\boldsymbol \tau_1,\ldots,\boldsymbol \tau_r$ are orthonormal a.s., i.e., the matrix $\widetilde{{\mathbf T}}=\widetilde{{\mathbf T}}(\omega)$ in \eqref{e:def_Gamma} is orthogonal a.s.;
\item[$(ii)$] For a fixed $q \in \{1,\hdots,r\}$, consider the associated index sets as defined in \eqref{e:def_indexsets} as well as their cardinalities $r_i$, $i = 1,2,3$, as in \eqref{e:r1,r2,r3}. Also consider $\widetilde{\mathbf T}$ as in \eqref{e:def_Gamma} and ${\mathbf R}$ as in \eqref{e:<p1,p2>=c12_2}. Then, the matrix $\boldsymbol \Gamma :=\mathbf R^* \widetilde{\mathbf T}$ is nonsingular a.s.\ and may be expressed as
\begin{equation}\label{e:gamma_breakdown}
{\boldsymbol \Gamma} = {\boldsymbol \Gamma}(\omega) = \begin{pmatrix}\boldsymbol \Gamma_{11} & \boldsymbol \Gamma_{12} & \boldsymbol \Gamma_{13}\\
\mathbf 0 & \boldsymbol \Gamma_{22} & \boldsymbol \Gamma_{23}\\
\mathbf 0 & \mathbf 0 & \boldsymbol \Gamma_{33}
\end{pmatrix},
\end{equation}
where the submatrices $\boldsymbol \Gamma_{i \kappa}$, $i,\kappa = 1,2,3$, are of size $r_i\times r_\kappa$. Moreover, the submatrices $\boldsymbol \Gamma_{ii}$, $i=1,2,3$, are nonsingular a.s.
\item [$(iii)$] For a fixed $q \in \{1,\hdots,r\}$, let $\mathcal I_0$ be the associated index set as defined in \eqref{e:def_indexsets} and let $\mathbf R_i$, $i=1,2,3$, be the matrices as in \eqref{e:R=(R1,R2,R3)}. Then, $\textnormal{span}\{{\boldsymbol \tau}_\ell(\omega): \ell \in {\mathcal I}_0\} = \textnormal{span}\{\mathbf R_2,\mathbf R_3\}\cap \mathbf R_3^\perp$ a.s., i.e., expression \eqref{e:span_equals_W} holds.
\end{itemize}
\end{lemma}

Some comments are in order on the statement of the lemma. Though the fixed-dimensional matrix $\widetilde{{\mathbf T}}(\nu_n(\omega))$ is not necessarily orthogonal, part $(i)$ of the lemma shows that it %
does become an orthogonal matrix in the limit. In regard to the orthonormal limiting column vectors ${\boldsymbol \tau}_1(\omega), \ldots, {\boldsymbol \tau}_r(\omega)$ as in \eqref{e:def_Gamma}, parts $(ii)$ and $(iii)$ of the lemma show that $\textnormal{span}\{{\boldsymbol \tau}_\ell(\omega): \ell \in {\mathcal I}_0\}$ generates the fixed-dimensional space needed in expressing $\lambda_{r-r_2 + \ell}(\Lambda)$ as in \eqref{e:lim_nu_rescaled_eigenvalue}.

On sufficient conditions for the same statements to hold along the \textit{main} sequence $n \in \bbN$, see Proposition \ref{p:|lambdaq(EW)-xiq(2^j)|_bound}.

We are now in a position to prove the lemma.\vspace{2mm}

\noindent {\sc Proof of Lemma \ref{l:fullrank_limit_P*U}}. First, we show $(i)$. By \eqref{e:subseq_condition_1_top}, for each $q = 1,\hdots,r$,
\begin{equation}\label{e:sum<pi(n),up-r+q(n)>2->0}
\sum^{p(\nu_n)}_{i=r+1} \langle {\mathbf p}_{  i}(\nu_n), {\boldsymbol {\mathfrak u}}_{p-r+q}(\nu_n)\rangle^2 \to 0 \quad \textnormal{a.s.}, \quad n \rightarrow \infty.
\end{equation}
 So, write ${\mathbf q}_1(n),\ldots, {\mathbf q}_r(n)$ for the (orthonormal) columns of ${\mathbf Q}(n)$. By relation \eqref{e:sum<pi(n),up-r+q(n)>2->0} and the fact that $\textnormal{span}\{{\mathbf P}(n)\}=\textnormal{span}\{{\mathbf Q}(n)\}$, as $n \rightarrow \infty$,
$\sum_{i=1}^r \langle \mathbf q_i(\nu_n), {\boldsymbol{\mathfrak u}}_{p-r+q}(\nu_n)\rangle^2\to 1$, $q=1,\ldots,r$, a.s. Then, by relations \eqref{e:tr(M*M)=sum_singval^2} and \eqref{e:T-tilde(n)}, as $n \rightarrow \infty$,
\begin{equation}\label{e:sum_of_singular_values_tends_to_r}
\sum^{r}_{s=1}\sigma_{s}^2\big(\widetilde{\mathbf T}(\nu_n)\big) = \sum_{i,\ell=1}^r \langle {\mathbf q}_i(\nu_n),{\boldsymbol {\mathfrak u}}_{p-r+\ell}(\nu_n)\rangle^2 \rightarrow r \quad \text{a.s.}
\end{equation}
However, the largest singular value $\sigma_{r}(\widetilde{\mathbf T}(n))$ of $\widetilde{\mathbf T}(n)$ satisfies the bound $\sigma_{r}^2\big(\widetilde{\mathbf T}(n)\big)\leq \|{\mathbf Q}^*(n)\|^2\|\widetilde{\mathbf U}_r(n)\|^2 = 1$  a.s. Hence, by \eqref{e:sum_of_singular_values_tends_to_r}, $\sigma_s^2(\widetilde{\mathbf T}(\nu_n))\to 1$ a.s., $s=1,\ldots,r$. By the continuity of singular values, this implies that $\widetilde{\mathbf T}$ is an orthogonal matrix with probability 1. This establishes $(i)$.

For statement $(ii)$, $\boldsymbol \Gamma$ is nonsingular a.s.\ as consequence of statement $(i)$ and of the invertibility of $\mathbf R$ (see condition \eqref{e:<p1,p2>=c12_2}). Now define the sequences of random matrices
\begin{equation}\label{e:Gamma(nu_n)=block-wise_decomp}
\boldsymbol \Gamma(\nu_n) := \mathbf R^*(\nu_n) \widetilde{\mathbf T}(\nu_n) =  \mathbf P^*(\nu_n) \widetilde{\mathbf U}_r(\nu_n) =:\big( \boldsymbol \Gamma_{i\kappa}(\nu_n)\big)_{i,\kappa=1,2,3},
\end{equation}
where each $\boldsymbol \Gamma_{i \kappa}(\nu_n)$ is a block of size $r_i\times r_\kappa$. Since $\lim_{n \rightarrow \infty}{\boldsymbol \Gamma}(\nu_n) = {\boldsymbol \Gamma}$ a.s., then \eqref{e:gamma_breakdown} is a consequence of \eqref{e:subseq_condition_2_top}. In addition, the classical formula $\det({\boldsymbol\Gamma}) = \prod^{3}_{i=1} \det({\boldsymbol \Gamma}_{ii})$ implies that ${\boldsymbol \Gamma}_{ii}$, $i=1,2,3$, are nonsingular a.s.

 In regard to statement $(iii)$, by expression \eqref{e:gamma_breakdown}, the columns of $\widetilde{\mathbf T}$ are such that
\begin{equation}\label{e:tau_ell_in_R3perp,ell_in_I0}
{\boldsymbol \tau}_\ell(\omega) \in \mathbf R_3^\perp, \quad \ell \in \mathcal I_0.
\end{equation}
Likewise, for $i \in \mathcal I_-$, $\boldsymbol \tau_i(\omega) \in \text{span}\{\mathbf R_2,\mathbf R_3\}^\perp$. Now note that $\textnormal{dim}\big(\text{span}\{\mathbf R_2,\mathbf R_3\}^\perp \big) = r_1$. Since $\textnormal{card}({\mathcal I}_{-}) = r_1$, the linear independence of the vectors $\{{\boldsymbol \tau}_{\ell}(\omega): \ell \in \mathcal I_-\}$ (part $(i)$ of this lemma) implies that $\textnormal{span}\{{\boldsymbol \tau}_{\ell}(\omega): \ell \in \mathcal I_-\} = \text{span}\{\mathbf R_2,\mathbf R_3\}^\perp$. Thus, by the orthogonality of the vectors $\boldsymbol \tau_1(\omega),\ldots,\boldsymbol \tau_r(\omega)$ (again by part $(i)$ of this lemma), $\{{\boldsymbol \tau}_{\ell}(\omega) : \ell \in \mathcal I_0\} \subseteq \textnormal{span}\{\mathbf R_2,\mathbf R_3\}$. Moreover, by relation \eqref{e:tau_ell_in_R3perp,ell_in_I0},
$$
\{{\boldsymbol \tau}_{\ell}(\omega): \ell \in \mathcal I_0\} \subseteq \text{span}\{\mathbf R_2,\mathbf R_3\} \cap {\mathbf R}^{\perp}_3.
$$
However, as a consequence of Lemma \ref{l:M_in_S>0(r2,R)}, $(ii)$ and $(iii)$, $\dim (\textnormal{span}\{\mathbf R_2,\mathbf R_3\}\cap \mathbf R_3^\perp) = r_2 = \textnormal{card}({\mathcal I}_0)$. Hence, expression \eqref{e:span_equals_W} is established. $\Box$\\

The following lemma is used in the proofs of Proposition \ref{p:conv_w-by-w_rescaled_eigenvalues}/Theorem \ref{t:lim_n_a_times_lambda/a^(2h+1)}, Corollary \ref{c:PWXP^*+R_asymptotics} and Lemma \ref{l:max|<up-r+q(n),pi(n)>|a(n)^(hi-hq)=OP(1)}.
\begin{lemma}\label{l:<p,w>=infinitesimal}
Suppose the assumptions of Lemma \ref{l:fullrank_limit_P*U}  hold.
Fix $q\in\{1,\ldots,r\}$ and consider the associated index set $\mathcal I_0$ (see \eqref{e:def_indexsets}). For a given $\ell\in\mathcal I_0$ and random vectors $\boldsymbol \tau_s = \boldsymbol \tau_s(\omega)$,  $s = \ell,\hdots,r_1 + r_2$, as in \eqref{e:def_Gamma}, take any
\begin{equation}\label{e:gamma-tilde}
\mathbf w = \mathbf w(\omega) \in \textnormal{span}\{\boldsymbol \tau_\ell,\ldots,\boldsymbol \tau_{r_1+r_2}\} \cap \bbS^{r-1}.
\end{equation}
Also, fix any random vector
\begin{equation}\label{e:x*}
\mathbf{x}_{*}= \mathbf{x}_{*}(\omega) = \big(x_{*,{r-r_3+1}}(\omega),\hdots,x_{*,r}(\omega)\big)^* \in \bbR^{r_3}.
\end{equation}
Consider a random subsequence $\nu_1(\omega) <\nu_2(\omega) <\ldots$ of $n' \in \bbN'$ as in the statement of Lemma \ref{l:fullrank_limit_P*U}. Then, there exists a sequence of unit vectors ${\mathbf v}(\nu_n) \in \textnormal{span}\{{\boldsymbol {\mathfrak u}}_{p-r+\ell}(\nu_n),\ldots,{\boldsymbol {\mathfrak u}}_{p}(\nu_n)\}$ satisfying, for large enough $n$,
\begin{equation}\label{e:<p,w>=infinitesimal_3}
\langle \mathbf{p}_i(\nu_n), {\mathbf v}(\nu_n) \rangle = \frac{x_{*,i}}{a(\nu_n)^{h_i - h_q}}, \quad i \in \mathcal{I}_+,
\end{equation}
and such that
\begin{equation}\label{e:<p,w>=infinitesimal_2}
\lim_{n\to\infty} {\mathbf Q}^*(\nu_n) {\mathbf v}(\nu_n) = {\mathbf w} \quad a.s.
\end{equation}
In particular, relations \eqref{e:pi(n),v(n)=x*/a^(hi-hq)} and \eqref{e:pick_the_vector_main_manuscript} hold in the proof of Theorem \ref{t:lim_n_a_times_lambda/a^(2h+1)}.
\end{lemma}
Lemma \ref{l:<p,w>=infinitesimal} is what makes the bounds \eqref{e:rescaled_lambda2_upper_bound1} and \eqref{e:rescaled_lambda_upperbound2} possible and useful. This is so because, by relation \eqref{e:<p,w>=infinitesimal_3}, there are no potentially explosive terms in ${\mathbf v}^*(\nu_n) \frac{{\mathbf W}(a(\nu_n)2^j)}{a(\nu_n)^{2h_q+1}}{\mathbf v}(\nu_n)$, and by relation \eqref{e:<p,w>=infinitesimal_2}, the sequence $\mathbf Q^*(\nu_n){\mathbf v}(\nu_n)$ converges to a target vector in the subspace $\textnormal{span}\{\boldsymbol \tau_\ell,\ldots,\boldsymbol \tau_{r_1+r_2}\}$. In particular, the upper bound can be made asymptotically sharp, as described in the proof of Theorem \ref{t:lim_n_a_times_lambda/a^(2h+1)}.\\

\noindent {\sc Proof of Lemma \ref{l:<p,w>=infinitesimal}}: For the fixed $\ell \in \mathcal I_0$ and for $\mathbf{x}_{*}=\mathbf{x}_{*}(\omega)$ as in \eqref{e:x*}, define
\begin{equation}\label{e:vartheta_n}
{\boldsymbol \vartheta}_{\nu_n} = \Big(\frac{x_{*,i}}{a(\nu_n)^{h_i-h_q}}\Big)_{i \in \mathcal{I}_+}\in \bbR^{r_3},
\end{equation}
which is a (random) vector containing the right-hand terms in \eqref{e:<p,w>=infinitesimal_3}. Now, for random vectors ${\mathbf w}$, ${\boldsymbol \tau}_i$, $i=\ell,\hdots,r$, as in \eqref{e:gamma-tilde}, let $\boldsymbol \alpha = (0,\ldots,0,\alpha_\ell,\ldots,\alpha_{r_1+r_2})^*\in\bbR^{r_2}$ be the unit vector of (generally random) coefficients such that
\begin{equation}\label{e:w=sum_i=ell-to-r1+r2_alpha-i*tau-i}
\mathbf w(\omega) =  \sum_{i=\ell}^{r_1+r_2}\alpha_i(\omega) {\boldsymbol \tau}_i(\omega) = \sum_{i=r_1+1}^{r_1+r_2}\alpha_i(\omega) {\boldsymbol \tau}_i(\omega) .
\end{equation}
The proof consists in constructing a sequence of random coefficient vectors
\begin{equation}\label{e:||c(nu_n)||=1}
\bbR^r \ni {\boldsymbol c}(\nu_n)=\big(0,\hdots,0,c_\ell(\nu_n),\hdots, c_{r}(\nu_n)\big)^* ,\quad \|{\boldsymbol c}(\nu_n)\|=1,
\end{equation}
such that the sequence of associated random vectors ${\mathbf v}(\nu_n) := \sum^{r}_{i=\ell}c_i(\nu_n) {\boldsymbol {\mathfrak u}}_{p-r+i}(\nu_n)$, $n \in \bbN$, satisfies relations \eqref{e:<p,w>=infinitesimal_3} and \eqref{e:<p,w>=infinitesimal_2}. This will be achieved by solving a linear system of the form
\begin{equation}\label{e:linear_system_upper_bound_schematic}
{\boldsymbol \vartheta}_{\nu_n} = {\mathbf P}^*_3(\nu_n) \widetilde{\mathbf U}_{r}(\nu_n) \big(0,\hdots,0,c_\ell(\nu_n),\hdots, c_{r}(\nu_n)\big)^* \in \bbR^{r_3}.
\end{equation}
In \eqref{e:linear_system_upper_bound_schematic}, ${\mathbf P}^*_3(\nu_n)$ and $\widetilde{\mathbf U}_{r}(\nu_n)$ are given by \eqref{e:P=(P1(n)_P2(n)_P3(n))} and \eqref{e:U_r(n)}, respectively. Also, as $n \rightarrow \infty$, $c_i(\nu_n) \rightarrow \alpha_{i}$ for $i = \ell,\hdots,r-r_3$ and $c_i(\nu_n) \rightarrow 0$ for $i \in {\mathcal I}_+$.\vspace{2mm}

So, consider relation \eqref{e:Gamma(nu_n)=block-wise_decomp}. By Lemma \ref{l:fullrank_limit_P*U}, $\boldsymbol \Gamma_{33} = \lim_{n\to\infty} \boldsymbol \Gamma_{33}(\nu_n)$ is nonsingular a.s. Therefore, by the continuity of the determinant, $\boldsymbol \Gamma_{33}(\nu_n)$ is nonsingular for all sufficiently large $n$. Hence, for each $s\in(0,1)$ there exists a vector ${\mathbf z}_{\nu_n} (s) \in \bbR^{r_3}$ satisfying
\begin{equation}\label{e:Gamma*z=etc}
\boldsymbol \Gamma_{33}(\nu_n) \mathbf z_{\nu_n} (s) =\boldsymbol\vartheta_{\nu_n} - \boldsymbol\Gamma_{32}(\nu_n)\cdot s \boldsymbol \alpha.
\end{equation}
Note that, for ${\boldsymbol \vartheta}_{\nu_n}$ as in \eqref{e:vartheta_n}, ${\boldsymbol \vartheta}_{\nu_n}  \to {\mathbf 0}$ a.s.\ as $n \rightarrow \infty$. In addition, as a consequence of Lemma \ref{l:fullrank_limit_P*U}, $\lim_{n\to\infty} \boldsymbol \Gamma_{32}(\nu_n)=\mathbf 0_{r_3\times r_2}$ a.s. Since, again, $\boldsymbol \Gamma_{33}$ is nonsingular a.s., then
\begin{equation}\label{e:|varsigma|->0}
\sup_{s \in [0,1]}\|{\mathbf z}_{\nu_n} (s)\| \rightarrow 0 \quad \textnormal{a.s.}, \quad n \rightarrow \infty.
\end{equation}
Thus, for any small $\varepsilon \in (0,1)$, $0 \leq \sup_{s \in [0,1]}\|{\mathbf z}_{\nu_n} (s)\| < \varepsilon < 1$ for large enough $\nu_n(\omega)$ a.s. On other hand, for each $n$, the function $f(s) = 1-s^2 - \|\mathbf z_{\nu_n} (s)\|^2$ depends continuously on $s$. Moreover, $f(0) > 1 - \varepsilon$ and $f(1) = - \|\mathbf z_{\nu_n}(1)\|^2$. Hence, $f$ must have a root $s_0=s_0(\nu_n)\in(0,1]$ for large enough $\nu_n(\omega)$. So, we define the unit vector
$$
{\mathbf c}^*(\nu_n) := \big({\mathbf 0}^*_{r_1},
s_0\boldsymbol \alpha^*, {\mathbf z}_{\nu_n}^* (s_0) \big)^*\in\bbR^r.
$$

Moreover, in view of \eqref{e:|varsigma|->0}, $s_0\to 1$ a.s.~as $n\to\infty$. This implies that
\begin{equation}\label{e:T(n)*c->w}
\mathbf Q^*(\nu_n) \widetilde{{\mathbf U}}_r(\nu_n) {\mathbf c}(\nu_n) =
\widetilde{\mathbf T}(\nu_n) \begin{pmatrix}
\mathbf 0_{r_1}\\
s_0\boldsymbol \alpha\\
\mathbf z_{\nu_n} (s_0)
\end{pmatrix}
\to \widetilde{\mathbf T} \begin{pmatrix}
\mathbf 0_{r_1}\\
\boldsymbol \alpha\\
\mathbf 0_{r_3}
\end{pmatrix} = \mathbf w \quad \text{a.s.}, \quad n\to\infty.
\end{equation}
Thus, if we define
\begin{equation}\label{e:def_v(nu_n)}
\mathbf v(\nu_n): = \widetilde{{\mathbf U}}_r(\nu_n){\mathbf c}(\nu_n) %
\end{equation}
then $\mathbf v(\nu_n)$ satisfies \eqref{e:<p,w>=infinitesimal_2}.  Also,  by construction, $s^2_0+ \|\mathbf z_{\nu_n}(s_0)\|^2 = 1$. Since $\mathbf w$ is a unit vector and the vectors $\boldsymbol \tau_i$ are orthonormal, then $\|\boldsymbol\alpha\|=1$ a.s. This shows that $\|\mathbf v(\nu_n)\|= s_0^2\|\boldsymbol \alpha\|^2 +  \|\mathbf z_{\nu_n} (s_0)\|^2 = 1$. Also, by the definition of $\boldsymbol \alpha$ (see \eqref{e:w=sum_i=ell-to-r1+r2_alpha-i*tau-i}), the first $\ell-1$ entries of the vector ${\mathbf c}^*(\nu_n)  \in\bbR^r$ are zero (cf.\ \eqref{e:||c(nu_n)||=1}), so $\mathbf v(\nu_n) \in \textnormal{span}\{{\boldsymbol {\mathfrak u}}_{p-r+\ell}(\nu_n),\ldots,{\boldsymbol {\mathfrak u}}_{p}(\nu_n)\}$.

We now show that \eqref{e:<p,w>=infinitesimal_3} holds for $\mathbf v(\nu_n)$ as in \eqref{e:def_v(nu_n)}. By expression \eqref{e:Gamma(nu_n)=block-wise_decomp}, we can write $\mathbf P_3^*(\nu_n) \widetilde{{\mathbf U}}_{r}(\nu_n)=\big( \mathbf \Gamma_{31}(\nu_n) ~ \mathbf \Gamma_{32}(\nu_n) ~  \mathbf \Gamma_{33}(\nu_n)\big)\in \mathcal M(r_3,r,\bbR)$. So, by relations \eqref{e:Gamma*z=etc} and \eqref{e:def_v(nu_n)},  a.s.\ for large enough $n$,
\begin{equation}\label{e:P_3^*(nu_n)v(nu_n)=vartheta_n}
\mathbf P_3^*(\nu_n)\mathbf v(\nu_n) =   \mathbf \Gamma_{32}(\nu_n) \cdot s_0\boldsymbol\alpha +  \mathbf \Gamma_{33}(\nu_n)\mathbf z_{\nu_n}(s_0)
= \boldsymbol\vartheta_{\nu_n} %
\end{equation}
(cf.\ \eqref{e:linear_system_upper_bound_schematic}). Together, relations \eqref{e:def_v(nu_n)} and \eqref{e:P_3^*(nu_n)v(nu_n)=vartheta_n} imply that \eqref{e:linear_system_upper_bound_schematic} holds. This establishes \eqref{e:<p,w>=infinitesimal_3}. $\Box$\\

The following corollary of the proof of Theorem \ref{t:lim_n_a_times_lambda/a^(2h+1)} is used several times throughout the proof of Theorem \ref{t:asympt_normality_lambdap-r+q} and also in the proof of Proposition \ref{p:|lambdaq(EW)-xiq(2^j)|_bound}.  %
\begin{corollary}\label{c:PWXP^*+R_asymptotics}
Fix any $j$ as in \eqref{e:def_j1,jm}. Let $\widetilde{ \mathbf W}(a(n)2^j)$ be the sequence of random matrices given by \eqref{e:W-tilde(a2^j)} (see Lemma \ref{l:gutted_log_eig_consistency}). Let $\xi_{\ell}(2^j)$, $\ell=1,\ldots,r$, be the functions appearing in \eqref{e:lim_n_a*lambda/a^(2h+1)} (see Theorem \ref{t:lim_n_a_times_lambda/a^(2h+1)}). Then, as $n \rightarrow \infty$,
\begin{equation*}\label{e:convergence_rescaled_eigenvalue_W-tilde}
\lambda_{p-r+\ell}\Big(\frac{\widetilde{\W}(a(n)2^j)}{a(n)^{2h_\ell+1}}\Big) \stackrel \bbP \to \xi_{\ell}(2^j), \quad \ell=1,\ldots,r.
\end{equation*}
In particular, for a fixed $q \in \{1,\hdots,r\}$ and its associated index sets as in \eqref{e:def_indexsets},
\begin{equation}\label{e:lambda_p-r+q->xi_based_on_residual_Rq(n)_Ba}
\lambda_{p-r+\ell}\Big(\frac{\widetilde{\W}(a(n)2^j)}{a(n)^{2h_q+1}}\Big) \stackrel{\bbP}\rightarrow\begin{cases}
0, & \ell \in \mathcal I_-;\\
 \xi_{\ell}(2^j), & \ell \in \mathcal I_0;\\
 \infty, & \ell \in \mathcal I_+,
\end{cases}
\quad n \rightarrow \infty.
\end{equation}
\end{corollary}
\begin{proof}
The proofs of Theorem \ref{t:lim_n_a_times_lambda/a^(2h+1)} and Proposition \ref{p:conv_w-by-w_rescaled_eigenvalues} can be repeated, \textit{mutatis mutandis}, with the more general matrices $\widetilde{\W}(a2^j)$ as in \eqref{e:W-tilde(a2^j)} in place of $\W(a2^j)$ (\textbf{n.b.}: even though $\W(a2^j) \in {\mathcal S}_{\geq 0}(p,\bbR)$, only the fact that $\W(a2^j) \in {\mathcal S}(p,\bbR)$ plays a role in the proof of Theorem \ref{t:lim_n_a_times_lambda/a^(2h+1)}). In particular, Lemmas \ref{l:gutted_log_eig_consistency}--\ref{l:<p,w>=infinitesimal} -- which are used in the proofs of Theorem \ref{t:lim_n_a_times_lambda/a^(2h+1)} and Proposition \ref{p:conv_w-by-w_rescaled_eigenvalues} -- are stated and proved for the matrices $\widetilde \W(a2^j)$.
\end{proof}

Proposition \ref{p:|lambdaq(EW)-xiq(2^j)|_bound}, stated and proved next, is used in the proofs of Theorem \ref{t:asympt_normality_lambdap-r+q} and Lemma \ref{l:max|<up-r+q(n),pi(n)>|a(n)^(hi-hq)=OP(1)}. By comparison to Lemma \ref{l:fullrank_limit_P*U}, which pertains to convergent \textit{sub}sequences $\widetilde{\mathbf T}(\nu_n)$, the proposition shows that, under the additional condition \eqref{e:xiq_distinct}, the \textit{main} sequence of random matrices $\widetilde{\mathbf T}(n)={\mathbf Q}^*(n) \widetilde{{\mathbf U}}_r(n)$ (in particular, ${\mathbf T}(n)$ as in \eqref{e:T(n)}) has a limit $\widetilde{\mathbf T}$ in probability that is deterministic. This yields the convergence of the vector ${\mathbf P}^*(n) {\boldsymbol{\mathfrak u}}_{p-r+q}(n)$, which plays a role analogous to that of the convergence of eigenvectors in fixed dimensions (cf.\ Proposition 1, $(iii)$, in Abry and Didier \cite{abry:didier:2018:n-variate}), the latter being meaningless in high dimensions. %

\begin{proposition}\label{p:|lambdaq(EW)-xiq(2^j)|_bound}
Fix any $j$ as in \eqref{e:def_j1,jm}. Let $\widetilde{ \mathbf W}(a(n)2^j)$ be the sequence of random matrices given by \eqref{e:W-tilde(a2^j)} (see Lemma \ref{l:gutted_log_eig_consistency}). Suppose condition \eqref{e:xiq_distinct} in Theorem \ref{t:asympt_normality_lambdap-r+q} holds.
Fix $q\in\{1,\ldots,r\}$, and let ${\mathbf B}(2^j)$ and ${\mathbf R}$ be as in \eqref{e:B(2^j)_full_rank} and \eqref{e:<p1,p2>=c12_2}, respectively. Then, there are deterministic vectors $\boldsymbol \tau_q$ and $\boldsymbol{\gamma}_{q}:=\mathbf R^*\boldsymbol \tau_q$, depending only on $\mathbf B(2^j)$ and $\mathbf R$, and a sequence of $(p-r+q)$--th unit random eigenvectors $\{{\boldsymbol{\mathfrak u}}_{p-r+q}(n)\}_{n\in\bbN}$ of $\widetilde{\mathbf W}(a(n)2^j)$ along which
\begin{equation}\label{e:<p,u>_to_gamma}
\plim_{n\to\infty} {\mathbf Q}^*(n) {\boldsymbol{\mathfrak u}}_{p-r+q}(n)=\boldsymbol\tau_q \in \bbR^r, \qquad \plim_{n\to\infty} {\mathbf P}^*(n) {\boldsymbol{\mathfrak u}}_{p-r+q}(n)= \boldsymbol{\gamma}_{q} \in \bbR^r.
\end{equation}
In particular, for some sequence of unit eigenvectors $\mathbf u_{p-r+q}(n)$ of $\W(a(n)2^j)$, under condition \eqref{e:xiq_distinct} the limits $\plim_{n\to\infty} {\mathbf Q}^*(n) {\mathbf u}_{p-r+q}(n)=\boldsymbol\tau_q$ and $\plim_{n\to\infty} {\mathbf P}^*(n)  \mathbf u_{p-r+q}(n)$ exist.
\end{proposition}

Note that both the statement and proof of Proposition \ref{p:conv_w-by-w_rescaled_eigenvalues} naturally apply to the sequence of random matrices $\widetilde{\mathbf W}(a(n)2^j)$. We use this fact in proving Proposition \ref{p:|lambdaq(EW)-xiq(2^j)|_bound}.\\

\noindent{\sc Proof of Proposition \ref{p:|lambdaq(EW)-xiq(2^j)|_bound}} For the fixed $q\in\{1,\ldots,r\}$, let $\mathcal I_0$ be as in \eqref{e:def_indexsets}, and also let
\begin{equation}\label{e:tau_q(n)=Q*(n)_u_p-r+q(n)}
\boldsymbol \tau_q(n) = \boldsymbol \tau_q(n,\omega) = {\mathbf Q}^*(n)  {\boldsymbol{\mathfrak u}}_{p-r+q}(n)
\end{equation}
be as in \eqref{e:T-tilde(n)}. The second statement in \eqref{e:<p,u>_to_gamma} is a direct consequence of the first one, since $ {\mathbf P}^*(n) {\boldsymbol{\mathfrak u}}_{p-r+q}(n) =\mathbf  R^*(n) \boldsymbol \tau_q(n) \stackrel\bbP \to \mathbf R^*\boldsymbol\tau_q=\boldsymbol \gamma_q$ as $n \rightarrow \infty$.

We now show the first statement in \eqref{e:<p,u>_to_gamma}. So, starting from \eqref{e:tau_q(n)=Q*(n)_u_p-r+q(n)}, it suffices to show that, for some deterministic vector ${\mathbf w}$ and for any arbitrary subsequence $n'\in\bbN'$ of $\bbN$, there exists a sub-subsequence $n'' \in \bbN'' \subseteq \bbN'$ such that
\begin{equation}\label{e:tau_q(n'')->w_a.s.}
\boldsymbol \tau_q(n'') \rightarrow {\mathbf w} \quad \text{a.s.}, \quad n'' \rightarrow \infty.
\end{equation}

In fact, first note that, by Corollary \ref{c:PWXP^*+R_asymptotics},
\begin{equation}\label{e:Prop_B.1_conv_rescaled_eigenval}
\frac{\lambda_{p-r+\ell}\big(\widetilde{{\mathbf W}}(a(n)2^j)\big)}{a(n)^{2h_{q}+1}} \stackrel{\bbP}\rightarrow \xi_{\ell}(2^j), \quad \ell \in {\mathcal I}_0, \quad n \rightarrow \infty.
\end{equation}
Now, in view of Lemmas \ref{l:|<p3,uq(n)>|a(n)^{h_3-h_1}=O(1)}, \ref{l:sum<pi(n),up-r+q(n)>2-o(a-varpi)_first} and condition \eqref{e:sqrt(K)(B^-B)->N(0,sigma^2)}, if needed we can refine the sub-subsequence $n''\in\bbN''$ so that relations \eqref{e:subseq_condition_1_top}, \eqref{e:subseq_condition_2_top} and \eqref{e:B-hat(2j)->B(2j)_prop} hold (cf.\ the proof of Proposition \ref{p:conv_w-by-w_rescaled_eigenvalues}).
By an application of Proposition \ref{p:conv_w-by-w_rescaled_eigenvalues} naturally adapted to the random matrix $\widetilde{{\mathbf W}}(a(n'')2^j)$,
\begin{equation}\label{e:Prop_B.1_conv_rescaled_eigenval_n''}
\frac{\lambda_{p-r+\ell}\big(\widetilde{{\mathbf W}}(a(n'')2^j)\big)}{a(n'')^{2h_{q}+1}} \rightarrow \lambda_{r-r_2+\ell}(\boldsymbol\Lambda) \quad \text{a.s.}, \quad \ell \in {\mathcal I}_0, \quad n'' \rightarrow \infty,
\end{equation}
where $\boldsymbol \Lambda$ is the deterministic matrix given by \eqref{e:def_Lambda}. Now recall that
\begin{equation}\label{e:Chi1,...,Chir2}
0 < \chi_1 < \chi_2 < \hdots < \chi_{\eta}
\end{equation}
are the $\eta \in \bbN$ distinct positive eigenvalues of the matrix $\boldsymbol \Lambda$ (see \eqref{e:chi1,...,chi_eta}). Then, by the limits \eqref{e:Prop_B.1_conv_rescaled_eigenval} and \eqref{e:Prop_B.1_conv_rescaled_eigenval_n''} under condition \eqref{e:xiq_distinct}, the scalars $\xi_\ell(2^j)$ in \eqref{e:Prop_B.1_conv_rescaled_eigenval} are distinct, i.e.,
\begin{equation}\label{e:chi1=xi_r1+1<...<chir2=xi_r1+r2}
\chi_1=\xi_{r_1+1}(2^j)< \ldots <\chi_{r_2}=\xi_{r_1+r_2}(2^j)
\end{equation}
(in particular, $\eta = r_2$ in \eqref{e:Chi1,...,Chir2}). Hence, each eigenspace
\begin{equation}\label{e:Mi_of_Lambda}
\mathcal M_i \textnormal{ of } {\boldsymbol \Lambda}
\end{equation}
(see \eqref{e:span(t_ell)_I0=W=M1+...+Meta}) corresponding to these $r_2$ distinct positive eigenvalues is one-dimensional, i.e.,
\begin{equation}\label{e:dim(Mi)=1}
\text{dim}(\mathcal M_i) = 1,\quad i = 1,\hdots,r_2.
\end{equation}
 In particular, we can fix a deterministic vector
\begin{equation}
{\mathbf w} \in {\mathcal M}_{q-r_1} \cap {\mathbb S}^{r-1}
\end{equation}
(\textbf{n.b.}: ${\mathbf w}$ will be the vector appearing in the limit \eqref{e:tau_q(n'')->w_a.s.}, hence the use of the same notation).

Now note that, for each $\ell = 1,\hdots,r$, the sequence $\boldsymbol \tau _{\ell}(n'')$ is a.s.\ bounded. Following the method of proof of Proposition \ref{p:conv_w-by-w_rescaled_eigenvalues} (see relations \eqref{e:nu_n(w)_contained_in_N'} and \eqref{e:nu_n_mainlimit}), we establish the almost sure convergence of $\boldsymbol\tau_{\ell}(n'')$ by considering subsequences for each $\omega \in \Omega$ a.s. To this end, for a fixed $\omega\in \Omega$, let $\nu_n(\omega)$ denote an arbitrary subsequence of $n''\in\bbN''$. Then, we can apply Bolzano-Weierstrass to further refine the subsequence $\nu'_n(\omega)$ so that the limits $\lim_{n\to\infty} \boldsymbol \tau_{\ell}(\nu'_n(\omega))=:\boldsymbol \tau_\ell(\omega)$ exist a.s.\ as in \eqref{e:Gamma_a.s._limit}. Suppose, for the moment, that
\begin{equation}\label{e:span(tau_r1+ell),ell=1,...,r2}
\text{span}\{\boldsymbol \tau_{r_1+i}(\omega)\} = \mathcal M_{i}, \quad i=1,\ldots,r_2.
\end{equation}
In view of \eqref{e:dim(Mi)=1}, this implies that the limiting vectors $\boldsymbol \tau_\ell(\omega)$, $\ell \in {\mathcal I}_0$, are deterministic up to a sign. Consequently, by multiplying ${\boldsymbol \tau}_q(\nu'_n(\omega))$ by $(-1)$ if needed, we obtain $\lim_{n\to\infty}\boldsymbol \tau_q(\nu'_n(\omega))=\boldsymbol\tau_q(\omega) = \mathbf w$. Since the sequence $\nu'_n(\omega)$ is arbitrary, \eqref{e:tau_q(n'')->w_a.s.} holds. This proves that $\boldsymbol \tau_q(n)\stackrel \bbP \to \mathbf w$ as $n \rightarrow \infty$, which establishes the first convergence statement in \eqref{e:<p,u>_to_gamma}.\vspace{2mm}

So, we now need to show \eqref{e:span(tau_r1+ell),ell=1,...,r2}. For this purpose, we use the function $\varphi$ as in \eqref{e:varphi(tau)=tau*_Lambda_tau} as well as the spectral (eigenspace) structure \eqref{e:Mi_of_Lambda} of the matrix ${\boldsymbol \Lambda}$ as in \eqref{e:def_Lambda}. We prove that, for $i = 1,\hdots,r_2$,
\begin{equation}\label{e:rescaled_eigen_p-r+i/a^(2*hq+1)->varphi(tau_i)_distinct_chis}
\frac{\lambda_{p-r+(r_1+i)}\big(\widetilde{\W}(a(\nu_n)2^j)\big)}{a(\nu_n)^{2h_q+1}}(\omega) \rightarrow \varphi({\boldsymbol \tau}_{r_1+i}(\omega)), \quad n \rightarrow \infty,
\end{equation}
where, for ${\boldsymbol \tau}_i={\boldsymbol \tau}_i(\omega)$ as in \eqref{e:Gamma_a.s._limit},
\begin{equation}\label{e:chi1=varphi(tau_1)=<...=<varphi(tau_ell-1)=chi_k_distinct_chis}
\chi_1 = \varphi({\boldsymbol \tau}_{r_1+1}) < \hdots < \varphi({\boldsymbol \tau}_{r_1+i}) = \chi_{i}.
\end{equation}
(cf.\ \eqref{e:rescaled_eigen_p-r+i/a^(2*hq+1)->varphi(tau_i)} and \eqref{e:chi1=varphi(tau_1)=<...=<varphi(tau_ell-1)=chi_k}). In fact, as a consequence of \eqref{e:span(t_ell)_I0=W=M1+...+Meta}, \eqref{e:dim(Mi)=1}, and the claim  \eqref{e:rescaled_eigen_p-r+i/a^(2*hq+1)->varphi(tau_i)_distinct_chis}--\eqref{e:chi1=varphi(tau_1)=<...=<varphi(tau_ell-1)=chi_k_distinct_chis}, ${\boldsymbol \tau}_{r_1+1}(\omega), \hdots, {\boldsymbol \tau}_{r_1+r_2}(\omega)$ are eigenvectors of ${\boldsymbol \Lambda}$ associated, respectively, with the eigenvalues \eqref{e:Chi1,...,Chir2}. In other words, relation \eqref{e:span(tau_r1+ell),ell=1,...,r2} holds.

In turn, we establish \eqref{e:rescaled_eigen_p-r+i/a^(2*hq+1)->varphi(tau_i)_distinct_chis}--\eqref{e:chi1=varphi(tau_1)=<...=<varphi(tau_ell-1)=chi_k_distinct_chis} by induction. Indeed, by \eqref{e:Prop_B.1_conv_rescaled_eigenval_n''} and \eqref{e:chi1=xi_r1+1<...<chir2=xi_r1+r2}, repeating the arguments for  \eqref{e:f(ul(omega)=f(u-tilde))}, we obtain that
$$
\chi_1 \leftarrow \frac{\lambda_{p-r+(r_1+1)}\big(\widetilde{{\mathbf W}}(a(\nu_n)2^j)\big)}{a(\nu_n)^{2h_{q}+1}}(\omega) \rightarrow \varphi({\boldsymbol \tau}_{r_1+1}(\omega)), \quad n \rightarrow \infty.
$$
In other words, the claim holds for $i = 1$. Next, for the induction hypothesis, assume that \eqref{e:rescaled_eigen_p-r+i/a^(2*hq+1)->varphi(tau_i)_distinct_chis}--\eqref{e:chi1=varphi(tau_1)=<...=<varphi(tau_ell-1)=chi_k_distinct_chis} hold $i = r_1+1,\hdots,\ell-1$. Then, the decomposition $\text{span}\{{\boldsymbol \tau}_{\ell}(\omega): \ell \in \mathcal I_0\} =\mathcal M_1 \oplus \ldots \oplus\mathcal M_{r_2}$ (see \eqref{e:span(t_ell)_I0=W=M1+...+Meta}) and relation \eqref{e:dim(Mi)=1} imply that the vectors $\boldsymbol \tau_\ell(\omega)$ satisfy
$\text{span}\{{\boldsymbol \tau}_{r_1+i}(\omega)\} = {\mathcal M}_i$, $i = 1,\hdots,\ell-1$.
In turn, this implies that case $(i)$ (see \eqref{e:case_(i)}) in the proof of Proposition \ref{p:conv_w-by-w_rescaled_eigenvalues} holds with $k = \ell$, i.e.,
\begin{equation*}%
\mathcal M_0\oplus \ldots \oplus \mathcal M_{\ell-1} = \text{span}\{\boldsymbol \tau_{r_1+1},\ldots,\boldsymbol\tau_{\ell-1}\} \subsetneq \mathcal M_0\oplus \ldots \oplus \mathcal M_{\ell}, \quad  \mathcal M_0:=\emptyset.
\end{equation*}
So, by following the same argument starting at relation \eqref{e:case_(i)}, we conclude that relation \eqref{e:prop_case_(i)_final_step} holds, namely,
\begin{equation*}%
\frac{\lambda_{p-r+\ell}\big(\widetilde{\W}(a(\nu_n)2^j)\big)}{a(\nu_n)^{2h_q+1}}(\omega) \rightarrow \varphi({\boldsymbol \tau}_{\ell}(\omega))=\chi_{\ell}, \quad n \rightarrow \infty.
\end{equation*}
This finishes the induction. Hence, \eqref{e:rescaled_eigen_p-r+i/a^(2*hq+1)->varphi(tau_i)_distinct_chis}--\eqref{e:chi1=varphi(tau_1)=<...=<varphi(tau_ell-1)=chi_k_distinct_chis} hold, and so does \eqref{e:span(tau_r1+ell),ell=1,...,r2}. $\Box$\\

The following result, Lemma \ref{l:max|<up-r+q(n),pi(n)>|a(n)^(hi-hq)=OP(1)}, is used in the proof of Theorem \ref{t:asympt_normality_lambdap-r+q}. It establishes that the terms \eqref{e:supR'max(angles*powerlaws)=O_P(1)}, which are shown to be bounded in probability in Lemma \ref{l:|<p3,uq(n)>|a(n)^{h_3-h_1}=O(1)}, $(i)$, in truth have limits in probability under the additional condition \eqref{e:xiq_distinct}. As with Proposition \ref{p:|lambdaq(EW)-xiq(2^j)|_bound}, when proving Lemma \ref{l:max|<up-r+q(n),pi(n)>|a(n)^(hi-hq)=OP(1)} we use the fact that the proof of Proposition \ref{p:conv_w-by-w_rescaled_eigenvalues} naturally applies to the sequence of random matrices $\widetilde{\mathbf W}(a(n)2^j)$ in place of ${\mathbf W}(a(n)2^j)$.

In the proof of Lemma \ref{l:max|<up-r+q(n),pi(n)>|a(n)^(hi-hq)=OP(1)}, as well as in that of the subsequent Lemma \ref{l:supR'max(angles*powerlaws)_bounded_for_subseq}, we refer to the functions
\begin{equation}\label{e:varphi,varphi-n}
 \varphi(\cdot,{\boldsymbol \tau}) \textnormal{ and }\widehat{\varphi}_n(\cdot,{\boldsymbol \tau})
\end{equation}
as defined by the expressions \eqref{e:def_varphi(x,u)} and \eqref{e:def_phi_n(x,u)}, respectively, with $\mathbf B_n$ naturally replacing $\widehat {\mathbf B}_a(2^j)$.

\begin{lemma}\label{l:max|<up-r+q(n),pi(n)>|a(n)^(hi-hq)=OP(1)} Fix any $j$ as in \eqref{e:def_j1,jm}. Let $\widetilde{ \mathbf W}(a(n)2^j)$ be the sequence of random matrices given by \eqref{e:W-tilde(a2^j)} (see Lemma \ref{l:gutted_log_eig_consistency}). Assume condition \eqref{e:xiq_distinct} in Theorem \ref{t:asympt_normality_lambdap-r+q} holds.
Fix $q \in \{1,\hdots,r\}$ and consider the associated index sets as in \eqref{e:def_indexsets}. Then, for the sequence of $(p-r+q)$--th unit eigenvectors $\{{\boldsymbol{\mathfrak u}}_{p-r+q}(n)\}_{n\in\bbN}$ of $\widetilde{\mathbf W}(a(n)2^j)$ described in Proposition \ref{p:|lambdaq(EW)-xiq(2^j)|_bound}, there are deterministic constants $x_{*,q,i}$, $i \in \mathcal{I}_+$, such that, as $n \rightarrow \infty$,
\begin{equation}\label{e:inner*scaling=o(1)}
\langle \p_i(n),{\boldsymbol{\mathfrak u}}_{p-r+q}(n) \rangle a(n)^{h_{i} - h_{q}} \stackrel{\bbP}\to x_{*,q,i},\quad i \in \mathcal{I}_+.
\end{equation}
\end{lemma}
\begin{proof}
For the fixed $q \in \{1,\hdots,r\}$, recall that we write $r_3 = \text{card}({\mathcal I}_+)$ (see \eqref{e:r1,r2,r3}). Let
\begin{equation}\label{e:x(n)}
\mathbf x(n)= \mathbf x(n,\omega):= \Big(\langle {\mathbf p}_i(n), {\boldsymbol{\mathfrak u}}_{p-r+q}(n) \rangle a(n)^{h_i - h_q} \Big)_{i \in{\mathcal{I}_+}} \in \bbR^{r_3}
\end{equation}
be a random vector containing the scalar terms appearing on the left-hand side of \eqref{e:inner*scaling=o(1)}. To establish \eqref{e:inner*scaling=o(1)}, it suffices to show that, for some deterministic vector ${\mathbf x}_* \in \bbR^{r_3}$ and any arbitrary subsequence $\{\mathbf x(n')\}_{n' \in \bbN'}$, there exists a further subsequence $n'' \in \bbN'' \subseteq \bbN'$ such that
\begin{equation}\label{e:x(n'')->x_*}
\mathbf x(n'') \rightarrow {\mathbf x}_*  \quad \text{a.s.}, \quad n'' \rightarrow \infty.
\end{equation}

In fact, under condition \eqref{e:xiq_distinct}, Proposition \ref{p:|lambdaq(EW)-xiq(2^j)|_bound} implies the existence of the limits $\boldsymbol \tau_\ell(n') =\mathbf  {\mathbf Q}^*(n') {\boldsymbol{\mathfrak u}}_{p-r+\ell}(n')\stackrel \bbP \to \boldsymbol \tau_\ell \in \bbR^r$, $\ell= 1,\hdots,r$, where each vector $\boldsymbol \tau_\ell$ is deterministic. By passing to a further subsequence $n''\in\bbN''\subseteq\bbN'$ if needed, we may assume that
\begin{equation}\label{e:lim_n''->infty_gamma-ell(n'')}
\lim_{n''\to\infty}\boldsymbol \tau_\ell(n'')=\boldsymbol \tau_\ell \quad \textnormal{ a.s.}, \quad \ell = 1,\hdots, r.
\end{equation}
For each fixed $\boldsymbol \tau \in \bbR^{r}$, recall that the vectors ${\mathbf x}_{*}(\boldsymbol \tau)$ and ${\mathbf x}_{*,n}({\boldsymbol \tau})$, defined in \eqref{e:def_x-star(u)}, are the minimizers of the functions $\varphi$ and ${\widehat \varphi}_n$, respectively, as in \eqref{e:varphi,varphi-n}. So, for the fixed $q \in \{1,\hdots,r\}$ and ${\boldsymbol \tau}_q$, ${\boldsymbol \tau}_q(n'')$ as in \eqref{e:lim_n''->infty_gamma-ell(n'')}, let
\begin{equation}\label{e:x_*(n),x_*_0}
\mathbf x_*=\mathbf x_*(\boldsymbol \tau_q)=(x_{*,q+1},\ldots,x_{*,r})\in \bbR^{r_3} \quad \textnormal{and} \quad \mathbf x_*(n'') = {\mathbf x}_{*,n}\big({\boldsymbol \tau}_q(n'')\big)\in \bbR^{r_3},
\end{equation}
where ${\mathbf x}_*$ is deterministic (\textbf{n.b.}: $\mathbf x_*$ will be the vector appearing in the limit \eqref{e:x(n'')->x_*}, hence the use of the same notation). Note that, by expression \eqref{e:def_x-star(u)},
\begin{equation}\label{e:x_*(n'')->x_*}
{\mathbf x}_*(n'') \rightarrow {\mathbf x}_* \quad \textnormal{a.s.}, \quad n'' \rightarrow \infty.
\end{equation}
By \eqref{e:sqrt(K)(B^-B)->N(0,sigma^2)}, without loss of generality we can assume that
\begin{equation}\label{e:B-hat(2j)->B(2j)_n''->infty}
\B_n \to \B(2^j) \quad \text{a.s.}, \quad n'' \to \infty.
\end{equation}

We now use the convergence \eqref{e:x_*(n'')->x_*} of ${\mathbf x}_*(n'')$ to establish the convergence \eqref{e:x(n'')->x_*} of ${\mathbf x}(n'')$. So, suppose, for the moment, that the functions $\varphi$ and ${\widehat \varphi}_{n''}$ as in \eqref{e:varphi,varphi-n} satisfy the system of inequalities
\begin{equation}\label{e:string_ineqs_n''}
\varphi(\mathbf x_*, \boldsymbol \tau_q )=\lim_{n'' \to \infty } {\widehat \varphi}_{n''}\big(\mathbf x_*(n''), {\boldsymbol \tau}_{q}(n'')\big)\leq \lim_{n''\to \infty }{\widehat \varphi}_{n''}\big(\mathbf x(n''), {\boldsymbol \tau}_{q}( n'')\big)
 \leq  \varphi(\mathbf x_*, \boldsymbol \tau_q ) \quad \text{a.s.}
\end{equation}
Following the method of proof of Proposition \ref{p:conv_w-by-w_rescaled_eigenvalues}  (see relations \eqref{e:nu_n(w)_contained_in_N'} and \eqref{e:nu_n_mainlimit}), we show the almost sure convergence of $\mathbf x(n'')$ by considering subsequences for each $\omega \in \Omega$ a.s.

More precisely, consider $\mathbf x(n'')$ as in \eqref{e:x(n)}. By Lemma \ref{l:supR'max(angles*powerlaws)_bounded_for_subseq} we can pass to a further subsequence (still denoted $n'' \in \bbN''$, for simplicity) such that, for $\omega \in \Omega$ a.s.\ and along any arbitrary subsequence $\nu_n(\omega)$ of $n''$, $\sup_{n \in \bbN}\|\mathbf x(\nu_n(\omega))\| < \infty$. Thus, we can apply the Bolzano-Weierstrass theorem to obtain a sub-subsequence (still denoted $\nu'_n(\omega)$, for simplicity) such that $\mathbf x(\nu'_n)\to \mathbf x_\infty \in \bbR^{r_3}$ a.s., where $\mathbf x_\infty$ is a possibly random limit vector. However, along this sub-subsequence $\nu_n'$, relations \eqref{e:lim_n''->infty_gamma-ell(n'')} and \eqref{e:x_*(n'')->x_*}, respectively, imply that $\lim_{n\to\infty}\boldsymbol \tau_q(\nu_n')={\boldsymbol \tau}_{q}$ a.s.\ and $\lim_{n\to\infty}\mathbf x_*(\nu_n')  = \mathbf x_*$ a.s. Thus, by \eqref{e:string_ineqs_n''},
\begin{equation}\label{e:varphi(x*,tau-q)=<lims=<varphi(x*,tau-q)}
\varphi(\mathbf x_*, \boldsymbol \tau_q\big)=\lim_{n\to \infty } {\widehat \varphi}_{\nu_n'}\big(\mathbf x_*(\nu_n'), {\boldsymbol \tau}_{q}(\nu_n')\big)\leq \lim_{n\to \infty }{\widehat \varphi}_{\nu_n'}\big(\mathbf x(\nu_n'), {\boldsymbol \tau}_{q}(\nu_n')\big)
 \leq  \varphi(\mathbf x_*, \boldsymbol \tau_q ).
\end{equation}
In other words, by \eqref{e:varphi,varphi-n}, \eqref{e:B-hat(2j)->B(2j)_n''->infty} and \eqref{e:varphi(x*,tau-q)=<lims=<varphi(x*,tau-q)}, $\varphi(\mathbf x_*, \boldsymbol \tau_q)=\lim_{n\to \infty }{\widehat \varphi}_{\nu_n'}\big(\mathbf x(\nu_n'),{\boldsymbol \tau}_{q}(\nu_n')\big)= \varphi(\mathbf x_\infty, \boldsymbol \tau_q)$ a.s.  Since the minimizer $\mathbf x_*$ of $\varphi(\mathbf x, \boldsymbol \tau_q )$ in $\mathbf x$ is unique, we conclude that $\mathbf x_\infty=\mathbf x_*$ a.s. Since the original random subsequence $\nu_n$ was arbitrary, we obtain \eqref{e:x_*(n'')->x_*}. Hence, we conclude that $\mathbf x(n)\stackrel \bbP \to \mathbf x_*$, as anticipated. This establishes \eqref{e:inner*scaling=o(1)}.\vspace{2mm}

So, we now need to show \eqref{e:string_ineqs_n''}. Refining the subsequence $n''$ if necessary,  by Lemmas \ref{l:|<p3,uq(n)>|a(n)^{h_3-h_1}=O(1)} and \ref{l:sum<pi(n),up-r+q(n)>2-o(a-varpi)_first} we may assume the convergence statements \eqref{e:subseq_condition_1_top} and \eqref{e:subseq_condition_2_top} hold, with ${\boldsymbol{\mathfrak u}}_{p-r+\ell}(n'')$ in place of ${\mathbf u}_{p-r+\ell}(n'')$, along the subsequence $n''$ (cf.\ the proof of Proposition \ref{p:conv_w-by-w_rescaled_eigenvalues}). Since, in addition, relation \eqref{e:lim_n''->infty_gamma-ell(n'')} holds, then the assumptions of Lemma \ref{l:<p,w>=infinitesimal} hold along the \textit{non}random subsequence $n''$. So, by Lemma \ref{l:<p,w>=infinitesimal}%
, there is a sequence of unit vectors
\begin{equation}\label{e:v(n'')}
\bbR^{p} \ni \mathbf v(n'') \in \text{span}\{ {\boldsymbol{\mathfrak u}}_{p-r+q}(n''),\ldots,{\boldsymbol{\mathfrak u}}_{p}(n'')\}
\end{equation}
such that, for some large $M(\omega)$, $n''\geq M(\omega)$ implies that
\begin{equation}\label{e:p*v=x/a}
\langle \p_i(n''),\mathbf v(n'')\rangle = \frac{x_{*,i}}{a^{h_i -h_q}},\quad i \in \mathcal I_+,
\end{equation}
and
\begin{equation}\label{e:def_w(n'')}
 {\mathbf Q}^*(n'')\mathbf v(n'')=:  \mathbf w(n'') \to \boldsymbol\tau_q \quad \text{a.s.}
\end{equation}
Recall the notation \eqref{e:W_PX(a2^j)} for $\widetilde{\Xi}_q(n'')$. Thus, for ${\mathbf x}_*(n'')$, ${\mathbf x}(n'')$ and $\boldsymbol \tau_q(n'')$ as in \eqref{e:x_*(n),x_*_0}, \eqref{e:x(n)} and \eqref{e:lim_n''->infty_gamma-ell(n'')}, respectively,
$$
{\widehat \varphi}_{n''}\big(\mathbf x_*(n''), {\boldsymbol \tau}_{q}(n'')\big)+ {\boldsymbol{\mathfrak u}}_{p-r+q}^*(n'') \widetilde{\Xi}_q(n'')
{\boldsymbol{\mathfrak u}}_{p-r+q}(n'')
$$
$$
\leq {\widehat \varphi}_{n''}\big(\mathbf x(n''), {\boldsymbol \tau}_{q}(n'')\big)+ {\boldsymbol{\mathfrak u}}_{p-r+q}^*(n'') \widetilde{\Xi}_q(n''){\boldsymbol{\mathfrak u}}_{p-r+q}(n'')
$$
$$
=\frac{\lambda_{p-r+q}\big(\mathbf{\widetilde{W}}(a(n'')2^j)\big)}{a(n'')^{2h_q + 1}}
\leq {\mathbf v}^*(n'') \frac{ \mathbf{\widetilde{W}}(a(n'')2^j)}{a(n'')^{2h_q + 1}}{\mathbf v}(n'')
$$
\begin{equation}\label{e:gn_string_of_inequalities_2}
= {\widehat \varphi}_{n''}(\mathbf x_*, \mathbf w(n'')\big) +
{\mathbf v}^*(n'') \widetilde{\Xi}_q(n''){\mathbf v}(n'').
\end{equation}
The first inequality in \eqref{e:gn_string_of_inequalities_2} is a consequence of the fact that $\mathbf x_*(n'')$ ($\neq {\mathbf x}(n'')$) minimizes ${\widehat \varphi}_{n''}({\mathbf x}, {\boldsymbol \tau}_{q}(n'')\big)$ in ${\mathbf x}$. The first equality stems from \eqref{e:rescaled_eigen_p-r+ell_in_terms_of_fn_and_varphin} naturally reinterpreted for the matrix $\widetilde{{\mathbf W}}(a(n'')2^j)$. The second inequality follows from \eqref{e:v(n'')}. In particular, since \eqref{e:p*v=x/a} implies $\|a^{\mathbf h-h_q{\mathbf I}}\mathbf P^*\mathbf v(n'')\| = O_\bbP(1),$ we have
\begin{equation}\label{e:v*(n'')Xi(n'')v(n'')=oP(1)}
{\mathbf v}^*(n'')\widetilde{\Xi}_q(n''){\mathbf v}(n'') = {\mathbf v}^*(n'')\Big(\frac{O_{\bbP}(1)}{a^{2h_q+1}}+
\frac{{\mathbf P}a^{{\mathbf h}}}{a^{h_q}}\frac{O_{\bbP}(1)}{a^{h_q + 1/2}}+\frac{O^*_{\bbP}(1)}{a^{h_q + 1/2}}\frac{a^{{\mathbf h}}{\mathbf P}^*}{a^{h_q}}\Big){\mathbf v}(n'')=o_{\bbP}(1).
\end{equation}
In view of relations \eqref{e:sqrt(K)(B^-B)->N(0,sigma^2)}, \eqref{e:v*(n'')Xi(n'')v(n'')=oP(1)}, \eqref{e:residual_of_W-tilde_along_eigenvector_directions} (see Lemma \ref{l:|<p3,uq(n)>|a(n)^{h_3-h_1}=O(1)}, $(ii)$) and \eqref{e:B-hat(2j)->B(2j)_n''->infty}, by passing to a further subsequence if needed we may assume that
\begin{equation}\label{e:u*(n'')(W/a^(2h_q+1)-Pa^hB-hata^hP*/a^(2h_q))u(n'')}
{\mathbf v}^*(n'')\widetilde{\Xi}_q(n'')
 {\mathbf v}(n'') \to 0 \quad \textnormal{a.s.}, \qquad
{\boldsymbol{\mathfrak u}}_{p-r+q}^*(n'')
\widetilde{\Xi}_q(n'')
{\boldsymbol{\mathfrak u}}_{p-r+q}(n'')
\to 0 \quad \textnormal{a.s.}
\end{equation}
Now note that, by relations \eqref{e:lim_n''->infty_gamma-ell(n'')}, \eqref{e:x_*(n'')->x_*}, \eqref{e:def_w(n'')}, \eqref{e:def_phi_n(x,u)} and \eqref{e:def_varphi(x,u)},
 \begin{equation}\label{e:system_limits_varphi-hat_n''}
 \lim_{n''\to\infty} {\widehat \varphi}_{n''}\big(\mathbf x_*(n''), {\boldsymbol \tau}_{q}(n'')\big)= \varphi (\mathbf{x}_*,\boldsymbol \tau_q)=\lim_{n''\to\infty} {\widehat \varphi}_{n''}\big(\mathbf x_*, \mathbf w(n'')\big) \quad \text{a.s.}
 \end{equation}
 Thus, based on \eqref{e:u*(n'')(W/a^(2h_q+1)-Pa^hB-hata^hP*/a^(2h_q))u(n'')}, \eqref{e:system_limits_varphi-hat_n''} and the string of inequalities \eqref{e:gn_string_of_inequalities_2}, we obtain
$$
\varphi (\mathbf x_*, \boldsymbol \tau_q)=\lim_{n'' \to \infty } {\widehat \varphi}_{n''}\big(\mathbf x_*(n''), {\boldsymbol \tau}_{q}(n'')\big)
$$
$$
\leq \lim_{n''\to \infty }{\widehat \varphi}_{n''}\big(\mathbf x(n''), {\boldsymbol \tau}_{q}( n'')\big)
 \leq \lim_{n''\to \infty }{\widehat \varphi}_{n''}\big(\mathbf x_*, {\mathbf w}( n'')\big) = \varphi (\mathbf x_*, \boldsymbol \tau_q ).
$$
Hence, \eqref{e:string_ineqs_n''} holds, as claimed.
\end{proof}

The following lemma is used in the proof of Lemma \ref{l:max|<up-r+q(n),pi(n)>|a(n)^(hi-hq)=OP(1)}.
\begin{lemma}\label{l:supR'max(angles*powerlaws)_bounded_for_subseq}
Let $n'' \in \bbN''$ be the sequence appearing in \eqref{e:x_*(n'')->x_*}. Then, we can pass to a subsequence (still denoted $n'' \in \bbN''$, for simplicity) such that, for $\omega \in \Omega$ a.s.\ and along any arbitrary random subsequence $\nu_n(\omega)$ of $n''$,
\begin{equation}\label{e:supR'max(angles*powerlaws)_bounded_for_subseq}
 \sup_{n \in \bbN} \max_{i \in \mathcal{I}_+}\big\{ |\langle {\mathbf p}_{i}(\nu_n(\omega)),{\boldsymbol {\mathfrak u}}_{p-r+q}(\nu_n(\omega))\rangle| \hspace{0.5mm}a(\nu_n(\omega))^{h_{i}-h_q}\big\} <\infty.
\end{equation}
\end{lemma}
\begin{proof}
Let $\mathbf x(n)$ be as in \eqref{e:x(n)}. Consider the matrix $\mathbf Q(n)$ as in \eqref{e:P(n)=Q(n)R(n)} and let $\boldsymbol \tau_q(n)= \mathbf Q^*(n){\boldsymbol {\mathfrak u}}_{p-r+q}(n)$ (see \eqref{e:T-tilde(n)}). For the function $ {\widehat \varphi}_{n}$ as defined in \eqref{e:varphi,varphi-n}, we may rewrite
\begin{equation}\label{e:rexxpress_W_tilde}
\frac{\lambda_{p-r+q}\big(\widetilde{\mathbf{W}}(a(n'')2^j)\big)}{a(n'')^{2h_q + 1}} = {\widehat \varphi}_{n''}(\mathbf x(n''), {\boldsymbol \tau}_{q}(n'')\big)+ {\boldsymbol {\mathfrak u}}^*_{p-r+q}(n'')  {\widetilde\Xi_q(n'')}{\boldsymbol {\mathfrak u}}_{p-r+q}(n'')
\end{equation}
(cf.\ \eqref{e:rescaled_eigen_p-r+ell_in_terms_of_fn_and_varphin}). Moreover, \eqref{e:rexxpress_W_tilde} is bounded above by
$$
\inf_{{\mathcal U}_{p-r+q}}\sup_{{\mathbf u} \in {\mathcal U}_{p-r+q} \cap \bbS^{p-1}} {\mathbf u}^*\frac{\widetilde{\mathbf W}(a(n'')2^j)}{a(n'')^{2h_q + 1}}{\mathbf u} \leq \sup_{{\mathbf u} \in {\textnormal{span}\{\mathbf p_{q+1},\ldots,\mathbf{p}_r\}}^{\perp} \cap \bbS^{p-1}} {\mathbf u}^*\frac{\widetilde{\mathbf W}(a(n'')2^j)}{a(n'')^{2h_q + 1}}{\mathbf u}
$$
$$
\leq \sup_{{\mathbf u} \in {\textnormal{span}\{\mathbf p_i,i \in \mathcal{I}_+\}}^{\perp} \cap \bbS^{p-1}} {\mathbf u}^*\frac{\widetilde{\mathbf W}(a(n'')2^j)}{a(n'')^{2h_q + 1}}{\mathbf u} =:\widetilde{ \mathbf u}^*_{n''}\hspace{1mm}\frac{\widetilde{\mathbf W}(a(n'')2^j)}{a(n'')^{2h_q + 1}}\hspace{1mm}\widetilde{ \mathbf u}_{n''}
$$
\begin{equation}\label{e:lambdaq_over_a^(2h1+1)=O(1)}
= \begin{pmatrix}
a(n'')^{\mathbf h_1 - h_q {\mathbf I}}\mathbf P^*_1(n'') \widetilde {\mathbf u}_{n''}\\
\mathbf P^*_2(n'') \widetilde {\mathbf u}_{n''}\\
\mathbf 0
\end{pmatrix}^*  \mathbf B_{n''} \begin{pmatrix}
a(n'')^{\mathbf h_1 - h_q {\mathbf I}}\mathbf P^*_1(n'') \widetilde {\mathbf u}_{n''}\\\
\mathbf P^*_2(n'') \widetilde {\mathbf u}_{n''}\\
\mathbf 0
\end{pmatrix} + \widetilde {\mathbf u}^*_{n''} \hspace{1mm}\widetilde\Xi_q(n'')\hspace{1mm}\widetilde {\mathbf u}_{n''}.
\end{equation}
In \eqref{e:lambdaq_over_a^(2h1+1)=O(1)}, ${\mathbf h}_1$ is given by \eqref{e:h1} and $\mathbf P^*_1(n'')$, $\mathbf P^*_2(n'')$ are as in \eqref{e:P=(P1(n)_P2(n)_P3(n))}.

However, since $\widetilde {\mathbf u}_{n''} \in {\textnormal{span}\{\mathbf p_i,i \in \mathcal{I}_+\}}^{\perp}$, relation \eqref{e:W_PX(a2^j)} implies that
\begin{equation}\label{e:u-tilde*_n_Xi-tilde_u-tilde_n}
\widetilde {\mathbf u}^*_{n''}\hspace{0.5mm}  {\widetilde\Xi_q(n'')}\hspace{0.5mm} \widetilde {\mathbf u}_{n''} = o_\bbP(1).
\end{equation}
Hence, by relations \eqref{e:Bn->B(2^j)}, \eqref{e:residual_of_W-tilde_along_eigenvector_directions} (see Lemma \ref{l:|<p3,uq(n)>|a(n)^{h_3-h_1}=O(1)}, $(ii)$) and \eqref{e:u-tilde*_n_Xi-tilde_u-tilde_n}, we can find a further subsequence (still denoted $n''$, for simplicity) such that the event
$$
{A_0} = \Big\{\omega \in \Omega: \hspace{2mm}\|{\mathbf B}_{n''}- {\mathbf B}(2^j)\|=o(1),
$$
\begin{equation}\label{e:A0}
\hspace{2mm} {\boldsymbol {\mathfrak u}}^*_{p-r+q}(n'')  {\widetilde\Xi_q(n'')}{\boldsymbol {\mathfrak u}}_{p-r+q}(n'') = o(1), \hspace{2mm}\widetilde {\mathbf u}^*_{n''} \hspace{0.5mm} {\widetilde\Xi_q(n'')}\hspace{0.5mm} \widetilde {\mathbf u}_{n''} = o(1)\Big\}
\end{equation}
occurs with probability 1. So, fix $\omega \in A_0$ and consider any subsequence $\nu_n(\omega)$. Let ${\mathbf B}_{n,33}$ be the $r_3\times r_3$ lower-right sub-block of $\mathbf B_{n}$. Note that $\|a(\nu_n)^{\mathbf h_1 - h_q {\mathbf I}}\mathbf R^*_1(\nu_n) {\boldsymbol \tau}_{q}(\nu_n)\|$, $\|\mathbf R^*_2(\nu_n) {\boldsymbol \tau}_{q}(\nu_n)\|$ are bounded a.s. Then, for ${\widehat \varphi}_{n}$ as in \eqref{e:varphi,varphi-n}, we can almost surely write
$$
{\widehat \varphi}_{\nu_n}\big(\mathbf x(\nu_n), {\boldsymbol \tau}_{q}(\nu_n)\big)
= \mathbf{x}^*(\nu_n) {\mathbf B}_{\nu_n,33} \hspace{0.5mm} \mathbf{x}(\nu_n)+ O(\|\mathbf x(\nu_n)\|)
$$
\begin{equation}\label{e:varphi-hat_n=quad_form+O(||x||)}
\geq \mathbf{x}^*(\nu_n) {\mathbf B}_{\nu_n,33} \hspace{0.5mm} \mathbf{x}(\nu_n) + O(1).
\end{equation}
Also, almost surely,
\begin{equation}\label{e:max(non-dominant_scaling)=O(1)}
\max\big\{\|a(\nu_n)^{\mathbf h_1 - h_q {\mathbf I}}\mathbf P^*_1(\nu_n) \widetilde {\mathbf u}_{\nu_n}\|, \|\mathbf P^*_2(\nu_n) \widetilde {\mathbf u}_{\nu_n}\|\big\}= O(1)
\end{equation}
(cf.\ \eqref{e:lambdaq_over_a^(2h1+1)=O(1)}). Thus, in view of \eqref{e:rexxpress_W_tilde}--\eqref{e:max(non-dominant_scaling)=O(1)}, we can almost surely write
$$
\|  \mathbf{x}(\nu_n)\|^2 \big\{\lambda_1( {\mathbf B}_{\nu_n,33}) + o(1)\big\} + O(1) \leq   {\widehat \varphi}_{\nu_n}\big(\mathbf x(\nu_n), {\boldsymbol \tau}_{q}(\nu_n)\big)+o(1)
$$
$$
=\frac{\lambda_{p-r+q}\big(\widetilde{\mathbf W}(a(\nu_n)2^j)\big)}{a(\nu_n)^{2h_q + 1}} \leq  \sup_{n \in \bbN}\frac{\lambda_{p-r+q}\big(\widetilde{\mathbf W}(a(\nu_n)2^j)\big)}{a(\nu_n)^{2h_q + 1}} < \infty.
$$
Since, in addition, $\lambda_1( {\mathbf B}_{\nu_n,33})$ is a.s.\ bounded away from zero for large enough $n$, this establishes \eqref{e:supR'max(angles*powerlaws)_bounded_for_subseq}, and hence, the claim.
\end{proof}

As a consequence of the following simple lemma, the mean-value theorem-type expansions appearing in the proof of Theorem \ref{t:asympt_normality_lambdap-r+q} are well defined with probability going to 1 as $n \rightarrow \infty$. To state and prove the lemma, it will be convenient to define the magnitude
\begin{equation}\label{e:eta0}
\eta_0 := \frac{1}{2} \min_{\ell \in {\mathcal I}_0 \backslash \{r_1+1\}}\{\xi_\ell(2^j)-\xi_{\ell-1}(2^j)\}, \quad \min_{\ell \in \emptyset}\{\xi_\ell(2^j)-\xi_{\ell-1}(2^j)\} := 1.
\end{equation}
Note that, under condition \eqref{e:xiq_distinct}, $\eta_0 > 0$.
\begin{lemma}\label{l:f1,f2,f3_well_defined}
Fix any $j$ as in \eqref{e:def_j1,jm}. Suppose ($A4$) and ($A5$) hold and let $\mathbf h$ be as in \eqref{e:h-bf=diag(h1,...,hr)}. Let ${\mathbf B}(2^j)$ be the matrix given by \eqref{e:B(2^j)_full_rank}. Let $\xi_{\ell}(2^j)$, $\ell=1,\ldots,r$, be the functions appearing in \eqref{e:lim_n_a*lambda/a^(2h+1)}, and suppose condition \eqref{e:xiq_distinct} in Theorem \ref{t:asympt_normality_lambdap-r+q} holds. For ${\mathbf B} \in {\mathcal S}_{\geq 0}(r,\bbR)$, ${\mathbf K}_1 \in {\mathcal S}_{\geq 0}(p,\bbR)$ and ${\mathbf K}_2 \in {\mathcal M}(r,p,\bbR)$, define the deterministic matrix
$$
{\mathcal S}(p,\bbR) \ni \overline{{\mathbf W}}\big(a(n)2^j,{\mathbf B},{\mathbf K}_1,{\mathbf K}_2\big)
$$
\begin{equation}\label{e:W-tilde(a2^j,B,K1,K2)}
= a(n)^{2h_q+1}\Big(\frac{{\mathbf P}(n)a(n)^{{\mathbf h}} {\mathbf B} a(n)^{{\mathbf h}}{\mathbf P}^*(n)}{a(n)^{2h_q}}  + {\mathbf K}_1+
\frac{{\mathbf P}(n)a(n)^{{\mathbf h}}}{a(n)^{h_q}}{\mathbf K}_2+{\mathbf K}^*_2 \frac{a(n)^{{\mathbf h}}{\mathbf P}^*(n)}{a(n)^{h_q}} \Big).
\end{equation}
For $\delta > 0$, $\zeta_{1} > 0$, $\zeta_{2} >0$, further define the matrix vicinities
$$
{\mathcal O}_{\delta,r} = \{ {\mathbf B} \in {\mathcal S}_{\geq 0}(r,\bbR): \|{\mathbf B}-{\mathbf B}(2^j)\| < \delta\}, \quad {\mathcal O}_{\zeta_1,p} = \{ {\mathbf K}_1 \in {\mathcal S}(p,\bbR): \|{\mathbf K}_1\| < \zeta_1\}
$$
\begin{equation}\label{e:vicinities_Odelta,r_Ozeta01,p_Ozeta02,r,p}
and \hspace{3mm}{\mathcal O}_{\zeta_2,r,p} = \{ {\mathbf K}_2 \in {\mathcal M}(r,p,\bbR): \|{\mathbf K}_2\| < \zeta_2\}
\end{equation}
(\textbf{n.b.}: $ p = p(n)$).
\begin{itemize}
\item [$(i)$] Then, there exist $n_0 \in \bbN$, $\delta_0 > 0$, $\zeta_{01} > 0$ and $\zeta_{02} > 0$ such that, for any $n \geq n_0$, ${\mathbf B} \in {\mathcal O}_{\delta_0,r}$, ${\mathbf K}_1 \in {\mathcal O}_{\zeta_{01},p}$ and ${\mathbf K}_2 \in {\mathcal O}_{\zeta_{02},r,p}$,
$$
\lambda_{p-r+q-1}\Big( \frac{\overline{{\mathbf W}}\big(a(n)2^j,{\mathbf B},{\mathbf K}_1,{\mathbf K}_2\big)}{a(n)^{2h_q+1}}\Big) + \eta_0
< \lambda_{p-r+q}\Big( \frac{\overline{{\mathbf W}}\big(a(n)2^j,{\mathbf B},{\mathbf K}_1,{\mathbf K}_2\big)}{a(n)^{2h_q+1}}\Big)
$$
\begin{equation}\label{e:lambda_are_simple}
< \lambda_{p-r+q + 1}\Big( \frac{\overline{{\mathbf W}}\big(a(n)2^j,{\mathbf B},{\mathbf K}_1,{\mathbf K}_2\big)}{a(n)^{2h_q+1}}\Big) - \eta_0.
\end{equation}
In particular, $\lambda_{p-r+q}\Big( \frac{\overline{{\mathbf W}}\big(a(n)2^j,{\mathbf B},{\mathbf K}_1,{\mathbf K}_2\big)}{a(n)^{2h_q+1}}\Big)$ is a simple eigenvalue.
\item [$(ii)$] In addition, by possibly picking a larger $n_0 \in \bbN$ and restricting the vicinities \eqref{e:vicinities_Odelta,r_Ozeta01,p_Ozeta02,r,p} obtained in $(i)$, for any $n \geq n_0$, ${\mathbf B} \in {\mathcal O}_{\delta_0,r}$, ${\mathbf K}_1 \in {\mathcal O}_{\zeta_{01},p}$ and ${\mathbf K}_2 \in {\mathcal O}_{\zeta_{02},r,p}$,
\begin{equation}\label{e:lambda_p-r+q_is_positive}
\lambda_{p-r+q}\Big( \frac{\overline{{\mathbf W}}\big(a(n)2^j,{\mathbf B},{\mathbf K}_1,{\mathbf K}_2\big)}{a(n)^{2h_q+1}}\Big) > 0.
\end{equation}
\end{itemize}
\end{lemma}
\begin{proof}
We first show $(i)$. By way of contradiction, suppose \eqref{e:lambda_are_simple} does not hold. Namely, for any $n_0 \in \bbN$, $\delta_0 > 0$, $\zeta_{01} > 0$, $\zeta_{02} > 0$, there exist $n = n(n_0) \geq n_0$, ${\mathbf B}_{\delta_0} \in {\mathcal O}_{\delta_0,r}$, ${\mathbf K}_{1,\zeta_{01}} \in {\mathcal O}_{\zeta_{01},p}$ and ${\mathbf K}_{2,\zeta_{02}} \in {\mathcal O}_{\zeta_{02},r,p}$ such that
\begin{equation}\label{e:lambda_p-r+ell=<lambda_p-r+ell-1_contradiction}
\lambda_{p-r+q}\Big( \frac{\overline{{\mathbf W}}\big(a(n)2^j,{\mathbf B}_{\delta_0},{\mathbf K}_{1,\zeta_{01}},{\mathbf K}_{2,\zeta_{02}}\big)}{a(n)^{2h_q+1}}\Big)
< \lambda_{p-r+q-1}\Big( \frac{\overline{{\mathbf W}}\big(a(n)2^j,{\mathbf B}_{\delta_0},{\mathbf K}_{1,\zeta_{01}},{\mathbf K}_{2,\zeta_{02}}\big)}{a(n)^{2h_q+1}}\Big) + \eta_0
\end{equation}
or
\begin{equation}\label{e:lambda_p-r+ell+1=<lambda_p-r+ell-1-eta0_contradiction}
\lambda_{p-r+q + 1}\Big( \frac{\overline{{\mathbf W}}\big(a(n)2^j,{\mathbf B}_{\delta_0},{\mathbf K}_{1,\zeta_{01}},{\mathbf K}_{2,\zeta_{02}}\big)}{a(n)^{2h_q+1}}\Big) - \eta_0 \leq \lambda_{p-r+q}\Big( \frac{\overline{{\mathbf W}}\big(a(n)2^j,{\mathbf B}_{\delta_0},{\mathbf K}_{1,\zeta_{01}},{\mathbf K}_{2,\zeta_{02}}\big)}{a(n)^{2h_q+1}}\Big).
\end{equation}
So, pick $\delta_0 = \delta_0(n) = 1/n$, $\zeta_{01} = \zeta_{01}(n) = 1/a(n)^{2h_q+1}$ and $\zeta_{02} = \zeta_{02}(n) = 1/a(n)^{h_q+1/2}$. Then, $\|{\mathbf B}_{\delta_0}-{\mathbf B}(2^j)\| \rightarrow 0$, ${\mathbf K}_{1,\zeta_{01}} = O(1)/a(n)^{2h_q+1}$ and ${\mathbf K}_{2,\zeta_{02}} = O(1)/a(n)^{h_q+1/2}$ as $n \rightarrow \infty$. However, for ${\mathbf M}_{n(n_0)} = a^{2h_q+1}\big(\frac{O(1)}{a^{2h_q+1}}+
\frac{{\mathbf P}a^{{\mathbf h}}O(1)}{a^{2h_q+1/2}}+\frac{O^*(1)a^{{\mathbf h}}{\mathbf P}^*}{a^{2h_q+1/2}} \big)$ (see expression \eqref{e:def_Mn}), Corollary \ref{c:PWXP^*+R_asymptotics} implies that
\begin{equation*}%
 \lambda_{p-r+\ell}\Big( \frac{\overline{{\mathbf W}}\big(a2^j,{\mathbf B}_{\delta_0},{\mathbf K}_{1,\zeta_{01}},{\mathbf K}_{2,\zeta_{02}}\big)}{a^{2h_q+1}}\Big)
   \to
\begin{cases}
0, & \ell\in \mathcal I_-;\\
\xi_\ell(2^j), & \ell\in \mathcal I_0;\\
\infty, & \ell\in \mathcal I_+,
\end{cases}
\end{equation*}
as $n \rightarrow \infty$, where $\xi_\ell(2^j)$, $\ell \in \mathcal{I}_0$, are given by \eqref{e:lim_n_a*lambda/a^(2h+1)}. Hence, bearing in mind \eqref{e:eta0} under condition \eqref{e:xiq_distinct}, we arrive at a contradiction with \eqref{e:lambda_p-r+ell=<lambda_p-r+ell-1_contradiction} or \eqref{e:lambda_p-r+ell+1=<lambda_p-r+ell-1-eta0_contradiction}. This establishes \eqref{e:lambda_are_simple} and, hence, $(i)$.

Statement \eqref{e:lambda_p-r+q_is_positive} can be shown by a similar argument by contradiction based on Corollary \ref{c:PWXP^*+R_asymptotics}. This establishes $(ii)$.
\end{proof}

The following mean value theorem-type relation is repeatedly used in the proof of Theorem \ref{t:asympt_normality_lambdap-r+q}.
\begin{lemma}\label{l:mean_value_theorem}
For $m \in \bbN$, let $G: {\mathcal T} \rightarrow \bbR$ be a differentiable function, where ${\mathcal T} \subseteq \bbR^{m}$ is a connected, open set in $\bbR^{m}$. Let $T_0, T_1 \in {\mathcal T}$. Then, there is a vector $\Theta = \big(\theta_{i}\big)_{i= 1,\hdots,m}$ in the segment $\{T \in {\mathcal T}: T = s T_0 + (1-s) T_1, s \in [0,1]\} \subseteq {\mathcal T}$ such that
\begin{equation}\label{e:mean_value_theorem}
G(T_1) - G(T_0) = \sum^{m}_{i=1}\frac{\partial}{\partial t_{i}}G[\Theta]\, \Delta_{i},
\end{equation}
where $\Delta = T_1 - T_0 = \{\Delta_{i}\}_{i = 1,\hdots,m} \in \bbR^{m}$.
\end{lemma}
\begin{proof}
Define the path $\bbR^{m} \ni F(s) = T_0 + s \Delta$, $s \in [0,1]$. Also define the real-valued, composite function $H(s) = G[F(s)]$. Then, by the mean value theorem and the chain rule, there is $\varsigma \in [0,1]$ such that $H(1) - H(0) = H'(\varsigma) = \sum^{m}_{i=1}\frac{\partial}{\partial t_{i}}\, G[F(\varsigma)]\, \Delta_{i}$. This shows \eqref{e:mean_value_theorem}.
\end{proof}

The following lemma is used in the proofs of Proposition \ref{p:conv_w-by-w_rescaled_eigenvalues}/Theorem \ref{t:lim_n_a_times_lambda/a^(2h+1)} and Theorem \ref{t:asympt_normality_lambdap-r+q} (in the former case, implicitly by means of the auxiliary results on matrices of the general form $\widetilde{{\mathbf W}}(a(n)2^j)$).
\begin{lemma}\label{l:|a(n)(-D)W_X,Z(a(n)2j))|=OP(1)}
Fix any $j$ as in \eqref{e:def_j1,jm}. Assume ($A1-A4$) hold. Let $\mathbf{W}_{X,Z}(a(n) 2^j)$ be as in \eqref{e:W=PWXP*+WZ+PWXZ+WXZ*P*}. Then,
\begin{equation}\label{e:|a(n)(-D)W_X,Z(a(n)2j))|=OP(1)}
\|a(n)^{-{({\mathbf H}+\frac{1}{2}{\mathbf I})}} \mathbf{W}_{X,Z}(a(n) 2^j))\| = O_\bbP(1).
\end{equation}
\end{lemma}
\begin{proof}
For simplicity, write $\mathbf W_{Z,X}(a2^j)=\mathbf W_{X,Z}(a2^j)^*$.  For any ${\mathbf u}\in \bbR^p$,
\begin{equation}\label{e:u*W-ZX*W-XZu}
0\leq {\mathbf u}^* \mathbf{W}_{Z,X}(a 2^j) a^{-{({\mathbf H}+\frac{1}{2}{\mathbf I})^*}}  a^{-{({\mathbf H}+\frac{1}{2}{\mathbf I})}} \mathbf{W}_{X,Z}(a 2^j){\mathbf u}
\end{equation}
$$
= \frac{1}{n_{a,j}^2}\sum_{k=1}^{n_{a,j}}\sum_{k'=1}^{n_{a,j}} {\mathbf u}^* D_Z(a2^j,k)D_X(a2^j,k)^*a^{-{({\mathbf H}+\frac{1}{2}{\mathbf I})}^*}a^{-{({\mathbf H}+\frac{1}{2}{\mathbf I})}} D_X(a2^j,k')D_Z(a2^j,k')^*{\mathbf u}
$$
$$
= \frac{1}{n_{a,j}^2}\sum_{k=1}^{n_{a,j}}\sum_{k'=1}^{n_{a,j}} {\mathbf u}^* D_Z(a2^j,k)D_Z(a2^j,k')^*{\mathbf u}\Big\langle a^{-{({\mathbf H}+\frac{1}{2}{\mathbf I})}} D_X(a2^j,k),a^{-{({\mathbf H}+\frac{1}{2}{\mathbf I})}} D_X(a2^j,k')\Big\rangle
$$
$$
\leq \frac{1}{n_{a,j}^2}\Bigg[\sum_{k=1}^{n_{a,j}}\sum_{k'=1}^{n_{a,j}} \bigg( {\mathbf u}^* D_Z(a2^j,k)D_Z(a2^j,k')^*{\mathbf u}\bigg)^2
$$
$$
\times \sum_{\ell=1}^{n_{a,j}}\sum_{\ell'=1}^{n_{a,j}}  \Big\langle a^{-{({\mathbf H}+\frac{1}{2}{\mathbf I})}} D_X(a2^j,\ell),a^{-{({\mathbf H}+\frac{1}{2}{\mathbf I})}} D_X(a2^j,\ell')\Big\rangle^2\Bigg]^{1/2}
$$
$$%
= \frac{1}{n_{a,j}}\sum_{k=1}^{n_{a,j}} \bigg( D_Z(a2^j,k)^*\mathbf u\bigg)^2 \Bigg[\frac{1}{n_{a,j}^2}\sum_{\ell=1}^{n_{a,j}}\sum_{\ell'=1}^{n_{a,j}} \Big\langle a^{-{({\mathbf H}+\frac{1}{2}{\mathbf I})}} D_X(a2^j,\ell),a^{-{({\mathbf H}+\frac{1}{2}{\mathbf I})}} D_X(a2^j,\ell')\Big\rangle^2\Bigg]^{1/2},
$$%
where the second inequality is a consequence of the Cauchy-Schwarz inequality. However,
\begin{equation}\label{e:(1/naj)*sum_(Dz*u)^2}
\frac{1}{n_{a,j}}\sum_{k=1}^{n_{a,j}} \bigg( D_Z(a2^j,k)^*\mathbf u\bigg)^2=\frac{1}{n_{a,j}}\sum_{k=1}^{n_{a,j}} \mathbf u ^* D_Z(a2^j,k) D_Z(a2^j,k)^*\mathbf u   =  \mathbf u^*\mathbf W_Z(a2^j)\mathbf u \leq O_{\bbP}(1),
\end{equation}
where the last equality results from \eqref{e:assumptions_WZ=OP(1)}. Moreover,
$$
\Bigg[\frac{1}{n_{a,j}^2}\sum_{\ell=1}^{n_{a,j}}\sum_{\ell'=1}^{n_{a,j}} \Big\langle a^{-{({\mathbf H}+\frac{1}{2}{\mathbf I})}} D_X(a2^j,\ell),a^{-{({\mathbf H}+\frac{1}{2}{\mathbf I})}} D_X(a2^j,\ell')\Big\rangle^2\Bigg]^{1/2}
$$
$$
 \leq \Bigg[\frac{1}{n_{a,j}^2}\sum_{\ell=1}^{n_{a,j}}\sum_{\ell'=1}^{n_{a,j}}  \Big\|a^{-{({\mathbf H}+\frac{1}{2}{\mathbf I})}} D_X(a2^j,\ell)\Big\|^2  \Big\|a^{-{({\mathbf H}+\frac{1}{2}{\mathbf I})}} D_X(a2^j,\ell')\Big\|^2 \Bigg]^{1/2}
$$
$$
=\frac{1}{n_{a,j}}\sum_{\ell=1}^{n_{a,j}}  \Big\|a^{-{({\mathbf H}+\frac{1}{2}{\mathbf I})}} D_X(a2^j,\ell)\Big\|^2
= \text{tr} \bigg(a^{-{({\mathbf H}+\frac{1}{2}{\mathbf I})}} \mathbf{W}_X(a2^j)  a^{-{({\mathbf H}+\frac{1}{2}{\mathbf I})^*}}\bigg)
$$
\begin{equation}\label{e:r_PH^2_B-hat}
\leq r\bigg\|a^{-{({\mathbf H}+\frac{1}{2}{\mathbf I})}} \mathbf{W}_X(a2^j)  a^{-{({\mathbf H}+\frac{1}{2}{\mathbf I})^*}}\bigg\| \leq r \| \mathbf P_H\|^2 \| \widehat{\mathbf B}_a(2^j)\| = O_\bbP(1),
\end{equation}
where the last equality follows from \eqref{e:sqrt(K)(B^-B)->N(0,sigma^2)} and \eqref{e:<p1,p2>=c12_2}. Thus, in view of \eqref{e:(1/naj)*sum_(Dz*u)^2} and \eqref{e:r_PH^2_B-hat}, by taking $\sup_{\mathbf u \in \mathcal S^{p-1}}$ in \eqref{e:u*W-ZX*W-XZu}, we obtain the bound $\|a(n)^{-{({\mathbf H}+\frac{1}{2}{\mathbf I})}} \mathbf{W}_{X,Z}(a(n) 2^j)\| \leq O_{\mathbb P}(1)$. This establishes \eqref{e:|a(n)(-D)W_X,Z(a(n)2j))|=OP(1)}.

\end{proof}

\section{Proofs: Section \ref{s:examples}}\label{s:proofs_examples}

In this section, we provide the proofs of statements made in Section \ref{s:examples}. The results are organized into two subsections, corresponding to Gaussian and non-Gaussian examples.

Hereinafter, for a stochastic process $Z$, we define the differencing operator by means of $\Delta Z(t):= Z(t)-Z(t-1)$. Likewise, the $k$--th order difference $\Delta^k Z=\Delta(\Delta^{k-1} Z)$ is defined iteratively.

Consider the following definition.
\begin{definition}\label{def:fractional_process}
For each $p \in \bbN$, let $Z_p = Z = \{Z(t)\}_{t \in \bbZ}$ be a second-order, $p$-variate process. We say the sequence of processes $\{Z_p\}_{p \in \bbN}$ has \textit{maximal memory parameter $d$ uniformly in $p$} (or \textit{maximal order $d$}, for short) if
\begin{itemize}
\item[$(i)$] for a fixed integer $k_0\geq 0$ and each $p$, the $k_0$--th differenced process $\Delta^{k_0} Z$ is (weakly) stationary with spectral density
\begin{equation}\label{e:M(d)_0}
|1-e^{\imag x}|^{2k_0}\mathfrak g_p(x) \in {\mathcal S}_{\geq 0}(p,\bbC), \quad x \in [-\pi,\pi);
\end{equation}
\item[$(ii)$] there is a smallest $d\geq 0$ satisfying  $d-k_0\in(-1/2,1/2)$ such that
\begin{equation}\label{e:M(d)_1}
\sup_{p \in \bbN} \hspace{2mm}\textnormal{ess}\sup_{|x|\leq \pi} |x|^{2d}\| \mathfrak g_{p}(x)\| <\infty.%
\end{equation}
\end{itemize}
\end{definition}
For each $p$, the parameter $d$ in \eqref{e:M(d)_1} expresses the statement $\| \mathfrak g_{p}(x)\|=O(|x|^{-2d})$ as $x\to 0$. Namely, it describes the largest possible scaling law in the behavior of $Z$. In particular, the condition $\textnormal{ess}\sup_{|x|\leq \pi} |x|^{2d}\| \mathfrak g_{p}(x)\| <\infty$ includes spectral densities for multivariate long memory processes whose memory exponents are no greater than $d$ (cf.\ Kechagias and Pipiras \cite{kechagias:pipiras:2015:ident}). Further note that, for each $p$, $\mathfrak g_p$ is typically called the \textit{generalized spectral density} (of $Z$). This terminology is used in the sequel.

\subsection{Section \ref{s:Gaussian}: proofs and auxiliary results}\label{s:proofs_and_auxiliary_results_Gaussian}

In the following proposition, we show that all classes of examples described in Section \ref{s:Gaussian} satisfy ($A2$) or ($A3$).
\begin{proposition}\label{p:examples}
Suppose assumptions $(W1-W3)$ are in place. Then, the following claims hold.
\begin{itemize}
\item [$(i)$] Fix $N_\psi\geq 2$ and let $\vartheta_0 := \min\{2h_1 + 1,h_1 + 3/2\} > 1$. Consider any dyadic sequence $\{a(n)\}_{n \in \bbN}$ such that
    \begin{equation}\label{e:condition_(A4)_ofBm}
    a(n)\leq \frac{n}{2^{j_m}}, \quad \frac{a(n)}{n} + \frac{n}{a(n)^{\vartheta_0}} \rightarrow \infty, \quad n \rightarrow \infty.
    \end{equation}
 Then, assumption $(A3)$ is satisfied for $X = \{X(t)\}_{t \in \bbZ}$ as in \eqref{e:oss_ofBm}.
    \item [$(ii)$]  For $N_\psi \geq 1$ and under assumption $(A4)$, assumption $(A2)$ is satisfied for $Z = \{Z(t)\}_{t \in \bbZ}$ as in \eqref{e:Z(t)=ARMA}.
\item [$(iii)$]  For $d$ as in \eqref{e:d<(3/2)*(h1+1/2)}, fix $N_\psi\geq 2$ satisfying $N_\psi>d+\frac12$. Suppose $\frac{p(n)}{n/a(n)}\to c>0$ (cf.\ \eqref{e:p(n),a(n)_conditions}) and, for $\alpha$ as in \eqref{e:psihat_is_slower_than_a_power_function}, assume that ${n a(n)^{-2\alpha}=O(1)}$. Also, fix
    \begin{equation}\label{e:varepsilon_in_(0,(3/2)(h1+1/2)-d)}
     0<\varepsilon <  \min\Big\{\frac{3}{2}(h_1+1/2), \, 2h_1+1/2\Big\}-d.
     \end{equation}
     Further define
    \begin{equation}\label{E:b=h1+1/2-epsilon}
    b = h_1 + 1/2 - \varepsilon.
    \end{equation}
    Then, there exist sequences $\{a(n)\}_{n \in \bbN}$ and $\{p(n)\}_{n \in \bbN}$ such that assumptions ($A2$), ($A3$) and ($A4$) are satisfied (the latter, for large enough $n$) for $Z = \{Z(t)\}_{t \in \bbZ}$ and $X = \{X(t)\}_{t \in \bbZ}$ as in \eqref{e:Z(t)=residual_factor_model}. In addition, the eigenvalues of the  scaling matrix $\widetilde{{\mathbf H}}$ (cf.\ \eqref{e:H=PHdiag(h1,...,hn)P^(-1)H}) for $X$ are given by \eqref{e:h-tilde_q=h_q-b>1/2}.
\end{itemize}
\end{proposition}
\begin{proof}
We first show ($i$). Fix $N_\psi\geq 2$ and let $\{a(n)\}_{n \in \bbN}$ be a dyadic sequence as in \eqref{e:condition_(A4)_ofBm}. Then, the associated auxiliary random matrix $\widehat{{\mathbf B}}_a(2^j)$ as in \eqref{e:B-hat_a(2^j)} satisfies assumption ($A3$) as a consequence of Theorem 3.1, Lemma C.2 (extended to dimension $r$) and Proposition 3.1 in Abry and Didier \cite{abry:didier:2018:dim2}.\vspace{2mm}

To show $(ii)$, consider $Z$ as in \eqref{e:Z(t)=ARMA}. In this case, $Z$ has spectral density
$$
\mathfrak g_p(x)= \frac{1}{2\pi}\bigg(\sum_{\ell \in \bbZ }\mathbf A_\ell(p) e^{i\ell x}\bigg) \Sigma_\varepsilon(p) \bigg(\sum_{\ell \in \bbZ }\mathbf A_\ell(p) e^{i\ell x}\bigg)^*.
$$
Thus,
$$
\textnormal{ess}\sup_{|x|\leq \pi}\|\mathfrak g_p(x)\| \leq (2\pi)^{-1}\|\Sigma_{\boldsymbol \varepsilon}(p)\| \bigg(\sum_{\ell \in \bbZ }\|\mathbf A_\ell(p)\| \bigg)^2.
$$
Together with condition \eqref{e:sup-p_Sigma(p)<infty}, this implies that $Z$ is of maximal order $d=0$ as in Definition \ref{def:fractional_process}. Hence, by Lemma \ref{l:assumptions_WZ=OP(1)_hold}, assumption ($A2$) holds under $(A4)$.\vspace{2mm}

We now turn to ($iii$). Consider $Z$ as in \eqref{e:Z(t)=residual_factor_model}. We break up the proof into two cases, based on the magnitude
\begin{equation}\label{e:eta-0}
\eta_0 := 2(d-h_1 + \varepsilon)
\end{equation}
(\textbf{n.b.}: relation \eqref{e:wZ_bdd_factormodel_example} below shows that $2(d-h_1 + \varepsilon)$ naturally appears in the expression for the order of magnitude of $\|\mathbf W_Z(a(n)2^j)\|$).

First, assume $\eta_0 \leq 1$. For the fixed $c > 0$, pick any sequences $\{a(n)\}_{n \in \bbN}$ and $\{p(n)\}_{n \in \bbN}$ such that assumption ($A4$)  and relation \eqref{e:condition_(A4)_ofBm} are satisfied. Since $\mathcal Z$ is of maximal order $d$,  and $N_\psi>d+1/2$,  then Lemma \ref{l:assumptions_WZ=OP(1)_hold} gives $\|\mathbf W_{\mathcal Z}(a(n)2^j)\|=O_{\bbP}(a(n)^{2d})$. By \eqref{E:b=h1+1/2-epsilon} and $n/a(n)=O(p(n))$ (see \eqref{e:p(n),a(n)_conditions}), we obtain
$$
\|\mathbf W_{Z}(a(n)2^j)\|=\frac{\|\mathbf W_{\mathcal Z }(a(n)2^j)\|}{p(n)\hspace{1mm}a(n)^{2b}} = \frac{O_\bbP(a(n)^{2d})}{(n/a(n))\hspace{1mm}a(n)^{2b}}
$$
\begin{equation}\label{e:wZ_bdd_factormodel_example}
= O_{\bbP}( a(n)^{2(d-b)+1}/n)=O_{\bbP}( a(n)^{2(d-h_1+\varepsilon)}/n )= O_{\bbP}(1).
\end{equation}
In \eqref{e:wZ_bdd_factormodel_example}, we used the fact that
\begin{equation}\label{e:strongfactor_scalingassumption}
\frac{a(n)^{2(d-h_1+\varepsilon)}}n =O(1).
\end{equation}
In addition, by a similar reasoning,
\begin{equation}\label{e:||EW_Z||=O_P(1)_factor}
\|\bbE \mathbf W_{Z}(a(n)2^j)\| = O_{\bbP}(1).
\end{equation}
Hence, assumption ($A2$) holds.

Moreover, for $b$ as in \eqref{E:b=h1+1/2-epsilon}, let $\widetilde{{\mathbf H}} = {\mathbf H} - b \hspace{0.5mm}{\mathbf I}$. For ${\mathbf W}_X(a(n)2^j)$ as in \eqref{e:W_X(a2^j)_factor_model_example}, define the auxiliary random matrix
\begin{equation}\label{e:B-hat_a(2^j)-X}
\widehat{{\mathbf B}}_a(2^j)_{X} = {\mathbf P}_H^{-1}\big\{a(n)^{-\widetilde{{\mathbf H}}-(1/2){\mathbf I}}\hspace{1mm}{\mathbf W}_X(a(n)2^j)\hspace{1mm}a(n)^{-\widetilde{{\mathbf H}}^*-(1/2){\mathbf I}} \big\}({\mathbf P}_H^*)^{-1}\in {\mathcal S}_{\geq 0}(r,\bbR).
\end{equation}
Also, let $\widehat{{\mathbf B}}_a(2^j)_{{\mathcal X}}$ be the auxiliary random matrix associated with the process ${\mathcal X}$ as in \eqref{e:ofBm_factor_model_example}. Then, in view of \eqref{e:Z(t)=residual_factor_model}, relation \eqref{e:h-tilde_q=h_q-b>1/2} holds and
\begin{equation}\label{e:B-hat_a(2^j)-X=B-hat_a(2^j)-mathcalX}
\widehat{{\mathbf B}}_a(2^j)_{X}= \widehat{{\mathbf B}}_a(2^j)_{{\mathcal X}}.
\end{equation}
Thus, since
\begin{equation}\label{eN-psi>=2}
N_{\psi}\geq 2,
\end{equation}
part $(i)$ implies that ($A3$) holds (in particular, using the same rate of convergence $\sqrt{n_{a,j}}$).

Alternatively, assume
\begin{equation}\label{e:eta-0}
\eta_0 > 1.
\end{equation}
Observe that \eqref{e:varepsilon_in_(0,(3/2)(h1+1/2)-d)} further gives
\begin{equation}\label{e:eta0_upperbound}
\eta_0<\min\big\{h_1+3/2, ~2h_1+1\big\}.
\end{equation}
For the fixed $c > 0$, further define
\begin{equation}\label{e:a(n)=2^(1/eta-0*log2n)}
a(n) := 2^{\lfloor \frac{1}{\eta_0}\log_2 n\rfloor}, \quad p(n) := \Big\lfloor \frac{c \hspace{0.4mm}n}{a(n)}\Big\rfloor.
\end{equation}
We claim that there are constants $0 < C_1 < C_2 < \infty $ such that
\begin{equation}\label{e:C1=<a(n)^eta_0/n=<C2}
C_1 \leq \frac{a(n)^{\eta_0}}{n} \leq C_2, \quad n \in \bbN.
\end{equation}
In fact, \eqref{e:C1=<a(n)^eta_0/n=<C2} is a consequence of exponentiating the inequalities
$$
-\eta_0 \leq \eta_0\Big(  \Big\lfloor \frac{1}{\eta_0}\log_2 n \Big\rfloor - \frac{1}{\eta_0}\log_2 n \Big) \leq \eta_0, \quad n \in \bbN.
$$
As a consequence of \eqref{e:C1=<a(n)^eta_0/n=<C2}, relation \eqref{e:strongfactor_scalingassumption} holds. By the same argument, \eqref{e:wZ_bdd_factormodel_example} also holds. Furthermore, a similar reasoning establishes \eqref{e:||EW_Z||=O_P(1)_factor}. In other words, ($A2$) is satisfied, as claimed.

Moreover, for $a(n)$ and $p(n)$ as in \eqref{e:a(n)=2^(1/eta-0*log2n)}, $\lim_{n \rightarrow \infty} \frac{p(n)\hspace{0.2mm}a(n)}{n} = c$. Also, by relations \eqref{e:varepsilon_in_(0,(3/2)(h1+1/2)-d)}, \eqref{e:eta-0}, \eqref{e:eta0_upperbound} and \eqref{e:C1=<a(n)^eta_0/n=<C2}, $\frac{n}{a(n)^{h_1 + 3/2}} \leq \frac{C^{-1}_1}{a(n)^{h_1 + 3/2 - \eta_0}}\rightarrow 0$ as $n \rightarrow \infty$. Therefore, ($A4$) holds for large enough $n$, as claimed. Also, for $\vartheta_0$ as in \eqref{e:condition_(A4)_ofBm},  expression \eqref{e:eta0_upperbound} implies $\eta_0< \vartheta_0$, giving
$$
\frac{n}{a(n)^{\vartheta_0}} \leq \frac{n}{a(n)^{\eta_0}}  a(n)^{\eta_0-\vartheta_0} \leq C a(n)^{\eta_0-\vartheta_0} \to 0, \quad n \rightarrow \infty.
$$
Thus, for $\widehat{{\mathbf B}}_a(2^j)_{X}$ as in \eqref{e:B-hat_a(2^j)-X}, again relations \eqref{e:B-hat_a(2^j)-X=B-hat_a(2^j)-mathcalX} and \eqref{eN-psi>=2} combined with part $(i)$ show that assumption $(A3)$ is satisfied. Furthermore, as in the previous case $\eta_0 \leq 1$, relation \eqref{e:h-tilde_q=h_q-b>1/2} holds for $\eta_0 > 1$. This concludes the proof of $(iii)$. $\Box$\\
\end{proof}

In the remainder of this section, we state or establish all the auxiliary results needed in the proof of Proposition \ref{p:examples}, namely, Lemmas \ref{l:assumptions_WZ=OP(1)_hold}--\ref{l:lemma_laurent_massart}.

In the following lemma, we provide a bound on $\|{\mathbf W}_{Z}(a(n)2^j)\|$ and $\|\bbE {\mathbf W}_{Z}(a(n)2^j)\|$  under conditions on the underlying noise process $Z$.
\begin{lemma}\label{l:assumptions_WZ=OP(1)_hold}
Suppose assumptions $(W1-W3)$ and $(A4)$ are in place. For each $p = p(n)$, let $ Z = \{Z(t)\}_{t \in \bbZ}$ be a Gaussian, $p$-variate stochastic processes satisfying Definition \ref{def:fractional_process}. If, in addition, $N_\psi> d + \frac{1}{2}$ and $n a(n)^{-2\alpha}=O(1)$, then
\begin{equation}\label{e:assumptions_WZ=OP(1)_hold}
a(n)^{-2d}\|\mathbf W_{Z}(a(n)2^j)\| = O_{\bbP}(1), \quad\textnormal{and}\quad a(n)^{-2d}\|\bbE \mathbf W_{Z}(a(n)2^j)\| = O(1).
\end{equation}
\end{lemma}
\begin{proof}
The second relation in \eqref{e:assumptions_WZ=OP(1)_hold} is a consequence of \eqref{e:norm_of_EW_bdd} (see Lemma \ref{l:Dzq_bound}).

So, we show the first relation in \eqref{e:assumptions_WZ=OP(1)_hold}. Define
\begin{equation}
\mathcal H_{2^j}(x) := \sum_{\ell \in \bbZ} h_{j,\ell}e^{-\imag x \ell}.
\end{equation}
By relation (16) in Moulines et al.~\cite{moulines:roueff:taqqu:2007:JTSA}, p.\ 161, we can express $\mathcal H_{2^j}(x) = (1- e^{- \imag x})^{N_{\psi}}\widetilde{\mathcal H}_{2^j}(x)$, where $\widetilde{\mathcal H}_{2^j}(x)$ is a trigonometric polynomial. Since $N_\psi \geq d-1/2$, expression (17) in Moulines et al.~\cite{moulines:roueff:taqqu:2007:JTSA} and the fact that $Z$ is Gaussian imply that the wavelet coefficients $D_Z(2^j,k)$ are well defined a.s., are Gaussian and satisfy
$$
\bbE D_Z(2^j,k)D_Z(2^{j'},k')^*=\int^{\pi}_{-\pi} e^{\imag x (2^j k - 2^{j'} k')} \widetilde {\mathcal H}_{2^j}(x)\overline{\widetilde {\mathcal H}_{2^{j'}}(x)}|1-e^{\imag x}|^{2N_{\psi}}\mathfrak g_p(x) \hspace{0.5mm}dx
$$
\begin{equation}\label{e:EDZDZ*}
= \int^{\pi}_{-\pi} e^{\imag x (2^j k - 2^{j'} k')} {\mathcal H}_{2^j}(x)\overline{{\mathcal H}_{2^{j'}}(x)}
\hspace{0.5mm}\mathfrak{g}_p(x) \hspace{0.5mm}dx.
\end{equation}
Thus, the wavelet random matrices ${\mathbf W}_{Z}(a(n)2^j)$ are well defined a.s. Then, the claim is a consequence of Lemma \ref{l:W_Z_above<exp} (see expression \eqref{e:lambda-p((2^ja(n))^(-2d)W-Z(a(n)2^j))<C}). $\Box$\\
 \end{proof}

We now set up some notation for the next lemma. For a fixed $\mathbf u\in \mathcal S^{p-1}$, let
\begin{equation}\label{e:def_Vn_(u)}
V_n(\mathbf u):=\Big( D_Z(a2^j,1)^*{\mathbf u}, D_Z(a2^j,2)^*{\mathbf u}, \hdots, D_Z(a2^j,n_{a,j})^*{\mathbf u}\Big)^*
\end{equation}
be the vector of available wavelet coefficients of $Z$ at scale $a(n)2^j$ projected in the direction $\mathbf u$. Still for $\mathbf u\in \mathcal S^{p-1}$, let
\begin{equation}\label{e:Gamma-n(u)}
\mathbf \Gamma_n(\mathbf u)=\bbE V_n(\mathbf u)V_n(\mathbf u)^*.
\end{equation}
Also, define
\begin{equation}\label{e:beta}
\beta_0:=\limsup_{n \rightarrow \infty}   a(n)^{-2d}  \sup_{\|\mathbf u\|=1}\| \boldsymbol \Gamma_n(\mathbf u)\|.
\end{equation}
Note that, by Lemma \ref{l:Dzq_bound}, $\beta_0 < \infty$.
In the following lemma, we establish a concentration inequality for the norm $\|(a(n)2^j)^{-2d}{\mathbf W}_{Z}(a(n)2^j)\|= \lambda_p\big((a(n)2^j)^{-2d}{\mathbf W}_{Z}(a(n)2^j)\big)$. The proof follows an $\varepsilon$-net argument, involving steps of approximation, concentration and union bound (cf.\ Vershynin \cite{vershynin:2018}, Sections 4.4 and 4.6, or Lugosi \cite{lugosi:2017}, pp.\ 13--14).
\begin{lemma} \label{l:W_Z_above<exp}
Suppose the assumptions of Lemma \ref{l:assumptions_WZ=OP(1)_hold}  hold. Let
\begin{equation}\label{e:m(a(n)2j)}
m(a(n)2^j) =  \| \bbE {\mathbf W}_{Z}(a(n)2^j)\|.
\end{equation}
With $\beta_0$ as in \eqref{e:beta}, take any $\beta>\beta_0$, and let
\begin{equation}\label{e:t>m(a(n)2j)}
t \geq 5\beta>5\beta_0.%
\end{equation}
Then, for large enough $n \in \bbN$,
$$
\bbP\Big((a(n)2^j)^{-2d} \big(\| {\mathbf W}_{Z}(a(n)2^j)\|-m(a(n)2^j)\big) >  t \Big)
$$
\begin{equation}\label{e:W_Z_upr_bnd_finite_samp}
\leq\exp\left\{p(n)\log9 - \frac{n}{ a(n)2^j}   \left(\frac{t- \beta}{ 4\beta  } - \sqrt{\frac{t- \beta}{ 4\beta}}\right) \right\}.
\end{equation}
In particular, let $C$ be any constant satisfying
\begin{equation}\label{e:C_bound}
C> \beta \bigg(2 + \left(1+\sqrt{1 + 2^{j+2}c\log 9}\right)^2\bigg),
\end{equation}
where $c \geq 0$ is as in \eqref{e:p(n),a(n)_conditions}. Then, as $n \rightarrow \infty$,
\begin{equation}\label{e:lambda-p((2^ja(n))^(-2d)W-Z(a(n)2^j))<C}
\lambda_p\big((2^ja(n))^{-2d}{\mathbf W}_{Z}(a(n)2^j)\big) < C
\end{equation}
with probability going to 1.
\end{lemma}
\begin{proof}

So, fix a small $\varepsilon > 0$. Also fix ${\mathbf u} \in \mathcal{S}^{p-1}$ and let
\begin{equation}\label{e:mathcalDz}
{\mathcal D}^*_{Z,n} = \frac{1}{\sqrt{\mathbf u^* \bbE \mathbf W_Z(a2^j)\mathbf u}}\Big( D_Z(a2^j,1)^*{\mathbf u}, D_Z(a2^j,2)^*{\mathbf u}, \hdots, D_Z(a2^j,n_{a,j})^*{\mathbf u}\Big) \in \bbR^{n_{a,j}}.
\end{equation}
Write
$$
{\boldsymbol \Sigma}_{{\mathcal D},n} = \bbE {\mathcal D}_{Z,n}{\mathcal D}^*_{Z,n}
$$
and consider its spectral decomposition ${\boldsymbol \Sigma}_{{\mathcal D},n} = \mathbf O_{n}\mathbf \Lambda_{{\mathcal D},n}\mathbf O^*_{n}$ for an orthogonal matrix $\mathbf O_{n}$. Recast
$$
\frac{1}{n_{a,j}}{\mathcal D}^*_{Z,n}{\mathcal D}_{Z,n} \stackrel{d}= \frac{1}{n_{a,j}} {\mathbf Z}^*_{n}\boldsymbol \Sigma_{{\mathcal D},n}{\mathbf Z}_{n}= \frac{1}{n_{a,j}} {\mathbf Z}^*_n \mathbf O_{n}\mathbf \Lambda_{{\mathcal D},n}\mathbf O^*_{n}{\mathbf Z}_{n} \stackrel{d}= \frac{1}{n_{a,j}} {\mathbf Z}^*_{n}\mathbf\Lambda_{{\mathcal D},n}{\mathbf Z}_{n} =: \sum^{n_{a,j}}_{k=1}\eta_{k,n}Z^2_k,
$$
where $\mathbf Z_n = (Z_1,\ldots,Z_{n_{a,j}})^*$ is a vector of i.i.d.\ standard normal random variables. Let
\begin{equation}\label{e:eta-vec}
{\boldsymbol \eta} = {\boldsymbol \eta}_n = \Big( \eta_{1,n}, \hdots, \eta_{n_{a,j},n}  \Big)
\end{equation}
be the vector of eigenvalues of the deterministic matrix {$\frac{1}{n_{a,j}}\mathbf \Sigma_{{\mathcal D},n}$}. Note that, by the stationarity of $\{ D_Z(2^j,k)\}_{k\in\bbZ}$ (see \eqref{e:EDZDZ*}), %
$$
 \mathbf u^* \bbE \mathbf W_Z(a2^j)\mathbf u =\frac{1}{n_{a,j}}\sum_{k=1}^{n_{a,j}}  \bbE\big(\mathbf u^*D_Z(a2^j,k)D_Z(a2^j,k)^*\mathbf u\big)=   \bbE\big( D_Z(a2^j,1)^*{\mathbf u} \big)^2. %
$$
Hence,
$$
\frac{1}{n_{a,j}}\bbE {\mathcal D}^*_{Z,n}{\mathcal D}_{Z,n} =\frac{1}{n_{a,j}{\mathbf u^* \bbE \mathbf W_Z(a2^j)\mathbf u}} \cdot n_{a,j} \bbE \big( D_Z(a2^j,1)^*{\mathbf u}\big)^2=1,
$$
whence
\begin{equation}\label{e:eta-n=<2pibeta/naj*m(a2j)}
\sum_{i=1}^{n_{a,j}} \eta_{i,n} =1%
\end{equation}
On the other hand, let
$$
m_n(\mathbf u):= (a2^j)^{-2d} {\mathbf u}^*\bbE {\mathbf W}_{Z}(a2^j) {\mathbf u}.
$$
In particular,
\begin{equation}\label{e:sup_m-n(u)}
\sup_{\|\mathbf u\|=1}m_n(\mathbf u)=\frac{m(a2^j)}{(a2^j)^{2d}}.
\end{equation}
Recall that $\mathbf \Gamma_n(\mathbf u)$ is defined by \eqref{e:Gamma-n(u)}. Then, by Lemma \ref{l:Dzq_bound}, for all large $n$,
\begin{equation}\label{e:eta_inf_bound}
\| \boldsymbol \eta\|_\infty=\frac{1}{n_{a,j}}\| {\boldsymbol \Sigma}_{{\mathcal D},n}\| = \frac{ \|\mathbf \Gamma_n(\mathbf u)\|  }{n_{a,j}{\mathbf u}^*\bbE {\mathbf W}_{Z}(a2^j) {\mathbf u}}\leq  \frac{ \beta  a^{2d}}{n_{a,j}{\mathbf u}^*\bbE {\mathbf W}_{Z}(a2^j) {\mathbf u}} = \frac{\beta }{n_{a,j}m_n(\mathbf u)}.
\end{equation}
So, for $\beta$ as in \eqref{e:t>m(a(n)2j)}, consider any $s>0$ such that
\begin{equation}\label{e:s>4pi(beta+epsilon)}
s>  2\beta.
\end{equation}
Define
$$
x = \left(\frac{-\|\boldsymbol\eta\|_2 + \sqrt{\|\boldsymbol\eta\|_2^2 + 2 \|\boldsymbol\eta\|_\infty s/m_n(\mathbf u)}}{2\|\boldsymbol\eta\|_\infty}\right)^2.%
$$
Then, $2 \|{\boldsymbol \eta}\|_2 \hspace{1mm}\sqrt{x}+ 2 \|{\boldsymbol \eta}\|_{\infty} \hspace{1mm}x = s/m_n(\mathbf u)$.
Thus, using \eqref{e:sup_m-n(u)} and \eqref{e:eta-n=<2pibeta/naj*m(a2j)},
$$
\bbP\bigg((a2^j)^{-2d}\Big({\mathbf u}^* {\mathbf W}_{Z}(a2^j){\mathbf u} -m(a2^j)\Big) > s\bigg)
$$
$$
\leq \bbP\bigg((a2^j)^{-2d}\Big({\mathbf u}^* {\mathbf W}_{Z}(a2^j){\mathbf u} - {\mathbf u}^*\bbE {\mathbf W}_{Z}(a2^j) {\mathbf u}\Big)  > s\bigg)
$$
$$
=\bbP\Big(\frac{{\mathbf u}^* {\mathbf W}_{Z}(a2^j){\mathbf u}}{{\mathbf u}^*\bbE {\mathbf W}_{Z}(a2^j)\mathbf {\mathbf u}} - 1  > \frac{s}{(a2^j)^{-2d}{\mathbf u}^*\bbE {\mathbf W}_{Z}(a2^j)\mathbf {\mathbf u}} \Big) =  \bbP\Big(\sum^{n_{a,j}}_{k=1}\eta_{k,n}(Z^2_{k}-1) \geq \frac{s}{m_n(\mathbf u)}  \Big)
$$
\begin{equation}\label{e:P(u*WZu-m>s)=<exp(-x)}
=\bbP\Big(\sum^{n_{a,j}}_{k=1}\eta_{k,n}(Z^2_{k}-1)\geq 2 \|{\boldsymbol \eta}\|_2 \hspace{1mm}\sqrt{x}+ 2 \|{\boldsymbol \eta}\|_{\infty} \hspace{1mm}x \Big) \leq \exp\{-x\},
\end{equation}
where the last inequality is a consequence of Lemma \ref{l:lemma_laurent_massart}. %

Now, note that $\|\boldsymbol\eta\|_2\leq \sqrt{n_{a,j}}\|\boldsymbol \eta\|_\infty$ and that $\sqrt{\|\boldsymbol\eta\|_2^2 +  2 \|\boldsymbol\eta\|_\infty s/m_n(\mathbf u)}\leq \|\boldsymbol\eta\|_2 + \sqrt{2 \|\boldsymbol\eta\|_\infty s/m_n(\mathbf u)}$. We obtain
$$x=\frac{1}{2\|\boldsymbol \eta\|_\infty^2}\Bigg( \|\boldsymbol\eta\|_2^2 +  \frac{\|\boldsymbol\eta\|_\infty s}{m_n(\mathbf u)} -\|\boldsymbol\eta\|_2 \sqrt{\|\boldsymbol\eta\|_2^2 +  \frac{2 \|\boldsymbol\eta\|_\infty s}{m_n(\mathbf u)}}\Bigg)$$
$$\geq \frac{  1}{2\|\boldsymbol\eta\|_\infty^2}\Bigg(\frac{\|\boldsymbol\eta\|_\infty s}{m_n(\mathbf u)} -\|\boldsymbol\eta\|_2 \sqrt{ \frac{2 \|\boldsymbol\eta\|_\infty s}{m_n(\mathbf u)}}\Bigg)
$$
$$
\geq \frac{1}{\|\boldsymbol\eta\|_\infty} \left(\frac{s}{2m_n(\mathbf u)} - \sqrt{n_{a,j}\frac{s\|\boldsymbol \eta\|_\infty}{2m_n(\mathbf u)}}\right) $$
$$
\geq  \frac{1}{\|\boldsymbol\eta\|_\infty} \left(\frac{s}{2m_n(\mathbf u)} - \sqrt{\frac{s \beta}{2m^2_n(\mathbf u)}}\right) = \frac{1}{\|\boldsymbol\eta\|_\infty m_n(\mathbf u)} \left(\frac{s}{2} - \sqrt{\frac{s \beta}{2}}\right)
$$
\begin{equation}\label{e:x>=naj(s/pibeta-sqrt(s/pibeta))}
\geq \frac{n_{a,j}}{\beta} \left(\frac{s}{2} - \sqrt{\frac{s \beta}{2}}\right).
\end{equation}
In \eqref{e:x>=naj(s/pibeta-sqrt(s/pibeta))}, the third inequality follows from relation \eqref{e:eta_inf_bound}, and the last inequality holds since $s/2- \sqrt{ {s \beta}/{2}}>0$ due to \eqref{e:s>4pi(beta+epsilon)}. Recall that $m(a2^j)$ is given by \eqref{e:m(a(n)2j)}. From \eqref{e:P(u*WZu-m>s)=<exp(-x)} and \eqref{e:x>=naj(s/pibeta-sqrt(s/pibeta))}, we arrive at
\begin{equation}\label{e:lambda_p(W_z)_bound}
\bbP\Big((a2^j)^{-2d}\Big({\mathbf u}^* {\mathbf W}_{Z}(a2^j){\mathbf u} - m(a2^j) \Big)   > s\Big)
\leq \exp\left\{-  n_{a,j} \left(\frac{s}{2\beta } - \sqrt{\frac{s}{2\beta}}\right)\right\}.
\end{equation}
We now appeal to the same argument as in Lugosi \cite{lugosi:2017}, pp.\ 13--14. In fact, let $\mathcal{N}$ be a $1/4$-net of the unit sphere. Then, it can be shown that $\text{card}(\mathcal{N})\leq 9^p$. By an application of the Cauchy-Schwarz inequality, we obtain $\|{\mathbf W}_Z(a2^j)\|\leq 2 \max_{\mathbf u\in\mathcal{N}}{\mathbf u}^*{\mathbf W}_Z(a2^j){\mathbf u}$. Let $t$ be as in $\eqref{e:t>m(a(n)2j)}$. Thus, by the union bound,
$$
 \bbP\Big((a2^j)^{-2d}\big(\|\mathbf {W}_{Z}(a2^j)\|  -m(a2^j)\big)  > t \Big) \leq 9^p \max_{{\mathbf u} \in {\mathcal N}} \bbP\Big({\mathbf u}^*{\mathbf W}_{Z}(a2^j){\mathbf u} > \frac{(a2^j)^{2d}t+ m(a2^j)}{2}\Big)
$$
\begin{equation}\label{e:net_bound_W_Z}
=9^p \max_{{\mathbf u} \in {\mathcal N}} \bbP\Big((a2^j)^{-2d}\big({\mathbf u}^*{\mathbf W}_{Z}(a2^j){\mathbf u} - m(a2^j) \big)> \frac{t-(a2^j)^{-2d}m(a2^j)}{2}\Big).
\end{equation}
Recall expression \eqref{e:m(a(n)2j)}. Note that, by relation \eqref{e:norm_of_EW_bdd} of Lemma \ref{l:Dzq_bound}, since $\beta>\beta_0$, for all large $n$,
\begin{equation}\label{e:m(a2j)->gamma}
(a2^j)^{-2d}m(a2^j)<\beta.
\end{equation}
In view of \eqref{e:t>m(a(n)2j)}, the inequality $\frac{t - (a2^j)^{-2d}m(a2^j)}{2} > \frac{5\beta -\beta }{2} =2\beta$ holds for all large $n$; i.e., \eqref{e:s>4pi(beta+epsilon)} holds with  $s=\frac{t - (a2^j)^{-2d}m(a2^j)}{2}$. Thus, applying \eqref{e:lambda_p(W_z)_bound} to \eqref{e:net_bound_W_Z}, we get
$$
\bbP\Big((a2^j)^{-2d}(\|\mathbf {W}_{Z}(a2^j)\|  -m(a2^j))  > t \Big)
$$
$$
\leq9^p \exp\left\{-  n_{a,j} \left(\frac{t- (a2^j)^{-2d}m(a2^j)}{4\beta } - \sqrt{\frac{t- (a2^j)^{-2d}m(a2^j)}{4\beta}}\right)\right\}.
$$
\begin{equation}\label{e:P((a2^j)^(-2d)W-Z(a2^j)|-m(a2^j)>t)}
\leq 9^p \exp\left\{-  n_{a,j} \left(\frac{t- \beta}{4\beta } - \sqrt{\frac{t- \beta}{4\beta}}\right)\right\}.
\end{equation}
In \eqref{e:P((a2^j)^(-2d)W-Z(a2^j)|-m(a2^j)>t)}, we used that $v\mapsto v - \sqrt v$ is increasing for all $v>1/4$, and that $\frac{t- (a2^j)^{-2d}m(a2^j)}{4\beta }> \frac{2\beta}{2\beta}>1$. By combining this with \eqref{e:lambda_p(W_z)_bound} and \eqref{e:net_bound_W_Z}, we arrive at \eqref{e:W_Z_upr_bnd_finite_samp}.

To show the statement regarding \eqref{e:C_bound},  recast
\begin{equation}\label{e:def_b(n)}
\exp\left\{p(n)\log9- \frac{n}{a(n)2^j} \left(\frac{t- \beta}{4\beta } - \sqrt{\frac{t- \beta}{4\beta}}\right)\right\}%
 =: \exp\left\{\frac{n}{a(n)} \hspace{0.5mm}b(n)\right\}.
\end{equation}
By \eqref{e:p(n),a(n)_conditions} and \eqref{e:m(a2j)->gamma},
$$
b(n) = \frac{p(n)}{n/a(n)}\log9- \frac{1}{2^j}\left(\frac{t- \beta}{4\beta } - \sqrt{\frac{t- \beta}{4\beta}}\right)%
$$
$$\rightarrow c\log 9 - \frac{1}{2^j} \left(\frac{t- \beta}{4\beta } - \sqrt{\frac{t- \beta}{4\beta}}\right)%
=: L(t), \quad n \rightarrow \infty.
$$
Thus, if $t> t_*:=\beta \bigg(1 + \left(1+\sqrt{1 + 2^{j+2}c\log 9} \right)^2\bigg)$ , %
 then $L(t)<0$, implying $b(n)<0$ for all large $n$. So, for $C$ satisfying \eqref{e:C_bound}, we have
 $$
 C>\beta + \beta\bigg(1 + \left(1+\sqrt{1 + 2^{j+2}c\log 9}\right)^2\bigg)=\beta +t_*.
 $$
Thus, for some $t>t_*$,  relation \eqref{e:norm_of_EW_bdd} (see Lemma \ref{l:Dzq_bound}) implies that $C>\beta +t \geq (a2^j)^{-2d}m(a2^j)+t$ for all large $n$.  Also, for any fixed $t>t_*$,  $L(t)<-2\delta<0$ for some sufficiently small $\delta = \delta(t)>0$, giving $b(n)<-\delta$ for all large $n$. Therefore, by \eqref{e:W_Z_upr_bnd_finite_samp} and \eqref{e:def_b(n)},
$$
\bbP\Big((a2^j)^{-2d}\lambda_{p}\big( {\mathbf W}_{Z}(a(n)2^j) \big) > C \Big) \leq \exp\left\{\frac{n}{a(n)} \hspace{0.5mm}b(n)\right\}\leq\exp\left\{-\frac{n}{a(n)} \hspace{0.5mm}\delta\right\}  \to 0,
$$
as $n \rightarrow \infty$. Consequently, $\lambda_{p}\big( {\mathbf W}_{Z}(a(n)2^j) \big)$ is bounded above by any such constant $C$ with probability going to 1, as claimed. $\Box$\\
\end{proof}

The following lemma is used in the proof of Lemma \ref{l:assumptions_WZ=OP(1)_hold} by means of Lemma \ref{l:W_Z_above<exp}. In the lemma, we establish some properties of the wavelet coefficients of colored noise. To state the lemma, consider the expressions \eqref{e:WZ(a(n)2^j)} and \eqref{e:def_Vn_(u)}. Now let $\mathbf u_0$ be a unit vector such that $\mathbf u_0^*  \bbE {\mathbf W}_{Z}(a(n)2^j)\mathbf u_0 = \| \bbE {\mathbf W}_{Z}(a(n)2^j)\|$. Then, we can see that
\begin{equation}\label{e:||EW_Z||_appears_on_the_diag}
\| \bbE {\mathbf W}_{Z}(a(n)2^j)\| \textnormal{ appears along the diagonal of }\mathbf \Gamma_n(\mathbf u_0),
\end{equation}
since, for any $k = 1,\hdots,n_{a,j}$,
$$
\bbE D_Z(a2^j,k)^* \mathbf u_0 {\mathbf u}^*_0 D_Z(a2^j,k) = \mathbf u_0^* \bbE D_Z(a2^j,k)D_Z(a2^j,k)^*{\mathbf u}_0.
$$
\begin{lemma}\label{l:Dzq_bound}
Suppose the assumptions of Lemma \ref{l:assumptions_WZ=OP(1)_hold} hold. For a fixed $\mathbf u\in \mathcal S^{p-1}$, let $\mathbf \Gamma_n(\mathbf u)$ be as in \eqref{e:Gamma-n(u)}. Then, for some  $C_0>0$ that is independent of $p$, $n$ and $j$,
\begin{equation}\label{e:max_eig_cov_univariate_noise=<C}
\big(a(n)2^j\big)^{-2d}\sup_{\|\mathbf u\|=1}\| \mathbf \Gamma_n(\mathbf u)\|\leq C_0. %
\end{equation}
In particular,
\begin{equation}\label{e:norm_of_EW_bdd}
\big(a(n)2^j \big)^{-2d} \| \bbE {\mathbf W}_{Z}(a(n)2^j)\| \leq C_0.
\end{equation}
\end{lemma}
\begin{proof}
Recall that any diagonal entry of $\Gamma_n(\mathbf u)$ is bounded above by the spectral norm $\| \Gamma_n(\mathbf u)\|$. Thus, in view of \eqref{e:||EW_Z||_appears_on_the_diag}, it suffices to establish the statement \eqref{e:max_eig_cov_univariate_noise=<C}. For simplicity we consider $2^j= 1$ and write $n_{a,j}=n_a$; the full statement can be then obtained by simply replacing $a(n)$ with $a(n)2^j$.

First observe, by \eqref{e:EDZDZ*}, the $(k,k')$--th entry of $\boldsymbol \Gamma_{n}(\mathbf u)$ is given as
$$
\mathbf u ^*\bbE D_Z(a,k)D_Z(a,k')^*\mathbf u = \int^{\pi}_{-\pi} e^{\imag x (a (k - k'))} |{{\mathcal H}_{a}(x)}|^2
\big(\mathbf u^*\mathfrak{g}_p(x)\mathbf u \big)dx
$$
$$
= \frac{1}{a}\int^{\pi a}_{-\pi a} e^{\imag y(k - k')} |{{\mathcal H}_{a}(y/a)}|^2
\big(\mathbf u^*\mathfrak{g}_p(y/a)\mathbf u \big)dy,
$$
where we used the change-of-variable $x=y/a$ on the second line. Also, recall from relation (78) of Moulines et al.\ \cite{moulines:roueff:taqqu:2007:JTSA},
for some $C_1>0$,
$$
\big|{\mathcal H}_a(x)\big| \leq C_1a^{1/2}|a x|^{N_\psi}(1+a|x|)^{-\alpha-N_\psi},\quad x \in (-\pi,\pi),
$$
where $\alpha$ is given by \eqref{e:psihat_is_slower_than_a_power_function}. Thus, for all $|y/a| \leq \pi$, we have
$$
|{\mathcal H}_a(y/a)|^2 \big(\mathbf u^*  \mathfrak{g}_p(y/a)\mathbf u \big) \leq |{\mathcal H}_a(y/a)|^2
\hspace{0.5mm} \|\mathfrak{g}_p(y/a)\|
$$
$$
\leq   C_2 a| y|^{2N_\psi}(1 + |y|)^{-2\alpha-2N_\psi} |y/a|^{-2d} \big(|y/a|^{2d}\|\mathfrak{g}_p(y/a)\| \big)%
$$
\begin{equation}\label{e:bound_H_a*quadform}
\leq C_3 a^{(2 d+1)} |y|^{2N_\psi-2d}(1 + |y|)^{-2\alpha-2N_\psi}.%
\end{equation}
In \eqref{e:bound_H_a*quadform}, $C_3>0$ is independent of $p$ and the last inequality follows from \eqref{e:M(d)_1}. This gives, for any $\mathbf v =(v_1,\ldots, v_{n_{a}})\in \bbR^{n_{a}}$ with $\|\mathbf v\|=1$, and using again the change-of-variables $x=y/a$,
$$
0\leq \mathbf v^* \boldsymbol \Gamma_{n}(\mathbf u)\mathbf v =\int^{\pi}_{-\pi} \sum_{\ell=1}^{n_{a}}\sum_{\ell'=1}^{n_{a}}  v_\ell v_{\ell'}e^{\imag x a (\ell -\ell')} |{\mathcal H}_a(x)|^2
 \big(\mathbf u^*\mathfrak{g}_p(x)\mathbf u \big)dx
 $$
 $$
=\frac{1}{a}\int^{\pi a}_{-\pi a} \Big|\sum_{\ell=1}^{n_{a}}  v_\ell e^{\imag y \ell} \Big|^2 |{\mathcal H}_a(y/a)|^2
\hspace{0.5mm} \big(\mathbf u^*\mathfrak{g}_p(y/a)\mathbf u \big) \hspace{0.5mm}dy
$$
$$
\leq C_3 a^{2d} \int^{\pi a}_{-\pi a} \Big|\sum_{\ell=1}^{n_{a}}  v_\ell e^{\imag y \ell} \Big|^2 \frac{| y|^{2N_\psi-2d}}{(1 + |y|)^{2\alpha+2N_\psi}} dy
$$
\begin{equation}\label{e:h(r)_fourier_prebound}
=C_3a^{2d}  \sum_{\ell=1}^{n_{a}}\sum_{\ell'=1}^{n_{a}} v_\ell v_{\ell'}\int^{\pi a}_{-\pi a} e^{\imag y (\ell - \ell')} \frac{| y|^{2N_\psi-2d}}{(1 + |y|)^{2\alpha+2N_\psi}}dy.
\end{equation}

Starting from the argument of the integral in \eqref{e:h(r)_fourier_prebound}, let $g(y)=| y|^{2N_\psi-2d}(1 + |y|)^{-2(\alpha+N_\psi)}$. Since $2(N_\psi-d)>1$,  then $g''$ is  integrable. Hence, its Fourier transform satisfies $|\widehat g(\ell)|=o(\ell^{-2})$ as $|\ell|\to\infty$. Therefore, $\sum_{\ell\in\bbZ}|\widehat g(\ell)|<\infty$. Furthermore, consider the truncated function $g_a(y)= g(y){\mathbf 1}_{\{|y|\leq \pi a\}}$ as well as its Fourier transform $\widehat g_a$. Then, we can bound
$$
|\widehat g_a(\ell)-\widehat g(\ell)| \leq 2 \int_{\pi a}^\infty g(y) dy \leq C_4 \frac{1}{a^{(2(\alpha+d)-1)}}.
$$
Thus, after dividing through by $a^{2d}$, \eqref{e:h(r)_fourier_prebound} is bounded, in absolute value, by
$$
C_3 \bigg|\sum_{\ell=1}^{n_{a}}\sum_{\ell'=1}^{n_{a}} v_\ell v_{\ell'}\widehat g_a(\ell-\ell') \bigg|= C_3\bigg|\widehat g_a(0) + \sum_{m=1}^{n_{a}-1}  2\Re(\widehat g_a(m)) \sum_{\ell=1}^{n_a-m} v_{\ell}v_{\ell+m} \bigg|
$$
$$
\leq C_5\bigg(|\widehat g_a(0)| + 2\sum_{m=1}^{n_{a}-1}|\widehat g_a(m)| \bigg)\leq C_6\bigg( \sum_{m=0}^{n_{a}-1}|\widehat g_a(m)| \bigg)
\leq C_6 \bigg(  \sum_{m=0}^{n_{a}-1}|\widehat g_a(m)-\widehat g(m)|+ \sum_{m=0}^{n_{a}-1}|\widehat g(m)|\bigg)
$$
\begin{equation}\label{e:bound_double-sum_g-hat}
\leq C_7 \bigg( \frac{n}{a^{2(\alpha+d)}}+\sum_{m=0}^\infty|\widehat g(m)|\bigg).%
\end{equation}
In \eqref{e:bound_double-sum_g-hat}, the first equality and inequality follow from $\sum_{\ell=1}^{n_a} v_\ell^2 =1$ and the Cauchy-Schwarz inequality, respectively. Since $d\geq 0$ (cf.\ \eqref{e:M(d)_1}), using $n a^{-2(\alpha+d)} \leq n a^{-2\alpha} =O(1)$, we obtain \eqref{e:max_eig_cov_univariate_noise=<C}.  $\Box$\\%
\end{proof}

The following lemma is used in the proof of Lemma \ref{l:W_Z_above<exp}. It provides a concentration inequality for centered quadratic forms. It corresponds to Lemma 1 in Laurent and Massart \cite{laurent:massart:2000} (see also Birg\'{e} and Massart \cite{birge:massart:1998}, Lemma 8, and Boucheron et al.\ \cite{boucheron:lugosi:massart:2013}, p.\ 39).
\begin{lemma}\label{l:lemma_laurent_massart}
(Laurent and Massart \cite{laurent:massart:2000}) Let $Z_1,\hdots,Z_n \stackrel{\textnormal{i.i.d.}}\sim {\mathcal N}(0,1)$ and $\eta_1,\hdots,\eta_n \geq 0$, not all zero. Let $\|{\boldsymbol \eta}\|_2$ and $\|{\boldsymbol \eta}\|_{\infty}$ be the Euclidean square and sup norms of the vector ${\boldsymbol \eta} = (\eta_1,\hdots,\eta_n)^*$. Also, define the random variable
$X = \sum^{r}_{i=1} \eta_{i} (Z^2_i - 1)$.
Then, for every $x > 0$,
\begin{equation}\label{e:laurent_massart_bound1}
\bbP\Big(X \geq 2 \|{\boldsymbol \eta}\|_2 \hspace{1mm}\sqrt{x}+ 2 \|{\boldsymbol \eta}\|_{\infty} \hspace{1mm}x \Big) \leq \exp\{-x\},
\end{equation}
\begin{equation}\label{e:laurent_massart_bound2}
\bbP\Big(X \leq - 2 \|{\boldsymbol \eta}\|_2 \hspace{1mm}\sqrt{x} \Big) \leq \exp\{-x\}.
\end{equation}
\end{lemma}

\subsection{Section \ref{s:non-Gaussian}: auxiliary results for the case of non-Gaussian $X$}

The following proposition is mentioned in Section \ref{s:non-Gaussian}. In the statement of the proposition, we make use of the following univariate construct. Suppose the univariate stochastic process ${\mathcal X}_{\tilde d,k} = \{{\mathcal X}_{\tilde d,k}(t)\}_{t \in \bbZ}$ has memory parameter $\tilde d >0$ (Roueff and Taqqu \cite{roueff:taqqu:2009}, Definition 1). In other words, for any integer $k > \tilde d - 1/2$, the $k$-th order difference process $\Delta^k {\mathcal X}_{\tilde d,k}$ is weakly stationary with spectral density
\begin{equation}\label{e:f_DeltaX}
f_{\Delta {\mathcal X}_{\tilde d,k}}(x) = |1-e^{-\imag x}|^{2(k-\tilde d)}f_*(x), \quad x \in [-\pi,\pi)
\end{equation}
(cf.\ Definition \ref{def:fractional_process}). In \eqref{e:f_DeltaX}, $f_*(x) \geq 0$ is a symmetric function that is continuous and nonzero at the origin. Moreover, we assume that, for some
\begin{equation}\label{e:beta_in_[varpi,2]}
\beta\in (0,2],
\end{equation}
the function $f_*$ satisfies
$$
|f_*(x)-f_*(0)| = O(|x|^{\beta}),\quad x\to 0.
$$
In addition, suppose $\Delta^k {\mathcal X}_{\tilde d,k}(t)$ is a linear process of the form
\begin{equation}\label{e:linear_process}
\Delta^k {\mathcal X}_{\tilde d,k}(t) = \sum_{\ell \in \bbZ} a_{\tilde d}(t - \ell) \xi_\ell, \quad \sum_{\ell \in \bbZ} a^2_{\tilde d}(\ell) < \infty.
\end{equation}
In \eqref{e:linear_process}, $\{\xi_\ell\}_{\ell \in \bbZ}$ is a sequence of i.i.d.\ random variables with mean zero, unit variance, and finite fourth moment.
\begin{proposition}\label{p:non-Gaussian_linear}
Let $0 < \tilde d_1 \hdots \leq  \tilde d_r$. Fix $k_\ell \in \bbN \cup \{0\}$, $\ell = 1,\hdots,r$, and $\beta$ satisfying \eqref{e:beta_in_[varpi,2]}. Further fix pairwise distinct $j_1,\hdots,j_m \in \bbN \cup \{0\}$. Suppose the following conditions are in place.
\begin{itemize}
\item [$(i)$] $X = \{X(t)\}_{t \in \bbZ} = \{(X_1(t),\hdots,X_r(t) )^*\}_{t \in \bbZ}$ is a $r$-variate stochastic process, where the entry-wise processes are independent and $X_{\ell} \stackrel{\textnormal{f.d.d.}}= {\mathcal X}_{\tilde d_{\ell},k_\ell}$, $\ell = 1,\hdots,r$;
\item [$(ii)$] conditions $(W1-W3)$ hold with $(1+\beta)/2 - \alpha < \tilde d_\ell \leq N_{\psi}$ and $k_\ell \leq N_{\psi}$, $\ell = 1,\hdots,r$;
\item [$(iii)$] $\{a(n)\}_{n \in \bbN}$ is a dyadic sequence such that $n/a(n) \rightarrow \infty$ and $(n/a(n))^{1/2}a(n)^{1-2\alpha - 2 \tilde d_1} \rightarrow 0$, as $n \rightarrow \infty$.
\end{itemize}
Then, with ${\mathbf H}=\textnormal{diag}(\tilde d_1-\frac{1}{2},\ldots,\tilde d_r-\frac{1}{2})$ the weak limit \eqref{e:sqrt(K)(B^-B)->N(0,sigma^2)} holds for some diagonal matrix $\Sigma_{B}(j_1,\hdots,j_m)$. Moreover,
\begin{equation}\label{e:|bfB_a(2^j)-B(2^j)|=O(shrinking)_3}
\|{{\mathbf B}}_a(2^j) - \mathbf B(2^j) \| = O(a(n)^{-\beta}), \quad  n\to\infty, \quad j=j_1,\hdots,j_m.
\end{equation}
In particular, conditions \eqref{e:sqrt(K)(B^-B)->N(0,sigma^2)} and \eqref{e:|bfB_a(2^j)-B(2^j)|=O(shrinking)} in assumption (A3) hold.
\end{proposition}
\begin{proof}
The limit \eqref{e:sqrt(K)(B^-B)->N(0,sigma^2)} is an immediate consequence of Theorem 2 in Roueff and Taqqu \cite{roueff:taqqu:2009}. For the limit \eqref{e:|bfB_a(2^j)-B(2^j)|=O(shrinking)_3}, observe that
$$
\E \W_X(2^j) = \textnormal{diag} \big( \bbE D_X(2^j,0)_1^2,\ldots, \bbE D_X(2^j,0)_r^2 \big).
$$
Moreover, in view of condition \eqref{e:beta_in_[varpi,2]}, under $(i)$, $(ii)$ and $(iii)$, expression (67) in Roueff and Taqqu \cite{roueff:taqqu:2009}  shows that, for some constants $K_\ell(j)$,
$$
\big|a(n)^{-2\tilde d_\ell}\bbE D_X(a(n) 2^j,0)_\ell^2- K_\ell(j)\big| = O( a(n)^{-\beta}), \qquad  \ell = 1,\ldots,r.
$$
Therefore, with ${\mathbf H}=\textnormal{diag}(\tilde d_1-\frac{1}{2},\ldots,\tilde d_r-\frac{1}{2})$,
$$
\|{{\mathbf B}}_a(2^j) - \mathbf B(2^j) \|
$$
$$
= \big\|\textnormal{diag} \big(  a(n)^{-2\tilde d_1}\bbE D_X(a(n) 2^j,0)_1^2- K_1(j) ,
$$
$$
\ldots, a(n)^{-2\tilde d_r}\bbE D_X(a(n) 2^j,0)_r^2- K_r(j)\big)\big\| = O(a(n)^{-\beta}),
$$
which shows \eqref{e:|bfB_a(2^j)-B(2^j)|=O(shrinking)_3}. $\Box$\\
\end{proof}

\bibliographystyle{agsm}

\bibliography{highdim}

\end{document}